\documentclass{amsart}

\usepackage{amssymb}
\usepackage{amsthm}
\usepackage{amsmath}
\usepackage{enumitem}
\usepackage[mathscr]{euscript}
\usepackage{mathtools}
\usepackage{stmaryrd}

\usepackage{cancel}
\usepackage{soul}
\usepackage{tikz}
\usepackage{tikz-cd}

\usepackage[colorinlistoftodos]{todonotes}
\usepackage{comment}
\usepackage[normalem]{ulem}
\usepackage[unicode]{hyperref}
\hypersetup{
    colorlinks,
    linkcolor=blue 
}
\usepackage{cleveref}

\theoremstyle{plain}

\newtheorem{ThmIntro}{Theorem}
\newtheorem{CoroIntro}[ThmIntro]{Corollary}

\newtheorem{proposition}{Proposition}[section]
\newtheorem{theorem}[proposition]{Theorem}
\newtheorem{lemma}[proposition]{Lemma}
\newtheorem{corollary}[proposition]{Corollary}
\theoremstyle{definition}
\newtheorem{example}[proposition]{Example}
\newtheorem{definition}[proposition]{Definition}

\newtheorem{remark}[proposition]{Remark}

\newtheorem{conjecture}[proposition]{Conjecture}

\newtheorem{fact}[proposition]{Fact}

\DeclareMathOperator{\Aut}{Aut}

\DeclareMathOperator{\Id}{Id}

\DeclareMathOperator{\SL}{\mathsf{SL}}
\DeclareMathOperator{\Sp}{Sp}
\DeclareMathOperator{\GL}{\mathsf{GL}}
\DeclareMathOperator{\SO}{\mathsf{SO}}

\DeclareMathOperator{\PSL}{\mathsf{PSL}}
\DeclareMathOperator{\PGL}{\mathsf{PGL}}
\DeclareMathOperator{\PU}{\mathsf{PU}}
\DeclareMathOperator{\Hom}{Hom}

\DeclareMathOperator{\Isom}{Isom}
\DeclareMathOperator{\Inn}{Inn}

\DeclareMathOperator{\Span}{Span}
\DeclareMathOperator{\Ad}{Ad}

\DeclareMathOperator{\rsp}{RSp}
\DeclareMathOperator{\ulim}{ulim}
\DeclareMathOperator{\RSp}{RSp}
\DeclareMathOperator{\Pos}{Pos}

\DeclareMathOperator{\Diam}{\lozenge}
\newcommand{\bdext}{{\Vc^{\mathrm{M}}_\rho}}

\DeclareMathOperator{\N}{{\mathbb{N}}}

\DeclareMathOperator{\R}{\mathbb{R}}

\DeclareMathOperator{\Cc}{\mathcal{C}}
\DeclareMathOperator{\Fc}{\mathcal{F}}

\DeclareMathOperator{\Ic}{\mathcal{I}}
\DeclareMathOperator{\Lc}{\mathcal{L}}

\DeclareMathOperator{\Oc}{\mathcal{O}}

\DeclareMathOperator{\Rc}{\mathcal{R}}

\DeclareMathOperator{\Tc}{\mathcal{T}}

\DeclareMathOperator{\Vc}{\mathcal{V}}

\DeclareMathOperator{\Cb}{\mathbb{C}}

\DeclareMathOperator{\Fb}{{\mathbb{F}}}

\DeclareMathOperator{\Hb}{\mathbb{H}}

\DeclareMathOperator{\Kb}{\mathbb{K}}
\DeclareMathOperator{\Lb}{\mathbb{L}}
\DeclareMathOperator{\Nb}{{\mathbb{N}}}

\DeclareMathOperator{\Rb}{{\mathbb{R}}}
\DeclareMathOperator{\Sb}{\mathbb{S}}

\DeclareMathOperator{\Zb}{\mathbb{Z}}
\DeclareMathOperator{\Qb}{\mathbb{Q}}

\DeclareMathOperator{\Rbom}{\mathbb{R}^\lambda_\omega}

\newcommand{\abs}[1]{\left|#1\right|}

\newcommand{\PG}{\textit{PG}}

\newcommand*\from{\colon}
\newcommand{\pr}{\textnormal{pr}}
\newcommand{\opp}{\textnormal{opp}}
\newcommand{\std}{\textnormal{std}}
\newcommand{\cl}{\textnormal{cl}}
\newcommand{\fr}{\textnormal{fr}}
\newcommand{\Tr}{\textnormal{Tr}}
\newcommand{\Stab}{\textnormal{Stab}}

\newcommand{\CL}{\textnormal{Col}}
\newcommand{\Prox}{\textnormal{Prox}}
\newcommand{\red}{\textnormal{red}}

\newcommand{\Qbar}{\overline{\Qb}^r}

\newcommand{\1}{\mathbf{1}}

\begin{document}

\def\subjclassname{\textup{2020} Mathematics Subject Classification}

\subjclass{
22E40, 
30F60, 
14P10. 
}


\title[Positive representations over real closed fields]{Positive representations over real closed fields}

\author{Xenia Flamm}
\address{Institut des Hautes \'Etudes Scientifiques\newline
\indent Max-Planck Institute for Mathematics in the Sciences}
\email{flamm@ihes.fr, xenia.flamm@mis.mpg.de}

\author{Nicolas Tholozan}
\address{École Normale Supérieure PSL, CNRS}
\email{nicolas.tholozan@ens.fr}

\author{Tianqi Wang}
\address{Yale University}
\email{tq.wang@yale.edu}

\author{Tengren Zhang}
\address{National University of Singapore}
\email{matzt@nus.edu.sg}

\date{\today}

\begin{abstract}
We develop the theory of $\Theta$-positive representations from  general Fuchsian groups to linear groups over real closed fields. Our definition, which does not assume the boundary map to be continuous, encompasses many generalizations of positive or Anosov representations that have been considered in the literature. 
\end{abstract}

\maketitle

\setcounter{tocdepth}{2}

\section{Introduction}
One of the striking recent developments of the field sometimes called \emph{higher Teichm\"uller--Thurston theory} is the theory of $\Theta$-positive representations of surface groups, introduced by Guichard--Wienhard \cite{guichard2025generalizing} and subsequently developed by Beyrer--Pozzetti \cite{beyrer2021positive} and Guichard--Labourie--Wienhard \cite{GLW}, among others.

Here, $\Theta$ refers to a subset of simple roots of a given real semisimple, algebraically connected, linear algebraic group~$G$.
Let $P_\Theta$ denote the associated parabolic subgroup of $G$. Roughly speaking, the group $G$ is said to admit a $\Theta$-positive structure if one can define a natural notion of ``cyclically ordered'' tuple (called \emph{positive tuple}) of ``hyper-transverse'' flags in the flag variety $\Fc_\Theta \coloneqq G/P_\Theta$. Such $\Theta$-positive structures have been classified by Guichard--Wienhard. This notion generalizes Lusztig's positivity for tuples of complete flags in $\Rb^d$ \cite{Lusztig_TotalPositivityReductiveGroups}.

Now, let $\Gamma$ be the fundamental group of a closed hyperbolic surface. Its Gromov boundary $\partial_\infty \Gamma$ identifies with the unit circle $\Sb^1 = \partial_\infty \Hb$, where $\Hb \subset \Cb$ denotes the Poincaré disc. A representation $\rho\colon \Gamma \to G$ is \emph{$\Theta$-positive} when there exists a continuous $\rho$-equivariant \emph{boundary map} $\xi_\rho\colon \partial_\infty \Gamma \to \Fc_\Theta$ that sends any cyclically ordered tuple to a positive tuple.

These $\Theta$-positive representations include the Hitchin representations studied by Labourie \cite{Lab} and Fock--Goncharov \cite{fock2006moduli}, the maximal representations introduced by Burger--Iozzi--Wienhard \cite{BIW03,BIW10}, as well as a new family of representations into $\SO(p,q)$ and representations into two new exceptional Lie groups. These representations are \emph{$\Theta$-Anosov} in the sense of Labourie \cite{Lab}. In particular, they are discrete and faithful. These recent developments culminated in the following theorem:

\begin{theorem}[\cite{GLW, BeyrerGuichardLabouriePozzettiWienhard_PositivityCrossRatiosCollarLemma}]
    Let $\Gamma$ be a closed surface group and $G$ a real semisimple linear algebraic group that admits a $\Theta$-positive structure. 
    Then the set of $\Theta$-positive representations of $\Gamma$ to $G$ forms a union of connected components of $\Hom(\Gamma,G)$.
\end{theorem}

In fact, in light of the independent work of Aparicio-Aroyo--Bradlow--Collier--Garcia-Prada--Gothen--Oliveira \cite{aparicio2019so}, the $\Theta$-positive condition is likely to characterize exactly the connected components of real representation varieties of closed surface groups containing only discrete and faithful representations.

\subsection{%
\texorpdfstring{%
A general notion of $\Theta$-positive representation}%
{A general notion of Theta-positive representation}}
\label{subsection:Intro:GeneralNotionThetaPosRepr}
The purpose of the present work is to extend the notion of $\Theta$-positive representation simultaneously in two directions:
\begin{itemize}
    \item For representations of fundamental groups of hyperbolic surfaces which are not necessarily closed anymore;
    \item For representations into linear semi-algebraic groups over real closed fields (see \Cref{section:SemiAlgGroups}) that are semi-simple  and semi-algebraically connected.
\end{itemize}

Throughout the paper, $\Gamma$ will denote a non-elementary \emph{Fuchsian group}, i.e.\ a discrete group of biholomorphisms of the Poincar\'e disc $\Hb$ which is not virtually abelian. Its limit set $\Lambda(\Gamma)$ is the accumulation set of any $\Gamma$-orbit in $\Sb^1 \simeq \partial_\infty \Hb$. The group $\Gamma$ is called \emph{of the first kind} when $\Lambda(\Gamma) = \Sb^1$ and \emph{of the second kind} otherwise.

Let $G$ be a real, algebraically connected, semisimple, linear algebraic group $G$ (in the main body of the paper, our results are stated more generally for linear semi-algebraic groups over the real algebraic numbers, see \Cref{section:SemiAlgGroups}). 
Suppose that $\Theta$ is a subset of the set of simple restricted roots of $G$ for which $G$ admits a $\Theta$-positive structure, see \Cref{subsection: PositiveStructures}. 
Let $P_\Theta$ be the associated standard parabolic subgroup and $\Fc_\Theta \coloneqq G/P_\Theta$ the associated flag variety.

Consider now any ordered field extension $\Fb$ of $\Rb$ and let $G_{\Fb}$, $P_{\Theta,\Fb}$ and $\Fc_{\Theta,\Fb}= G_{\Fb}/P_{\Theta,\Fb}$ denote the $\Fb$-points of $G$, $P_\Theta$ and $\Fc_\Theta$ respectively.
As pointed out by Guichard--Wienhard \cite{guichard2025generalizing}, the $\Theta$-positivity of tuples of flags is a real semi-algebraic condition, and thus does make sense over $\Fc_{\Theta, \Fb}$.
We will investigate the following definition, which already appears in \cite{GLW}:

\begin{definition} \label{def: positive rep intro}
    A representation $\rho \from \Gamma \to G_{\Fb}$ is \emph{$\Theta$-positive} if there exists a $\Gamma$-invariant non-empty subset $D\subset \Lambda(\Gamma)$ and a $\rho$-equivariant map $\xi \from D \to \Fc_{\Theta,\Fb}$ mapping cyclically ordered tuples in $D$ to positive tuples.
\end{definition}

A key feature of this definition is that we do not assume anything about the domain of definition $D$ of the boundary map $\xi$. Our first concern will be about extending this boundary map to the whole circle. 

In \Cref{s: Fuchsian groups}, we recall that any Fuchsian group $\Gamma$ acting on $\Sb^1$ is semiconjugated to a Fuchsian group $\Gamma'$ of the first kind. We will see (\Cref{propo: equivalence positivity first and second kind Fuchsian groups}) that the $\Theta$-positive representations of $\Gamma$ and $\Gamma'$ are the same, which often allows us to assume that $\Gamma$ is of the first kind.

Since a $\rho$-equivariant map must send a point $x$ to a point fixed by $\rho(\Stab_\Gamma(x))$, let us investigate first the points with non-trivial stabilizer. Recall that these are of two kind: \emph{parabolic fixed points} and \emph{hyperbolic fixed points}. If $\gamma$ is an infinite order element of $\Gamma$, we denote by $\gamma^+ \in \Sb^1$ its attracting fixed point if $\gamma$ is hyperbolic and its unique fixed point if $\gamma$ is parabolic. The sets
\[\Lambda_h \coloneqq \{\gamma^+ \mid \gamma\in \Gamma \textrm{ hyperbolic}\} \textnormal{ and } \Lambda_p \coloneqq \{\gamma^+\mid \gamma\in \Gamma \textrm{ parabolic}\}\]
are $\Gamma$-invariant subsets of $\Sb^1$, and $\Lambda_h$ is always non-empty (as long as $\Gamma$ is non-elementary).\\

In \Cref{section : proximal representations}, we prove that $\Theta$-positivity implies a weak form of proximality over real closed fields, and deduce that positive representations have a boundary map which is canonically defined on $\Lambda_h$.

\begin{ThmIntro} \label{thm-intro: weak proximality}
Let $\Gamma$ be a Fuchsian group of the first kind and $\rho$ a representation of $\Gamma$ into $G_{\Fb}$, where $\Fb$ is a real closed field. Then the following are equivalent:
\begin{enumerate}
    \item
    \label{thm-intro: weak proximality: item positive}
    $\rho$ is $\Theta$-positive;
    \item 
    \label{thm-intro: weak proximality: item positive on fixed points}
    for every hyperbolic element $\gamma \in \Gamma$, $\rho(\gamma)$ is \emph{weakly $\Theta$-proximal}, and the map $\xi^h_\rho \from \Lambda_h \to \Fc_{\Theta,\Fb}$ sending $\gamma^+$ to the weakly attracting flag of $\rho(\gamma)$ sends cyclically ordered tuples in $\Lambda_h$ to positive tuples in $\Fc_{\Theta,\Fb}$.
\end{enumerate}
\end{ThmIntro}

The notions of weakly $\Theta$-proximal element and weakly attracting flag will be defined in \Cref{sec: proximality}.
Here let us explain them in the particular case where $G_{\Fb}= \SL_d(\Fb)$ and $\Theta$ is the set of all simple roots. Then $g \in \SL_d(\Fb)$ will be called \emph{weakly $\Theta$-proximal} if it is diagonalizable over $\Fb$ with distinct eigenvalues satisfying $\vert \lambda_1 \vert_{\Fb} > \ldots > \vert \lambda_d\vert_{\Fb} $ (here and elsewhere, $\vert x\vert_{\Fb} = \max\{x,-x\}\in \Fb$). The flag obtained by taking sums of eigenspaces in decreasing order is then called the \emph{weakly attracting flag} of $g$. It is indeed a fixed point of $g$ in $\Fc_{\Theta, \Fb}$. However, it is an actual attracting point only when the ratios between consecutive eigenvalues are \emph{big elements} in $\Fb$, meaning that for any $x\in \Fb$, there exists $n\in \N$ such that
\[x \leq \left(\frac{\lambda_i}{\lambda_{i+1}}\right)^n~.\]
In that case, we call $g$ \emph{strongly $\Theta$-proximal}. The two notions are equivalent if and only if $\Fb$ is \emph{Archimedean} (i.e.\ a subfield of $\Rb$).

One of the perhaps most surprising features of the non-Archimedean case is that positive boundary maps may not extend to parabolic fixed points. In \Cref{subsection: ExNonFramablePosRepr}, we describe positive representations into $\PGL_2(\Fb)$ for which the image of parabolic elements do not fix a point in $\mathbf P^1(\Fb)$. Burger--Iozzi--Parreau--Pozzetti reached the same conclusion in \cite{BIPP_RSCMaximalRepr}. 

In order to extend the boundary map everywhere, we thus have no other choice but to impose that it is a priori defined on $\Lambda_p$. Following Fock--Goncharov's terminology, we call \emph{$\Theta$-framing} a $\rho$-equivariant map from $\Lambda_p$ to $\Fc_{\Theta,\Fb}$ and define:

\begin{definition}
   Let $\Gamma$ be a Fuchsian group of the first kind containing parabolic elements (equivalently, such that $\Lambda_p$ is non-empty). A representation $\rho \from \Gamma \to G_{\Fb}$ is \emph{$\Theta$-positively frameable} if there exists a $\rho$-equivariant map $\xi \from \Lambda_p \to \Fc_{\Theta, \Fb}$ mapping cyclically ordered tuples in $\Lambda_p$ to positive tuples.
\end{definition}

Finally, in order to extend positive boundary maps to a point $x\in \Sb^1$ with trivial stabilizer, the only condition is that the image of $x$ should belong to a countable intersection of ``generalized intervals'' called \emph{diamonds} (see \Cref{s: Diamonds}). The key condition here will be that the field $\Fb$ is \emph{Cantor complete} (see \Cref{dfn: Cantor complete}), which ensures that the intersection of a nested sequence of closed non-empty bounded intervals is non-empty. Note that every ordered field embeds in a Cantor complete real closed field \cite{Shelah_QuiteCompleteRealClosedFields}. We can now state our strongest extension theorem: 
\begin{ThmIntro} \label{thm-intro: extension boundary map}
    Let $\Gamma$ be a Fuchsian group of the first kind, $\Fb$ a Cantor complete real closed field and $\rho$ a representation of $\Gamma$ into $G_{\Fb}$. Assume that either $\Lambda_p = \emptyset$ and $\rho$ is $\Theta$-positive, or $\Lambda_p\neq \emptyset$ and $\rho$ is $\Theta$-positively frameable. Then there exists a $\rho$-equivariant positive map $\xi \colon \Sb^1\to \Fc_{\Theta, \Fb}$ (defined on the whole circle).
\end{ThmIntro}

The converse is immediate by restriction of the boundary map. If $\Lambda_p$ is non-empty but $\rho$ is only assumed to be $\Theta$-positive, then the map $\xi$ might only be extended to $\Sb^1 \setminus \Lambda_p$.

\subsection{Examples}
There have already been various efforts to generalize positivity beyond the case of closed surface groups and real algebraic groups. To our knowledge, \Cref{def: positive rep intro} encompasses all these generalizations. Let us review these before investigating further properties of positive representations.

\subsubsection*{Maximal representations}
In \cite{BIW03,BIW10}, Burger--Iozzi--Wienhard introduce \emph{maximal representations} of a finitely generated surface group (possibly with parabolic elements) into a real Lie group of Hermitian type $G$. They prove that such representations admit a positive equivariant boundary map into the Shilov boundary of $G$, which is a priori only measurable (in particular, only defined on a set of full measure).

\subsubsection*{Fock--Goncharov's positive framed representations}
In \cite{fock2006moduli}, Fock--Goncharov define positive representations of punctured surface groups into a real split simple Lie group $G$ as those admitting a \emph{positive framing} $\xi\colon \Lambda_p \to \Fc_\Delta$, where $\Fc_\Delta$ is the full flag variety.
Our definition of \emph{$\Theta$-positively frameable representation} is a straightforward generalization of Fock--Goncharov's positivity to all flag varieties $\Fc_\Theta$ admitting a $\Theta$-positive structure, and to all Fuchsian groups $\Gamma$ of the first kind with parabolic elements. 
This generalization was already investigated (over the reals and when $\Gamma$ is finitely generated) by Guichard--Rogozinnikov--Wienhard \cite{GRW}, where they describe the topology of conjugacy classes of $\Theta$-positively framed representations.

\subsubsection*{Hitchin representations of Fuchsian groups}
In \cite{CZZ}, the fourth author, together with Canary and Zimmer, introduced \emph{cusped Hitchin representations} of a finitely generated Fuchsian group $\Gamma$ into $\SL_d(\Rb)$. Importantly, their definition assumes the existence of a \emph{continuous}, positive, equivariant, boundary map from the limit set $\Lambda(\Gamma)$ to the variety of complete flags in $\Rb^d$. In particular, the precise geometry of the hyperbolic surface $ S=\Gamma \backslash \Hb$ constrains the behavior of the representation on boundary curves:
\begin{itemize}
    \item If $\gamma\in \Gamma$ is represented by a curve in $S$ bounding a funnel, then $\rho(\gamma)$ must be loxodromic;
    \item If $\gamma \in \Gamma$ is represented by a curve in $S$ bounding a cusp, then $\rho(\gamma)$ must be unipotent.
\end{itemize}
We immediately get that a cusped Hitchin representation is positive in the sense of \Cref{def: positive rep intro}.
Conversely, they also proved \cite[Theorem 9.2]{CZZ} that a positive representation of $\Gamma$ is a cusped Hitchin representation if and only if it is \emph{type-preserving}, which is a particular case of \Cref{thm-intro: type preserving} below. 

\subsubsection*{Hitchin representations over real closed fields}
In \cite{Flamm}, the first author studied Hitchin representations of closed surface groups into $\SL_d(\Fb)$. She defined them as those representations arising as ``limits'' of real Hitchin representations (see \Cref{section:RSC}), and characterized them as those admitting a positive boundary map defined on $\Lambda_h$.
These are thus exactly the $\Delta$-positive representations into $\SL_d(\Fb)$ in the sense of \Cref{dfn: positive representations}, where $\Delta$ is the set of all simple roots.

\subsection{Continuity of the boundary map}

An important feature of our general definition of $\Theta$-positivity is that the boundary map $\xi$ is not necessarily continuous. There are two motivations for that:

\begin{itemize}
    \item For fundamental groups of punctured surfaces, imposing the continuity of the boundary map as in \cite{CZZ} constrains the behavior of the representation and does not provide a unified framework to account for continuous transitions between, for instance, Fuchsian representations with hyperbolic and parabolic monodromy at the punctures.

    \item In the case of closed surface groups, one cannot hope to get continuous boundary maps into flag varieties over non-Archimedean ordered fields since the boundary of the closed surface group is a topological circle while the flag varieties are totally discontinuous.
\end{itemize}

In \cite{GLW}, the authors proceed to prove that, for a closed surface group, the positive boundary map extends continuously to the whole circle. We will give another proof here. More precisely, we will prove the following for real positive representations:

\begin{ThmIntro} \label{thm-intro: Continuity boundary map R}
    Let $\rho \colon \Gamma \to G$ be a $\Theta$-positive representation of a Fuchsian group of the first kind. Then there exists a unique left (resp.\ right) continuous positive $\rho$-equivariant boundary map $\xi_\rho^l$ (resp.\ $\xi_\rho^r$) from $\Sb^1$ to $\Fc_\Theta$.
    Moreover, $\xi_\rho^l$ and $\xi_\rho^r$ coincide at conical limit points of $\Gamma$. In particular, if $\Gamma$ is cocompact, the positive boundary map $\xi$ is unique and continuous, and the representation $\rho$ is $\Theta$-Anosov.
\end{ThmIntro}

\begin{remark}
      For maximal representations into Hermitian Lie groups, the left and right continuous boundary maps were already constructed in the work of Burger--Iozzi--Wienhard \cite{BIW03,BIW10}. There, they claim that the continuity of the boundary map for closed surface groups \emph{follows} from the Anosov property proven in their other paper with Labourie \cite{BILW}.
\end{remark}

More generally, if $\Gamma$ is a finitely generated Fuchsian group (not necessarily of the first kind), recall that the hyperbolic orbifold $\Gamma\backslash \Hb$ is the union of a compact core with finitely many cusps and funnels. We call a representation $\rho\colon \Gamma \to G_{\Fb}$ \emph{type-preserving} if the image of every generator of the fundamental group of a cusp is unipotent and the image of every generator of the fundamental group of a funnel is $\Theta$-proximal. The following theorem holds for real representations of finitely generated Fuchsian groups.

\begin{ThmIntro}
\label{thm-intro: type preserving}
    Let $\Gamma$ be a finitely generated Fuchsian group and $\Lambda(\Gamma)\subset \Sb^1$ its limit set. A representation $\rho \colon \Gamma \to G$ is $\Theta$-positive and type-preserving if and only if there exists a continuous $\rho$-equivariant positive map $\xi \colon \Lambda(\Gamma) \to \Fc_{\Theta}$.
\end{ThmIntro}

One of the motivations of this paper was to understand precisely why the continuity of the boundary map would be true for closed surface groups and real representations, but fail for surface groups with punctures and for representations over other ordered fields. In the course of the proof we will see precisely where the hypotheses come in: 
 \begin{itemize}
     \item A Fuchsian group is cocompact if and only if every point $x\in \Sb^1$ is a conical limit point, and the crucial property that we will use is the existence of $a\neq b\in \Sb^1$ and $(\gamma_n)\in \Gamma^{\N}$ such that $\gamma_n \cdot a$ and $\gamma_n \cdot b$ respectively converge to $x$ \emph{from the right} and \emph{from the left}.
     \item Ultimately, the existence of left and right continuous extensions will rely on the fact that a bounded increasing sequence converges, a property which is well-known to characterize $\Rb$ among ordered fields.
 \end{itemize}

\subsection{Positivity and extended geometric finiteness}
Recall that a \emph{relatively hyperbolic pair} is the data of a finitely generated group $\Gamma$ and a collection $\mathcal P$ of finitely many conjugacy classes of subgroups that we will call \emph{cusp groups}, such that there exists a proper isometric action of $\Gamma$ on a Gromov hyperbolic space $X$ for which each $P\in \mathcal P$ fixes a point $x_P \in \partial_\infty X$ and such that $\Gamma$ acts cocompactly on the complement of a $\Gamma$-invariant collection of horoballs centered at the $x_P$. The prototypical example (which is the one we are interested in here) is when $\Gamma$ is a Fuchsian group with finite covolume, $\mathcal P$ is the collection of conjugacy classes of cyclic subgroups generated by unipotent elements, and $X= \Hb$. 

Recently, several definitions have been proposed to generalize the notion of Anosov representation to relatively hyperbolic pairs. While the notion of \emph{relatively Anosov representation}, introduced independently by Kapovich--Leeb \cite{KL} and Zhu \cite{Zhu}, forces the image of the cusp groups to be quasi-unipotent, Weisman introduced in \cite{Weisman} the broader notion of \emph{extended geometrically finite (EGF) representation}, which is less constraining on the image of cusp groups and allows to study transitions between representations such that cusp groups have unipotent image and representations for which this image is diagonalisable.

However, perhaps surprisingly, Weisman did not prove that Burger--Iozzi--Wien-hard's maximal representations or Fock--Goncharov's positive representations of fundamental groups of hyperbolic surfaces with cusps are extended geometrically finite. Here we remedy this.
Weisman's EGF condition asks for the existence of a $\rho$-equivariant, continuous, transverse, surjective \emph{$\Theta$-boundary extension} $\zeta\from \Vc \subset \Fc_\Theta \to \partial_\infty X$, together with a ``dynamics preserving'' condition.
Informally, this condition requires that the dynamical behavior of $\rho(\Gamma)$ on $\Vc$ extends to some open neighborhood in $\Fc_\Theta$ (see \Cref{s: Real case}). Here we prove that every positive representation $\rho$ into $G_{\Fb}$ admits a \emph{positive $\Theta$-boundary extension}; namely, a pair of $(\Vc, \zeta)$ where $\Vc\subset \Fc_{\Theta,\Fb}$ is $\rho$-invariant, and $\zeta\colon \Vc\to \Sb^1$ is a $\rho$-equivariant continuous map under which every preimage of a cyclically ordered tuple is positive (see \Cref{def: Boundary Extension}).

\begin{ThmIntro} \label{thm: Boundary Extension Intro}
Let $\Gamma$ be a Fuchsian group of the first kind and $\rho\colon \Gamma \to G_{\Fb}$ a representation. If $\rho$ admits a positive $\Theta$-boundary extension, then $\rho$ is $\Theta$-positive.

Conversely, if $\rho$ is $\Theta$-positive, then there exists a unique positive $\Theta$-boundary extension $(\bdext, \zeta_\rho^{\mathrm{M}})$ which is \emph{maximal} in the following sense:
\begin{itemize}
    \item For every positive $\Theta$-boundary extension $(\Vc, \zeta)$, we have \[\Vc \subset \bdext \quad\text{and}\quad\zeta= {\zeta_\rho^{\mathrm{M}}}\vert_{ \Vc}~;\]
    \item For every positive boundary map $\xi \from D \subset \Sb^1 \to \Fc_{\Theta,\Fb}$, we have $\xi(D)\subset \bdext $ and $\zeta_\rho^{\mathrm{M}} \circ \xi = \Id\vert_{ D}$.
\end{itemize}
\end{ThmIntro}

Note that our definition of $\Theta$-boundary extension does not make any assumption on  the image of $\zeta$ in general. However, we will see in \Cref{sss:Surjectivity Cantor Complete} that $\zeta^{\mathrm M}_\rho$ is surjective onto $\Sb^1$ when the field $\Fb$ is Cantor complete.

As a corollary, for real representations, we obtain: 
\begin{ThmIntro} \label{thm-intro: Theta-positive EGF}
    Let $\Gamma$ be a finitely generated Fuchsian group of the first kind and $\rho \from \Gamma \to G$ be a $\Theta$-positive representation. Then $\rho$ is EGF in the sense of Weisman (relatively to parabolic subgroups) and $\Theta$-divergent.
\end{ThmIntro}

The $\Theta$-divergent condition is discussed in \Cref{subsection: divergence for sequences}.
Informally, it requires that every diverging sequence $(\gamma_n)_{n\in \Nb} \subset \Gamma$, has a subsequence $(\gamma_{k_n})$ such that $\rho(\gamma_{k_n})$ converges to a constant map on some non-empty open subset in $\Fc_{\Theta,\Rb}$.
The combination of the EGF and divergent conditions has been studied by the third author in \cite{W23b}, where he proves among other things that it admits nice dynamical characterizations in terms of dominated splittings.

\subsection{%
\texorpdfstring{%
$\Theta$-positivity is open and closed}%
{Theta-positivity is open and closed}}
When $\Gamma$ is a cocompact Fuchsian group, one of the most important results of the theory of positive representations is that the set of $\Theta$-positive representations $\Pos_\Theta(\Gamma,G)$ is both open and closed in $\Hom(\Gamma,G)$. This was proven in a series of papers by Guichard, Wienhard, Beyrer, Pozzetti and Labourie \cite{GLW, beyrer2021positive,BeyrerGuichardLabouriePozzettiWienhard_PositivityCrossRatiosCollarLemma}.

Here, we extend this result to a finite covolume Fuchsian group with parabolic elements.
Note that positivity imposes a condition on images of parabolic elements, which we call \emph{$\Theta$-positively translating} (see \Cref{dfn:PosRotPosTransl}).
This condition is neither open nor closed.
However, we prove that positive representations form connected components of \emph{relative representation varieties}: let $\rho_0 \colon \Gamma \to G_{\Fb}$ be a representation, and define
\[\Hom_{\rho_0}(\Gamma, G_{\Fb}) = \{\rho \mid \rho(\gamma) \textrm{ conjugate to $\rho_0(\gamma)$ for all $\gamma \in \Gamma$ parabolic}\}\]
(see \Cref{subs: Semialg Framed Repr} for precisions). Let us also denote by $\Pos_\Theta^{\fr}(\Gamma, G_{\Fb})$ the set of $\Theta$-positively frameable representations.

\begin{ThmIntro} \label{thm-intro: Positivity closed}
    The set $\Pos_\Theta^{\fr}(\Gamma,G_{\Fb})\cap \Hom_{\rho_0}(\Gamma, G_{\Fb})$ is open and closed in\break $\Hom_{\rho_0}(\Gamma,G_{\Fb})$.
\end{ThmIntro}

The proof of the closedness follows closely the arguments of Beyrer--Guichard--Labourie--Pozzetti--Wienhard. We rely in particular on their collar lemma (see \Cref{lem: collar lemma} and \cite[Theorem C, Corollary D]{BeyrerGuichardLabouriePozzettiWienhard_PositivityCrossRatiosCollarLemma}).

\subsection{A conditional improvement}

\Cref{thm-intro: Positivity closed} only deals with deformations of representations that preserve the conjugacy class of the image of the cusps.
More generally, we will show that the set of $\Theta$-positive representations of $\Gamma$ is closed in $\Hom(\Gamma, G_{\Fb})$ (for any non-elementary Fuchsian group $\Gamma$, see \Cref{thm: closedness}).
The openness on the other hand, is less clear. We will show that the set of \emph{$\Theta$-positively framed representations} (i.e.\ pairs of a representation and a positive framing) is open in the set of framed representations (for a suitable topology).
It is tempting to believe that the same is true after projecting to the space of frameable representations.
We prove this under some condition on the flag variety.

\begin{definition}
    Let $\Fc_{\Theta,\Fb}$ be a flag variety with a positive structure. We say that $\Fc_{\Theta,\Fb}$  satisfies the \emph{$\PG$-condition} if there exists $N>0$ such that, for every positive tuple $(x_1,\ldots ,x_N)\in \Fc_{\Theta,\Fb}$, every $y\in \Fc_{\Theta,\Fb}$ is transverse to at least one of the $x_i$.
\end{definition}

Here ``PG'' stands for ``positive is generic''.
We will see that this condition is satisfied in many cases, for the complete flag varieties of $\SL(n,\Rb)$, $\textnormal{Sp}(2k,\Rb)$, $\SO(k,k+1)$ and $G_2$ (following from work of Saldanha--Shapiro--Shapiro \cite{SaldanhaShapiroShapiro_FinitenessGrassmannConvexityRevisited, SaldanhaShapiroShapiro_FinitenessGrassmannConvexity}) and for the Shilov boundaries of of Hermitian groups of tube type (see \Cref{subs:PGConditionHermFlagVarities}).
The general case remains open and we conjecture the following.

\begin{conjecture}\label{conj:PG-condition}
    Every flag variety with a $\Theta$-positive structure satisfies the $\PG$-condition.
\end{conjecture}

Under the $\PG$-condition, we can remove the restriction on the images of parabolic elements in \Cref{thm-intro: Positivity closed}, and prove the following.

\begin{ThmIntro}
\label{thm-intro:PG-condition}
    Assume that $\Fc_{\Theta,\Fb}$ satisfies the $\PG$-condition, and let $\Gamma$ be a non-cocompact lattice. Then:
    \begin{enumerate}
        \item
        \label{thm-intro:PG-condition: posfram}
        If $\rho \colon \Gamma\to G_{\Fb}$ is $\Theta$-positive and $\Theta$-frameable, then every $\Theta$-framing of $\rho$ is positive. In particular, $\rho$ is $\Theta$-positively frameable.
        \item 
        \label{thm-intro:PG-condition: open in frameable}
        The set $\Pos_\Theta^{\fr}(\Gamma,G_{\Fb})$ is open and closed in the closed set $\Hom_\Theta^{\fr}(\Gamma,G_{\Fb})$ of $\Theta$-frameable representations.
    \end{enumerate}
\end{ThmIntro}

\subsection{%
\texorpdfstring{%
$\Theta$-positivity and the real spectrum compactification}%
{Theta-positivity and the real spectrum compactification}}
Here we assume that $\Gamma$ is a lattice, i.e.\ a finitely generated Fuchsian group of the first kind. The space $\Hom(\Gamma,G)$ thus has the structure of a real affine algebraic variety and, as such, admits a \emph{real spectrum compactification} $\RSp \Hom(\Gamma,G)$.
It is proven in \cite{BurgerIozziParreauPozzetti_RSCCharacterVarieties2} that points in $\RSp \Hom(\Gamma,G)$ correspond to representations of $\Gamma$ into $G_{\Fb}$ for $\Fb$ an ordered field containing $\R$, modulo an equivalence relation which will be made precise in \Cref{section:RSC}.

Let $\Fb$ be a real closed field containing $\Rb$.

\begin{ThmIntro} \label{thm-intro: Positivity semi-algebraic}
    The set $\Pos_\Theta(\Gamma,G)$ is semi-algebraic. Denoting by $\left(\Pos_\Theta(\Gamma, G)\right)_{\Fb}$ its $\Fb$-extension, we have:
    \[\left(\Pos_\Theta(\Gamma, G)\right)_{\Fb} = \Pos_\Theta(\Gamma, G_{\Fb})\]
    if $\Gamma$ is cocompact, and
    \[\left(\Pos_\Theta(\Gamma, G)\right)_{\Fb} = \Pos_\Theta^{\fr}(\Gamma, G_{\Fb})\]
    otherwise.
\end{ThmIntro}

Informally, this means that a finite number of real semi-algebraic conditions characterize $\Theta$-positively frameable representations over any real closed field, and implies that any real semi-algebraic condition which holds over $\Theta$-positive representations into $G_{\Rb}$ holds for $\Theta$-positively frameable representations into $G_{\Fb}$.
This was proven by the first author in \cite{Flamm} for Hitchin representations of closed surface groups.
\footnote{To be precise, she defines Hitchin representations into $\SL_d(\Fb)$ as the $\Fb$-extension of the Hitchin components, and proves that these are exactly the representations satisfying the second condition of \Cref{thm-intro: weak proximality}.}

\begin{remark}
    In contrast, when $\Gamma$ contains parabolic elements, the set $\Pos_\Theta(\Gamma,G_{\Fb})$ is \emph{not} the $\Fb$-extension of a real semi-algebraic set, as reflected by the fact that $\Theta$-positive representations are automatically frameable over $\Rb$, while they need not be over non-Archimedean fields.
\end{remark}

In order to prove the theorem, we will use a cutting and gluing lemma of independent interest, see \Cref{propo: cutting lemma}. Roughly speaking it says that the restriction of a $\Theta$-positive representation to a subsurface is still $\Theta$-positive and, conversely, if a representation $\rho$ of $\pi_1(\Sigma)$ is $\Theta$-positive in restriction to the complement $\Sigma'$ of a simple closed curve, then one only needs to check the $\Theta$-positivity of one $4$-tuple to guarantee that $\rho$ is $\Theta$-positive.\\

This semi-algebraicity combined with the closedness \Cref{thm-intro: Positivity closed} has strong consequences regarding the real spectrum compactification of $\Hom(\Gamma, G)$, as was already shown in \cite{Flamm} for Hitchin representations.
Applying the Tarski--Seidenberg principle, we obtain our last characterization of $\Theta$-positive representations:

\begin{CoroIntro}\label{coro:Positive = limit of positive}
    Let $\rho$ be a representation of $\Gamma$ into $G_{\Fb}$. Then the following are equivalent:
    \begin{enumerate}
        \item $\rho$ is $\Theta$-positive, and is frameable if $\Gamma$ is not cocompact, 
        \item $\rho$, seen as a point of $\RSp\Hom(\Gamma,G)$, belongs to the closure of $\Pos_\Theta(\Gamma, G)$.
    \end{enumerate}
\end{CoroIntro}

The direction (2)$\implies$(1) was already proven in \cite[Theorem 1.11, Theorem 6.20]{BurgerIozziParreauPozzetti_RSCCharacterVarieties2}.

\begin{CoroIntro} \label{coro: positive form connected components}
    If $\Gamma$ is a cocompact Fuchsian group, the set of (equivalence classes of) $\Theta$-positive representations of $\Gamma$ into $G_{\Fb}$ for all real closed fields $\Fb$ containing $\Rb$ forms a union of connected components of $\RSp\Hom(\Gamma, G)$ consisting only of injective representations with discrete image.
\end{CoroIntro}
For finite volume Fuchsian groups with parabolic elements, an analogous corollary holds in relative representation varieties, see \Cref{rem:connected components of positive representations non uniform lattice}.

Note that $\Theta$-positivity of a representation is invariant under conjugation. Here, it is more convenient to state these results in the representation variety $\Hom(\Gamma,G)$, but we will include in \Cref{section:RSC} the analogous theorems for the character variety $\mathfrak X(\Gamma,G) \coloneqq \Hom(\Gamma,G)/\!\!/G$. In fact, we will prove that every $\Theta$-positive representation has bounded centralizer and closed conjugation orbit in $\Hom(\Gamma,G)$, so that the quotient map $\Hom(\Gamma,G) \to \mathfrak X(\Gamma,G)$ is a naive quotient in restriction to $\Pos_\Theta(\Gamma,G)$.

\subsection{Related works}

Let us mention that the idea to investigate positivity over real closed fields is not new. Already in \cite{BurgerPozzetti}, the authors investigate non-Archimedean limits of maximal representations in order to prove a collar lemma.
The work \cite{ BurgerIozziParreauPozzetti_RSCCharacterVarieties1} introduced the real spectrum compactification of character varieties and announced several results about Hitchin and maximal representations. Some are proven in \cite{BurgerIozziParreauPozzetti_RSCCharacterVarieties2}, and the Hitchin case is treated in \cite{Flamm}.

The papers and \cite{guichard2025generalizing,GLW,BeyrerGuichardLabouriePozzettiWienhard_PositivityCrossRatiosCollarLemma} have also set the ground for the present work by pointing out that many notions and statements about $\Theta$-positivity are semi-algebraic.
We will use in particular the semi-algebraic formulation of the collar lemma from \cite{BeyrerGuichardLabouriePozzettiWienhard_PositivityCrossRatiosCollarLemma}.
Finally the work in preparation \cite{BIPP_RSCMaximalRepr} will study maximal representations over real closed fields and recover many of the present results in that case. In particular, the authors noticed the subtle difference between $\Theta$-positive and $\Theta$-positively frameable representations in the non-Archimedean world.

\subsection{Structure of the paper}

The article is organized as follows:
\begin{itemize}
    \item \Cref{s: Preliminaries} recalls some miscellaneous background: about Fuchsian groups, about ordered fields, about real algebraic geometry and about flag varieties.

    \item \Cref{s: Positivity on flag varieties} recalls the definition and main properties of $\Theta$-positive structures on flag varieties and their generalization over ordered fields, and proves several useful lemmas.
    We finish by proving the $\PG$-condition in the case of Hermitian Lie groups.
    
    \item In \Cref{s: Definitions Positive representations}, we introduce $\Theta$-positive representations and $\Theta$-positive framings. We describe elements in their images and we prove their various equivalent characterizations, namely \Cref{thm-intro: weak proximality}, \Cref{thm-intro: extension boundary map}, and \Cref{thm: Boundary Extension Intro}. We also exhibit an example of a $\Theta$-positive representation which is not frameable.
    We conclude with the proof of \Cref{thm-intro:PG-condition}~(\ref{thm-intro:PG-condition: posfram}).
   
    \item In \Cref{section:PositivityClosed}, we prove that the $\Theta$-positive condition is closed in a fairly general setting, and deduce that $\Theta$-positive representations are irreducible.

    \item In \Cref{s: Real case} we focus on the real case: we prove the continuity properties of the boundary maps in that case (\Cref{thm-intro: Continuity boundary map R}) and deduce that real $\Theta$-positive representations are EGF and divergent (\Cref{thm-intro: Theta-positive EGF}). We also discuss the relation with relatively Anosov representations and prove \Cref{thm-intro: type preserving}.

    \item In \Cref{section:Positivitysemi-algebraic}, we return to representations over general ordered fields and prove that, for finitely generated groups,  $\Theta$-positivity is characterized by a finite number of semi-algebraic conditions (\Cref{thm-intro: Positivity semi-algebraic}). In the process  we prove a cutting and gluing lemma for positive representations, and prove the openness property of positivity in relative character varieties, concluding the proof of \Cref{thm-intro: Positivity closed} and \Cref{thm-intro:PG-condition}~(\ref{thm-intro:PG-condition: open in frameable}).

    \item Finally, in \Cref{s: Real Sectrum}, we recall the definition of the real spectrum compactification of representation varieties, and explain the consequences of our results for the subset of the real spectrum compactification corresponding to $\Theta$-positive representations (Corollaries~\ref{coro:Positive = limit of positive} and~\ref{coro: positive form connected components}).

    \item We include an appendix proving some statements about Fuchsian groups, which can be interpreted as a description of $\Theta$-positive representations into $\PSL_2(\Rb)$.
\end{itemize}

\subsection*{Acknowledgments}
The first two authors are grateful to Silvain Rideau-Kikuchi for clarifying questions concerning the Tarski--Seidenberg principle.
They further thank Marc Burger and Beatrice Pozzetti for sharing their results and explaining their work to us.
The first author thanks Max Riestenberg for pointing us to the result of the $\PG$-condition for $\SL_n(\R)$.
The second author warmly thanks Olivier Benoist for patiently answering his questions and sharing his insight on various aspects of real algebraic geometry.
The authors express their gratitude to Marc Burger for detailed feedback on a first version of this article.
The authors thank the IHP and IHES for excellent working conditions.

The authors acknowledge support of the Institut Henri Poincaré (UAR 839 CNRS-Sorbonne Université), and LabEx CARMIN (ANR-10-LABX-59-01).
This project has received funding from the European Research Council (ERC) under the European Union’s Horizon 2020 research and innovation programme (grant agreement No 101018839), ERC GA 101018839. The fourth author was partially supported by the NUS-MOE grants A-8001950-00-00, A-8000458-00-00, and A-8004148-00-00.

\tableofcontents

\section{Preliminaries} \label{s: Preliminaries}

We collect in this section some background about the tools at play in this paper: Fuchsian groups, ordered fields, real algebraic geometry, semisimple linear semi-algebraic groups, their Jordan projection, parabolic subgroups and flag varieties.

\subsection{Fuchsian groups}
\label{s: Fuchsian groups}
For details on Fuchsian groups we refer to \cite{Katok}.\break
Throughout the paper we denote by $\Hb$ the hyperbolic plane.
Via the Poincaré disc model, its group of orientation preserving isometries is identified with the group $\PU(1,1)$ acting by homographies on the unit disc in $\Cb$. We see the unit circle $\Sb^1$ as the boundary at infinity of $\Hb$. We denote by $o$ the center of the unit disk.

A \emph{Fuchsian group} $\Gamma$ is a discrete subgroup of $\Isom_+(\Hb)$. The \emph{limit set} of a Fuchsian group $\Gamma$, denoted $\Lambda(\Gamma)$ (or simply $\Lambda$ when it does not bring any confusion), is the set of limit points in $\Sb^1$ of some (hence any) $\Gamma$-orbit in $\Hb$. This limit set is a closed, $\Gamma$-invariant subset of $\Sb^1$. It is empty when $\Gamma$ is finite and consists of $1$ or $2$ points when $\Gamma$ is infinite and virtually cyclic. Otherwise, it is infinite and $\Gamma$ is called \emph{non-elementary}.

The stabilizer in $\Gamma$ of a point $x\in \Sb^1$, when non-trivial, is an infinite cyclic subgroup of $\Gamma$. Conversely, every $\gamma \in \Gamma$ of infinite order fixes at most two points in $\Sb^1$:
\begin{itemize}
    \item either $\gamma$ has two distinct fixed points $\gamma^+$ and $\gamma^-$ and 
    \[\gamma^{\pm n} \cdot x \underset{n\to +\infty}{\longrightarrow} \gamma^\pm\]
    for all $x\notin \{\gamma^-,\gamma^+\}$; in this case, $\gamma$ is called \emph{hyperbolic}, and $\gamma^+$ (resp.\ $\gamma^-$) is called the \emph{attracting fixed point} (resp.\  \emph{repelling fixed point}) of $\gamma$;
    \item or $\gamma$ has a unique fixed point $\gamma^+= \gamma^-$, and $\gamma^{\pm n}\cdot x \underset{n\to +\infty}{\longrightarrow} \gamma^+$ for all $x\in \Sb^1$; in this case $\gamma$ is called \emph{parabolic} and $\gamma^+$ the \emph{parabolic fixed point} of $\gamma$.
\end{itemize}

We call a parabolic element $\gamma$ a \emph{positive parabolic} (resp.\ \emph{negative parabolic}) if $\gamma$ translates counter-clockwise (resp. clockwise) on $\Sb^1 \setminus \gamma^+$.

Let $\Lambda_p=\Lambda_p(\Gamma)$ be the set of fixed points of parabolic elements in $\Gamma$, and let $\Lambda_h=\Lambda_h(\Gamma)$ be the set of fixed points of hyperbolic elements in $\Gamma$. Both $\Lambda_p$ and $\Lambda_h$ are $\Gamma$-invariant subsets of $\Lambda$. If $\Gamma$ is non-elementary, then $\Gamma$ acts minimally on $\Lambda$, and so $\Lambda_h$ and $\Lambda_p$ are dense in $\Lambda$ as long as they are non-empty. Furthermore, when $\Gamma$ is non-elementary, $\Lambda_h$ is always non-empty, and the set $\{(\gamma^-, \gamma^+)\mid \gamma \in \Gamma \textnormal { hyperbolic}\}$ is dense in $\Lambda^2$.

A non-elementary Fuchsian group is called \emph{of the first kind} if its limit set is all of $\Sb^1$ and \emph{of the second kind} if its limit set is properly contained in $\Sb^1$.

\subsubsection{Fuchsian groups acting on their limit sets}
The main object of interest to us will be the action of a Fuchsian group $\Gamma$ on its limit set $\Lambda$. 
We will now review some other viewpoints on these objects. 

Let us start with the definition of a convergence group action:

\begin{definition}[Convergence group]
    Let $M$ be a compact metrizable topological space. A discrete group $\Gamma$ acting on $M$ by homeomorphisms acts as a \emph{convergence group} if for every unbounded sequence $(\gamma_n)_{n\in \Nb}\subset \Gamma$, there exists a subsequence $(\gamma_{k_n})_{n\in \Nb}$ and two points $a,b \in M$ such $\gamma_{k_n}\vert_{M\setminus \{b\}}$ converges uniformly on compact sets to the constant map $\equiv a$.
\end{definition}

The action of a Fuchsian group on $\Sb^1$ is a convergence group action. More precisely, we have the following:
\begin{fact}[Convergence action] \label{fact: Fuchsian groups are convergence groups}
    Let $\Gamma$ be a Fuchsian group and $(\gamma_n)_{n\in \Nb} \subset \Gamma$ such that \[\gamma_n\cdot o \underset{n\to +\infty}{\longrightarrow} x \in \Sb^1 \textrm{ and } \gamma_n^{-1}\cdot o \underset{n\to +\infty}{\longrightarrow} y\in \Sb^1~.\] Then $\gamma_n\cdot z \underset{n\to +\infty}{\longrightarrow} x$ for all $z\in \Sb^1\setminus \{y\}$.
\end{fact}

 We call such a sequence $(\gamma_n)_{n\in \Nb}$ an \emph{attracting sequence} in $\Gamma$, $x$ the \emph{attracting point} and $y$ the \emph{repelling point} of $(\gamma_n)_{n\in \Nb}$.

 \begin{example}
     If $\gamma \in \Gamma$ is hyperbolic or parabolic, then the sequence $(\gamma^n)_{n\in \Nb}$ is attracting, with attracting point $\gamma^+$ and repelling point $\gamma^-$.
 \end{example}

The second important property of the action of $\Gamma$ on $\Lambda$ is that it preserves a \emph{total cyclic ordering} on $\Lambda$:

\begin{definition}[Cyclic order]
\label{dfn:CyclicOrder}
    A \emph{total cyclic order} on a set $M$ is the data of a subset of \emph{cyclically ordered triples} in $\{(x,y,z)\in M^3\}$ with the following properties:
    \begin{itemize}
        \item If $(x,y,z)$ is cyclically ordered then so is $(y,z,x)$,
        \item If $(x,y,z)$ is cyclically ordered, then $(x,z,y)$ is not cyclically ordered,
        \item If $(x,y,t)$ and $(x,t,z)$ are cyclically ordered, then so is $(x,y,z)$,
        \item If $x,y,z$ are pairwise distinct, then either $(x,y,z)$ or $(z,y,x)$ is cyclically ordered.
    \end{itemize}
    An $n$-tuple $(x_1,\ldots, x_n)\in M^n$ is \emph{cyclically ordered} if $(x_i,x_j,x_k)$ is cyclically ordered for all $i<j<k$.
    If $M$ and $M'$ are sets with a total cyclic order, a map $f\colon M\to M'$ is \emph{monotone} if $(x,y,z)$ is cyclically ordered whenever $(f(x),f(y),f(z))$ is cyclically ordered.
\end{definition}

If $M$ is a cyclically ordered set, we set 
\[(x,y) = \{z\in M \mid (x,z,y) \textrm{ cyclically ordered}\} \textrm{ and } [x,y] = \{x\} \cup (x,y) \cup \{y\}~.\]

\begin{remark}
Let $M$ and $M'$ be cyclically ordered sets. Notice that if a map $f\colon M\to M'$ is monotone and injective, then it preserves the cyclic order, i.e.\ if $(a,b,c)$ is cyclically ordered, then $(f(a),f(b),f(c))$ is cyclically ordered.
\end{remark}

We equip $\Sb^1$ with the counter-clockwise cyclic order. By restriction, this induces a $\Gamma$-invariant cyclic order on $\Lambda(\Gamma)$ for any Fuchsian group $\Gamma$. \\

The following lemma will allow us to restrict to Fuchsian groups of the first kind when necessary. It is proven in \Cref{appendix} (see \Cref{coro:Fuchsian group conjugated to first kind}).

\begin{lemma}\label{Fuchsian groups lemma}
    Let $\Gamma$ be a non-elementary Fuchsian group. Then there exists a Fuchsian group $\Gamma'$ of the first kind and an isomorphism $\iota\colon \Gamma \to \Gamma'$ with the following properties:
        \begin{enumerate}
            \item there exists a continuous, $\iota$-equivariant, monotonic, surjective map \[\alpha\colon\Lambda(\Gamma)\to  \Lambda(\Gamma') = \Sb^1~;\]
            \item there exists a left-continuous $\iota^{-1}$-equivariant, monotonic, injective map \[\beta \colon \Lambda(\Gamma') = \Sb^1 \to \Lambda(\Gamma)~.\] 
        \end{enumerate}
\end{lemma}

We refer to the isomorphism $\iota$ given by the above lemma as a \emph{semi-conjugacy} between $\Gamma$ and $\Gamma'$. \Cref{appendix} contains more details about semi-conjugacy between Fuchsian groups and establishes an equivalence between the following objects (see \Cref{thm: informal theorem semi-conjugacy}):
\begin{itemize}
    \item semi-conjugacy classes of Fuchsian groups,
    \item conjugacy classes of  orientation preserving minimal convergence actions on $\Sb^1$,
    \item diffeomorphism classes of hyperbolic orbifolds.
\end{itemize}

\subsubsection{Conical limit points}
Let us recall for future use the definition of a conical limit point of a Fuchsian group $\Gamma$.
\begin{definition}
    A point $x\in \Lambda(\Gamma)$ is a \emph{conical limit point} if there exists a sequence $(\gamma_n)_{n\in \Nb}\subset \Gamma$ such that $\gamma_n\cdot o \underset{n\to +\infty}{\longrightarrow} x$ and such that the distance between $\gamma_n\cdot o$ and the geodesic ray $[o,x)$ from $o$ to $x$ remains bounded.
\end{definition}

Let $(x_n)_{n\in \Nb}$ be a sequence in $\Sb^1$ converging to a point $x$. We say that $x_n$ converges \emph{from the left} (resp.\ \emph{from the right}) if, for some (hence any) $z\neq x$, $x_n \in (z,x)$ (resp.\ $x_n \in (x,z)$) for $n$ large enough.

\begin{lemma} \label{lem: Monotonous limit points}
    Let $x$ be a point in $\Lambda(\Gamma)$.
    \begin{enumerate}
        \item If $(\gamma_n)_{n\in\Nb}$ is a sequence such that $\gamma_n\cdot o\to x$ as $n\to +\infty$, then up to extracting a subsequence, there exists a non-empty interval $U\subset \Sb^1$ such that $\gamma_n\cdot U$ converges to $x$ from the left or from the right.
        \item If, moreover, $x$ is a conical limit point, then there exist a sequence $(\gamma_n)_{n\in\Nb}$ and two disjoint non-empty intervals $U$ and $V$ such that $\gamma_n\cdot V$ converges to $x$ from the left and $\gamma_n\cdot U$ converges to $x$ from the right.
    \end{enumerate}
    
\end{lemma}
\begin{proof}
   (1) Up to extracting a subsequence of $(\gamma_n)_{n\in\Nb}$, we can assume that $\gamma_n^{-1}\cdot o$ converges to some $y\in \Sb^1$ and that $\gamma_n^{-1}\cdot x$ converges to some $y'\in\Sb^1$. Let $U=(a,b)$ be a non-empty interval whose closure does not contain $y$ nor $y'$, and $a\neq x$. By \Cref{fact: Fuchsian groups are convergence groups}, the sequences $(\gamma_n^{-1}\cdot a)_{n\in\Nb}$ and $(\gamma_n^{-1}\cdot o)_{n\in\Nb}$ converge to the same point, so for $n$ large enough, $U$ does not contain $\gamma_n^{-1}\cdot a$ nor $\gamma_n^{-1}\cdot x$. Up to extracting a further subsequence, one can assume that either 
   \[(a,\gamma_n\cdot z,x)=\gamma_n\cdot (\gamma_n^{-1} \cdot a, z, \gamma_n^{-1}\cdot x)\]
   is cyclically ordered for all $n\in \Nb$ and all $z\in U$, or 
   \[(x,\gamma_n\cdot z,a)=\gamma_n\cdot (\gamma_n^{-1} \cdot x, z, \gamma_n^{-1}\cdot a)\]
   is cyclically ordered for all $n\in \Nb$ and all $z\in U$. Since $y\notin\overline{U}$, \Cref{fact: Fuchsian groups are convergence groups} implies that $\gamma_n\cdot U$ converges to $x$. Thus, $\gamma_n\cdot U$ converges to $x$ from the left or the right.

   (2) Assume now that $x$ is a conical limit point. Then there is a sequence $(\gamma_n)_{n\in\Nb}$ so that $\gamma_n\cdot o\to x$ as $n\to+\infty$ and  $d(\gamma_n\cdot o, [o,x))$ is bounded. As before, by extracting a subsequence, we may assume that $\gamma_n^{-1}\cdot o$ converges to some $y\in\Sb^1$ and $\gamma_n^{-1}\cdot x$ converges to some $y'\in\Sb^1$. Then, by elementary hyperbolic geometry, the angle $\angle_o(\gamma_n^{-1} \cdot o, \gamma_n^{-1}\cdot x) = \angle_{\gamma_n\cdot o}(o,x)$ is bounded away from $0$. Since $\gamma_n^{-1}\cdot o \underset{n\to +\infty}{\longrightarrow} y$ and $\gamma_n^{-1}\cdot x \underset{n\to +\infty}{\longrightarrow} y'$, we deduce that $y\neq y'$. We can thus choose two intervals $U=(a,b)$ and $V= (c,d)$ such that $(y,a,b,y',c,d)$ is cyclically ordered, and $a\neq x\neq c$. Then for all $z\in U$ and $z'\in V$ and all $n$ large enough, $(\gamma_n^{-1}\cdot a, z, \gamma_n^{-1} \cdot x)$ and $(\gamma_n^{-1}\cdot x, z', \gamma_n^{-1} \cdot c)$ are cyclically ordered. As before, we conclude that $\gamma_n\cdot V$ and $\gamma_n\cdot U$ respectively converge to $x$ from the left and from the right.
   \end{proof}

\subsubsection{The geometrically finite case}
The previous discussion is more concrete and explicit when the group $\Gamma$ is finitely generated and non-elementary (equivalently, \emph{geometrically finite}). In that case, up to passing to a finite index subgroup, $\Gamma$ is also torsion free and the quotient $S = \Gamma \backslash \Hb$ is a complete hyperbolic surface. This surface is the union of a compact core and finitely many ends which are either \emph{funnels} (quotients of a hyperbolic half-plane by a hyperbolic isometry) or \emph{cusps} (quotients of a horodisc by a parabolic isometry). The group $\Gamma$ is of the first kind if and only if $S$ has finite volume -- i.e.\ all its ends are cusps. We refer to these groups as \emph{lattices} in $\Isom_+(\Hb)$.

When $S$ is compact, $\Gamma$ is isomorphic to a closed surface group. In that case, the group $\Gamma$ is Gromov hyperbolic with boundary homeomorphic to $\Sb^1$. There are only two invariant cyclic ordering on $\partial_\infty \Gamma$ and choosing one amounts to choosing an orientation of $S$. Modulo a compatible choice of orientation, the action of $\Gamma$ on $\Sb^1$ is conjugated to the action of $\Gamma$ on $\partial_\infty \Gamma$ by a homeomorphism preserving the orientation, and the same will hold for every Fuchsian group isomorphic to~$\Gamma$. Finally, all such groups are of the first kind.

When $S$ is non-compact, $\Gamma$ is a free group with at least $2$ generators. It is again Gromov hyperbolic, and its Gromov boundary $\partial_\infty \Gamma$ is a Cantor set. There are then several possible cyclic orderings on $\partial_\infty \Gamma$, but only one is compatible with the cyclic ordering of $\Lambda(\Gamma)$, and choosing this cyclic ordering amounts to prescribing the topological type of $S$. For instance, a one-holed torus and a three-holed sphere, whose fundamental groups are both free groups in two generators, will give rise to different cyclic orderings of its boundary at infinity. More precisely, we have the following:

\begin{fact}
    Let $\Gamma$ be a Fuchsian group isomorphic to a free group with finitely many generators. Then there exists a unique total cyclic order on $\partial_\infty \Gamma$ for which there is a continuous, monotonic, $\Gamma$-equivariant map $\eta\colon \partial_\infty \Gamma \to \Lambda(\Gamma)$.

    If $\Gamma'$ is another Fuchsian group isomorphic to $\Gamma$, then the cyclic orderings on $\partial_\infty \Gamma \simeq \partial_\infty \Gamma'$ are compatible if and only if the surfaces $\Gamma \backslash \Hb$ and $\Gamma'\backslash \Hb$ are homeomorphic. 
\end{fact}

The map $\eta$ in the above fact is a homeomorphism if and only if $\Gamma$ is \emph{convex cocompact}, i.e.\ the surface $S = \Gamma \backslash \Hb$ does not have cusps. More generally, it has the following behavior: for all $x\neq y \in \partial_\infty \Gamma$, $\xi(x) \neq \xi(y)$ unless $x$ and $y$ are the attracting and repelling fixed points of an element $\gamma \in \Gamma$ in $\partial_\infty \Gamma$ representing a simple closed curve bounding an end of $S$. In that case, up to the correct choice of orientation of $\gamma$, the cyclic ordering on $\partial_\infty \Gamma$ is such that for all $z\notin \{x,y\}$, $(x,y,z)$ is cyclically ordered, and $\xi$ maps $x$ and $y$ to the attracting and repelling fixed points of $\gamma$ in $\Sb^1$. In particular, $\xi(x) \neq \xi(y)$ when $\gamma$ is hyperbolic (equivalently, the curve bounds a funnel), and $\xi(x) = \xi(y)$ when $\gamma$ is parabolic (equivalently, the curve bounds a cusp).

\subsection{Ordered fields}
\label{sec: ordered fields}
For an introduction to ordered fields we recommend \cite[Chapter 1]{Scheiderer_RAG}.
Let $\Fb$ be an ordered field, i.e.\ a field together with a total order that is compatible with the field operations.
It is \emph{real closed} if every positive element is a square and every odd degree polynomial has a root. Note that every ordered field has a \emph{real closure}, i.e.\ an algebraic field extension that is real closed and whose order extends the original one.
One says that $\Fb$ is \emph{non-Archimedean} if there exists $x \in \Fb$ with $x>n$ for all $n\in \Nb$.
Otherwise we say that $\Fb$ is Archimedean. 
Every Archimedean field is a subfield of $\Rb$ \cite[1.1.18 Theorem]{Scheiderer_RAG}.

\begin{example}
\label{example:OrderedFields}
We list here some relevant examples:
\begin{itemize}
\item Obviously, the field of real numbers $\Rb$ is real  closed.
\item The field of rational numbers $\Qb$, equipped with its unique order, is not real closed. Its real closure is the field $\overline{\Qb}^r \coloneqq \overline{\Qb} \cap \Rb$ of real algebraic numbers.
\item The field $\Rb(X)$ of rational fractions over $\Rb$ has a unique order for which $X>0$ but $X<\lambda$ for all $\lambda \in \Rb_{>0}$. It is not real closed since, for instance, $X$ is positive but does not have a square root.
\item The field of \emph{real Puiseux series} is the set of formal series
\[
\Rb(X)^\wedge \coloneqq \Bigl\{  \sum_{k=k_0}^{\infty} c_k X^{k/m} \, \Bigm| \, k_0 \in \Zb, \, m \in \Nb\setminus\{0\} , \, c_k \in \Rb \Bigr\},
\]
together with formal addition and multiplication.
An element\break $\sum_{k=k_0}^{\infty} c_k X^{k/m}$ is positive if $c_{k_0} > 0$.
With this order $\Rb(X)^\wedge$ is real closed, see e.g.\ \cite[Theorem 2.91]{BasuPollackRoy_AlgorithmsRealAlgebraicGeometry}. It contains the field $\Rb(X)$ whose real closure is thus the subfield of the real Puiseux series that are algebraic over $\Rb(X)$ (\cite[Example 1.3.6~(b)]{BochnakCosteRoy_RealAlgebraicGeometry}).
\end{itemize}
\end{example}

Another important class of examples of real closed fields are the Robinson fields, which are quotients of order-convex subrings of the hyperreals.
Before we can define them, we need some notations.
Fix a non-principal ultrafilter $\omega$ on $\Nb$, i.e.\ a finitely additive probability measure on $\Nb$ with values in $\{0,1\}$ that gives probability zero to finite sets.
A point $x$ in a topological space $X$ is the \emph{$\omega$-limit} of a sequence $(x_n)_{n\in\Nb}$ in $X$, if for every neighborhood $U$ of $x$, one has $x_n\in U$ for $\omega$-almost all $n$, and we write $\lim_\omega x_n=x$.
In this case, $x$ is an accumulation point of the sequence $(x_n)$, namely there is a subsequence converging to $x$ in the usual sense.
An important feature is that every sequence in a compact space has an $\omega$-limit.

\begin{example}[Hyperreals]
The \emph{field of hyperreals} is defined as the ultraproduct 
\[\Rb_\omega\coloneqq \Rb^{\Nb}/\sim\]
endowed with pointwise addition and multiplication, where $(x_n)_{n \in \Nb} \sim (y_n)_{n\in \Nb}$ if and only if the two sequences coincide $\omega$-almost everywhere.
The corresponding equivalence class in $\Rb_\omega$ will be denoted by $[x_n]$.
We can endow $\Rb_\omega$ with an order $\leq_\omega$ coming from the order on $\Rb$, more precisely, we have $[x_n]\leq_\omega [y_n]$ if and only if $x_n \leq y_n$ for $\omega$-almost all $n \in \Nb$.
There is a canonical order-preserving embedding $\Rb \hookrightarrow \Rb_\omega$ by sending a real number $t$ to the equivalence class of the constant sequence $(t)_{n\in \N}$, also denoted $t$.
Since $\Rb$ is real closed, so is $\Rb_\omega$.
Note that $\Rb_\omega$ is non-Archimedean, since for example the class of the sequence $x_n = n$ is larger than any natural number.
\end{example}

\begin{example}[Robinson fields]
\label{example:RobinsonField}
Let $(\lambda_n)_{n \in \Nb}$ be a sequence of scalars with $\lambda_n \geqslant 1$ for all $n\in \Nb$ and $\lambda_n \to +\infty$ as $n \to +\infty$.
Then $[\lambda_n ] \eqqcolon \lambda \geqslant_\omega t$ for all $t \in \Rb$.
The set
\[\Oc^\mathbf{\lambda} \coloneqq \{ x \in \Rb_\omega \mid -\lambda^m <_\omega x <_\omega \lambda^m  \textnormal{ for some } m \in \Zb\} \]
is an order-convex subring of $\Rb_\omega$ with maximal ideal
\[\Ic^\mathbf{\lambda} \coloneqq \{ x \in \Rb_\omega \mid -\lambda^m <_\omega x <_\omega \lambda^m \textnormal{ for all } m \in \Zb\}. \]
The quotient $\Rbom \coloneqq \Oc^\mathbf{\lambda}/\Ic^\mathbf{\lambda}$ is the \emph{Robinson field} associated to the non-principal ultrafilter $\omega$ and the scaling sequence $\mathbf{\lambda}$.
Since $\Oc^\lambda$ is order-convex (i.e.\ 
$0 \leq x \leq y$ and $y \in \mathcal{O}$, implies that also $x \in \mathcal{O}$), the order on $\Rb_\omega$ descends to an order on $\Rbom$.
With this order $\Rbom$ is a non-Archimedean real closed field.
We refer to \cite[\S 2.1]{BurgerIozziParreauPozzetti_RSCCharacterVarieties2} for more details.
\end{example}

We are particularly interested in ordered fields that admit order-compatible absolute values.

\begin{definition}
Let $\Fb$ be a field. 
An \emph{absolute value} on $\Fb$ is a map $\abs{\cdot} \from \Fb \to \Rb_{\geqslant 0}$ satisfying
\begin{enumerate}
\item $\abs{x} = 0 \iff x =0$,
\item $\abs{xy}=\abs{x}\abs{y}$ for all $x,y \in \Fb$,
\item $\abs{x+y} \leq \abs{x} + \abs{y}$ for all $x,y \in \Fb$.
\end{enumerate}
The absolute value is called \emph{trivial} if $\abs{x}=1$ for all $x \in \Fb^\times$, and \emph{ultrametric} if $\abs{x+y} \leq \max\{\abs{x},\abs{y}\}$ for all $x,y \in \Fb$.
\end{definition}

If $\Fb$ is additionally  ordered, then we say that the absolute value is \emph{order-compatible} if for all $x, y \in \Fb$ with $0 \leq x \leq y$ we have $\abs{x}\leq \abs{y}$.
We recall the following definition from \cite[\S 5]{Brumfiel_RSCTeichmullerSpace}.
\begin{definition}
\label{def: big element}
    Let $\Fb$ be an ordered field.
    Then $b\in \Fb$ is a \emph{big element} if for all $x\in \Fb$ there exists $n\in \N$ such that $x \leq b^n$.
\end{definition}

If $\Fb$ admits a big element $b$, one can define an order-compatible absolute value $\abs{\cdot}_b\from \Fb \to \Rb_{\geqslant 0}$ by setting
\[ \abs{x}_b \coloneqq \exp\left(\inf\left\{\frac{p}{q} \in \Qb \mid x^q<b^p\right\}\right),\]
mimicking the definition of the standard logarithm.
If $\Fb$ is Archimedean we have $\abs{x}_b=\abs{x}^{\frac{1}{\log(b)}}$, where $\abs{\cdot}$ is the standard absolute value on $\Rb$.
As soon as $\Fb$ is non-Archimedean, $\abs{\cdot}_b$ is an ultrametric absolute value and every non-trivial order-compatible absolute value on $\Fb$ is of the form $\abs{\cdot}_b$ for some big element $b\in\Fb$ \cite[Lemma 10.11]{Flamm}.

\begin{example}
    In $\Rb(X)$, endowed with the order from \Cref{example:OrderedFields}, the element $\tfrac{1}{X}$ is a big element, and $\abs{\tfrac{P}{Q}}_{1/X} = e^{-\textnormal{deg } (P)+ \textnormal{deg }(Q)}$.
    In the Robinson field $\Rb_\omega^\lambda$, $\lambda$ is a big element, and its associated absolute value is $\abs{x}_\lambda= \lim_\omega \vert x_n\vert^{1/\lambda_n}$.
    
    The field of hyperreals $\Rb_\omega$ does not admit a big element, and in fact no non-trivial order-compatible absolute value.
\end{example}

The order on $\Fb$ induces the \emph{order topology}, where a subbasis of open sets is given by the open intervals $(a,b) = \{ x \in \Fb \mid a<x<b\}$.
The resulting topology turns $\Fb$ into a Hausdorff topological field, which is not locally compact and totally disconnected whenever $\Fb \neq \Rb$, see e.g.\ \cite[Remark 1.2.12 (4), Exercise 1.2.2]{Scheiderer_RAG}.
When $\Fb$ admits a big element $b$, its order topology is equivalent to the topology induced by the distance $d(x,y) \coloneqq |x-y|_b$, and is thus metrizable.

\subsection{Semi-algebraic sets and maps}
\label{section:SemiAlgSets}
We summarize general definitions and results from real algebraic geometry.
We refer the reader to \cite{BochnakCosteRoy_RealAlgebraicGeometry}, in particular Chapters 1, 2 and 5, for more details and proofs.
The main objects of study in real algebraic geometry are semi-algebraic sets.

From now on let $\Fb$ be a real closed field.

\begin{definition}
\label{dfn_SemiAlgSet}
A subset $\mathcal{B} \subset \Fb^n$ is a \emph{basic semi-algebraic set}, if there exists a polynomial $f \in \Fb[X_1,\ldots,X_n]$ such that
\[ \mathcal{B} = \mathcal{B}(f)= \{ x \in \Fb^n \mid f(x)>0\}.\]
A subset $X \subset \Fb^n$ is \emph{semi-algebraic} if it is a Boolean combination of basic semi-algebraic sets, i.e.\ $X$ is obtained by taking finite unions and intersections of basic semi-algebraic sets and their complements. A map $f \from X \to Y$ between two semi-algebraic sets $X \subset \Fb^n$ and $Y \subset \Fb^m$ is called \emph{semi-algebraic} if its graph $\textrm{Graph}(f) \subset X \times Y$ is semi-algebraic in $\Fb^{n+m}$.
\end{definition}

We endow $\Fb^n$ with the product topology.
For this topology, polynomial functions are continuous, thus the basic semi-algebraic sets are open and form a basis of the product topology.
In the case $\Fb=\Rb$, this topology agrees with the analytic topology on $\Rb^n$.
Semi-algebraic subsets of $\Fb^n$ are endowed with the subspace topology, and we denote this the \emph{semi-algebraic} topology.

Any algebraic subset of $\Fb^n$ is semi-algebraic, and any polynomial or rational map is a semi-algebraic map. However, in the setting of real closed fields, the Tarski--Seidenberg theorem (which fails for algebraic sets) is an important reason to consider semi-algebraic sets as opposed to algebraic sets:

\begin{theorem}[Tarski--Seidenberg theorem, {\cite[Theorem 5.2.1]{BochnakCosteRoy_RealAlgebraicGeometry}}]
\label{thm_TarskiSeidenbergproj}
Let\break $\pr \from \Fb^n \to \Fb^{n-1}$ denote the projection onto the first $n-1$ coordinates. Then, for any semi-algebraic set $X\subset \Fb^n$, the set $\pr(X) \subset \Fb^{n-1}$ is semi-algebraic.
\end{theorem}

As a consequence of the Tarski--Seidenberg theorem, one can prove that closures and interiors of semi-algebraic sets are also semi-algebraic sets, see \cite[Propositions 2.2.2]{BochnakCosteRoy_RealAlgebraicGeometry}, that the composition of two semi-algebraic maps is a semi-algebraic map \cite[Propositions 2.2.6]{BochnakCosteRoy_RealAlgebraicGeometry}, and that the images and preimages of semi-algebraic sets under semi-algebraic maps are also semi-algebraic sets \cite[Proposition 2.2.7]{BochnakCosteRoy_RealAlgebraicGeometry}.

Recall that if $\Fb\neq \Rb$, then $\Fb$ is totally-disconnected for the order topology on $\Fb$.
However we have the following notion of connectedness for semi-algebraic sets.

\begin{definition} \label{dfn_SemiAlgConn}
A semi-algebraic set $X \subset \Fb^n$ is \emph{semi-algebraically connected} if it cannot be written as the disjoint union of two non-empty semi-algebraic subsets of $\Fb^n$ both of which are closed (and hence open) in $X$. 
\end{definition}

This notion of semi-algebraic connectedness agrees with connectedness over $\Rb$, and in general gives rise to well-behaved components:

\begin{theorem}[{\cite[Theorem 2.4.4 and Theorem 2.4.5]{BochnakCosteRoy_RealAlgebraicGeometry}}]
\label{thm_RConnSemiAlgConn}
Every semi-algebraic subset $X\subset\Fb^n$ can be written as a disjoint union of a finite number of semi-algebraically connected semi-algebraic sets, each of which is both closed and open in $X$.
When $\Fb=\Rb$, a semi-algebraic set of $\Rb^n$ is semi-algebraically connected if and only if it is connected for the analytic topology.
\end{theorem}

The finite number of semi-algebraically connected semi-algebraic sets given by \Cref{thm_RConnSemiAlgConn} are called the \emph{semi-algebraically connected components} of $X$.
When $\Fb=\Rb$ these agree with the connected components for the analytic topology.

Finally, over $\Rb$ a subset in $\Rb^n$ is compact if and only if it is closed and bounded.
The following proposition justifies why closed and bounded semi-algebraic sets are the right analogue of compact sets in real algebraic geometry.
\begin{proposition}[{\cite[Proposition 2.5.7]{BochnakCosteRoy_RealAlgebraicGeometry}}]
\label{propo_ProjClosedBd}
    Let $X$ be a closed and bounded semi-algebraic subset of $\Fb^n$ and $\pr\from \Fb^n \to \Fb^{n-1}$ the projection on the space of the first $n-1$ coordinates. 
    Then $\pr(X)$ is a closed and bounded semi-algebraic set.
\end{proposition}

\subsection{Real closed extensions}
\label{s: extensions semi-algebraic sets}
In this section, let $\Fb$ be a real closed field and $\Kb$ a real closed extension of $\Fb$. 

\begin{definition}
\label{dfn_ExtSemiAlgSets}
Let $X \subset \Fb^n$ be a semi-algebraic set given as
\[ X = \bigcup_{i=1}^s\bigcap_{j=1}^{r_i} \{x \in \Fb^n \mid f_{ij}(x) \ast_{ij} 0 \},\]
with $f_{ij} \in \Fb[X_1,\ldots,X_n]$ and $\ast_{ij}$ is either $<$ or $=$ for $i = 1, \ldots,s$ and $j=1,\ldots,r_i$.
The \emph{$\Kb$-extension $X_{\Kb}$} of $X$ is the set given by the same Boolean combination of sign conditions as $X$, more precisely
\[X_{\Kb} = \bigcup_{i=1}^s\bigcap_{j=1}^{r_i} \{x \in \Kb^n \mid f_{ij}(x) \ast_{ij} 0 \}.\]
\end{definition}

Note that $X_{\Kb}$ is semi-algebraic. Also, using the Tarski--Seidenberg theorem (\Cref{thm_TarskiSeidenbergproj}), one can show that $X_{\Kb}$ depends only on the set $X$, and not on the Boolean combination describing it, see \cite[Proposition 5.1.1]{BochnakCosteRoy_RealAlgebraicGeometry}.

The Tarski--Seidenberg theorem also implies a useful transfer principle. In order to describe this, we need several concepts from logic.

\begin{definition}
    A \emph{first-order formula of the language of ordered fields with parameters in $\Fb$},     abbreviated as a \emph{formula of $\mathcal L(\Fb)$}, is a formula written with a finite number of conjunctions, disjunctions, negations, implications, biconditionals, and universal or existential quantifiers on variables, starting from atomic formulas which are formulas of the kind
    \[f(x_1,\ldots, x_n) = 0 \quad \textnormal{or} \quad g(x_1,\ldots,x_n) > 0, \]
    where $f$ and $g$ are polynomials with coefficients in $\Fb$. 
    The \emph{free variables} of such a formula are those variables of the polynomials appearing in the formula which are not quantified. If such a formula has no free variables, then we call it a \emph{sentence of $\mathcal L(\Fb)$}.
\end{definition}

The Tarski--Seidenberg theorem, restated in first order logic, says that every formula of $\mathcal L(\Fb)$ is equivalent to a quantifier-free one. In particular, every sentence of $\mathcal L(\Fb)$ is equivalent to a quantifier-free one, and so its truth value is decidable. As a consequence, we have:

\begin{proposition}[Tarski--Seidenberg transfer principle, {\cite[Proposition 5.2.3]{BochnakCosteRoy_RealAlgebraicGeometry}}]
\label{thm_TarskiSeidenberg}
    If $\Phi$ is a sentence of $\mathcal{L}(\Fb)$, then $\Phi$ holds true in $\Fb$ if and only if it holds true in $\Kb$.
\end{proposition}

Using this one can prove an extension theorem for semi-algebraic maps.

\begin{theorem}[{\cite[Propositions 5.3.1, 5.3.3, 5.3.5]{BochnakCosteRoy_RealAlgebraicGeometry}}]
\label{thm_ExtSemiAlgMaps}
Let $X \subset \Fb^n$ and $Y \subset \Fb^m$ be two semi-algebraic sets, and $f \from X \to Y$ a semi-algebraic map.
Then the $\Kb$-extension of the graph of $f$ is the graph of a semi-algebraic map $f_{\Kb} \from X_{\Kb} \to Y_{\Kb}$, that is called the \emph{$\Kb$-extension} of $f$.
Furthermore, $f$ is injective (respectively surjective, respectively bijective) if and only if $f_{\Kb}$ is injective (respectively surjective, respectively bijective), and $f$ is continuous if and only if $f_{\Kb}$ is continuous.
\end{theorem}

We noted previously that closures and interiors of semi-algebraic sets are semi-algebraic, and that the semi-algebraic connected components of a semi-algebraic set are also semi-algebraic. Thus, the transfer principle also has the following consequences.

\begin{theorem}[{\cite[Propositions 5.3.5 and 5.3.6]{BochnakCosteRoy_RealAlgebraicGeometry}}]
\label{thm_ExtConnComp}
Let $X \subset \Fb^n$ be semi-algebraic.
Then 
\begin{enumerate}
    \item 
    \label{thm_ExtConnComp: closedness}
    $X$ is closed (respectively open) if and only if $X_{\Kb}$ is closed (respectively open);
    \item 
    \label{thm_ExtConnComp: closed and bounded}
    $X$ is closed and bounded if and only if $X_{\Kb}$ is closed and bounded;
    \item $X$ is semi-algebraically connected if and only if $X_{\Kb}$ is semi-algebraically connected;
    \item if $C_1, \ldots, C_m$ are the semi-algebraically connected components of $X$, then $(C_1)_{\Kb}, \ldots, (C_m)_{\Kb}$ are the semi-algebraically connected components of $X_{\Kb}$.
\end{enumerate}
\end{theorem}

\subsection{Completeness of ordered fields}
\label{subsection:CompletenessOrderedFields}

In the real case, the extension of positive boundary maps defined on a dense subset of $\Sb^1$ relies on the completeness properties of the reals.
In the case of non-Archimedean fields there are several non-equivalent notions of completeness.
We focus on the following, which will allow us to extend boundary maps.

\begin{definition}
\label{dfn: Cantor complete}
    An ordered field is \emph{Cantor complete} if every nested sequence of closed bounded intervals has non-empty intersection.
\end{definition}

Robinson fields (\Cref{example:RobinsonField}) are also examples of Cantor complete ordered fields.
This is stated without proof in e.g.\ \cite[Example 6.10]{hall2015ordered}.
We thus give a proof here, which is adapted from the proof that the hyperreals are Cantor complete in \cite[Theorem 11.10.1]{Goldblatt}.

\begin{proposition}
    Robinson fields are Cantor complete.
\end{proposition}
\begin{proof}
    Let $[x^k,y^k] \subset \Rbom$ be a nested sequence of non-empty closed intervals, i.e.\ $x^k, y^k \in \Rbom$ with $x^k \leq x^{k+1} \leq y^{k+1} \leq y^k$ for all $k \in \Nb$.
    Without loss of generality we can assume that all of the above inequalities are strict.
    For all $k \in \Nb$, represent $x^k$ by a sequence $(x_n^k)_{n \in \Nb}$ of elements in $\Rb$, and similarly $y^k$ by a sequence $(y_n^k)_{n \in \Nb}$ of elements in $\Rb$.

    We now argue that we can assume that for every $k,n\in\Nb$, we have
    \[x_n^1<x_n^2<\dots<x_n^k<y_n^k<\dots<y_n^2<y_n^1.\]
   Note that for all $k \in \Nb$ there exists an $\omega$-full set $A_k \subset \Nb$ such that if $n\in A_k$, then $x_n^k <y_n^k$, $x_n^\ell<x_n^k$, and $y_n^k<y_n^\ell$ for all $\ell<k$. For each $n\in\Nb$, set $s_{n,0}\coloneqq0$, and for each $m\in\Nb$, iteratively define 
   \[s_{n,m}\coloneqq\inf\{k \mid k> s_{n,m-1}\text{ and }n\in A_k\}\in\Nb\cup\{\infty\}.\]
   Then for each $n,m\in\Nb$, by replacing every term of the (possibly empty, finite, or infinite) sequences 
   \[(x_n^k)_{s_{n,m-1}<k<s_{n,m}}\quad\text{and}\quad(y_n^k)_{s_{n,m-1}<k<s_{n,m}},\] 
   we may assume that they are increasing and decreasing respectively, and
   \begin{itemize}
       \item if $s_{n,m-1}+1<s_{n,m}<\infty$, then for all $s_{n,m-1}<k<s_{n,m}$, we have
       \[x_n^{s_{n,m-1}}<x_n^k<x_n^{s_{n,m}}\quad\text{and}\quad y_n^{s_{n,m}}<y_n^k<y_n^{s_{n,m-1}},\]
       \item if $s_{n,m-1}<s_{n,m}=\infty$, then for all $k,\ell>s_{n,m-1}$, we have
       \[x_n^{s_{n,m-1}}<x_n^k<y_n^\ell<y_n^{s_{n,m-1}}.\]
   \end{itemize}
   For each $k$, the set of positive integers $n$ such that $x_n^k$ is replaced by the above procedure lies in the complement of $A_k$. Thus, the replaced sequence also represents the point $x_\omega^k$ in $\Rb^\lambda_\omega$. On the other hand, by construction, the required inequalities hold.
   
   For all $n \in \Nb$ we can now consider the nested intersection $\bigcap_{k \in \Nb} [x_n^k,y_n^k]$.
    Since $\Rb$ is Cantor complete, $\bigcap_{k \in \Nb} [x_n^k,y_n^k] \neq \emptyset$, and we choose elements $z_n \in \bigcap_{k \in \Nb} [x_n^k,y_n^k]$ for all $n \in \Nb$.
    It follows that $\ulim z_n \in \bigcap_{k\in \Nb}[x^k, y^k]$, which finishes the proof.
\end{proof}

\begin{remark}
   On the other hand, the field of real Puiseux series (\Cref{example:OrderedFields}) is not Cantor complete. 
Indeed, we have $\bigcap_{k \in \Nb} [k,X^{-1/k}] = \emptyset$. 
\end{remark}

The reals are Cantor-complete, and have the stronger property that any intersection of a nested sequence of closed intervals is an interval. This is far from true over other Cantor complete fields. In fact, this property is equivalent to the supremum property, which is known to characterize $\Rb$ among ordered fields.
For example for a non-Archimedean ordered field $\Fb$ consider $\bigcap_{k \in \Nb}[-1/k,1/k]$, which is equal to the set of \emph{infinitesimals} of $\Fb$, namely $\{x \in \Fb\mid \max\{x,-x\} \leq r, \ \forall\ r\in \Qb_{>0}\}$.
This set is not an interval.\\

The non-empty intersection property of nested subsets can be extended to all closed and bounded semi-algebraic subsets in higher dimensions, since morally closed and bounded semi-algebraic sets are the right analogue of compact sets in real algebraic geometry.

\begin{proposition}\label{propo: cantor complete closed sets intersect non-empty}
    Let $\Fb$ be a real closed Cantor complete field, and $(C_k)$ a sequence of closed and bounded semi-algebraic subsets of $\Fb^n$ with $C_{k+1}\subset C_k$.
    Then 
    \[\bigcap_{k\in \Nb} C_k \neq \emptyset~.\]
\end{proposition}
\begin{proof}
    We proceed by induction on $n$.
    For the base case $n=1$, each $C_k$ is a finite union of disjoint closed intervals, each of which is a semi-algebraically connected component of $C_k$. Choose recursively, for each $k\in\Nb$, a semi-algebraically connected component $I_k \subset C_k$ such that $I_{k}\subset I_{k-1}$ and $I_k \cap C_n \neq \emptyset$ for all $n\geqslant k$. By the Cantor completeness of $\Fb$, $\bigcap_{k\in \Nb} I_k$ is non-empty and contained in $\bigcap_{k\in \Nb} C_k$.

    Assume now $n > 1$.
    We consider the projection $\pr \from \Fb^n \to \Fb^{n-1}$ onto the first $n-1$ coordinates.
    By the Tarski--Seidenberg principle (\Cref{thm_TarskiSeidenbergproj}), $\pr(C_k)$ is a semi-algebraic subset of $\Fb^{n-1}$ for all $k\in \N$.
    Furthermore, it is closed and bounded by \Cref{propo_ProjClosedBd}, and $\pr(C_{k+1})\subset \pr(C_k)$.
    By the induction hypothesis, there exists $x\in \bigcap_{k\in \Nb} \pr(C_k)$. Set $D_k \coloneqq C_k \cap \pr^{-1}(x)$. Then $D_k$ is a closed non-empty bounded semi-algebraic subset of $\pr^{-1}(x)\simeq \Fb$, and $D_{k+1} \subset D_k$. By the base case, $\bigcap_{k\in \Nb} D_k \subset \bigcap_{k\in \Nb} C_k$ is non-empty.
\end{proof}

Since the closure of a semi-algebraic set is semi-algebraic, this readily yields:

\begin{corollary}
\label{propo:NestedIntersectionPropertyCantorComplete}
    Let $\Fb$ be a real closed Cantor complete field, and $(U_k)$ a sequence of open and bounded semi-algebraic subsets of $\Fb^n$ with $\overline{U}_{k+1}\subset U_k$.
    Then 
    \[\bigcap_{k\in \Nb} U_k \neq \emptyset~.\]
\end{corollary}

\subsection{Semisimple linear semi-algebraic groups}
\label{section:SemiAlgGroups}
Let $\Fb$ be any real closed field. We will now define semisimple linear semi-algebraic groups over $\Fb$ and discuss some of their basic properties.
For an introduction to linear algebraic groups one can consult \cite{borelBook}.
For a detailed account in the context of real closed fields and semisimple linear semi-algebraic groups we highly recommend \cite{Appenzeller_GroupsRCF}.

\begin{definition}
\label{dfn: linear semi-algebraic group}
A \emph{linear semi-algebraic group} over $\Fb$ is a subgroup $G_{\Fb}\subset\GL_n(\Fb)$ (for some $n\in\Nb$) for which there exists an algebraic subgroup $\mathbf G\subset\GL_n$ defined over $\Fb$ so that $G_{\Fb}$ lies in $\mathbf{G}(\Fb)$ and contains the semi-algebraically connected component of $\Id$ in $\mathbf G(\Fb)$. We refer to the minimal such $\mathbf G$ as the linear algebraic group \emph{compatible} with $G_{\Fb}$. We also say that $G_{\Fb}$ is \emph{semisimple} if $\mathbf G$ is semisimple.
\end{definition}

\begin{example}\label{example: group}
    Let $\mathfrak g$ be a semisimple Lie algebra over $\Fb$. The group $\Aut(\mathfrak g)\subset \GL(\mathfrak g)$ is (the group of $\Fb$ points of) a linear semisimple algebraic group over $\Fb$. The semi-algebraically connected component $\Inn(\mathfrak g)$ of $\Aut(\mathfrak g)$ is a semisimple linear semi-algebraic group. 
\end{example}

In the case when $G_{\Fb}$ is semi-algebraically connected, then its compatible linear algebraic group $\mathbf G$ is connected in the algebraic sense (equivalently over $\Rb$, the set of $\Cb$-points $\mathbf{G}(\Cb)$ of $\mathbf{G}$ is connected for the analytic topology).

While a linear semi-algebraic group over $\Fb$ is not necessarily algebraic, it is a finite union of semi-algebraically connected components of $\GL_n(\Fb)$, and is therefore semi-algebraic. All group operations in a linear semi-algebraic group are continuous semi-algebraic maps.

Any semisimple algebraic group $\mathbf{G}$ defined over $\Fb$ is isomorphic to the $\Fb$-extension of an algebraic group defined over $\overline{\Qb}^r$, and two algebraic groups over $\overline{\Qb}^r$ are isomorphic over $\overline{\Qb}^r$ if and only if their $\Fb$-extension are isomorphic over $\Fb$. There is thus no loss of generality in assuming that $\mathbf G$ is defined over $\overline{\Qb}^r$, which we do from now on.

Now, the semi-algebraically connected components of $\mathbf G(\Fb)$ are defined over $\overline{\Qb}^r$. Hence any linear semi-algebraic group $G_{\Fb}$ over $\Fb$ compatible to $\mathbf G$ is the $\Fb$-extension of a semi-algebraic linear group $G$ over $\overline{\Qb}^r$. 

From now on, we thus fix a semisimple linear semi-algebraic group $G$ over $\overline{\Qb}^r$, with associated linear algebraic group $\mathbf G$, and we denote by $G_{\Fb} \subset \mathbf G(\Fb)$ its $\Fb$-extension for every real closed field $\Fb$.\\

Recall that the algebraic group $\GL_1$ (defined over $\Qb$) has as its $\Fb$-points the multiplicative group $\Fb^\times$. Let $\mathbf A \subset \mathbf G$ be a \emph{maximal split torus}, i.e.\ an algebraic subgroup isomorphic to $\GL_1^d$ over $\Qbar$, of maximal dimension $d$.
We denote by $A_{\Fb}$ the semi-algebraically connected component of $\mathbf A(\Fb)$, which is contained in $G_{\Fb}$. We call $A_{\Fb}$ the \emph{(multiplicative) Cartan subspace of $G_{\Fb}$.}

The adjoint action of $\mathbf A(\overline{\Qb}^r)$ on the Lie algebra $\mathfrak g$ of $\mathbf G$ is simultaneously diagonalizable, so we may decompose
\[\mathfrak g=\mathfrak g_1+\sum_{\alpha\in\Phi}\mathfrak g_\alpha,\]
where $\mathfrak g_\alpha=\{X\in\mathfrak g \mid \Ad(s)X=\alpha(s)X\}$ for any character $\alpha \colon \mathbf A(\overline{\Qb}^r) \to \GL_1$, and $\Phi$ is the set of \emph{multiplicative roots}, i.e.\ those non-trivial characters $\alpha$ for which $\mathfrak g_\alpha$ is non-zero.
There exists a subset $\Phi^+\subset\Phi$ such that 
\begin{itemize}
    \item $\Phi=\Phi^+\sqcup\Phi^-$, where $\Phi^-\coloneqq \{-\alpha \mid \alpha\in\Phi^+\}$.
\item if $\alpha,\alpha'\in\Phi^+$ such that $\alpha\cdot\alpha'\in\Phi$, then $\alpha\cdot\alpha'\in\Phi^+$.
\end{itemize}
Fix once and for all the choice of such a $\Phi^+$, and refer to the roots in $\Phi^+$ as the \emph{positive roots} of $\mathbf G$. We say that a positive root is \emph{simple} if it cannot be written as a product of two positive roots, and we let $\Delta\subset\Phi^+$ denote the set of simple roots. Every positive root can be written uniquely (up to re-ordering) as a product of simple roots.

The \emph{Weyl group} $W$ of $G$ is the quotient of the normalizer in $\mathbf G$ of $\mathbf A$ by the centralizer in $\mathbf G$ of $\mathbf A$. Observe that $W$ acts on $\mathbf A$ algebraically, and the induced action on the set of characters of $\mathbf A$ leaves $\Phi$ invariant. Let $w_0\in W$ be the \emph{longest element}, i.e.\ the unique element such that  $w_0\cdot\Phi^+=\Phi^-$. Using $w_0$, we may define the \emph{opposition involution} \[\iota\colon \Delta\to\Delta\]
given by $\iota(\alpha)=(\alpha\circ w_0)^{-1}$.

Now, given a real closed field $\Fb$, each multiplicative root $\alpha\in\Phi$ defines a character
\[\alpha_{\Fb} \colon A_{\Fb} \to \GL_1(\Fb).\]
Let us set
\[\Phi_{\Fb}\coloneqq \{\alpha_{\Fb}\mid\alpha\in\Phi\},\quad\Phi_{\Fb}^\pm\coloneqq\{\alpha_{\Fb}\mid\alpha\in\Phi^\pm\},\quad\text{and}\quad\Delta_{\Fb}\coloneqq\{\alpha_{\Fb} \mid \alpha\in\Delta\}.\]
We refer to the characters in $\Phi_{\Fb}$, $\Phi_{\Fb}^+$, $\Phi_{\Fb}^-$, and $\Delta_{\Fb}$ as the \emph{roots}, \emph{positive roots}, \emph{negative roots}, and \emph{simple roots} of $G_{\Fb}$. 

The Cartan subspace $A_{\Fb}$ is invariant under the natural action of the Weyl group $W$ on $\mathbf A(\Fb)$. Furthermore,
\[C_{\Fb}\coloneqq \{s\in A_{\Fb}\mid\alpha(s)\ge 1\text{ for all }\alpha_{\Fb}\in\Delta_{\Fb}\}\]
is a semi-algebraic fundamental domain for the $W$-action on $A_{\Fb}$. We refer to $C_{\Fb}$ as the \emph{positive (multiplicative) Weyl chamber} of $G_{\Fb}$.

\begin{remark}
    Classically, (restricted) roots and Weyl chambers are defined on the Lie algebra of $\mathbf A$, but this requires taking a logarithm. In contrast, multiplicative roots are algebraic and thus the positive multiplicative Weyl chamber is semi-algebraic and defined consistently over any real closed field: $C_{\Fb}$ \emph{is indeed} the $\Fb$-extension of the positive Weyl chamber over $\overline{\Qb}^r$. 
\end{remark}

\subsection{Multiplicative Jordan projection}
\label{section: Jordan projection}
An element $g\in G_{\Fb}\subset \GL_n(\Fb)$ is called
\begin{itemize}
    \item \emph{hyperbolic} if $g$ is diagonalizable over $\Fb$ with positive eigenvalues,
    \item \emph{elliptic} if $g$ is diagonalizable over $\Fb[\sqrt{-1}]$ with eigenvalues of modulus $1$, and
    \item \emph{unipotent} if $(g-\Id)^n=\Id$.
\end{itemize} Let $\mathcal H_{\Fb}$ (respectively $\mathcal E_{\Fb}$, $\mathcal U_{\Fb}$) denote the set of elements in $G_{\Fb}$ that are hyperbolic (respectively, elliptic, unipotent). These semi-algebraic subsets of $G_{\Fb}$ are the $\Fb$-extensions of the corresponding subset $\mathcal H, \mathcal E, \mathcal U$ over $\overline{\Qb}^r$, and the intersection of any two of these sets is the identity.

The next proposition is a generalization of the Jordan decomposition theorem to the $\Fb$-extension $G_{\Fb}$ of $G$, see \cite[Lemma 4.5, Lemma 4.6, Proposition 4.7]{BurgerIozziParreauPozzetti_RSCCharacterVarieties2}.

\begin{proposition}[(Multiplicative) Jordan decomposition and projection]
\label{propo: Jordan proj}
Suppose that $G_{\Fb}$ is semi-algebraically connected. For every $g \in G_{\Fb}$ there exist unique commuting
elements $g_e \in \mathcal{E}_{\Fb}$, $g_h \in \mathcal{H}_{\Fb}$ and $g_u \in \mathcal{U}_{\Fb}$ with $g = g_e g_h g_u$ (the \emph{Jordan decomposition} of $g$). Furthermore, we have the following:
\begin{enumerate}
\item The $G_{\Fb}$-conjugacy class of any element in $\mathcal{H}_{\Fb}$ meets $C_{\Fb}$ at exactly one point,
\item The map $J_{\Fb} \colon G_{\Fb} \to C_{\Fb}$ such that $J_{\Fb}(g)$ is conjugated to $g_h$ for every $g\in G_{\Fb}$ is semi-algebraic and continuous.
\end{enumerate}
The map $J_{\Fb}$ is called the \emph{(multiplicative) Jordan projection} of $G_{\Fb}$.
\end{proposition}

Once again, notations are consistent: the Jordan projection $J_{\Fb}$ is the $\Fb$-extension of the Jordan projection over $\overline{\Qb}^r$.

\subsection{Parabolic subgroups and flag varieties}
\label{sec: flag}
Fix a non-empty subset of simple roots $\Theta\subset \Delta$. Define the nilpotent subalgebras
\begin{align}\label{equation: u} \mathfrak u_\Theta\coloneqq \sum_{\alpha\in\Phi^+\setminus\Span(\Delta\setminus\Theta)}\mathfrak g_\alpha\quad\text{and}\quad\mathfrak u_\Theta^{\rm opp}\coloneqq\sum_{\alpha\in\Phi^-\setminus\Span(\Delta\setminus\Theta)}\mathfrak g_\alpha.\end{align}
The \emph{$\Theta$-standard parabolic subgroup}, denoted $\mathbf P_\Theta\subset\mathbf G$, and its \emph{opposite}, denoted $\mathbf P_\Theta^{\rm opp}\subset\mathbf G$, are the normalizers of $\mathfrak u_\Theta$ and $\mathfrak u_\Theta^{\rm opp}$ respectively. Let $\mathbf U_\Theta$ and $\mathbf U_\Theta^{\rm opp}$ denote the unipotent radicals of $\mathbf P_\Theta$ and $\mathbf P_\Theta^{\rm opp}$ respectively, and set $\mathbf L_\Theta\coloneqq \mathbf P_\Theta\cap\mathbf P_\Theta^{\rm opp}$. We also define $\mathbf A_\Theta \coloneqq \bigcap_{\alpha \in \Delta \setminus \Theta} \ker(\alpha) \subset \mathbf L_\Theta$.
Then $\mathbf{L}_\Theta$ is the centralizer of $\mathbf{A}_\Theta$ in $\mathbf{G}$. The groups $\mathbf P_\Theta$, $\mathbf P_\Theta^{\rm opp}$, $\mathbf U_\Theta$, $\mathbf U_\Theta^{\rm opp}$, $\mathbf L_\Theta$, and $\mathbf A_\Theta$ are all algebraic subgroups of $\mathbf G$. Recall that $\iota \colon \Delta\to\Delta$ denotes the opposition involution. One can verify that ${\bf P}_\Theta^{\rm opp}$ is conjugate in ${\bf G}$ to ${\bf P}_{\iota(\Theta)}$. We say that $\Theta$ is \emph{symmetric} if $\iota(\Theta)=\Theta$, or equivalently, if $\mathbf P_{\Theta}$ is conjugate to $\mathbf P_{\Theta}^{\rm opp}$. 

Observe that
\[P_{\Theta,\Fb}\coloneqq G_{\Fb} \cap \mathbf P_\Theta(\Fb),\quad P_{\Theta,\Fb}^{\rm opp}\coloneqq G_{\Fb} \cap \mathbf P_\Theta^{\rm opp}(\Fb),\quad U_{\Theta,\Fb} \coloneqq G_{\Fb}\cap \mathbf U_\Theta(\Fb)\]
\[U_{\Theta,\Fb}^{\rm opp}\coloneqq G_{\Fb} \cap \mathbf U_\Theta^{\rm opp}(\Fb),\quad L_{\Theta,\Fb}\coloneqq G_{\Fb}\cap \mathbf L_\Theta(\Fb), \quad\text{and}\quad A_{\Theta,\Fb} \coloneqq G_{\Fb} \cap \mathbf A_\Theta(\Fb)\]  
are all semi-algebraic  subgroups of $G_{\Fb}$, and are the $\Fb$-extensions of the corresponding groups defined over $\overline{\Qb}^r$. We refer to $P_{\Theta,\Fb}$ as the \emph{$\Theta$-standard parabolic subgroup of $G_{\Fb}$} and $P_{\Theta,\Fb}^{\rm opp}$ as its \emph{opposite}. Then $U_{\Theta,\Fb}$ and $U_{\Theta,\Fb}^{\rm opp}$ are respectively the unipotent radicals of $P_{\Theta,\Fb}$ and $P_{\Theta,\Fb}^{\rm opp}$. Since, over nilpotent Lie algebras, the exponential map is polynomial, we have continuous semi-algebraic bijections \[\exp \colon \mathfrak u_{\Theta,\Fb}\to U_{\Theta,\Fb}\quad\text{and}\quad\exp\colon\mathfrak u_{\Theta,\Fb}^{\rm opp}\to U_{\Theta,\Fb}^{\rm opp}.\]

Suppose now that $G_{\Fb}$ is semi-algebraically connected, in which case $\mathbf G$ is connected. The $\Theta$-flag manifold of $\mathbf G$ is the quotient
\[\mathbf F_{\Theta}\coloneqq \mathbf G/\mathbf P_{\Theta},\]
which is known to be a projective variety over $\overline{\Qb}^r$ such that the ${\bf G}$-action on ${\bf F}_\Theta$ is algebraic.
We may now define the \emph{$\Theta$-flag variety} of $G_{\Fb}$ as
\[\Fc_{\Theta,\Fb} \coloneqq \mathbf F_\Theta(\Fb).\]
This has the structure of a semi-algebraic set, and the $G_{\Fb}$-action on $\Fc_{\Theta,\Fb}$ is semi-algebraic (see e.g.\ \cite[Proposition 1.7.6]{Scheiderer_RAG})\footnote{More precisely, one can prove the existence of an affine chart $\mathbf O \subset \mathbf F_\Theta$, defined over $\overline{\Qb}^r$, such that $\mathbf O(\Fb)$ contains all the $\Fb$-points of $\mathbf F_\Theta$, identifying $\mathbf F_\Theta(\Fb)$ with a closed and bounded semi-algebraic set.}.

Over $\Rb$, the flag variety $\Fc_{\Theta, \Rb}$ is compact and connected, hence $G_{\Rb}$ (which contains the identity component of $\mathbf G(\Rb)$) acts transitively on $\mathbf F_{\Theta,\Rb}$ with stabilizer $P_{\Theta,\Rb}$. Hence the same holds over every real closed field $\Fb$, and we deduce the identification:
\[\mathcal F_{\Theta, \Fb} = G_{\Fb}/ P_{\Theta, \Fb}~.\]

As before, our notations are consistent: $\Fc_{\Theta,\Fb}$ is the $\Fb$-extension of $\Fc_\Theta \coloneqq \Fc_{\Theta, \overline{\Qb}^r}$.

Let us denote by $p_\Theta \in \Fc_\Theta \subset \Fc_{\Theta,\Fb}$ and $p_\Theta^\opp\in \Fc_{\iota(\Theta)}\subset \Fc_{\iota(\Theta),\Fb}$ the points fixed respectively by $P_\Theta$ and $P_{\Theta}^\opp$. We say that $x_1\in \Fc_{\Theta,\Fb}$ and $x_2\in \Fc_{\iota(\Theta),\Fb}$ are \emph{transverse} and write $x_1\pitchfork x_2$ if there is some $g\in G_{\Fb}$ such that $g\cdot (x_1,x_2) = (p_\Theta, p_{\Theta}^\opp)$. The set of transverse pairs is the unique open $G_{\Fb}$-orbit in $\Fc_{\Theta,\Fb} \times \Fc_{\iota(\Theta),\Fb}$. When $\Theta\subset\Delta$ is symmetric, we have $\Fc_{\Theta}=\Fc_{\iota(\Theta)}$, so transversality makes sense for pairs of flags in $\Fc_{\Theta,\Fb}$. 

Observe that  $P_{\Theta,\Fb}$ acts transitively on $\{y\in \Fc_{\iota(\Theta),\Fb} \mid p_\Theta \pitchfork y\}$, and the stabilizer in $P_{\Theta,\Fb}$ of $p_{\Theta}^\opp$ is $P_{\Theta,\Fb} \cap P_{\Theta,\Fb}^\opp = L_{\Theta,\Fb}$. 
The \emph{Levi decomposition} of parabolic subgroups states that the map
\[U_{\Theta,\Fb}\times L_{\Theta,\Fb}\to P_{\Theta,\Fb}\]
given by $(u,\ell)\mapsto u\ell$ is a semi-algebraic  bijection. We deduce that the map
\begin{eqnarray*}
U_{\Theta,\Fb}& \to &\{x\in\Fc_{\iota(\Theta),\Fb} \mid x\pitchfork p_\Theta\}\\
u & \mapsto & u \cdot p_\Theta^\opp
\end{eqnarray*}
is a semi-algebraic bijection. Similarly, the map 
\begin{eqnarray*}
U_{\Theta,\Fb}^\opp& \to &\{x\in\Fc_{\iota(\Theta),\Fb} \mid x\pitchfork p_\Theta^\opp\}\\
u & \mapsto & u \cdot p_\Theta
\end{eqnarray*}
is a semi-algebraic bijection.

\begin{example}
    The group $\PSL_2(\Fb)$ is the semi-algebraically connected component containing the identity of the $\Fb$-points of the algebraic group $\mathsf{PGL}_2$. It admits (up to conjugation) a unique parabolic subgroup
    \[P_{\Theta,\Fb} \coloneqq \left\{ \begin{pmatrix} a & b \\ 0 & a^{-1}\end{pmatrix}\middle\vert\ a\in \Fb^*, b\in \Fb\right\}/\{\pm \Id\}~,\]
    with unipotent radical 
    \[U_{\Theta,\Fb} \coloneqq \left\{\begin{bmatrix} 1 & b \\ 0 & 1\end{bmatrix}\middle\vert\ b\in \Fb\right\}~.\]
    The associated flag variety identifies with the projective space $\mathbf{P}^1(\Fb) \simeq \Fb \sqcup \{\infty\}$, and 
    \[p_\Theta = \left[ \begin{matrix} 1\\0 \end{matrix}\right] \simeq \infty \quad \textrm{ and }  p_\Theta^\opp = \left[ \begin{matrix} 0\\1 \end{matrix}\right] \simeq 0~.\]
    Two points in $\mathbf{P}^1(\Fb)$ are transverse if and only if they are distinct. Finally, $U_\Theta\simeq \Fb$ acts on $\Fb = \mathbf{P}^1(\Fb) \backslash \{\infty\} = \{x\in \mathbf{P}^1(\Fb) \mid x \pitchfork p_\Theta\}$ by translations.
\end{example}

The flag manifold $\Fc_{\Theta,\Fb}$ can also be described using the Lie algebra $\mathfrak g_{\Fb}$ as followed. Since ${\bf G}$ is semisimple, the group $\Inn(\mathfrak g)\subset\Aut(\mathfrak g)$ of inner automorphisms is a connected algebraic subgroup whose Lie algebra is $\mathfrak g$. Since $G_{\Fb}$ is semi-algebraically connected, it follows that $\Ad(G_{\Fb})$ lies in $\Inn(\mathfrak g_{\Fb})$, and so as $G_{\Fb}$-spaces, we have
\[\Fc_{\Theta,\Fb}\cong\{\Ad(g)\cdot\mathfrak p_{\Theta,\Fb}\mid g\in G_{\Fb}\}\cong\{g\cdot\mathfrak p_{\Theta,\Fb}\mid g\in \Inn(\mathfrak g_{\Fb})\}.\]

Let $\Aut_1(\mathfrak g_{\Fb})$ denote the subgroup of $\Aut(\mathfrak g_{\Fb})$ that preserves the type of all parabolic subalgebras, i.e.\ 
\[\Aut_1(\mathfrak g_{\Fb})\coloneqq\left\{g\in\Aut(\mathfrak g_{\Fb})\;\middle\vert\begin{array}{c}\text{there exists }g'\in\Inn(\mathfrak g_{\Fb})\text{ such that}\\
\text{for all }\alpha\in\Phi,\,\,g(\mathfrak g_{\alpha,\Fb})=g'(\mathfrak g_{\alpha,\Fb})\end{array}\right\}.\footnote{Equivalently, $\Aut_1(\mathfrak{g}_{\Fb})$ is the subgroup of $\Aut(\mathfrak{g}_{\Fb})$ acting trivially on the Dynkin diagram of the restricted root system of $\mathbf{G}$.}\]
From the description of $\Fc_{\Theta,\Fb}$ above, we may extend the $G_{\Fb}$-action on $\Fc_{\Theta,\Fb}$ (which factors through $\Ad(G_{\Fb})$) to a $\Aut_1(\mathfrak g_{\Fb})$-action. Notice that $\Aut_1(\mathfrak g_{\Fb})$ contains $\Inn(\mathfrak g_{\Fb})$, so it is necessarily a union of semi-algebraically connected components of $\Aut(\mathfrak g_{\Fb})$, and hence is a linear semi-algebraic group. We will see later that this extension is useful because the $\Aut_1(\mathfrak g_{\Fb})$ action has better transitivity properties on certain classes of subsets of $\Fc_{\Theta,\Fb}$ called diamonds.

\subsection{Proximalities}\label{sec: proximality}
We now introduce notions of proximality of elements in $G_{\Fb}$, related to their action on $\Fc_{\Theta,\Fb}$. Recall that $C_{\Fb} \subset A_{\Fb}$ is the multiplicative Weyl chamber of $G_{\Fb}$ and $J_{\Fb}\colon G_{\Fb}\to C_{\Fb}$ is the Jordan projection (see \Cref{propo: Jordan proj}). Let $\Theta\subset\Delta$ be non-empty.

\begin{definition}[Weak $\Theta$-proximality] An element $g \in G_{\Fb}$ is \emph{weakly $\Theta$-proximal over $\Fb$} if
$\alpha_{\Fb}(J_{\Fb}(g)) > 1$ for all $\alpha \in \Theta$.
\end{definition}

Suppose that $G_{\Fb}$ is semi-algebraically connected. Let $g\in G_{\Fb}$ be a weakly $\Theta$-proximal element, $g= g_h g_e g_u$ its Jordan decomposition, and $h\in G_{\Fb}$ be such that $h^{-1} g_h h\in C_{\Fb}$. Then the pair
\[(g^+,g^-) = h\cdot (p_\Theta, p_\Theta^\opp)\in \Fc_{\Theta,\Fb} \times \Fc_{\iota(\Theta),\Fb}\]
does not depend on the choice of $h$. Indeed, if $h'\in G_{\Fb}$ satisfies $h'^{-1}g_hh'=h^{-1}g_hh$, then $h'^{-1}h$ stabilizes $C_{\Fb}$, and hence fixes both $p_\Theta$ and $p_\Theta^{\rm opp}$.
The points $g^+$ and $g^-$ are transverse and fixed by $g$. We call them respectively the \emph{weak attracting flag} and \emph{weak repelling flag} of $g$.

The term ``weakly proximal'' is meant to make the distinction with a stronger form of proximality, which we mention here though it will not appear in the rest of the paper. We call an element $g\in G_{\Fb}$ \emph{strongly $\Theta$-proximal} if $\alpha_{\Fb}(J_{\Fb}(g))$ is a \emph{big element} of $\Fb$.

Over $\Rb$, weak and strong proximality are equivalent to the classical notion of proximality (see \cite{benoist1997proprietes}). In that case, $g^+ \in \Fc_{\Theta, \Rb}$ is the unique attracting fixed flag of $g$, and $g^n \cdot x$ converges to $g^+$ for any $x$ transverse to $g^-$. One could prove that the same dynamical behavior persists over a real closed field $\Fb$, when $g$ is strongly $\Theta$-proximal.
In contrast, when $g$ is only weakly $\Theta$-proximal, the basin of attraction of $g^+$ might be empty.

On the other hand, the notion of weak proximality will be useful to us because it is semi-algebraic. More precisely, let us denote by $\Prox_\Theta(G_{\Fb})$ the set of weakly $\Theta$-proximal elements of $G_{\Fb}$.

\begin{proposition} \label{prop: continuity g->g+ hyperbolic}
    The set $\Prox_\Theta(G_{\Fb})$ is an open semi-algebraic subset of $G_{\Fb}$ and the map
    \[\begin{array}{ccc}
        \Prox_\Theta(G_{\Fb}) & \to & \Fc_{\Theta,\Fb}\\
        g & \mapsto & g^+
    \end{array}\]
    is semi-algebraic and continuous.
\end{proposition}

\begin{proof}
    Since the map $\alpha\circ J: G \to \overline \Qb^r$ is a continuous semi-algebraic map defined over $\overline \Qb^r$, the set 
    \[\Prox_\Theta(G) = (\alpha \circ J)^{-1}(\overline{\Qb}^r_{>1}) \subset G\]
    is an open semi-algebraic subset whose $\Fb$-extension is
    \[\Prox_\Theta(G_{\Fb}) = (\alpha_{\Fb} \circ J_{\Fb})^{-1}(\Fb_{>1}) \subset G_{\Fb}~,\]
    which is also open and semi-algebraic, see \Cref{thm_ExtConnComp}.

    The graph of the map $g\mapsto g^+$ is the set 
    \[\{(g,h\cdot p_\Theta)\mid (g,h)\in \Prox_\Theta(G_{\Fb}) \times G_{\Fb}, \ h g_h h^{-1} \in C_{\Fb}\},\]
    which is semi-algebraic and defined over $\overline{\Qb}^r$.
    Hence the map $g\mapsto g^+$ is the $\Fb$-extension of the same map over $\overline{\Qb}^r$. To prove that $g\mapsto g^+$ is continuous, it is thus enough to prove it over $\Rb$, see \Cref{thm_ExtSemiAlgMaps}. 

    Now, for $g\in \Prox_\Theta(G_{\Rb})$, the point $g^+$ is the unique attracting fixed point of $g$ in $\Fc_{\Theta,\Rb}$ since weak and strong proximality are equivalent here. The continuity of the map $g\mapsto g^+$ then follows from the classical stability theorem for attracting fixed points in differentiable dynamics.
\end{proof}

\subsection{%
\texorpdfstring{%
Divergence for sequences in $G_{\Rb}$}%
{Divergence for sequences in G\_R}}
\label{subsection: divergence for sequences}
In this section, we restrict to $\Fb= \Rb$ and recall some further dynamical properties of groups acting on flag varieties in the real case.

Let $K_{\Rb}$ be a maximal compact subgroup of $G_{\Rb}$ and $\mathfrak k_{\Rb}\subset \mathfrak g_{\Rb}$ its Lie algebra.
We can choose $K_{\Rb}$ so that $\mathfrak k_{\Rb}$ is orthogonal (with respect to the Killing form on $\mathfrak g_{\Rb}$) to the Lie algebra of the Cartan subspace $A_{\Rb}$.
Let us now recall the definition of the Cartan projection.

\begin{theorem}\label{KAK}
For every $g\in G_{\Rb}$, there exists $m,\ell\in K_{\Rb}$ and a unique $\mu(g)\in C_{\Rb}$ such that $g=m\mu(g)\ell$. Furthermore,
\begin{enumerate}
    \item the map $\mu\colon G_{\Rb}\to C_{\Rb}$ given by $g\mapsto \mu(g)$ is continuous. 
    \item if $\Theta\subset\Delta$ is non-empty, $p_{\Theta,\Rb}\in\Fc_{\Theta,\Rb}$ is the point fixed by $P_{\Theta,\Rb}$, and $\alpha_{\Rb}(\mu(g))>1$ for all $\alpha\in\Theta$, then $\mathcal U_{\Theta}(g)\coloneqq m (p_{\Theta,\Rb}) \in \Fc_{\Theta,\Rb}$ depends only on $g$ (and not on the choice of $m$).
\end{enumerate}
\end{theorem}
The map $\mu$ in \Cref{KAK} is the \emph{(multiplicative) Cartan projection}. 

Let $\Theta\subset\Delta$ be a non-empty subset. The following lemma is well-known, see for example Kapovich--Leeb--Porti \cite[Section~4]{KLP17} or Canary--Zhang--Zimmer \cite[Proposition~2.3]{CZZ2}. 

\begin{lemma}\label{lemma: divergence in flag manifold}
Let $x\in\Fc_{\Theta,\Rb}$, let $y\in\Fc_{\iota(\Theta),\Rb}$, and let $(g_n)_{n\in\Nb}$ be a sequence in $G_{\Rb}$. For each $n$, let $g_n=m_n\mu(g_n)\ell_n$ be a Cartan decomposition. Then the following are equivalent:
\begin{enumerate}
\item $\alpha(\mu(g_n))\to+\infty$, $\mathcal U_\Theta(g_n)\to x$ and $\mathcal U_{\iota(\Theta)}(g_n^{-1})\to y$ as $n \to +\infty$.
\item The sequence $(g_n)_{n\in\Nb}$, when restricted to $\{z\in \Fc_{\Theta,\Rb} \mid z\pitchfork y\}$, converges uniformly on compact sets to the constant map with image $\{x\}$.
\item The sequence $(g_n^{-1})_{n\in\Nb}$, when restricted to $\{z\in \Fc_{\iota(\Theta),\Rb} \mid z\pitchfork x\}$, converges uniformly on compact sets to the constant map with image $\{y\}$.
\item
\label{lemma: divergence in flag manifold: open sets}
There are open sets $U\subset\Fc_{\Theta,\Rb}$ and $V\subset\Fc_{\iota(\Theta),\Rb}$ such that $g_n(z)\to x$ for all $z\in U$ and $g_n^{-1}(z)\to y$ for all $z\in V$.
\end{enumerate}
\end{lemma}

If these conditions are satisfied, we say that the sequence $(g_n)_{n\in \Nb}$ is a \emph{$\Theta$-attracting sequence} in $G_{\Rb}$, and call $x$ and $y$ the \emph{attracting flag} and \emph{repelling flag} of $(g_n)_{n\in \Nb}$ respectively.

\begin{example}
    If $g \in G_{\Rb}$ is $\Theta$-proximal, then the sequence $(g^n)_{n\in \Nb}$ is $\Theta$-attracting, with attracting flag $g^+$ and repelling flag $g^-$.
\end{example}

\begin{definition}\label{definition: divergent sequence and representation}
    A sequence $(g_n)_{n \in \Nb} \in G_{\Rb}^{\Nb}$ is \emph{$\Theta$-divergent} if $\alpha(\mu(g_n)) \to +\infty$ as $n\to +\infty$ for all $\alpha \in \Theta$.
    An element $g \in G_{\Rb}$ is \emph{$\Theta$-divergent} if the sequence $(g^n)_{n \in \Nb}$ is so.
    A representation $\rho$ of a discrete group $\Gamma$ into $G_{\Rb}$ is \emph{$\Theta$-divergent} if $(\rho(\gamma_n))_{n \in \Nb}$ is $\Theta$-divergent for every sequence $(\gamma_n)_{n\in\Nb}\subset \Gamma$ that leaves every finite set.
\end{definition}

The following is an easy consequence of \Cref{lemma: divergence in flag manifold}.

\begin{proposition}\label{prop: dynamics divergent sequence}
    Any $\Theta$-divergent sequence in $G_{\Rb}$ has a $\Theta$-attracting subsequence.
\end{proposition}
\begin{proof}
Let $(g_n)_{n\in \Nb} \in G_{\Rb}^{\Nb}$ be a $\Theta$-divergent sequence. Since $\Fc_{\Theta,\Rb}$ and $\Fc_{\iota(\Theta),\Rb}$ are compact, we can choose a subsequence $(k_n)_{k\in \Nb}$ such that $\mathcal U_\Theta(g_{k_n})$ and $U_{\iota(\Theta)}(g_{k_n}^{-1})$ converge respectively to points $x$ and $y$.
\end{proof}

\section{%
\texorpdfstring{%
$\Theta$-positive structures on flag varieties}%
{Theta-positive structures on flag varieties}}
\label{s: Positivity on flag varieties}

In this section, we recall the foundational work of Guichard and Wienhard \cite{guichard2025generalizing}, who introduced the notion of $\Theta$-positive tuples on certain real flag varieties, and we explain that it adapts well over any real closed field.

Throughout this section, $G$ is a semi-algebraically connected, semisimple linear semi-algebraic group over $\overline{\Qb}^r$.
We also fix a non-empty subset $\Theta$ of the simple roots of $G$, which we assume to be symmetric (see Sections~\ref{section:SemiAlgGroups} and~\ref{sec: flag} for definitions). 
As always, let $\Fb$ be a real closed field, and denote the $\Fb$-extensions of all semi-algebraic maps and sets defined over $\overline{\Qb}^r$ with a subscript $\Fb$.

\subsection{%
\texorpdfstring{%
$\Theta$-positive structures for general ordered fields}%
{Theta-positive structures for general ordered fields}}
\label{subsection: PositiveStructures}

Recall that $P_{\Theta,\Fb}$ and $P_{\Theta,\Fb}^\opp$ denote the $\Theta$-standard parabolic subgroup of $G_{\Fb}$ and its opposite respectively, $U_{\Theta,\Fb}$ and $U_{\Theta,\Fb}^\opp$ denote their respective unipotent radicals, and $p_\Theta$ and $p_\Theta^\opp$ denote their respective fixed points in $\Fc_{\Theta,\Fb} =\Fc_{\iota(\Theta),\Fb}$.
Set
\[U_{\Theta,\Fb}^\pitchfork \coloneqq \{u \in U_{\Theta,\Fb} \mid u\cdot p_\Theta^\opp \pitchfork p_\Theta^\opp\}.\]
It follows from the discussion in \Cref{sec: flag} that the map $U_{\Theta,\Fb}^\pitchfork\to\Fc_{\Theta,\Fb}$ given by $u\mapsto u\cdot p_\Theta^\opp$ is a semi-algebraic bijection onto the set of flags that are transverse to both $p_\Theta$ and $p_\Theta^\opp$.

\begin{definition}
\label{def: theta positive structure}
We say that $G_{\Fb}$ admits a \emph{$\Theta$-positive structure} if there exists a semi-algebraically connected component $C$ of $U_{\Theta,\Fb}^{\pitchfork}$ that is a semigroup, i.e.\ for all $g,h \in C$ we have $gh \in C$.
\end{definition}

\begin{example}
    When $G_{\Fb} = \PSL_2(\Fb)$, the set
    \[U_\Theta^\pitchfork = \left\{u_b \coloneqq \begin{bmatrix} 1 & b\\ 0 & 1\end{bmatrix} \,\middle\vert\, \ b\in \Fb^*\right\}\]
    has two connected components, both of which are semigroups:
    \[U_\Theta^{>0} \coloneqq \{ u_b\mid b>0\} \quad \textrm{ and } \quad U_\Theta^{<0} \coloneqq \{u_b\mid b<0\}~.\]
\end{example}

Since transversality is defined over $\overline{\Qb}^r$, the set $U_{\Theta, \Fb}^\pitchfork$ is the $\Fb$-extension of $U_\Theta^\pitchfork \coloneqq U_{\Theta, \Qbar}^\pitchfork$. Suppose now that $G$ admits a $\Theta$-positive structure, i.e.\ there is a semi-algebraically connected component $C$ of $U_\Theta^\pitchfork$ that is a semigroup. Then the $\Fb$-extension $C_{\Fb}$ of $C$ is a semi-algebraically connected component of $U_{\Theta,\Fb}^\pitchfork$ that is a semigroup, and so $G_{\Fb}$ admits a $\Theta$-positive structure. Conversely, since semi-algebraically connected components of $U_{\Theta,\Fb}^\pitchfork$ are defined over $\overline{\Qb}^r$ (see \Cref{thm_RConnSemiAlgConn}), every $\Theta$-positive structure on $G_{\Fb}$ is the $\Fb$-extension of a $\Theta$-positive structure on $G$. 
Hence $G_{\Fb}$ admits a $\Theta$-positive structure if and only if $G_{\Rb}$ does. Now, over $\Rb$, semi-algebraically connected components are the same as analytic components, so \Cref{def: theta positive structure} is equivalent to Guichard and Wienhard's notion of $\Theta$-positivity (\cite[Theorem 12.2]{guichard2025generalizing}). 
In conclusion:
\begin{itemize}
    \item The $\Theta$-positive structures in $G_{\Rb}$ classified by Guichard and Wienhard in \cite{guichard2025generalizing} are semi-algebraic and defined over $\overline{\Qb}^r$, and
    \item any $\Theta$-positive structure over $G_{\Fb}$ as in \Cref{def: theta positive structure} is the $\Fb$-extension of one of Guichard and Wienhard's $\Theta$-positive structures.
\end{itemize}

Real flag varieties admitting a positive structures are classified by Guichard--Wienhard. The list is as follows.

\begin{theorem}[{\cite[Theorem 1.1]{guichard2025generalizing}}]
\label{thm: GW classification positive structures}
    Let $G_{\Rb}$ be a simple real Lie group.
    Then $G_{\Rb}/P_{\Theta,\Rb}$ admits a positive structure if an only if either
    \begin{itemize}
        \item $G_{\Rb}$ is a split real form, $\Theta=\Delta$ (in which case $\Fc_{\Theta,\Rb}$ is the complete flag variety of $G_{\Rb}$); or
        \item $G_{\Rb}$ is Hermitian of tube type and of real rank $r$ and $\Theta=\{\alpha_r\}$, where $\alpha_r$ is the long simple restricted root (in which case $\Fc_{\Theta,\Rb}$ is the Shilov boundary of the symmetric space associated to $G_{\Rb}$); or
        \item 
        $G_{\Rb}$ is locally isomorphic to $\SO(p+1,p+k)$, $p>1$, $k>1$ and $\Theta=\{\alpha_1,\ldots,\alpha_p\}$, where $\alpha_1, \ldots, \alpha_p$ are the long simple restricted roots; or
        \item $G_{\Rb}$ is the real form of $F_4$, $E_6$, $E_7$, or $E_8$ whose system of restricted roots is of type $F_4$, and $\Theta=\{\alpha_1,\alpha_2\}$, where $\alpha_1, \alpha_2$ are the long simple restricted roots.
    \end{itemize}  
\end{theorem}

\begin{remark}
    Over a general ordered field $\Kb$ (not necessarily real closed), the only satisfying notion of semi-algebraically connected component of a semi-algebraic set $X_{\Kb}$ is the $\Kb$-points of a semi-algebraically connected component of $X_{{\overline \Kb}^r}$.
    (By Tarski--Seidenberg's quantifier elimination the semi-algebraically connected components of $X_{{\overline \Kb}^r}$ can be described by polynomials with coefficients in $\Kb$.) 
    With this definition, one can make sense of \Cref{def: theta positive structure} over any ordered field.
    However, this does not lead to any new $\Theta$-positive structures.
\end{remark}

\subsection{%
\texorpdfstring{%
Structure of $U_{\Theta,\Fb}^\pitchfork$}%
{Structure of U {Theta,F} pitchfork}}\label{hats}
In their work, Guichard and Wienhard \cite{guichard2025generalizing} analyzed the structure of $U_{\Theta,\Rb}^\pitchfork$, which we describe briefly here. 

Notice that the adjoint action of $L_{\Theta,\Fb}$ on the Lie algebra $\mathfrak u_{\Theta,\Fb}$ restricts to a semisimple action of the Abelian group $A_{\Theta,\Fb}$ on $\mathfrak u_{\Theta,\Fb}$, so we may decompose $\mathfrak u_{\Theta,\Fb}$ into its weight spaces
\[\mathfrak u_{\Theta,\Fb}=\bigoplus_{\beta\in \Phi_{\Theta}}\mathfrak u_{\beta,\Fb}.\]
It is straightforward to verify that for each $\beta\in\Phi_\Theta$, 
\[\mathfrak u_{\beta,\Fb}=\bigoplus_{\alpha\in\Phi\colon\alpha|_{A_{\Theta,\Fb}}=\beta}\mathfrak g_{\alpha,\Fb}.\]
Hence, if $\alpha\in\Phi$ satisfies $\alpha|_{A_{\Theta,\Fb}}=\beta$, then we will often abuse notation by denoting $\mathfrak u_{\alpha,\Fb}\coloneqq\mathfrak u_{\beta,\Fb}$.

Let $L_{\Theta,\Fb}^\circ$ denote the semi-algebraically connected component of $L_{\Theta,\Fb}$ that contains the identity. The following theorem is a summary of results due to Guichard and Wienhard.

\begin{theorem}{\ }\label{theorem: GWU}
\begin{enumerate}
\item $G_{\Rb}$ admits a $\Theta$-positive structure if and only if for each $\alpha\in\Theta$, there is a $L_{\Theta,\Rb}^\circ$-invariant sharp closed convex cone in $\mathfrak u_{\alpha,\Rb}$ \cite[Definition 3.1, Theorem 12.2]{guichard2025generalizing}.
\item If a $L_{\Theta,\Rb}^\circ$-invariant sharp closed convex cone in $\mathfrak u_{\alpha,\Rb}$ exists, then it is unique up to negation \cite[Remark 3.2, Proposition 3.6]{guichard2025generalizing}. 
\item Suppose that $G_{\Rb}$ admits a $\Theta$-positive structure. For each $\alpha\in\Theta$, choose a $L_{\Theta,\Rb}^\circ$-invariant sharp closed convex cone $c_\alpha\subset\mathfrak u_{\alpha,\Rb}$. 
The semigroup generated by \[\bigcup_{\alpha\in\Theta}\{\exp(X)\mid X\in c_\alpha\}\]
is closed, and its interior is a connected component of $U_{\Theta,\Fb}^\pitchfork$ \cite[Theorem 1.4]{guichard2025generalizing}.
Furthermore, every connected component of $U_{\Theta,\Fb}^\pitchfork$ that is a semigroup arises in this way \cite[Theorem 12.2]{guichard2025generalizing}.
\end{enumerate}
\end{theorem}

Let $\mathcal S_{\Fb}$ denote the set of semi-algebraically connected components of $U_{\Theta,\Fb}^\pitchfork$ that are semigroups. Notice that the $L_{\Theta,\Fb}$-action on $U_{\Theta,\Fb}$ by conjugation leaves $U_{\Theta,\Fb}^\pitchfork$ invariant, and thus induces an action of $L_{\Theta,\Fb}$ on $\mathcal S_{\Fb}$,
\[\Phi_{\Fb}\colon L_{\Theta,\Fb}\to {\rm Sym}(\mathcal S_{\Fb})\]
given by $\Phi_{\Fb}(g)\cdot C=gCg^{-1}$, where ${\rm Sym}(\mathcal S_{\Fb})$ denotes the symmetric group on the finite set $\mathcal S_{\Fb}$. Let $L_{\Theta,\Fb}^*$ denote the kernel of $\Phi_{\Fb}$. The following corollary is a consequence of the structure of $U_{\Theta,\Fb}^\pitchfork$ described by \Cref{theorem: GWU} and the transfer principle.

\begin{corollary}\label{cor: GW consequences}
If $G_{\Fb}$ admits a $\Theta$-positive structure, then the set $\mathcal S_{\Fb}$ admits a free and transitive action by the group $\{\pm1\}^{\Theta}$. Furthermore, the image of the homomorphism $\Phi_{\Fb}$
lies in $\{\pm1\}^{\Theta}$ (viewed as a subset of ${\rm Sym}(\mathcal S_{\Fb})$). In particular, the action of $L_{\Theta,\Fb}/L_{\Theta,\Fb}^*$ on $\mathcal S_{\Fb}$ is free.
\end{corollary}

\begin{proof}
By the transfer principle, it suffices to prove the corollary for $\Fb=\Rb$. In this case, the first statement is an immediate consequence of \Cref{theorem: GWU}. 

For each $\alpha\in\Theta$, let $c_\alpha\subset\mathfrak u_{\alpha,\Rb}$ be a $L_{\Theta,\Rb}^\circ$-invariant sharp closed convex cone. Since $L_{\Theta,\Rb}^\circ\subset L_{\Theta,\Rb}$ is a normal subgroup, it is straightforward to verify that for any $g\in L_{\Theta,\Rb}$ and $\alpha\in\Theta$, $\Ad(g)\cdot c_\alpha$ is also a $L_{\Theta,\Rb}^\circ$-invariant sharp closed convex cone, so \Cref{theorem: GWU}(2) implies that $\Ad(g)\cdot c_\alpha=\pm c_\alpha$. Hence, the adjoint action of $L_{\Theta,\Rb}$ on $\bigoplus_{\alpha\in\Theta}\mathfrak u_{\alpha,\Rb}\subset\mathfrak u_{\Theta,\Rb}$ gives a homomorphism 
\[L_{\Theta,\Rb}\to\{\pm1\}^{\Theta}\subset{\rm Sym}(\mathcal S_{\Rb}).\]
It is straightforward to check that the above $L_{\Theta,\Rb}$-action on $\mathcal S_{\Rb}$ agrees with $\Phi_{\Rb}$.
\end{proof}

\begin{remark}
Guichard and Wienhard \cite[Section 5.1]{guichard2025generalizing}  gave a description of the $L_{\Theta,\Rb}^\circ$-invariant sharp closed convex cone in $\mathfrak u_{\alpha,\Rb}$, which is semi-algebraic over $\overline{\Qb}^r$. As such, the transfer principle implies that \Cref{theorem: GWU} in fact holds with $\Rb$ replaced with an arbitrary real closed field $\Fb$.
\end{remark}

In general, the $L_{\Theta,\Fb}$-action on $\mathcal S_{\Fb}$ is not transitive. On the other hand, since $G_{\Fb}$ is semi-algebraically connected, $\Ad(L_{\Theta,\Fb})$ lies in the Levi factor $\hat L_{\Theta,\Fb}$ of $\Aut_1(\mathfrak g_{\Fb})$ (with respect to the maximal torus and the choice of positive roots obtained by pushing forward the ones we chose for $G_{\Fb}$ via $\Ad$). Via the semi-algebraic isomorphism 
\[\exp \colon \mathfrak u_{\Theta,\Fb}\to U_{\Theta,\Fb},\] 
we have a semi-algebraic action of $\hat L_{\Theta,\Fb}$ on $U_{\Theta,\Fb}$, which induces an action 
\[\hat\Phi_{\Fb} \colon \hat L_{\Theta,\Fb}\to{\rm Sym}(\mathcal S_{\Fb})\]
of $\hat L_{\Theta,\Fb}$ on $\mathcal S_{\Fb}$ that extends the $L_{\Theta,\Fb}$ action, i.e.\ $\hat\Phi_{\Fb}\circ\Ad=\Phi_{\Fb}$. Let $\hat L_{\Theta,\Fb}^*$ denote the kernel of $\hat\Phi_{\Fb}$. 

From the definition of $\Aut_1(\mathfrak g_{\Fb})$, one sees that $\hat L_{\Theta,\Fb}$ leaves invariant every root space of $\mathfrak g$, and hence $\mathfrak u_{\alpha,\Fb}$ for all $\alpha\in\Theta$. Since $\Ad(L_{\Theta,\Fb}^\circ)\subset \hat L_{\Theta,\Fb}$ is a normal subgroup, the same argument as the proof of \Cref{cor: GW consequences} implies that the image of $\hat\Phi_{\Fb}$ lies in $\{\pm1\}^{\Theta}$. 

An application of the transfer principle to a result of Guichard and Wienhard \cite[Proposition 5.2]{guichard2025generalizing} yields the following:

\begin{theorem}
\label{thm:ActionLThetaOnSemigroups}
 If $G_{\Fb}$ admits a $\Theta$-positive structure, then the image of $\hat\Phi_{\Fb}$ is $\{\pm1\}^\Theta$. In particular, the action of $\hat{L}_{\Theta,\Fb}/\hat{L}_{\Theta,\Fb}^*$ on $\mathcal S_{\Fb}$ is transitive and free.
\end{theorem}

Notice that for every $C\in\mathcal S_{\Fb}$, the set
\[C^{-1}\coloneqq\{u^{-1}\mid u\in C\}\]
also lies in $\mathcal{S}_{\Fb}$. As a particular case of \Cref{thm:ActionLThetaOnSemigroups}, we have the following corollary.

\begin{corollary}
There exists $w_0\in \hat L_{\Theta,\Fb}$ such that 
$w_0\cdot C=C^{-1}$ for all $C\in\mathcal S_{\Fb}$. Furthermore, $w_0$ is unique up to right multiplication by an element in $\hat{L}_{\Theta,\Fb}^*$.
\end{corollary}

\begin{proof}
By the transfer principle, it suffices to prove the corollary in the case $\Fb=\Rb$.
By \Cref{thm:ActionLThetaOnSemigroups} there is some $w_0\in \hat L_{\Theta,\Rb}$ such that $\hat\Phi_{\Rb}(w_0)=(-1,-1,\dots,-1)$. For any $C\in\mathcal S_{\Rb}$, \Cref{theorem: GWU} implies that for each $\alpha\in\Theta$, there is a $L_{\Theta,\Rb}^\circ$-invariant sharp closed convex cone $c_\alpha\subset\mathfrak u_{\alpha,\Rb}$ such that $C$ is the interior of the semigroup generated by $\bigcup_{\alpha\in\Theta}\{\exp(X) \mid X\in c_\alpha\}$.
Then $C^{-1}$ is the interior of the semigroup generated by $\bigcup_{\alpha\in\Theta}\{\exp(X) \mid X\in -c_\alpha\}$, so by the definition of $\hat\Phi_{\Rb}$, we have $C^{-1}=w_0\cdot C$.
\end{proof}

If $C\in\mathcal S_{\Fb}$, then its closure $\overline{C}$ is necessarily a closed semi-algebraic semigroup of $U_{\Theta,\Fb}$. As another consequence of \Cref{theorem: GWU}, we have the following sharpness result for $\overline{C}$:

\begin{proposition}
\label{propo:AcuteSemigroup}
    For each $C\in\mathcal S_{\Fb}$, we have $\overline{C}\cap\overline{C}^{-1}=\{\Id\}$. In particular, if $u,u^{-1}\in\overline{C}$, then $u=\Id$. 
\end{proposition}
\begin{proof} 
By the transfer principle, it suffices to prove this statement in the case $\Fb=\Rb$.
Denote by $c_\alpha\subset\mathfrak u_{\alpha,\Rb}$ the  $L_{\Theta,\Rb}^\circ$-invariant sharp closed convex cone such that $\overline{C}$ is the semigroup generated by $\bigcup_{\alpha\in\Theta}\{\exp(X) \mid X\in c_\alpha\}$.
Suppose that $u\in \overline{C}\cap\overline C^{-1}$.
We can write
\[u = \prod_{i=1}^n \exp(X_i)= \prod_{i=j}^m \exp(-Y_j)\]
with $X_i \in c_{\alpha_i}$ and $-Y_j \in -c_{\alpha_j}$.
Applying the logarithm and calculating modulo $[\mathfrak u, \mathfrak u]$ we obtain
\[\log u = \sum_{i=1}^n X_i = \sum_{j=1}^m -Y_j \mod [\mathfrak u_{\Theta,\Rb}, \mathfrak u_{\Theta,\Rb}]~.\]
Recall that $\mathfrak u_{\Theta,\Rb} = [\mathfrak u_{\Theta,\Rb}, \mathfrak u_{\Theta,\Rb}] \oplus \bigoplus_{\alpha \in \Theta} \mathfrak u_{\alpha,\Rb}$.
In particular, this implies that
\[\sum_{i=1}^n X_i = \sum_{j=1}^m -Y_j \in \bigoplus_{\alpha \in \Theta} \mathfrak u_{\alpha,\Rb}~.\]
By the sharpness of $c_\alpha \subset \mathfrak u_{\alpha,\Rb}$ for all $\alpha \in \Theta$, we conclude that
\[\sum_{i=1}^n X_i = \sum_{j=1}^m -Y_j=0~.\]
Again sharpness implies that $X_i=0$ for all $i=1,\ldots,n$, in particular $u = \Id$.
\end{proof}

\subsection{Diamonds and positive triples of flags}
\label{s: Diamonds}
From now on, we assume that $G$ admits a $\Theta$-positive structure and we choose once and for all a semi-algebraically connected component $U_{\Theta,\Fb}^{>0}$ of $U_{\Theta,\Fb}^\pitchfork$ which is a semigroup, and let $U_{\Theta,\Fb}^{<0}\coloneqq(U_{\Theta,\Fb}^{>0})^{\rm opp}$. Also, let $U_{\Theta,\Fb}^{\ge 0}$ and $U_{\Theta,\Fb}^{\le 0}$ denote the closures in $U_{\Theta,\Fb}$ of $U_{\Theta,\Fb}^{>0}$ and $U_{\Theta,\Fb}^{<0}$ respectively.

\begin{definition}
The \emph{standard diamond} and \emph{standard opposite diamond} are the subsets of $\Fc_{\Theta,\Fb}$ given by
\[\Diam_{\std} \coloneqq \{u\cdot p_\Theta^\opp \mid u\in U_{\Theta,\Fb}^{>0}\}\quad\text{and}\quad\Diam_{\std}^\opp \coloneqq \{ u \cdot p_\Theta^\opp \mid u \in U_{\Theta,\Fb}^{<0}\}~.\]
A \emph{diamond} in $\mathcal F_{\Theta,\Fb}$ is then a $g$-translate of the standard diamond for some $g \in \Aut_1(\mathfrak{g_{\Fb}})$. If $\Diam=g\cdot \Diam_{\std}$ for some $g \in \Aut_1(\mathfrak{g_{\Fb}})$, we say that the pair of points $(g\cdot p_{\Theta},g\cdot p_{\Theta}^{\rm opp})$ are \emph{extremities} of $\Diam$ and $g\cdot \Diam_{\std}^{\rm opp}$ is an \emph{opposite} of $\Diam$. 
\end{definition}

\begin{remark}\label{GW diamonds}
Our notion of diamonds is not exactly the same as the notion used in Guichard--Wienhard \cite{guichard2025generalizing}. For them, a diamond is the data of a triple $(\Diam,x,y)$, where $\Diam$ is a diamond as defined above and $(x,y)$ are extremities of $\Diam$. They then say that if $g \cdot (\Diam_{\std},p_{\Theta},p_{\Theta}^{\rm opp})=(\Diam,x,y)$ for some $g\in \Aut_1(\mathfrak g_{\Fb})$, then $g \cdot (\Diam_{\std}^{\rm opp},p_{\Theta},p_{\Theta}^{\rm opp})$ is an \emph{opposite} of $(\Diam,x,y)$. For the purpose of discussing semi-algebraicity however, it is more convenient for us that diamonds are subsets of $\Fc_{\Theta,\Fb}$ instead of triples, hence the change in terminology.

With Guichard and Wienhard's definitions, it is easy to see that diamonds have unique opposites. Indeed, if $g,g'\in \Aut_1(\mathfrak g_{\Fb})$ such that $g \cdot (\Diam_{\std},p_{\Theta},p_{\Theta}^{\rm opp})=g' \cdot (\Diam_{\std},p_{\Theta},p_{\Theta}^{\rm opp})$ then the freeness in \Cref{cor: GW consequences} implies that $g^{-1}g'\in\hat L_{\Theta,\Fb}^*$, and so $g \cdot (\Diam_{\std}^{\rm opp},p_{\Theta},p_{\Theta}^{\rm opp})=g' \cdot (\Diam_{\std}^{\rm opp},p_{\Theta},p_{\Theta}^{\rm opp})$. With our definition, this says that if we choose a pair of extremities $(a,b)$ for a diamond $\Diam$, then there is a unique opposite to $\Diam$ with extremities $(a,b)$. 

However, without specifying the extremities, the uniqueness of opposites of a diamond is not as clear (in general, a diamond does not determine its extremities). However, we will prove later that this uniqueness holds, see \Cref{prop: transverse set to a full diamond}.
\end{remark}

If $(x,y)$ is a transverse pair of flags, then any diamond with extremities $(x,y)$ is a semi-algebraically connected component of the set of flags in $\Fc_{\Theta,\Fb}$ that are transverse to both $x$ and $y$, and hence is semi-algebraic, see \Cref{thm_RConnSemiAlgConn}.

Since $\Aut_1(\mathfrak g_{\Fb})$ acts transitively on the set of transverse pairs of flags in $\Fc_{\Theta,\Fb}$, the description of the structure of $U_{\Theta,\Fb}^\pitchfork$ given in \Cref{subsection: PositiveStructures} implies that any transverse pair of flags $(x,y)$ in $\Fc_{\Theta,\Fb}$ are the extremities of exactly $2^{|\Theta|}$ different diamonds. Furthermore, \Cref{thm:ActionLThetaOnSemigroups} implies that the subgroup of $\Aut_1(\mathfrak g_{\Fb})$ that fixes both $x$ and $y$ acts transitively on the collection of diamonds that have $(x,y)$ as extremities.

With the notion of diamonds, one can define positive triples of flags.

\begin{definition}
A triple $(x,y,z)\in \mathcal F_{\Theta,\Fb}^3$ is \emph{positive} if $x$ and $z$ are transverse and $y$ belongs to a diamond with extremities $(x,z)$. 
\end{definition}

Guichard and Wienhard \cite{GW} (after applying the transfer principle) verified that positivity of a triple is invariant under permutation:
\begin{proposition}[{\cite[Proposition 13.15]{guichard2025generalizing}}]\label{permutation positive triple}
   Let $(x_1,x_2, x_3)\in \mathcal F_{\Theta,\Fb}^3$ be a positive triple. Then $(x_{\sigma(1)}, x_{\sigma(2)}, x_{\sigma(3)})$ is positive for every permutation $\sigma$ of $\{1,2,3\}$.
\end{proposition}

If $(x,y,z)$ is a positive triple of flags in $\Fc_{\Theta,\Fb}$, we denote the diamond that contains $y$ with extremities $(x,z)$ by $\Diam_y(x,z)$, and denote by $\Diam_y^{\rm opp}(x,z)$ the unique diamond opposite to $\Diam_y(x,z)$ with extremities $(x,z)$, see \Cref{GW diamonds}. Notice that if $g\in \Aut_1(\mathfrak g_{\Fb})$, then 
\[g\cdot\Diam_y(x,z)=\Diam_{g\cdot y}(g\cdot x,g\cdot z)\quad\text{and}\quad g\cdot\Diam_y^{\rm opp}(x,z)=\Diam_{g\cdot y}^{\rm opp}(g\cdot x,g\cdot z).\]

The following is a basic property of diamonds proven in Guichard--Wienhard \cite{guichard2025generalizing} and Guichard--Labourie--Wienhard \cite{GLW} (after applying the transfer principle).

\begin{proposition}[{\cite[Propositions 13.7]{guichard2025generalizing}}]\label{proposition: opposite diamonds}
If $(x,y,z)\in\Fc_{\Theta,\Fb}^3$ is a positive triple of flags, then there is a unique diamond $\Diam \subset \Diam_y(x,z)$ with extremities $(x,y)$. Furthermore, $\Diam=\Diam_z^{\rm opp}(x,y)$. 
\end{proposition}

\begin{remark}
The second statement of \Cref{proposition: opposite diamonds} was not stated explicitly in {\cite[Propositions 13.7]{guichard2025generalizing}}, but it follows from the following computation. By translating by an element in $\Aut_1(\mathfrak g_{\Fb})$, we may assume that $(x,z)=(p_\Theta,p_\Theta^{\rm opp})$ and $\Diam_y(x,z)=\Diam  _{\std}$. Then there is a unique $u\in U_{\Theta,\Fb}^{>0}$ such that $u\cdot z=y$. The fact that $U_{\Theta,\Fb}^{>0}$ is a semigroup implies that $u\cdot\Diam_{\std}$ is the unique diamond in $\Diam_{\std}$ with extremities $(x,y)$. It thus suffices to verify that $\Diam_z^{\rm opp}(x,y)=u\cdot\Diam_{\std}$. Notice that $u^{-1}\in U_{\Theta,\Fb}^{<0}$, so we have that $u^{-1} \cdot z \in \Diam_\std^\opp$, hence
\[u\cdot \Diam_\std= u\cdot \Diam^\opp_{u^{-1}\cdot z}(x, z)=\Diam_z^{\rm opp}(x,y).\]
\end{remark}

\subsection{Positive quadruples of flags}

On a flag variety with a $\Theta$-positive structure, the notion of positive triple extends to a notion of positive $n$-tuple for all $n\geqslant 3$. We start by defining positive quadruples, which play a special role in the theory.

\begin{definition}
A quadruple $(x,y,z,t)\in \mathcal F_{\Theta,\Fb}^4$ is called \emph{positive} if the following equivalent conditions are satisfied:
\begin{itemize}
    \item $y$ and $t$ lie in opposite diamonds with extremities $(x,z)$;
    \item there exist $g\in \Aut_1(\mathfrak g_{\Fb})$ and $u,v\in U_{\Theta,\Fb}^{>0}$  such that
    \[(x,y,z,t) = g\cdot (p_\Theta^\opp, u \cdot p_\Theta^\opp, u v\cdot p_\Theta^\opp, p_\Theta)~.\]
\end{itemize}
\end{definition}

Unlike for triples, positivity of a quadruple is not invariant under all permutations. Guichard and Wienhard \cite{guichard2025generalizing} (after applying the transfer principle) verified that it is nevertheless invariant under the dihedral subgroup of permutations:

\begin{proposition}[{\cite[Proposition 13.19]{guichard2025generalizing}}]
    If $(x,y,z,t)\in\Fc_{\Theta,\Fb}^4$ is a positive quadruple, then so are $(t,z,y,x)$ and $(y,z,t,x)$.
\end{proposition}

As an immediate corollary, we deduce the following:

\begin{corollary}\label{opposites transverse}
If the diamonds $\Diam$ and $\Diam'$ in $\Fc_{\Theta,\Fb}$ are opposite, then every flag in $\Diam$ is transverse to every flag in $\Diam'$. 
\end{corollary}

The following is a consequence of results of Guichard and Wienhard \cite{guichard2025generalizing} and Guichard, Labourie and Wienhard \cite{GLW} (after applying the transfer principle):

\begin{proposition}[{\cite[Lemma 13.21]{guichard2025generalizing}},{\cite[Corollary 3.9]{GLW}}]\label{lem: closure of diamond}
If $(x,y,z,t)\in\Fc_{\Theta,\Fb}^4$ is a positive quadruple of flags, then 
\[\overline{\Diam}_x^{\rm opp}(y,z)\subset \Diam_z(x,t).\]
\end{proposition}

\Cref{lem: closure of diamond} can be used to deduce the following statement about diamonds associated to positive triples of flags.

\begin{proposition} \label{lem: intersection closures diamonds}
    Let $(x,y,z)$ be a positive triple in $\Fc_{\Theta,\Fb}$.
    Then 
    \[\overline{\Diam}_z^\opp(x,y) \cap\overline{\Diam}_x^\opp(y,z) = \{y\}. \]
\end{proposition}
\begin{proof}
    Choose $t \in \Fc_{\Theta,\Fb}$ such that $(x,y,z,t)$ is positive. Without loss of generality, we can assume that $t = p_\Theta$, $y= p_\Theta^\opp$, $x \in \Diam_\std$ and $z\in \Diam_\std^\opp$. By \Cref{lem: closure of diamond} we have $\overline{\Diam}_z^\opp(x,y) \subset \Diam_x(z,t)$ and $\overline{\Diam}_x^\opp(y,z) \subset \Diam_z(x,t)$. In particular, both closures of diamonds are contained in the locus transverse to $t$, that is, $U_{\Theta,\Fb} \cdot p_\Theta^\opp$.

    By \Cref{proposition: opposite diamonds}, $\Diam_z^\opp(x,y)=\Diam_t^\opp(y,x) \subset \Diam_x(y,t)=\Diam_\std = U_{\Theta,\Fb}^{>0} \cdot p_\Theta^\opp$. Since $\overline{\Diam}_z^\opp(x,y) \subset U_{\Theta,\Fb} \cdot p_\Theta^\opp$, it follows that  \[\overline \Diam_z^\opp(x,y) \subset U_{\Theta,\Fb}^{\geqslant 0} \cdot p_\Theta^\opp~.\] Similarly, we have 
    \[\overline \Diam_x^\opp(y,z) \subset U_{\Theta,\Fb}^{\leq 0} \cdot p_\Theta^\opp~.\]
Since $U_{\Theta,\Fb}^{\geqslant 0} \cap U_{\Theta,\Fb}^{\leq 0} = \{\Id\}$ by \Cref{propo:AcuteSemigroup}, we conclude that 
    \[\overline \Diam_z^\opp(x,y)\cap \overline \Diam_x^\opp(y,z) = \{p_\Theta^\opp\} = \{y\}~. \qedhere\]
\end{proof}

\subsection{Positive tuples of flags}
We can now define positive $n$-tuples of flags:

\begin{definition}
\label{def:PositiveTuples}
    Let $n\ge 3$. An $n$-tuple $(x_1,\ldots, x_n)\in \Fc_{\Theta,\Fb}^n$ is \emph{positive} if the following equivalent conditions are satisfied:
    \begin{enumerate}
        \item For all integers $1\le i<n$, there exists a diamond $\Diam_{i,n}$ with extremities $(x_i,x_n)$ such that 
        \begin{itemize}
            \item for all integers $i<k<n$, $x_k \in \Diam_{i,n}$,
            \item for all integers $1\le k<i$, $x_k$ lies in the diamond with extremities $(x_i,x_n)$ that is opposite to $\Diam_{i,n}$;
        \end{itemize} 
        \item There exists $g\in \Aut_1(\mathfrak g_{\Fb})$ and $(u_2,\ldots ,u_{n-1})\in (U_{\Theta,\Fb}^{>0})^{n-2}$ such that
        \[(x_1,\ldots, x_n) = g\cdot (p_\Theta^\opp, u_2\cdot p_\Theta^\opp, u_2u_3\cdot p_\Theta^\opp, \ldots, u_2\ldots u_{n-1}\cdot p_\Theta^\opp, p_\Theta)~.\]
    \end{enumerate}
     We write $(\Fc_{\Theta,\Fb}^n)_{>0}$ for the set of positive $n$-tuples in $\Fc_{\Theta,\Fb}$. 
\end{definition}

Again, Guichard and Wienhard \cite{guichard2025generalizing} verified (after applying the transfer principle) that positivity is invariant under dihedral permutations:

\begin{proposition}[{\cite[Lemma 13.24]{guichard2025generalizing}}]
    If $(x_1,x_2, \ldots, x_n)$ is a positive $n$-tuple, then so are $(x_n,\ldots, x_2, x_1)$ and $(x_2,\ldots, x_n, x_1)$.
\end{proposition}

In particular, $(x_1,\dots,x_n)\in\Fc_{\Theta,\Fb}^n$ is positive if and only if for all $i<j$, there exists a diamond $\Diam_{i,j}$ with extremities $x_i$ and $x_j$ such that 
\begin{itemize}
    \item for all $i<k<j$, $x_k \in \Diam_{i,j}$,
    \item for all $k<i$ and $k>j$, $x_k$ lies in the diamond with extremities $(x_i,x_j)$ that is opposite to $\Diam_{i,j}$ .
\end{itemize}  

\begin{example}
Recall that $\mathbf{P}^1(\Fb) \simeq \Rb \sqcup \{\infty\}$ is the unique flag variety of the group $G_{\Fb} = \PSL_2(\Fb)$. The order on $\Fb$ induces a cyclic order on $\mathbf{P}^1(\Fb)$, invariant under the action of $\PSL_2(\Fb)$. The elements of $\Aut_1(\mathfrak{g}(\Fb)) = \PGL_2(\Fb)$ either preserve or reverse the cyclic order. A tuple $(x_1,\ldots, x_n) \in \mathbf{P}^1(\Fb)$ is positive if and only if either $(x_1,\ldots, x_n)$ or $(x_n, \ldots, x_1)$ is cyclically ordered.
\end{example}

Verifying positivity of an $n$-tuple with the definition can be cumbersome, and it is often useful to characterize it by the positivity of well-chosen sub-tuples. Later we will use the following recursive characterization, which is a rephrasing of a result of Guichard, Labourie, and Wienhard \cite{GLW}.

\begin{proposition}[{\cite[Proposition 3.1~(4)]{GLW}}]\label{lem: AddingElemToPosTuple}
    Let $(x_0,\ldots, x_n,x_{n+1}) \in \mathcal F_{\Theta,\Fb}^{n+2}$. If $(x_0, \ldots, x_n)$ is a positive tuple and there exists $1 \leq j \leq n-1$ such that\break $(x_0, x_j, x_n, x_{n+1})$ is positive, then $(x_0, \ldots, x_n, x_{n+1})$ is positive.
\end{proposition}

As consequences of \Cref{lem: AddingElemToPosTuple}, we have the following pair of corollaries, both of which are proven via a straightforward induction argument using \Cref{lem: AddingElemToPosTuple} as the inductive step. 

\begin{corollary}\label{corollarly: gluing}
    A tuple of flags $(a,x_1, \ldots, x_k, b, y_1, \ldots ,y_l)$ in $\Fc_{\Theta,\Fb}$ is positive if and only if $(a, x_1, \ldots, x_k, b)$ and $(b, y_1, \ldots, y_l, a)$ are positive tuples, and there exists $1\leq i \leq k$ and $1\leq j \leq l$ such that $(a,x_i,b,y_j)$ is positive. 
\end{corollary}

\begin{corollary}\label{check quadruples}
    A tuple of flags $(x_1,\dots,x_n)$ is positive if and only if for all $1\le i<j<k<l\le n$ the quadruple $(x_i, x_j, x_k, x_l)$ is positive.
\end{corollary}

\Cref{lem: AddingElemToPosTuple} also has the following implication for intersections of diamonds.

\begin{corollary}\label{intersection of diamonds}
    If $(a,b,c,d)$ is a positive quadruple of flags in $\Fc_{\Theta,\Fb}$, then 
    \[\Diam_c(b,d)\cap\Diam_b(a,c)=\Diam_a^{\rm opp}(b,c).\]
\end{corollary}

\begin{proof}
Let $e\in\Fc_{\Theta,\Fb}$. Since $(a,b,c,d)$ is positive, $e\in \Diam_c(b,d)\cap\Diam_b(a,c)$ if and only if $(a,b,e,d)$ and $(a,e,c,d)$ are positive. By \Cref{lem: AddingElemToPosTuple}, this is equivalent to requiring that the tuple $(a,b,e,c,d)$ is positive, i.e.\ $e\in\Diam_a^{\rm opp}(b,c)$.
\end{proof}

We verify that positivity of tuples of flags is an open and semi-algebraic condition:

\begin{proposition}
\label{propo: positive n-tuples is semi-algebraic}
    The set of positive $n$-tuples of flags is an open semi-algebraic subset of $\Fc_{\Theta,\Fb}^n$.
\end{proposition}
\begin{proof}
    The fact that $(\Fc_{\Theta,\Fb}^n)_{>0}\subset\Fc_{\Theta,\Fb}^n$ is open is an immediate consequence of \cite[Proposition 13.27]{guichard2025generalizing} and the transfer principle. We now verify that\break $(\Fc_{\Theta,\Fb}^n)_{>0}\subset\Fc_{\Theta,\Fb}^n$ is semi-algebraic. Since $\Aut_1(\mathfrak g_{\Fb})$ and $ U_{\Theta,\Fb}^{n-2}$ are semi-algebraic groups whose action on $\Fc_{\Theta,\Fb}^n$ is semi-algebraic, the map 
    \[\pi \colon  \Aut_1(\mathfrak g_{\Fb}) \times U_{\Theta,\Fb}^{n-2}  \mapsto  \Fc_{\Theta,\Fb}^n\]
    given by
    \[(g,u_1,\ldots, u_{n-2}) \mapsto g\cdot (p_\Theta, p_\Theta^\opp, u_1 \cdot p_\Theta^\opp, u_1 u_2 \cdot p_\Theta^\opp, \ldots, u_1 \cdots u_n \cdot p_\Theta^\opp)
    \]
    is semi-algebraic. Thus, $(\Fc_{\Theta,\Fb}^n)_{>0} = \pi( \Aut_1(\mathfrak{g}_{\Fb}) \times (U_{\Theta,\Fb}^{>0})^{n-2})$ is semi-algebraic.
\end{proof}

\begin{remark}
    The case of $\PSL_2(\Fb)$ is a particular case of a common feature to Hermitian groups of tube type. For such groups, there exists a partial cyclic order on $\Fc_{\Theta, \Fb}$ which is preserved by $G_{\Fb}$ but reversed by some element of $\Aut_1(\mathfrak{g}(\Fb))$. Positive tuples are tuples that are cyclically ordered up to dihedral permutation and thus form two semi-algebraically connected components which are preserved by $G_{\Fb}$. Choosing one component amounts to choosing a cyclic order or, equivalently, a consistent choice of diamond with two given extremities. 

    For other groups $G_{\Fb}$ that admit a $\Theta$-positive structure that are not Hermitian, the group $\Aut_1(\mathfrak g_{\Fb})$ can be semi-algebraically connected (for example, $G_{\Fb}=\SL_{2k+1}(\Fb)$), in which case the set of positive tuples is connected and there is no consistent way of choosing a diamond between two points. Yet, it is a good mental picture to think of positive tuples as ``cyclically ordered tuples'', up to an indeterminacy which is lifted once we are given a sub-triple.

    In fact, one could always construct a finite covering $\Fc_{\Theta,\Rb}^*$ of $\Fc_{\Theta,\Rb}$ (namely, $G_{\Rb} / P_{\Theta,\Rb}^*$ where $P_{\Theta, \Rb}^* = L_{\Theta_{\Rb}}^* U_\Theta$) over which the $\Theta$-positive structure defines a partial cyclic order, and such that a tuple in $\Fb_{\Theta,\Rb}$ is positive if and only if admits a cyclically ordered lift to $\Fb_{\Theta,\Rb}^*$. While we believe that developing this point of view could be useful, we decided to stick with the classical presentation of $\Theta$-positivity in the present paper. 
\end{remark}

\subsection{Uniqueness of opposites}

The next proposition implies that every diamond has a unique opposite.

\begin{proposition} \label{prop: transverse set to a full diamond}
    Every diamond $\Diam$ in $\Fc_{\Theta,\Fb}$ has a unique opposite, which is the set
    \[\{x\mid \forall y\in \overline \Diam, x\pitchfork y\}~,\]
    where $\overline{\Diam}$ denotes the closure of $\Diam$ in $\Fc_{\Theta,\Fb}$.
\end{proposition}

\begin{proof}
Let $\Diam$ be a diamond in $\Fc_{\Theta,\Fb}$, and let $(a,b)$ be a choice of extremities of $\Diam$. By applying an element in $\Aut_1(\mathfrak g_{\Fb})$, we may assume that $(\Diam,a,b)=(\Diam_{\std},p_\Theta,p_\Theta^\opp)$. It suffices to show that 
\[\Diam_{\std}^{\rm opp}=\{x\mid \forall y\in \overline \Diam_{\std}, x\pitchfork y\}.\]

First, we prove that $\Diam_{\std}^{\rm opp}\subset \{x\mid \forall y\in \overline \Diam, x\pitchfork y\}$. Pick any $c\in\Diam_{\std}^{\rm opp}$, and choose $d,e\in\Diam_{\std}^{\rm opp}$ such that $(a,d,c,e,b)$ is positive. Since $(a,d,c,e)$ is positive, it follows from definition that $c\in\Diam_a^{\rm opp}(d,e)$. 
Then by \Cref{opposites transverse}, $c$ is transverse to every flag in $\Diam_a(d,e)$. Since $(a,d,e,b)$ is positive, \Cref{lem: closure of diamond} implies that
\[\overline{\Diam}_e^{\rm opp}(a,b)\subset\Diam_a(d,e),\]
and so $c$ is transverse to every flag in $\overline{\Diam}_e^{\rm opp}(a,b)=\overline{\Diam}_{\std}$. 

Next, we prove that $\Diam_{\std}^{\rm opp}\supset \{x\mid \forall y\in \overline \Diam, x\pitchfork y\}$. Let $x\in\Fc_{\Theta,\Fb}$ be transverse to every point in $\overline \Diam$. In particular, $x$ is transverse to $p_\Theta$ and $p_\Theta^\opp$, hence we can write $x= v\cdot p_\Theta^\opp$ for some $v\in U_{\Theta,\Fb}^\pitchfork$. It suffices to show that $v\in U_{\Theta,\Fb}^{<0}$. For every $u\in U_{\Theta,\Fb}^{>0}$, $x = v\cdot p_\Theta^\opp$ is transverse to $u\cdot p_\Theta^\opp$, i.e.\ $v^{-1} u \in U_{\Theta,\Fb}^\pitchfork$. We thus get that $v^{-1}\cdot U_{\Theta,\Fb}^{>0} \subset U_{\Theta,\Fb}^\pitchfork$. We claim that $v^{-1}\cdot U_{\Theta,\Fb}^{>0} \subset U_{\Theta,\Fb}^{>0}$. Indeed, since $U_{\Theta,\Fb}^{>0}\subset U_{\Theta,\Fb}$ is open, for each $u\in U_{\Theta,\Fb}^{>0}$, we can find some $l\in T_{\Theta,\Fb}$ such that $l v l^{-1}$ is sufficiently close to the identity so that $lv^{-1}l^{-1} u$ is still in $U_{\Theta,\Fb}^{>0}$. Conjugating by $l^{-1}$, we get that $v^{-1} (l^{-1} u l)\in U_{\Theta,\Fb}^{>0}$. Hence $v^{-1}\cdot U_{\Theta,\Fb}^{>0}\cap U_{\Theta,\Fb}^{>0} \neq \emptyset$. Since $v^{-1}\cdot U_{\Theta,\Fb}^{>0} \subset U_{\Theta,\Fb}^\pitchfork$ and $U_\Theta^{>0}$ is a connected component of $U_{\Theta,\Fb}^\pitchfork$, we conclude that $v^{-1}\cdot U_{\Theta,\Fb}^{>0} \subset U_{\Theta,\Fb}^{>0}$. This is possible only if $v^{-1}\in U_{\Theta,\Fb}^{>0}$, or equivalently, if $v\in U_{\Theta,\Fb}^{<0}$.
\end{proof}

Henceforth, we denote the opposite of any diamond $\Diam$ in $\Fc_{\Theta,\Fb}$ by $\Diam^{\rm opp}$. Notice that $(\Diam^{\rm opp})^{\rm opp}=\Diam$.

As mentioned above, a diamond does not determine its extremities in general. However, using \Cref{prop: transverse set to a full diamond}, one can prove that if two pairs of extremities of a diamond share a common point, then they must be equal:

\begin{proposition}
    \label{lem: ExtremitiesDiamonds}
    If $(x,y)$ and $(x,z)$ are extremities of the same diamond, then $y=z$.
\end{proposition}
\begin{proof}
    Let $\Diam$ be  the diamond in $\Fc_{\Theta,\Fb}$ that have both $(x,y)$ and $(x,z)$ as extremities, and let $w\in\Diam$. By \Cref{proposition: opposite diamonds}, 
    \[\Diam^\opp_w(x,y) \subset \Diam_y(x,w)\quad\text{and}\quad\Diam^\opp_w(x,z) \subset \Diam_z(x,w).\] 
    At the same time, \Cref{prop: transverse set to a full diamond} implies that $\Diam^\opp_w(x,y)=\Diam^{\rm opp}=\Diam^\opp_w(x,z)$, so $\Diam_z(x,w)$ and $\Diam_y(x,w)$ intersect non-trivially. Since $\Diam_z(x,w)$ and $\Diam_y(x,w)$ are both semi-algebraically connected components of the set of flags in $\Fc_{\Theta,\Fb}$ that are transverse to both $w$ and $x$, they must agree.

    Applying an element $g\in \Aut_1(\mathfrak g_{\Fb})$, we can thus assume that $x=p_\Theta$, $w=p_\Theta^{\rm opp}$, and $\Diam_z(x,w)=\Diam_y(x,w)=\Diam_{\rm std}$. Let $u_y, u_z \in U_{\Theta,\Fb}^{>0}$ such that $y = u_y \cdot w$ and $z = u_z \cdot w$. Notice that $u_y\cdot\Diam_{\std}$ and $\Diam^\opp_w(x,y)$ both lie in $\Diam_{\std}$ and have $(x,y)$ as extremities, so \Cref{proposition: opposite diamonds} implies that $u_y\cdot\Diam_{\std}=\Diam^\opp_w(x,y)$. Similarly, $u_z \cdot \Diam_\std=\Diam^\opp_w(x,z)$. Since $\Diam^\opp_w(x,y) =\Diam^\opp_w(x,z)$, we have that $u_y^{-1} u_z$ and $u_z^{-1} u_y$ both leave $\Diam_{\std}$ invariant.
    
    Observe that $\{u\in U_{\Theta,\Fb}\mid u\cdot\Diam_{\std}=\Diam_{\std}\}\subset U_{\Theta,\Fb}^{\ge 0}$; indeed, if $u\in U_{\Theta,\Fb}$ satisfies $u\cdot\Diam_{\std}=\Diam_{\std}$, then $u\cdot p_\Theta^{\rm opp}\in \overline{\Diam}_{\std}=U_{\Theta,\Fb}^{\ge 0}\cdot p_\Theta^{\rm opp}$, and so $u\in U_{\Theta,\Fb}^{\ge 0}$. 
    Hence, $u_y^{-1} u_z, u_z^{-1} u_y \in U_{\Theta,\Fb}^{\geqslant 0}$.
    Then by \Cref{propo:AcuteSemigroup}, $u_z=u_y$, and thus $y=z$.
\end{proof}

\subsection{Positive maps and sequences}

The cyclic invariance of positivity for $n$-tuples allows to define more generally a \emph{positive map} from a cyclically ordered set $\Lambda$ to $\Fc_\Theta$. 

\begin{definition}
    Let $\Lambda$ be a cyclically ordered set.
    A map from $\Lambda$ to $\Fc_\Theta$ is \emph{positive} if it maps any cyclically order $n$-tuple to a positive $n$-tuple.
\end{definition}

By \Cref{check quadruples}, to verify that a map is positive, it it is enough to verify that it sends cyclically ordered quadruples to positive quadruples.

\begin{example}
    Recall that $\mathbf{P}^1(\Fb)$ admits two natural cyclic orders, which are reverses of each other. Pick one of them. A tuple $(x_1,\ldots, x_n)\in (\mathbf{P}^1(\Fb))^n$ is positive if and only if either $(x_1,\ldots, x_n)$ or $(x_n,\ldots, x_1)$ is cyclically ordered. Hence, a map $f\colon \Lambda \to \mathbf{P}^1(\Fb)$ is positive if and only if it is monotonous and injective.
\end{example}

If $\Lambda$ is a set with a total order $<$, then this total order induces a canonical cyclic order on $\Lambda$ for which a tuple is cyclically ordered if it is increasing up to a cyclic permutation. Namely:
\[\mathcal C = \{(a,b,c)\mid a<b<c \textrm{ or } b<c<a \textrm{ or } c<a<b\}~.\]
Thus it makes sense to discuss the notion of a positive map from a totally ordered set (such as $\Nb$ or $\Zb$ equipped with their standard order) to $\Fc_{\Theta,\Fb}$. With this, we may define the notion of a positive (infinite or bi-infinite) sequence.

\begin{definition}
    A sequence $(a_n)_{n\in \Nb} \in \Fc_\Theta^{\Nb}$ (resp.\ $(a_n)_{n\in \Zb} \in \Fc_\Theta^{\Zb}$) is \emph{positive} if the map $n\mapsto a_n$ is positive.
\end{definition}

A positive sequence $(a_n)_{n\in \Zb}$ in $\Fc_\Theta$ gives rise to a nested sequence of diamonds \[\overline \Diam_{a_0}^\opp(a_{-(n+1)}, a_{n+1}) \subset \Diam_{a_0}^\opp(a_{-n}, a_n) \quad (n\geqslant 1)\] by \Cref{lem: closure of diamond}. Hence, as a direct consequence of \Cref{propo:NestedIntersectionPropertyCantorComplete}, we obtain the following.

\begin{corollary}\label{corollary: diamond nonempty intersection over cantor complete fields}
    Let $\Fb$ be a real closed Cantor complete field and $(a_n)_{n\in \Zb}\in \Fc_\Theta^{\Zb}$ be a positive sequence. Then the intersection of diamonds \[\bigcap_{n\geqslant 1} \Diam_{a_0}^\opp(a_{-n},a_n)\] is non-empty.
\end{corollary}

Note however, that this intersection may or may not be a (``degenerate'') diamond or semi-algebraic itself; compare already the considerations in \Cref{subsection:CompletenessOrderedFields} in the one-dimensional case, where this intersection might not be an interval.
However, if $a_{+\infty}\coloneqq\lim_{n\to\infty}a_n$ and $a_{-\infty}\coloneqq \lim_{n\to\infty}a_{-n}$ both exist, then the intersection
\[\bigcap_{n\geqslant 1} \Diam_{a_0}^\opp (a_{-n}, a_n)\]
 is, in some sense, a ``degenerate diamond'' between $a_{+\infty}$ and $a_{-\infty}$. We will need the following more precise statements:

\begin{proposition}
\label{onesided squeeze}
Let $(a_n)_{n\geqslant 0}$ be a positive sequence in $\Fc_{\Theta,\Fb}$ that converges to a flag $a_\infty \in \Fc_{\Theta,\Fb}$. Then $(a_\infty, a_1, a_2,\ldots , a_n, \ldots )$ is positive, and for any $x\in \Diam_{a_2}^\opp(a_1, a_\infty)$, we have \[\bigcup_{n\geqslant 2} \Diam_x^\opp(a_1,a_n) = \Diam_x^\opp(a_1,a_\infty)~, \quad\text{and}\quad\bigcap_{n\geqslant 2} \overline{\Diam}_{a_1}^\opp(x,a_n) = \overline{\Diam}_{a_1}^\opp(x,a_\infty)~.\]
In particular, if $(b_n)_{n\in\Nb}$ is a sequence for which there is an increasing sequence of integers $(k_n)_{n\in\Nb}$ such that $b_n\in \Diam_{a_1}^{\rm opp}(a_{k_n},a_\infty)$ for all $n\in\Nb$, then $b_n\to a_\infty$ as $n\to\infty$.
\end{proposition}

\begin{proof}
First, we show that $(a_\infty,a_1,a_2,\dots,a_n,\dots)$ is positive. Observe that for all positive integers $n\ge 3$, $(a_n,a_{n+1},a_0,a_1)$ is positive, so \Cref{lem: closure of diamond} implies that  
\[\overline{\Diam}_{a_1}^{\rm opp}(a_0,a_{n+1})\subset\Diam_{a_{n+1}}(a_1,a_n)=\Diam_{a_2}^{\rm opp}(a_1,a_n).\]
For all integers $m\ge n+2$, $(a_0,a_1,a_{n+1},a_m)$ is positive, so $a_m\in \overline{\Diam}_{a_1}^{\rm opp}(a_0,a_{n+1})$, which implies that $a_\infty\in \overline{\Diam}_{a_1}^{\rm opp}(a_0,a_{n+1})$. It follows from the above inclusion that $(a_1,a_2,a_n,a_\infty)$ is positive for all positive integers $n\ge 3$. Since $(a_1,\dots,a_n)$ is positive, \Cref{lem: AddingElemToPosTuple} implies that $(a_\infty,a_1,\dots,a_n)$ is positive for all integers $n\ge 3$. The positivity of  $(a_\infty,a_1,a_2,\dots,a_n,\dots)$ follows.

Next, we prove the first equality of the lemma. By \Cref{lem: AddingElemToPosTuple}, the sequence $(a_\infty,x,a_1,a_2,\dots,a_n,\dots)$ is positive. In particular, for all integers $n\ge 2$, 
\[\Diam_x^{\rm opp}(a_1,a_n)=\Diam_{a_\infty}^{\rm opp}(a_1,a_n)\subset\Diam_{a_n}(a_1,a_\infty)=\Diam_x^{\rm opp}(a_1,a_\infty),\]
where the first equality holds because $(a_1,a_n,a_\infty,x)$ is positive, the inclusion holds by \Cref{proposition: opposite diamonds} and the positivity of the triple $(a_1,a_n,a_\infty)$, and the the second equality holds because $(a_1,a_n,a_\infty,x)$ is positive. Taking union over all integers $n\ge 2$ gives
\[\bigcup_{n\geqslant 2} \Diam_x^\opp(a_1,a_n)\subset \Diam_x^\opp(a_1,a_\infty).\]
Thus, to prove the first equality, we need to show the reverse inclusion, i.e.\ that if $b\in\Fc_{\Theta,\Fb}$ is a flag such that $(x,a_1,b,a_\infty)$ is positive, then $(x,a_1,b,a_n)$ is positive for some integer $n\ge 2$. Since $a_n\to a_\infty$ as $n\to\infty$, this is immediate from \Cref {propo: positive n-tuples is semi-algebraic}. 

To prove the second equality, note that we can rewrite the first equality as
\[\bigcup_{n\geqslant 3} \Diam_{a_2}(a_1,a_n) = \Diam_{a_2}(a_1,a_\infty)~.\]
Applying it to the sequence $(x,a_1, a_2, \ldots a_n,\ldots)$, we get
\begin{equation} \label{eq: union diamonds}
    \bigcup_{n\geqslant 2} \Diam_{a_1}(x,a_n) = \Diam_{a_1}(x,a_\infty)~.
\end{equation}

If $A$ is a subset of $\Fc_{\Theta,\Fb}$, define
\[A^\pitchfork \coloneqq \{z\in \Fc_{\Theta,\Fb}\mid z\pitchfork y \textrm{ for all } y\in A\}~.\]
The equality \eqref{eq: union diamonds} gives
\[\Diam_{a_1}(x,a_\infty)^\pitchfork = \left(\bigcup_{n\geqslant 2} \Diam_{a_1}(x,a_n)\right)^\pitchfork = \bigcap_{n\geqslant 2} \Diam_{a_1}(x,a_n)^\pitchfork~.\]
Since $\Diam^\pitchfork = \overline{\Diam}^\opp$ for every diamond $\Diam$ by \Cref{prop: transverse set to a full diamond}, we conclude that
\[\overline{\Diam}^\opp_{a_1}(x,a_\infty) = \bigcap_{n\geqslant 2} \overline{\Diam}^\opp_{a_1}(x,a_n)~.\]

Finally, observe that the assumption on the sequence $(b_n)_{n\in\Nb}$ implies that for all integers $n \ge 2$, $\Diam_{a_1}^{\rm opp}(a_n,a_\infty)$ contains a tail of $(b_n)_{n\in\Nb}$. At the same time, 
\[\overline{\Diam}_{a_1}^{\rm opp}(a_n,a_\infty)\subset \overline{\Diam}_{a_1}^{\rm opp}(a_n,x)\cap\overline{\Diam}_x^{\rm opp}(a_1,a_\infty),\] 
so every sub-sequential limit of $(b_n)_{n\in\Nb}$ lies in 
\[\overline{\Diam}_{a_1}^\opp(a_\infty,x)\cap \overline{\Diam}_x^{\rm opp}(a_1,a_\infty)=\{a_\infty\},\]
where the equality is by \Cref{lem: intersection closures diamonds}.
\end{proof}

\begin{corollary}\label{lem: squeeze lemma for diamond not real}
    Let $(a_n)_{n\in \Zb}$ be a positive sequence in $\Fc_{\Theta,\Fb}^{\Zb}$. Suppose that there is some $a_{+\infty},a_{-\infty}\in\Fc_{\Theta,\Fb}$ such that $a_n\to a_{+\infty}$ and $a_{-n}\to a_{-\infty}$ as $n\to+\infty$. 
    \begin{enumerate}
    \item If $a_{+\infty}=a_{-\infty}\eqqcolon a_\infty$, then
    \[\bigcap_{n\geqslant 1} \Diam_{a_0}^\opp (a_{-n}, a_n) = \{a_\infty\}~.\]
    \item If $(a_{-\infty},a_{-1},a_1,a_{+\infty})$ is positive, then 
    \[\bigcap_{n\geqslant 1} \Diam_{a_0}^\opp (a_{-n}, a_n) = \overline{\Diam}_{a_0}^\opp(a_{-\infty},a_{+\infty})~.\]
    \end{enumerate}
\end{corollary}

\begin{proof}
    Since $a_{-n} \to a_{-\infty}$ as $n\to +\infty$, \Cref{onesided squeeze} gives
    \[\bigcap_{n\geqslant 1} \Diam_{a_0}^\opp (a_{-n}, a_n)\subset \bigcap_{n\geqslant 1} \Diam_{a_0}^\opp (a_{-n}, a_1) \subset \overline \Diam_{a_0}^\opp(a_{-\infty}, a_1).\]
    Similarly, \[\bigcap_{n\geqslant 1} \Diam_{a_0}^\opp (a_{-n}, a_n)\subset \bigcap_{n\geqslant 1} \Diam_{a_0}^\opp (a_{-1}, a_n) \subset \overline \Diam_{a_0}^\opp(a_{-1}, a_{+\infty})~,\]
    so we have 
\[\bigcap_{n\geqslant 1} \Diam_{a_0}^\opp (a_{-n}, a_n) \subset \overline \Diam_{a_0}^\opp(a_{-\infty}, a_1)\cap \overline{\Diam}_{a_0}^\opp(a_{-1}, a_{+\infty}).\]
If $a_\infty=a_{+\infty}=a_{-\infty}$, then by \Cref{lem: intersection closures diamonds}, $\overline \Diam_{a_0}^\opp(a_\infty, a_1)\cap \overline{\Diam}_{a_0}^\opp(a_{-1}, a_\infty) = \{a_\infty\}$, so (1) holds.
If $(a_{-\infty},a_{-1},a_1,a_{+\infty})$ is positive, then \Cref{intersection of diamonds} implies that $\overline \Diam_{a_0}^\opp(a_{-\infty}, a_1)\cap \overline{\Diam}_{a_0}^\opp(a_{-1}, a_{+\infty}) = \overline{\Diam}_{a_0}^\opp(a_{-\infty},a_{+\infty})$, so (2) holds.
\end{proof}

Over $\Rb$, the least upper bound property implies that positive sequences have limits. Combining this with \Cref{onesided squeeze} yields the following consequence.

\begin{corollary}[{\cite[Proposition 3.14]{GLW}}]
\label{proposition: least upper bounded property for positive sequence}
    If $(a_n)_{n\geqslant 0}$ is a positive sequence in $\Fc_{\Theta,\Rb}$, then it converges to a point $a_\infty \in \Fc_{\Theta,\Rb}$. In particular, the sequence $(a_\infty, a_1, a_2,\ldots , a_n, \ldots )$ is positive, and for any $x\in \Diam_{a_2}^\opp(a_1, a_\infty)$, we have \[\bigcup_{n\geqslant 2} \Diam_x^\opp(a_1,a_n) = \Diam_x^\opp(a_1,a_\infty)~,\quad\text{and} \quad \bigcap_{n\geqslant 2} \overline{\Diam}_{a_1}^\opp(a_n, x) = \overline{\Diam}_{a_1}^\opp(a_\infty, x)~.\]
\end{corollary}

\subsection{Semi-positive tuples and maps}

When proving closedness of $\Theta$-positive representations, we will need to consider limits of positive tuples, which we will call \emph{semi-positive tuples}.
We collect here the properties that will be needed in \Cref{s: Closedness} to prove that the set of $\Theta$-positive representations is closed.

\begin{definition}
    A tuple $(x_1,\ldots, x_n) \in \Fc_{\Theta,\Fb}$ is \emph{semi-positive} if it belongs to the closure of the set of positive tuples.
    
    A map $\xi$ from a cyclically ordered set to $\Fc_{\Theta,\Fb}$ is semi-positive if the image of every cyclically ordered tuple is semi-positive.
\end{definition}

Note that the set of semi-positive $n$-tuples $(\Fc_{\Theta,\Fb}^n)_{\geqslant 0}$ is semi-algebraic, since the closure of a semi-algebraic set is semi-algebraic, see \Cref{section:SemiAlgSets}.

We are interested in conditions that guarantee that a semi-positive tuple is, in fact positive. The following proposition is a rephrasing of the fact that positive tuples form a union of semi-algebraically connected components of pairwise transverse tuples. This was proven by Guichard and Wienhard \cite{guichard2025generalizing} (after applying the transfer principle).

\begin{proposition}[{\cite[Lemma 13.27]{guichard2025generalizing}}]\label{prop: transverse + semi positive implies positive}
    Let $(x_1,\ldots, x_n) \in \Fc_{\Theta,\Fb}$ be a semi-positive tuple. If $x_i$ is transverse to $x_j$ for all $i \neq j$, then $(x_1,\ldots, x_n)$ is positive.
\end{proposition}

In general, the flags in a semi-positive tuple need not be pairwise transverse. However, using \Cref{prop: transverse + semi positive implies positive}, we deduce the following result, which allows us to use pairwise transversality of certain pairs in a semi-positive tuple to deduce transversality of other pairs.

\begin{lemma}
\label{lem: semi-positive + a bit transverse implies transverse}
    Let $(x_1,\ldots, x_6) \in \Fc_{\Theta,\Fb}$ be a semi-positive $6$-tuple. If $x_1$ is transverse to $x_6$ and $x_3$ is transverse to $x_4$, then $x_2$ is transverse to $x_5$.
\end{lemma}

\begin{proof}
    By the transfer principle, it is enough to work over $\Rb$ (or more generally a field whose topology is metrizable).
    For each $i=1,\dots,6$, let $(x_i^n)_{n\ge 1}$ be sequences in $\Fc_{\Theta,\Rb}$ converging to $x_i$ and such that $(x_1^n,\ldots, x_6^n)$ is positive. Since $G_{\Rb}$ acts transitively on pairs of transverse flags and since $x_1$ and $x_6$ are transverse, we can assume without loss of generality that $x_1^n=x_1$ and $x_6^n = x_6$ for all $n$. Since there are finitely many diamonds between $x_1$ and $x_6$, up to extracting a subsequence, we can assume that 
    \[\Diam_{x_2^n}(x_1,x_6) = \Diam_{x_2^m}(x_1,x_6)\]
    for all positive integers $m,n$. 
    
    Choose a points $p,q \in \Diam_{x_2^n}^\opp(x_1,x_6)$ such that $(x_1,p,q,x_6)$ is positive. Since $x_3^n,x_4^n\in \Diam_{x_2^n}(x_1,x_6)$ for all integers $n\ge 1$, it follows from \Cref{lem: closure of diamond} that
    \[x_3,x_4\in \overline{\Diam}_{x_2^n}(x_1,x_6)=\overline{\Diam}_p^{\rm opp}(x_1,x_6)\subset\Diam_{x_1}(p,q).\]
    In particular, both $x_3$ and $x_4$ are transverse to both $p$ and $q$. By assumption, $x_3$ is transverse to $x_4$, so \Cref{prop: transverse + semi positive implies positive} implies that $(p,x_3,x_4,q)$ is positive. Thus, we may choose points $r,s\in\Fc_{\Theta,\Fb}$ such that $(p,x_3,r,s,x_4,q)$ is positive. By the openness in \Cref{propo: positive n-tuples is semi-algebraic}, for all sufficiently large integers $n$, $(p,x_3^n,r,s,x_4^n,q)$ is positive and so \Cref{lem: AddingElemToPosTuple} implies that 
    \[(p,x_1,x_2^n,x_3^n,r,s,x_4^n,x_5^n,x_6,q)\]is positive. In particular, $x_2\in\overline{\Diam}_q^{\rm opp}(p,r)$ and $ x_5\in\overline{\Diam}_p^{\rm opp}(s,q)$. It now follows from \Cref{lem: closure of diamond} and \Cref{opposites transverse} that $x_2$ and $x_5$ are transverse. 
\end{proof}

\subsection{%
\texorpdfstring{%
$\PG$-condition for flag varieties}%
{PG-condition for flag varieties}}
\label{subs:PGConditionHermFlagVarities}

Let $\Fc_{\Theta,\Fb}$ be a flag variety over a real closed field $\Fb$ admitting a positive structure. Let us recall and specify the definition of the \emph{$\PG$-condition} from the introduction:
\begin{definition}
    Given an integer $N>0$, we say that $\Fc_{\Theta,\Fb}$ satisfies the \emph{$\PG_N$-condition} if, for every positive $N$-tuple $(x_1,\ldots, x_N) \in \Fc_{\Theta,\Fb}^N$ and every $y\in \Fc_{\Theta,\Fb}$, there exists $i\in \{1,\ldots, N\}$ such that $y$ is transverse to $x_i$.

    We say that $\Fc_{\Theta,\Fb}$ satisfies the  $\PG$-condition if it satisfies the $\PG_N$-condition for large enough $N$.
\end{definition}

The list of possible $\Fc_{\Theta,\Fb}$ that admit a positive structure is given in \Cref{thm: GW classification positive structures}.
To support \Cref{conj:PG-condition} that the $\PG$-condition always holds, we explain here that it is satisfied by two families of positive structures.

Note that the $\PG_N$-condition is a semi-algebraic condition over $\Qbar$. Hence it is valid over $\Rb$ if and only if it is valid over any real closed field by the transfer principle.

\subsubsection{The split real case}

Let us consider here $G$ a split real form (defined over $\Qbar$) of a simple algebraic group and denote by $\Fc_\Delta^G$ its complete flag variety (here, $\Delta$ denotes the set of all simple roots of $G$). This flag variety admits a positive structure, whose construction dates back to the work of Lusztig \cite{Lusztig_TotalPositivityReductiveGroups}.
Here we investigate the $\PG$-condition for these flag varieties.\\

Let us start with the case of $\SL_n$. We briefly recall that, in that case, we can take $P_\Delta$ and $P_\Delta^\opp$ to be the group of lower triangular and upper triangular matrices respectively, and $U_\Delta$ and $U_\Delta^\opp$ to be respectively the subgroups of $P_\Delta$ and $P_\Delta^\opp$ with diagonal coefficients equal to $1$.
Lusztig's positive structure on $\Fc_\Delta^{\SL_n}$ is then defined by the semigroup $U_\Delta^{>0}$ of \emph{totally positive} unipotent matrices, i.e.\ matrices $u\in U_{\Delta}$ for which all minors that are not obviously $0$ are positive.

The $\PG$-condition is satisfied for the complete flag variety of $\Fc_\Delta^{\SL_n}$. This is essentially due to Saldanha--Shapiro--Shapiro \cite{SaldanhaShapiroShapiro_FinitenessGrassmannConvexity}, as was already pointed out by Riestenberg and Smillie in \cite[Theorem 72]{RiestenbergSmillie_UniversalHTSpaces}.

\begin{theorem}[Saldanha--Shapiro--Shapiro \cite{SaldanhaShapiroShapiro_FinitenessGrassmannConvexity}] \label{theo-PG for SLn}
The complete flag variety of $\SL_n(\Rb)$ satisfies the $\PG$-condition.
\end{theorem}

This theorem is not explicitly stated as such in \cite{SaldanhaShapiroShapiro_FinitenessGrassmannConvexity}, but follows from \Cref{lemma: Vanishing minor convex curve} below.
Following \cite{SaldanhaShapiroShapiro_FinitenessGrassmannConvexity}, we call a curve $u\colon \Rb \to U_{\Delta,\Rb}$ \emph{flag-convex} if
\[u(t)^{-1} u'(t) = \left(\begin{matrix}
        0 & & \cdots & & 0 \\
        a_1(t) & \!\! 0 & & & \\
        0 & a_2(t) & 0 & & \vdots \\
        \vdots & \ddots & \ddots & \ddots & \\
        0 & \cdots & 0 & a_{n-1}(t) & 0\end{matrix}\right)\]
    for some positive functions $a_1, \ldots, a_{n-1}\colon \Rb \to \Rb_{>0}$.

\begin{lemma}[{\cite[Theorem 1 and Lemma 2]{SaldanhaShapiroShapiro_FinitenessGrassmannConvexity}}] \label{lemma: Vanishing minor convex curve}
    Let $u\colon \Rb \to \mathrm U_{\Delta,\Rb}$ be a $\mathcal C^1$ flag-convex curve and $k$ an integer between $1$ and $n-1$. Then the function
\[m_k\colon t\mapsto \det\left(u(t)_{n-k+i,j}\right)_{1\leq i,j\leq k}\]
vanishes at at most $r(k,n)$ points, where 
\[r(1,n) = n-1\]
and
\[r(k,n) = \frac{(n-k+1)^{2k-3}}{2^{k-3}}\]
for $k > 1$.
\end{lemma}

Let us now explain how to deduce Theorem \ref{theo-PG for SLn} from this lemma.

\begin{proof}[Proof of Theorem \ref{theo-PG for SLn}]
Let us fix
\begin{equation} \label{eq: condition N for SLn}
    N > \sum_{k=1}^{n-1} r(k,n)~.
\end{equation}
Let $(x_1,\ldots, x_N)$ be a positive tuple in $\Fc_{\Delta,\Rb}^{\SL_n}$ and $y$ another flag. We want to prove that $y$ is transverse to one of the $x_i$, i.e.\ for all $1\leq k\leq n-1$ the $k$-dimensional subspace $x_i^k$ of $x_i$ is in direct sum with the $(n-k)$-dimensional subspace $y^{n-k}$ of $y$.\\

Let us fix an auxiliary flag $z$ which is transverse to $y$ and such that $(x_1,\ldots, x_N,z)$ is positive. 
After applying an element of $\SL_n(\Rb)$ we can assume that $x_1$ is the standard flag
\[p = \Span(e_1)\subset \Span(e_1,e_2) \subset \ldots \subset \Span(e_1,\ldots, e_{n-1})~,\]
$z$ is the standard opposite flag
\[p^\opp=\Span(e_n) \subset \Span(e_{n-1},e_n) \subset \ldots, \subset \Span(e_2,\ldots, e_n)~,\]
and $x_i$ belongs to the standard diamond for all $i=2,\ldots, N$.

We can thus write $x_i = u_i\cdot x_1$ for a unique $u_i\in U_{\Delta,\Rb}$, and the positivity of the tuple $(x_1,\ldots, x_N,z)$ ensures that $u_{i-1}^{-1}u_{i}$ belongs to the sub-semigroup $U_{\Delta,\Rb}^{>0}$ of totally positive unipotent lower triangular matrices for all $i=2,\ldots,N$.

Since $y$ is also transverse to $z$, we can write $y= v\cdot x_1$ for some $v\in U_{\Delta,\Rb}$.
As explained in the proof of \cite[Lemma 2]{SaldanhaShapiroShapiro_FinitenessGrassmannConvexity}, one can then find a smooth flag-convex curve $u \colon \R \to U_{\Delta,\Rb}$ such that $u(i) = v^{-1}u_i$ for all $i=1,\ldots,N$.

Since $N>\sum_{k=1}^{n-1} r(k,n)$, applying \Cref{lemma: Vanishing minor convex curve}, we deduce that there exists some integer $i_0 \leq N$ such that $m_k(i_0)\neq 0$ for all $1\leq k \leq n-1$.
By definition of $m_k$, this means that $v^{-1} u_{i_0}(\Span(e_1,\ldots, e_k))$ is transverse to $\Span(e_1,\ldots, e_{n-k})$ for all $k$.
Hence
$u_{i_0}(\Span(e_1,\ldots, e_k)) = x_{i_0}^k$ is transverse to $v\cdot \Span(e_1,\ldots, e_{n-k}) = y^{n-k}$ for all $k$, i.e.\ $y$ is transverse to $x_{i_0}$.

Hence $\Fc_{\Delta,\Rb}^{\SL_n}$ satisfies the $\PG_N$-condition for $N$ as in \eqref{eq: condition N for SLn}.
\end{proof}

\begin{remark}
    The value of $N$ for which we obtain the $\PG_N$-condition grows at least like 
    \[r\left(\frac{n}{2}, n\right) \gg \left( \frac{n}{2}\right)^n\]
    as $n$ grows.
    
    The so-called \emph{Grassmannian conjecture}, formulated by Shapiro--Shapiro in \cite{ShapiroShapiro_GrassmannConvConj}, predicts that the optimal value for $r(k,n)$ in \Cref{lemma: Vanishing minor convex curve} is 
    \[r(k,n) = k(n-k)~.\]
    This would imply the $\PG_N$-property of $\Fc_\Delta^{\SL_n}$ for $N> \sum_{k=1}^{n-1} k(n-k) = \frac{n^3-n}{6}$, a much better bound.
\end{remark}

The $\PG$-condition passes to sub-flag varieties. To make this more precise, let us introduce the notion of \emph{positive embedding} of flag varieties with positive structures.

\begin{definition}
    Let $G$ and $H$ be semisimple algebraic groups (defined over $\Qbar$) and $\Fc_\Theta^G$ and $\Fc_{\Theta'}^H$ two flag varieties of $G$ and $H$ respectively admitting a positive structure. We say that $\Fc_\Theta^G$ admits a \emph{positive embedding} into $\Fc_{\Theta'}^H$ if there exists an algebraic representation $\rho\colon G\to H$ and a $\rho$-equivariant embedding $\phi\colon\Fc_\Theta^G \to \Fc_{\Theta'}^H$ such that:
    \begin{itemize}
        \item for every $x,y\in \Fc_\Theta^G$, $x$ and $y$ are transverse if and only if $\phi(x)$ and $\phi(y)$ are transverse,
        \item for every tuple $(x_1,\ldots, x_k)\in \Fc_\Theta^G$, $(x_1,\ldots, x_k)$ is positive if and only if $(\phi(x_1),\ldots, \phi(x_k))$ is positive.
    \end{itemize}
\end{definition}

Complete flag varieties of real split linear algebraic groups that embed in $\SL_n$ have been classified by Sambarino in \cite{Sambarino_ZariskiClosuresPosRepr}. 

\begin{fact} \label{fact: list Positive embeddings}
We have the following positive embeddings:
\begin{itemize}
    \item $\Fc_\Delta^{\Sp_{2n}}$ embeds positively in $\Fc_\Delta^{\SL_{2n}}$,
    \item $\Fc_\Delta^{\SO_{n,n+1}}$ embeds positively in $\Fc_\Delta^{\SL_{2n+1}}$,
    \item $\Fc_\Delta^{\mathrm{G}_2}$ embeds positively in $\Fc_\Delta^{\SL_{7}}$.
\end{itemize}
\end{fact}

Now, we have the following easy proposition:

\begin{proposition} \label{prop: PG positive embedding}
    Suppose $\Fc_\Theta^G$ embeds positively in $\Fc_{\Theta'}^H$. If $\Fc_{\Theta'}^H$ satisfies the $\PG_N$ condition, then so does $\Fc_\Theta^G$.
\end{proposition}

\begin{proof}
    Let $(x_1,\ldots, x_N) $ be a positive $N$-tuple in $\Fc_\Theta^G$ and $y$ another point in $\Fc_\Theta^G$.
    Let $\phi$ be a positive embedding of $\Fc_\Theta^G$ into $\Fc_{\Theta'}^H$.
    Then $(\phi(x_1), \ldots, \phi(x_N))$ is a positive $N$-tuple.
    Since $\Fc_{\Theta'}^H$ satisfies the $\PG_N$-condition, there exists $1\leq i \leq N$ such that $\phi(x_i)$ is transverse to $\phi(y)$.
    Hence $y$ is transverse to $x_i$ by one of the defining properties of positive embeddings.
\end{proof}

Combining this proposition with \Cref{theo-PG for SLn} and \Cref{fact: list Positive embeddings}, we obtain the following.

\begin{corollary}
    For every real closed field $\Fb$, the complete flag varieties of $\SL_n(\Fb)$, $\Sp_{2n}(\Fb)$, $\SO_{n,n+1}(\Fb)$ and $\mathrm{G}_2(\Fb)$ satisfy the $\PG$-condition.
\end{corollary}

\subsubsection{The hermitian case}
Let us now turn to the case when $G_{\Rb}$ is a simple Hermitian Lie group of tube type.

The \emph{Shilov boundary} of the symmetric space of $G$ is the flag variety $\check{S} \coloneqq G_{\Rb}/P_{\Theta,\Rb}$ for $\Theta=\{\alpha_r\}$, where $\alpha_r$ is the simple long restricted\footnote{This is the same over $\Rb$ and $\overline{\Qb}^r$\cite[\S 5.2]{Appenzeller_GroupsRCF}.} root (defined over $\overline{\Qb}^r$). The Shilov boundary carries a $2$-cocycle with integral values, the \emph{generalized Maslov index}, whose properties are summarized in the following proposition.

\begin{proposition}[{\cite[Th\'eor\`eme~3.3, Th\'eor\`eme~3.5]{Clerc_MaslovTube}}, {\cite[Theorem~6.2]{Clerc_MaslovNonTube}}]
\label{propo:MaslovIndex}
    There exists a function $m \colon \check{S}^3 \to \Rb$, called the \emph{generalized Maslov index}, with the following properties:
    \begin{enumerate}
        \item $m$ takes integral values \footnote{That $m$ takes integral values is not true if we do not assume $G_{\Rb}$ to be of tube type.} in the interval $[-r,r]$, where $r$ is the real rank of $G_{\Rb}$;

        \item $m$ is $G_{\Rb}$-invariant;
        
        \item For all $(x_1,x_2,x_3)\in \check{S}^3$, and any permutation $\sigma$ of $\{1,2,3\}$,
        \[m(x_{\sigma(1)},x_{\sigma(2)}, x_{\sigma(3)})= \textnormal{sgn}(\sigma)m(x_1,x_2,x_3)~,\]
        where $\textnormal{sgn}$ denotes the sign of $\sigma$;
        \item
        For all $(x_1,x_2,x_3,y)\in \check{S}^4$,
        $m(x_1, x_2, x_3) = m(x_1, x_2, y) + m(x_2, x_3, y) + m(x_3, x_1, y)$;
        \item 
        If $m(x_1,x_2,x_3)$ is \emph{maximal} (i.e.\ equal to $r$), then $x_i \pitchfork x_j$ for all $i\neq j \in \{1,2,3\}$.
    \end{enumerate}
\end{proposition}

We will make use of the cocycle property of $m$, and that maximality implies transversality to prove our conjecture in this case.

\begin{theorem}
    Let $N>2r$.
    Let $(x_1,\ldots,x_N) \in \check{S}^N$ be a positive $N$-tuple, and $y \in \check{S}$.
    Then $y$ is transverse to at least one of the $x_i$.
\end{theorem}
\begin{proof}
    Since $(x_1,\ldots,x_N)$ is positive, up to reversing the order of the tuple, we can assume that $m(x_1,x_i,x_{i+1})=r$ for all $i=2,\ldots,N-1$.
    By iterating the cocycle property of $m$, \Cref{propo:MaslovIndex}, setting $x_{N+1}=x_1$, we get 
    \begin{align*}
        \frac{1}{N}\sum_{i=1}^{N} m(x_i,x_{i+1},y)=
        \frac{1}{N}\sum_{i=2}^{N-1} m(x_1,x_i,x_{i+1})=\frac{(N-2)r}{N} =r-\frac{2r}{N}>r-1,
    \end{align*}
    where the last inequality holds because $N>2r$.
    Suppose that all $(x_i,x_{i+1},y)$ are non-maximal, thus $m(x_i,x_{i+1},y)\leq r-1$ for all $i$, as $m$ takes only integer values (\Cref{propo:MaslovIndex}).
    This yields  
    \[r-1 < \frac{1}{N}\sum_{i=1}^{N} m(x_i,x_{i+1},y)\leq r-1,\]
    a contradiction.
    Thus at least one of the triples $(x_i,x_{i+1},y)$ is maximal for some $i$, and in particular $y$ is transverse to $x_i$ and $x_{i+1}$ by \Cref{propo:MaslovIndex}.
\end{proof}

The statement of the above theorem is given in first order logic. Applying the Tarski--Seidenberg principle (\Cref{thm_TarskiSeidenberg}) thus yields the following.

\begin{corollary}
    Let $G$ be a Hermitian group of tube type of rank $r$ and $\Fb$ a real closed field.
    Then $\check{S}_{\Fb}$ satisfies the $\PG_N$ condition for $N>2r$.
\end{corollary}

\section{%
\texorpdfstring{%
$\Theta$-positive representations for general real closed fields}%
{Theta-positive representations for general real closed fields}}
\label{s: Definitions Positive representations}

Let $G$ denote a semi-algebraically connected, semisimple linear semi-algebraic group over $\overline{\Qb}^r$, let $\Theta$ be a non-empty symmetric collection of simple roots of $G$, and let $\Fc_{\Theta}$ denote the associated flag variety. As before, let $\Fb$ be a real closed field, and denote the $\Fb$-extensions of all semi-algebraic maps and sets defined over $\overline{\Qb}^r$ with a subscript $\Fb$. In this section, we assume that $G_{\Fb}$ admits a $\Theta$-positive structure.

Let $\Gamma\subset \Isom_+(\Hb)$ be a non-elementary Fuchsian group.
We will be particularly interested in the case where $\Gamma$ is finitely generated, but for now let us consider the general case. We finally introduce the central definition of this paper: 

\begin{definition}\label{dfn: positive representations}
    A representation $\rho \from \Gamma\to G_{\Fb}$ is \emph{$\Theta$-positive} if there exists a non-empty $\Gamma$-invariant subset $D$ of $\Lambda(\Gamma)$ and a $\rho$-equivariant, positive map $\xi \from D \to \Fc_{\Theta,\Fb}$.
\end{definition}

Recall from \Cref{Fuchsian groups lemma} that any Fuchsian group $\Gamma$ is semi-conjugated to a Fuchsian group $\Gamma'$ of the first kind. Our first remark will be that positive representations of $\Gamma$ and $\Gamma'$ are the same:

\begin{proposition}\label{propo: equivalence positivity first and second kind Fuchsian groups}
    Let $\iota \colon \Gamma \to \Gamma'$ be a semi-conjugacy between $\Gamma$ and $\Gamma'$, where $\Gamma'$ is a Fuchsian group of the first kind. Then a representation $\rho \colon \Gamma' \to G_{\Fb}$ is $\Theta$-positive if and only if $\rho\circ \iota$ is $\Theta$-positive.
\end{proposition}
\begin{proof}
    Let $\alpha\colon\Lambda(\Gamma) \to \Sb^1$ be a continuous surjective monotonous $\iota$-equivariant map given by \Cref{Fuchsian groups lemma}.
    We first prove that $\alpha$ is injective in restriction to any $\Gamma$-orbit. Suppose for the purpose of contradiction that there is some $x\in\Lambda(\Gamma)$ and some $\gamma\in\Gamma$ such that $x\neq \gamma\cdot x$ and $\alpha(\gamma\cdot x)= \alpha(x)$. Since $\alpha$ is $\iota$-equivariant, $\alpha(x)$ is a fixed point of $\iota(\gamma)$. Since $x\neq \gamma \cdot x$, the monotonicity of $\alpha$ implies that $\alpha^{-1}(\alpha(x))$ must be a $\gamma$-invariant interval containing $x$ in its interior. Since $x\in\Lambda(\Gamma)$, this means that $\alpha^{-1}(\alpha(x))\cap\Lambda(\Gamma)$ is open. By the density of $\Lambda_h(\Gamma)\subset\Lambda(\Gamma)$, there are hyperbolic elements $\gamma_1,\gamma_2\in\Gamma$ such that $\gamma_1^+$ and $\gamma_2^+$ are distinct and lie in $\alpha^{-1}(\alpha(x))$.
    This means that $\alpha(\gamma_1^+)=\alpha(x)=\alpha(\gamma_2^+)$, so $\iota(\gamma_1)$ and $\iota(\gamma_2)$ share a fixed point, and thus commute. However, $\gamma_1$ and $\gamma_2$ do not commute because $\gamma_1^+\neq\gamma_2^+$, which contradicts the fact that $\iota$ is an isomorphism. 
    
    If $\rho\colon\Gamma'\to G_{\Fb}$ is positive, then there is a positive $\rho$-equivariant map $\xi \colon D \to \Fc_{\Theta,\Fb}$ defined on a $\Gamma'$-invariant subset $D\subset\Sb^1$. Let $x\in \Lambda(\Gamma)$ be a point in $\alpha^{-1}(D)$. Since $\alpha|_{\Gamma\cdot x}$ is $\iota$-equivariant and injective, $\xi \circ \alpha\vert_{ \Gamma \cdot x}$ is a positive $\rho\circ \iota$-equivariant map. Hence $\rho \circ \iota$ is positive.

    Conversely, suppose that $\rho\circ\iota\colon\Gamma\to G_{\Fb}$ is positive. Then there is a positive, $\rho\circ\iota$-equivariant map and $\xi\colon D\to \Fc_{\Theta,\Fb}$ defined on $\Gamma$-invariant subset $D\subset \Lambda(\Gamma)$. Choose $x\in D$. Then $\xi\circ (\alpha|_{\Gamma\cdot x})^{-1}$ is a positive $\rho$-equivariant map defined on $\alpha(\Gamma\cdot x)$, so $\rho$ is positive.
\end{proof}

Since $\Theta$-positive representations of $\Gamma$ and $\Gamma'$ are the same, any statement about $\Theta$-positive representations that does not involve the limit sets or the distinction between parabolic and hyperbolic elements can be proven for either $\Gamma$ or $\Gamma'$. On the other hand, if $\Gamma$ is of the second kind, some hyperbolic elements of $\Gamma$ (that we sometimes call \emph{peripheral}) are mapped to parabolic elements of $\Gamma'$ by $\iota$, so that positivity ``sees'' these elements as parabolic. For this reason, it is often more convenient to assume that $\Gamma$ is of the first kind.

\subsection{Discreteness}

Let us start by proving that $\Theta$-positive representations over any ordered field are faithful and discrete (for the order topology, see \Cref{sec: ordered fields}).

\begin{proposition}
\label{propo: positive implies injective and discrete}
    Let $\rho \colon \Gamma \to G_{\Fb}$ be $\Theta$-positive.
    Then $\rho$ is injective and its image is discrete in $G_{\Fb}$. 
\end{proposition}

\begin{remark}
    An important difference between $\Rb$ and non-Archimedean ordered fields is that, for non-Archimedean ordered fields, $G_{\Fb}$ is not locally compact. In particular, the discreteness of $\rho(\Gamma)$ does not imply that any bounded subset of $G_{\Fb}$ contains finitely many points in $\rho(\Gamma)$. Equivalently, a ball in the Bruhat--Tits building may very well contain infinitely many points of a given $\rho(\Gamma)$ orbit. In fact, in many situations, there will exist a point $x$ in the Bruhat--Tits building of $G_{\Fb}$ and a sequence $(\gamma_n)_{n\in \Nb} \subset \Gamma$ such that $\lim_{n\to +\infty}d(x,\rho(\gamma_n)\cdot x) = 0$. This is the case for instance when $\Gamma$ is a closed surface group and $G_{\Fb} = \PSL_2(\Fb)$. 
\end{remark}

\begin{proof}[Proof of \Cref{propo: positive implies injective and discrete}]
    We can assume without loss of generality that $\Gamma$ is of the first kind. Let $D$ be a $\Gamma$-invariant subset of $\Sb^1$ and $\xi\colon D \to \Fc_{\Theta,\Fb}$ a $\rho$-equivariant positive map. Let $a_1,a_2,a_3$ be $3$ distinct points in $D$ and, for $k\in \{1,2,3\}$, let $I_k = (b_k,c_k)$ be an interval of length less than some $\varepsilon \in \Rb_{>0}$ containing $a_k$ and with extremities in $D$.

    If $g\in \Isom_+(\Hb)$ satisfies that $g(I_k)\cap I_k \neq \emptyset$ for all $k$, then there exists $(a_1',a_2',a_3')$ $\varepsilon$-close to $(a_1,a_2,a_3)$ which is moved by less than $\varepsilon$ under $g$. Since $\Isom_+(\Hb)$ acts freely and properly on triples of distinct points in $\Sb^1$, this forces $g$ to be very close to the identity. In particular, since $\Gamma\subset\Isom_+(\Hb)$ is discrete, for $\varepsilon$ small enough, we deduce that every $\gamma\in \Gamma\setminus \{\Id_\Gamma\}$ satisfies $\gamma(I_k)\cap I_k = \emptyset$ for some $k\in \{1,2,3\}$.

    Now, for such sufficiently small $\varepsilon$, consider the diamond $\Diam_k = \Diam_{\xi(a_k)}(\xi(b_k), \xi(c_k))$. Then for every $\gamma \in \Gamma\setminus \{\Id_\Gamma\}$, there is some $k\in \{1,2,3\}$ such that $\gamma(I_k)\cap I_k = \emptyset$, in which case  
    \[\rho(\gamma)(\Diam_k) \cap \Diam_k = \Diam_{\xi(\gamma\cdot a_k)}(\xi(\gamma\cdot b_k), \xi(\gamma\cdot c_k)) \cap \Diam_{\xi(a_k)}(\xi(b_k), \xi(c_k)) = \emptyset\]
    since the tuple $(\xi(b_k), \xi(a_k), \xi(c_k), \xi(\gamma \cdot b_k), \xi(\gamma \cdot a_k), \xi(\gamma \cdot c_k))$ is positive. As such, if we set $U\coloneqq \{g\in G_{\Fb} \mid g(\Diam_k)\cap \Diam_k \neq \emptyset \textrm{ for } k=1,2,3\}$, then $U$ is an open neighborhood of $\Id_{G_{\Fb}}$ whose preimage by $\rho$ is $\{\Id_\Gamma\}$. Hence $\rho$ is injective and has discrete image.
\end{proof}

\subsection{Positively translating elements}
\label{subsec:PosTranslElem}

An (almost) immediate consequence of $\Theta$-positivity of a representation $\rho$ is that the image of an infinite order element admits a positive orbit. In this section, we investigate such elements in more details.

\begin{definition}
\label{dfn:PosRotPosTransl}
    Given $g\in G_{\Fb}$, we say that a point $x\in \mathcal F_{\Theta,\Fb}$ has \emph{positive $g$-orbit} if the sequence $(g^n\cdot x)_{n\in \Zb}$ is positive, and that $g$ is \emph{$\Theta$-positively rotating} if there exists a point $x\in \mathcal F_{\Theta,\Fb}$ with positive $g$-orbit.
    
    We say that $g$ is \emph{$\Theta$-positively translating} if there exists $p$ and $x$ in $\mathcal F_{\Theta,\Fb}$ such that $p$ is fixed by $g$ and $(x,g\cdot x, g^2\cdot x, p)$ is positive. In that case, we call $(x,p)$ a \emph{$\Theta$-positively translating pair} for $g$.
\end{definition}

Notice that if $g$ is $\Theta$-positively rotating and $x\in\Fc_{\Theta,\Fb}$ is a point such that the sequence $(g^n\cdot x)_{n\in \Nb}$ is positive, then $x$ has a positive $\langle g\rangle$-orbit, i.e.\ for all positive integers $n$, the tuple $(g^{-n}\cdot x,\dots,g^{-1}\cdot x,x,g\cdot x,\dots,g^n\cdot x)$ is positive.

\begin{figure}[ht]
    \centering
    \begin{tikzpicture}[scale=1.5]
        \draw (0,0) circle (1cm);
        \draw (-.3,{-sqrt(1-.3*.3)}) node{$\bullet$};
        \draw[below] (-.3,{-sqrt(1-.3*.3)}) node{$x$};
        \draw (-.9,{sqrt(1-.9*.9)}) node{$\bullet$};
        \draw[left] (-.9,{sqrt(1-.9*.9)}) node{$g \cdot x$};
        \draw (-.1,{sqrt(1-.1*.1)}) node{$\bullet$};
        \draw[above] (-.1,{sqrt(1-.1*.1)}) node{$g^2 \cdot x$};
        \draw (.4,{sqrt(1-.4*.4)}) node{$\bullet$};
        \draw[right] (.4,{sqrt(1-.4*.4)+.2}) node{$g^3 \cdot x$};
        \draw (.8,{sqrt(1-.8*.8)}) node{$\bullet$};
        \draw[right] (.8+.05,{sqrt(1-.8*.8)}) node{$g^4 \cdot x$};
        \draw[right] (1,0) node{$\vdots$};
        \draw (.7,{-sqrt(1-.7*.7)}) node{$\bullet$};
        \draw[right] (.7,{-sqrt(1-.7*.7)}) node{$g^n \cdot x$};
        \draw[below] (.2,{-sqrt(1-.2*.2)}) node{$\ldots$};
        \draw (4,0) circle (1cm);
        \draw (5,0) node{$\bullet$};
        \draw[right] (5,0) node{$p$};
        \draw (3.7,{-sqrt(1-.3*.3)}) node{$\bullet$};
        \draw[below] (3.7,{-sqrt(1-.3*.3)}) node{$x$};
        \draw (3.1,{sqrt(1-.9*.9)}) node{$\bullet$};
        \draw[left] (3.1,{sqrt(1-.9*.9)}) node{$g \cdot x$};
        \draw (3.9,{sqrt(1-.1*.1)}) node{$\bullet$};
        \draw[above] (3.9,{sqrt(1-.1*.1)}) node{$g^2 \cdot x$};
    \end{tikzpicture}
    \caption{A $\Theta$-positively rotating element (left) and a $\Theta$-positively translating element (right)}
\end{figure}

From the definition above it is clear that $\Theta$-positively translating is a semi-algebraic condition over $\Qbar$. In contrast, the following proposition, together with the transfer principle, shows that ``$\Theta$-positively rotating'' is not characterized by real semi-algebraic conditions.

\begin{proposition}
\label{propo: PosTransPosRot}
    If $g \in G_{\Fb}$ is $\Theta$-positively translating, then it is $\Theta$-positively rotating. If $\Fb$ is Archimedean, then the converse is also true.
    
    In contrast, if $\Fb$ is non-Archimedean, there exists $g\in \PSL_2(\Fb)$ such that the action of $g$ on $\mathbf{P}^1(\Fb)$ is $\Theta$-positively rotating but does not have a fixed point.
\end{proposition}
\begin{proof}
    If $g$ is $\Theta$-positively translating, then, by $G_{\Fb}$-equivariance, we have that $(g^k\cdot x, g^{k+1}\cdot x, g^{k+2}\cdot x, p)$ is positive for all $k\in \Zb$, and the transitivity property of positivity (\Cref{lem: AddingElemToPosTuple}) implies that the whole $\langle g\rangle$-orbit of $x$ is positive. Hence $g$ is $\Theta$-positively rotating. 

    Let $g\in G_{\Rb}$ be $\Theta$-positively rotating, and let $x\in \mathcal F_{\Theta,\Rb}$ be a point whose $\langle g\rangle$-orbit is positive.
    By \Cref{proposition: least upper bounded property for positive sequence}, the sequence $(g^n\cdot x)_{n\in \Nb}$ has a limit $p$ as $n\to +\infty$, which is fixed by $g$ and satisfies that $(x,g\cdot x, g^2\cdot x,\ldots, p)$ is positive. Hence $g$ is $\Theta$-positively translating.
    If now $\Fb \subset \Rb$ is any real closed subfield, and $g\in G_{\Fb}$ a $\Theta$-positively rotating element, then $g$ viewed as an element in $G_{\Rb}$ is $\Theta$-positively translating by the above argument, hence it is $\Theta$-positively translating over $\Fb$ by the transfer principle.

    Let now $\Fb$ be a non-Archimedean real closed field.
    Let $\varepsilon \in \Fb$ be positive and smaller than every positive real number, and set
    \[g = \begin{bmatrix}
                    \sqrt{1-\varepsilon^2} & \varepsilon \\ 
                    -\varepsilon & \sqrt{1-\varepsilon^2}
    \end{bmatrix} \in \PSL_2(\Fb)~. \]
    For every $a \in \Fb \subset \mathbf P^1(\Fb)$, we have
    \[
    g\cdot a = \frac{\sqrt{1-\varepsilon^2} \ a + \varepsilon}{\sqrt{1-\varepsilon^2}-\varepsilon a}=a + \varepsilon \frac{1+a^2}{\sqrt{1-\varepsilon^2} - \varepsilon a}~.
    \]
    Hence, if $a$ is positive and smaller than any positive real number, then $g\cdot a > a$ and $g\cdot a$ is positive and smaller than any positive real number.
    By induction, we conclude that the $\langle g\rangle$-orbit of $a$ is positive.
    On the other hand, $\det(g) = 1$ and $\vert \textnormal{Tr}(g) \vert_{\Fb} <2$, hence $g$ does not have a fixed point in $\mathbf{P}^1(\Fb)$.
\end{proof}

\begin{remark}
\label{rem: infinitesimal rotation}
    Informally, the element $g$ in the above proof is an ``infinitesimal rotation''. It can be thought of as a limit of rotations of smaller and smaller angles. The orbits of such rotations remain positive for a longer and longer time, so that in the limit the full orbit is positive. 
\end{remark}

Recall from \Cref{hats} that $\hat{L}_{\Theta,\Fb}$ is the Levi factor of $\Aut_1(\mathfrak g_{\Fb})$ and $\hat{L}_{\Theta,\Fb}^*\subset\hat{L}_{\Theta,\Fb}$ is the stabilizer of every semi-algebraically connected component of $U_{\Theta,\Fb}^\pitchfork$.

\begin{proposition}\label{propo: PosTransUL}
    An element $g \in \Ad(G_{\Fb})$ is $\Theta$-positively translating if and only if it  is conjugated in $\Aut_1(\mathfrak{g}_{\Fb})$ to $ul$ for some $u\in U_{\Theta,\Fb}^{>0}$ and some $l\in \hat{L}_{\Theta,\Fb}^*$.
\end{proposition}
\begin{proof}
    Assume $g$ is $\Theta$-positively translating. Let $p,x\in \mathcal F_{\Theta,\Fb}$ be such that $g\cdot p = p$ and $(x,g\cdot x, g^2 \cdot x, p)$ is positive.
    In particular, $x$ and $p$ are transverse.
    Since $(x,g\cdot x, p)$ is positive, up to translating by an element in $\Aut_1(\mathfrak g_{\Fb})$, we can thus assume that $x= p_{\Theta}^\opp$, $p = p_\Theta$ and $g\cdot x \in \Diam_\std$.
    Since $g$ fixes $p = p_\Theta$, we can write $g=ul$ with $u\in U_{\Theta,\Fb}$ and $l\in \hat{L}_{\Theta,\Fb}$.
    Then $g\cdot x = u\cdot p_\Theta \in \Diam_\std$, so we may deduce that $u \in U_{\Theta,\Fb}^{>0}$.
    Finally, we have $(x, g\cdot x, g^2 \cdot x, p) = (p_\Theta^\opp, u\cdot p_\Theta^\opp, u (lul^{-1}) \cdot p_\Theta^\opp, p_\Theta)$.
    The positivity of this quadruple then implies that $lul^{-1} \in U_{\Theta,\Fb}^{>0}$.
    Hence $l$ belongs to $\hat{L}_{\Theta,\Fb}^*$ by definition.
\end{proof}

Note that the point $p$ fixed by a $\Theta$-positively translating element is not unique. However, over $\Rb$, there is a preferred choice of such a fixed point, given by the limit of the $\langle g\rangle$-orbit of $x$.

\begin{proposition}
\label{propo: OverRPosRotImpliesDiv}
    Let $g\in G_{\Rb}$ be a $\Theta$-positively rotating element.
    Then $g$ is $\Theta$-divergent and there exists a unique pair $(g^+,g^-)\in \Fc_{\Theta,\Rb}^2$ such that, for every $y\in \Fc_{\Theta,\Rb}$ with positive $\langle g \rangle$-orbit, the sequences $(g^n\cdot y)_{n\in\Nb}$ and $(g^{-n}\cdot y)_{n\in\Nb}$ converge to $g^+$ and $g^-$ respectively.
\end{proposition}

\begin{proof}
    Let $x$ be a point in $\mathcal F_{\Theta,\Rb}$ with positive $\langle g\rangle$-orbit.
    By \Cref{proposition: least upper bounded property for positive sequence}, the sequences $(g^n\cdot x)_{n\in \Nb}$ and $(g^{-n}\cdot x)_{n\in \Nb}$ have limits, denoted $g^+(x)$ and $g^-(x)$ respectively, as $n\to +\infty$. Observe that if $w\in\Diam_{g^{-1}\cdot x}^{\rm opp}(x,g\cdot x)$, then \Cref{lem: AddingElemToPosTuple} implies that the bi-infinite sequence
    \[(\dots,g^{-1}\cdot x,g^{-1}\cdot w,x,w,g\cdot x,g\cdot w,\dots)\]
    is positive, and contains $(g^n\cdot x)_{n\in\Zb}$ as a bi-infinite subsequence. Thus, \Cref{onesided squeeze} implies that $(g^n\cdot w)_{n\in\Nb}$ and $(g^{-n}\cdot w)_{n\in\Nb}$ converge to $g^+(x)$ and $g^-(x)$ respectively. 
    
    Since $\Diam^{\rm opp}_{g^{-1}\cdot x}(x,g\cdot x)\subset\Fc_{\Theta,\Rb}$ is open, \Cref{lemma: divergence in flag manifold} implies that $g$ is $\Theta$-divergent, and that for every $z\in\Fc_{\Theta,\Rb}$ that is transverse to $g^+(x)$ and $g^-(x)$, the sequences $(g^n\cdot z)_{n\in\Nb}$ and $(g^{-n}\cdot z)_{n\in\Nb}$ converge to $g^+(x)$ and $g^-(x)$ respectively. Thus, for any $y\in \Fc_{\Theta,\Rb}$ with a positive $\langle g\rangle$-orbit, the set $\Diam_{g^{-1}\cdot y}(y,g\cdot y)$ contains a flag $z$ that is transverse to both $g^+(x)$ and $g^-(x)$, so
    \[g^+(y)=\lim_{n\to+\infty}g^n\cdot z=g^+(x)\quad\text{and}\quad g^-(y)=\lim_{n\to+\infty}g^{-n}\cdot z=g^-(x).\]
    The pair $(g^+,g^-)\coloneqq(g^+(x),g^-(x))$ thus has the required property.
\end{proof}

Over a non-Archimedean field, given a $\Theta$-positively translating element $g$ and a $\Theta$-positively translating pair $(x,p)$ for $g$, the sequence $g^n\cdot x$ may not have a limit.
For instance, let 
\[g= \begin{bmatrix} 1 & 1 \\ 0 & 1\end{bmatrix} \in \PSL_2(\Fb).\]
Then $g$ fixes $\infty \in \mathbf{P}^1(\Fb)$ and $(0, g\cdot 0, g^2 \cdot 0, \infty)$ is positive, hence $g$ is $\Theta$-positively translating.
Yet, $g^n\cdot 0 = n$ does not converge in $\mathbf{P}^1(\Fb) $ (since $\Fb$ contains elements larger than any integer).

Nevertheless, the following proposition characterizes $g^+$ semi-algebraically as the ``minimal'' fixed point with respect to which $g$ translates positively. This allows to define $g^+$ and $g^-$ over any real closed field.

\begin{proposition}
    \label{propo:PosTransUniqueFP}
    Let $g\in G_{\Fb}$ be $\Theta$-positively translating.
    Then there exists a unique point $p_0\in\Fc_{\Theta,\Fb}$ such that $(x,p_0)$ is a $\Theta$-positively translating pair for $g$, and $\Diam_{g\cdot x} (x,p_0) \subset \Diam_{g\cdot x}(x,p)$ for all $\Theta$-positive translating pairs $(x,p)$ for $g$.
\end{proposition}

\begin{proof}
    \Cref{propo:PosTransUniqueFP} is a (admittedly, rather elaborate) sentence in the first order logic of ordered fields with coefficients in $\overline{\Qb}^r$.
    By the Tarski--Seidenberg principle, it is thus enough to verify it over $\Rb$. 

     Over $\Rb$, if $p$ is fixed by $g$ and $(x, g\cdot x, g^2 \cdot x, p)$ is positive, then $g^n\cdot x \in \Diam_{g\cdot x}(x,p)$ and by \Cref{propo: OverRPosRotImpliesDiv}, $g^n\cdot x\to g^+$. Then by \Cref{proposition: least upper bounded property for positive sequence}, 
     \[\Diam_{g\cdot x}(x,g^+) = \bigcup_{n\geqslant 2} \Diam_{g\cdot x}(x, g^n\cdot x) \subset \Diam_{g\cdot x}(x,p).\] 
     Thus $p_0\coloneqq g^+$ satisfies the desired properties. 

     For the uniqueness, suppose that $p_1$ is a fixed point of $g$ such that if $(x,p)$ is a $\Theta$-positively translating pair for $g$, then $(x,p_1)$ is a $\Theta$-positive translating pair for $g$ and $\Diam_{g\cdot x}(x,p_1) \subset \Diam_{g\cdot x}(x,p)$ for every $\Theta$-positively translating pair $(x,p)$.
     Then (by mutual inclusion), $\Diam_{g\cdot x}(x,p_0) = \Diam_{g\cdot x}(x,p_1)$, and so $p_0 = p_1$ by \Cref{lem: ExtremitiesDiamonds}.
\end{proof}

From now on, we denote by $g^+$ the point $p_0$ given by \Cref{propo:PosTransUniqueFP}, and set $g^- \coloneqq (g^{-1})^+~$, which we call respectively the \emph{forward and backward fixed points} of the $\Theta$-positively translating element $g \in G_{\Fb}$. The proof of \Cref{propo:PosTransUniqueFP} shows that these notations are compatible with the ones of \Cref{propo: OverRPosRotImpliesDiv} when $g \in G_{\Rb}$. Furthermore, since the description of $p_0$ in \Cref{propo:PosTransUniqueFP} is semi-algebraic, the map from the set of $\Theta$-positively translating elements in $G_{\Fb}$ to $\Fc_{\Theta,\Fb}$ that sends $g$ to $g^+$ is semi-algebraic. 

As a corollary, we obtain the following characterization of $\Theta$-positively translating elements that are proximal.

\begin{corollary}\label{corol: CharWeakThetaProxForPosTrans}
    Let $g\in G_{\Fb}$ be $\Theta$-positively translating.
    Then $g$ is weakly $\Theta$-proximal if and only if its backward and forward fixed points $g^-$ and $g^+$ are transverse. In that case, $g^+$ and $g^-$ are respectively the weak attracting and repelling fixed flags of $g$. Finally, if $(x,p)$ is a translating pair for $g$, then
    \[(g^-,g^{-1}\cdot x, x, g\cdot x, g^+)\]
    is positive.
\end{corollary}

\begin{proof}
    Again, this statement is semi-algebraic and it is thus enough to prove it in the case $\Fb=\Rb$.

    Recall that $g\in G_{\Rb}$ is $\Theta$-proximal if and only if it has attracting and repelling fixed flags $a_g$ and $r_g$, i.e.\ $a_g$ and $r_g$ are transverse, and there are open sets $V^+$ and $V^-$ containing $a_g$ and $r_g$ respectively, so that $(g^n|_{V^+})_{n\in\Nb}$ converges uniformly on compact sets to the constant map whose image is $a_g$, and $(g^{-n}|_{V^-})_{n\in\Nb}$ converges uniformly on compact sets to the constant map whose image is $r_g$. 

    At the same time, since $g\in G_{\Rb}$ is $\Theta$-positively translating, \Cref{propo: OverRPosRotImpliesDiv} implies that there is an open set $U\subset\Fc_{\Theta,\Fb}$ (namely, the set of points in $\Fc_{\Theta,\Fb}$ with positive $\langle g\rangle$-orbits) such that for all $p\in U$, $g^n\cdot p\to g^+$ and $g^{-n}\cdot p\to g^-$ as $n\to +\infty$.

    Thus, if $g$ is $\Theta$-proximal, then \Cref{lemma: divergence in flag manifold}, implies that $g^+=a_g$ and $g^-=r_g$. In particular, $g^+$ and $g^-$ are transverse. Conversely, if $g^+$ and $g^-$ are transverse, then there are open sets $V^+$ and $V^-$ containing $g^+$ and $g^-$ respectively, such that every point in $V^+$ is transverse to $g^-$ and every point in $V^-$ is transverse to $g^+$. It follows from \Cref{lemma: divergence in flag manifold} that the sequence $(g^n|_{V^+})_{n\in\Nb}$ converges uniformly on compact sets to the constant map whose image is $g^+$, and $(g^{-n}|_{V^-})_{n\in\Nb}$ converges uniformly on compact sets to the constant map whose image is $g^-$, so $g$ is $\Theta$-proximal. 

    If $(x,p)$ is a translating pair for $g$, then 
    \[(g^{-n}\cdot x,g^{-1}\cdot x,x,g\cdot x,g^n\cdot x)\] 
    is positive for all integers $n\ge 2$, so $(g^-,g^{-1}\cdot x,x,g\cdot x,g^+)$ is semi-positive. Since $g^-,g^+\in\overline{\Diam}_x^{\rm opp}(g^{-2}\cdot x,g^2\cdot x)$, \Cref{lem: closure of diamond} implies that $g^-,g^+\in\Diam_x^{\rm opp}(g^{-1}\cdot x,g\cdot x)$, so by \Cref{opposites transverse}, $g^-$ and $g^+$ are transverse to $g^{-1}\cdot x$, $x$, and $g\cdot x$. Since $g^-$ and $g^+$ are transverse, \Cref{prop: transverse + semi positive implies positive} implies that $(g^-,g^{-1}\cdot x,x,g\cdot x,g^+)$ is positive.
\end{proof}

Finally, let us prove the following characterization of proximal elements:
\begin{proposition}\label{lem: CharPosTransWeaklyProx}
    Let $g$ be an element of $G_{\Fb}$ for which there exist $x,y\in \mathcal F_{\Theta,\Fb}$ such that $(x,g\cdot x, g^2\cdot x, g\cdot y, y)$ is positive. Then $g$ is $\Theta$-positively translating. Moreover, the intersections
    \[\bigcap_{n\in \Nb} g^n\cdot \Diam_{g\cdot x}(x,y)~ \quad\text{and}\quad \bigcap_{n\in \Nb} g^{-n}\cdot \Diam_{g\cdot x}^\opp(x,y)\]
    are nested and contain $g^+$ and $g^-$ respectively. In particular, $g$ is weakly $\Theta$-proximal, and the tuple \[(g^-, g^{-n}\cdot x,\ldots , x, \ldots, g^n\cdot x, g^+ , g^n \cdot y, \ldots y, \ldots , g^{-n}\cdot y)\] is positive for all integers $n>0$.
\end{proposition}

\begin{figure}[ht]
    \centering
    \begin{tikzpicture}[scale=2]
        \draw (0,0) circle (1cm);
        \draw (0,1) node{$\bullet$};
        \draw[above] (0,1) node{$x$};
        \draw (0,-1) node{$\bullet$};
        \draw[below] (0,-1) node{$y$};
        \draw (.4,{sqrt(1-.4*.4)}) node{$\bullet$};
        \draw[right] (.4,{sqrt(1-.4*.4)+.2}) node{$g  x$};
        \draw (.6,{sqrt(1-.6*.6)}) node{$\bullet$};
        \draw[right] (.6+.05,{sqrt(1-.6*.6)+.05}) node{$g^2  x$};
        \draw[right] (.8,{sqrt(1-.8*.8)+.05}) node{\rotatebox{-20}{$\ddots$}};
        \draw (.95,{sqrt(1-.95*.95)}) node{$\bullet$};
        \draw[right] (.95,{sqrt(1-.95*.95)}) node{$g^n  x$};
        \draw (1,0) node{$\bullet$};
        \draw[right, blue] (1,0) node{$g^+$};
        \draw (.8,{-sqrt(1-.8*.8)}) node{$\bullet$};
        \draw[right] (.8,{-sqrt(1-.8*.8)}) node{$g^n  y$};
        \draw[right] (.4+.05,{-sqrt(1-.4*.4)}) node{\rotatebox{-110}{$\ddots$}};
        \draw (.2,{-sqrt(1-.2*.2)}) node{$\bullet$};
        \draw[below] (.2+.05,{-sqrt(1-.2*.2)}) node{$g y$};
        \draw[left] (-.5,{sqrt(1-.5*.5)+.05}) node{\rotatebox{70}{$\ddots$}};
        \draw (-.95,{sqrt(1-.95*.95)}) node{$\bullet$};
        \draw[left] (-.95,{sqrt(1-.95*.95)}) node{$g^{-n}  x$};
        \draw (-.8,-{sqrt(1-.8*.8)}) node{$\bullet$};
        \draw[left, blue] (-.8,-{sqrt(1-.8*.8)}) node{$g^-$};
        \draw (-.6,{-sqrt(1-.6*.6)}) node{$\bullet$};
        \draw[left] (-.6,{-sqrt(1-.6*.6)}) node{$g^{-n}  y$};
        \draw[left] (-.2,{-sqrt(1-.2*.2)}) node{\rotatebox{20}{$\ddots$}};
    \end{tikzpicture}
    \caption{The orbit of $g$}
\end{figure}

\begin{proof}
Note that the hypothesis of the proposition is semi-algebraic, and the conclusion of the proposition is a countable family of semi-algebraic statements. As such, by the transfer principle, it suffices to prove the proposition over $\Rb$. Thus, for the remainder of this proof we will assume that $\Fb=\Rb$.

First, we prove by induction that for every integer $n\ge 1$, the tuple
\begin{align}\label{positive g orbit of x and y}
    (g^n \cdot y, g^{n-1}\cdot y, \ldots, y, x, g\cdot x, \ldots, g^n\cdot x)
\end{align}
is positive. The base case holds by assumption. For the inductive step, observe that the quadruples 
\[(g^n\cdot y, g^{n-1}\cdot y, g^n \cdot x, g^{n+1}\cdot x) = g^{n-1}\cdot (g\cdot y, y, g\cdot x, g^2\cdot x)\]
and
\[(g^{n+1}\cdot y,g^n\cdot y, g^n\cdot x, g^{n+1}\cdot x) = g^n\cdot (g\cdot y,y,x, g\cdot x)\]
are both positive. Applying \Cref{lem: AddingElemToPosTuple} to the tuple \eqref{positive g orbit of x and y} twice yields the inductive step.

It follows immediately from the positivity of the tuple \eqref{positive g orbit of x and y} that $g$ is $\Theta$-positively rotating, and the intersections in the statement of the proposition are nested. Since $\Fb=\Rb$, \Cref{propo: PosTransPosRot} implies that $g$ is $\Theta$-positively translating, and
\[g^+ =\lim_{n \to +\infty} g^n\cdot x \in \bigcap_{n\in \Nb} g^n \cdot \overline{\Diam}_{g\cdot x}(x,y)=\bigcap_{n\in \Nb} g^n \cdot \Diam_{g\cdot x}(x,y).\]
Similarly, \[g^- \in \bigcap_{n\in \Nb} g^{-n} \cdot \Diam_{g\cdot x}^\opp(x,y)~.\] 
It follows that any pair of flags in the set 
\[\{g^-,g^+\}\cup\{g^n\cdot x\mid n\in\Zb\}\cup\{g^n\cdot y \mid n\in\Zb\}\]
are contained in opposite diamonds and are thus transverse by \Cref{opposites transverse}. In particular $g^-$ and $g^+$ are transverse, so \Cref{corol: CharWeakThetaProxForPosTrans} implies that $g$ is weakly $\Theta$-proximal. Furthermore, for each positive integer $n$, the positivity of the tuple \eqref{positive g orbit of x and y} implies that
\[(g^-, g^{-n}\cdot x,\ldots , x, \ldots, g^n\cdot x, g^+ , g^n \cdot y, \ldots y, \ldots , g^{-n}\cdot y)\]
is semi-positive, so \Cref{prop: transverse + semi positive implies positive} implies that it is positive.
\end{proof}

\subsection{%
\texorpdfstring{%
Weakly $\Theta$-proximal representations}%
{Weakly Theta-proximal representations}}
\label{section : proximal representations}
We will now discuss the proximality properties of the elements in the images of $\Theta$-positive representations. To do so, it is convenient to introduce the following notion.

\begin{definition}
    A representation $\rho \colon \Gamma \to G_{\Fb}$ is \emph{weakly $\Theta$-proximal} if $\rho(\gamma)$ is weakly $\Theta$-proximal for every hyperbolic element $\gamma \in \Gamma$.
\end{definition}

Recall that $\Lambda_h\subset \Sb^1$ denotes the set of attracting fixed points of hyperbolic elements in $\Gamma$. It is a standard fact that for any Fuchsian group $\Gamma$ and any hyperbolic elements $\gamma,\gamma'\in\Gamma$, if $\gamma^+=\gamma'^+$, then $\gamma$ and $\gamma'$ commute. As a consequence, if $\rho\colon\Gamma\to G_{\Fb}$ is weakly $\Theta$-proximal, then we may define a $\rho$-equivariant map 
\[\xi^h_\rho \from \Lambda_h \to \Fc_{\Theta,\Fb}\] 
that sends $\gamma^+\in\Lambda_h$ to the weakly attracting flag of $\rho(\gamma)$. We refer to $\xi^h_\rho$ as the \emph{$\Theta$-proximal limit map} of $\rho$.

The main goal of this subsection is to prove \Cref{thm-intro: weak proximality}, which we restate here.

\begin{theorem}[\Cref{thm-intro: weak proximality}]\label{theorem: limit map extends to proximal limit map}
    Let $\Gamma$ be a Fuchsian group of the first kind and $\rho$ a representation of $\Gamma$ into $G_{\Fb}$, where $\Fb$ is a real closed field.
    Then the following are equivalent:
    \begin{enumerate}
        \item
        \label{weak proximality: item positive}
        $\rho$ is $\Theta$-positive;
        \item 
        \label{weak proximality: item positive on fixed points}
        $\rho$ is weakly $\Theta$-proximal, and the $\Theta$-proximal limit map $\xi^h_\rho$ of $\rho$ is positive.
    \end{enumerate}
\end{theorem}

\begin{remark}
If $\Gamma$ is non-elementary and of the second kind, \Cref{Fuchsian groups lemma} implies that there is a Fuchsian group $\Gamma'$ of the first kind and a semi-conjugacy $\iota \colon \Gamma\to\Gamma'$. Observe that if $\gamma$ is a \emph{peripheral hyperbolic element} of $\Gamma$, i.e.\ a hyperbolic element preserving a connected component of $\Sb^1\setminus \Lambda(\Gamma)$, then $\iota(\gamma)\in \Gamma'$ is parabolic. It thus follows from \Cref{propo: equivalence positivity first and second kind Fuchsian groups} that \Cref{theorem: limit map extends to proximal limit map} would hold if we replace $\Lambda_h$ by the set of fixed points of non-peripheral hyperbolic elements.
\end{remark}

The proof of \Cref{theorem: limit map extends to proximal limit map} will require two lemmas. The first gives the proximality properties of elements in the image of $\rho$.

\begin{lemma}\label{propo: positive implies weak proximal}
    Let $\Gamma$ be of the first kind, and $\rho \from \Gamma \to G_{\Fb}$ a $\Theta$-positive representation. Then:
        \begin{enumerate}
        \item for every hyperbolic element $\gamma \in \Gamma$, $\rho(\gamma)$ is weakly $\Theta$-proximal and $\Theta$-positively translating.
        \item for every parabolic element $\eta \in \Gamma$, $\rho(\eta)$ is $\Theta$-positively rotating.
        \end{enumerate}
\end{lemma}

\begin{proof}
    Let $D$ be a non-empty $\Gamma$-invariant subset of $\Sb^1$ and $\xi \colon D \to \Fc_{\Theta,\Fb}$ a $\rho$-equivariant positive map.
    Let $\gamma \in \Gamma$ be a hyperbolic element.
    Choose $x \neq y \in D$ such that $(\gamma^-,x,\gamma^+,y)$ is cyclically ordered. Then $(x, \gamma \cdot x, \gamma^2 \cdot x, \gamma \cdot y, y)$ is cyclically ordered.
    By positivity and $\rho$-equivariance of $\xi$, the tuple 
    $(\xi(x), \rho(\gamma) \cdot \xi(x), \rho(\gamma)^2 \cdot \xi(x), \rho(\gamma) \cdot \xi(y), \xi(y))$ is positive.
    Thus \Cref{lem: CharPosTransWeaklyProx} implies (1).

    Let $\eta \in \Gamma$ be a parabolic element. For any $x \in D\setminus \eta^+$, either $(x, \eta\cdot x, \ldots, \eta^n\cdot x)$ is cyclically ordered for every $n$ (if $\eta$ is a positive parabolic) or $(\eta^n\cdot x, \ldots , \eta\cdot x, x)$ is cyclically ordered for all $n$ (if $\eta$ is a negative parabolic). In any case, since positivity is invariant under dihedral permutations, we deduce that the sequence $(\rho(\eta)^n \cdot \xi(x))_{n \in \Nb} = (\xi(\eta^n\cdot x))_{n\in \Nb}$ is positive, hence (2) holds. 
\end{proof}

\begin{remark}
If $\Gamma$ is of the first kind and $\eta\in\Gamma$ is a parabolic element, then there are $\Theta$-positive representations $\rho\colon\Gamma\to G_{\Fb}$ such that $\rho(\eta)$ is not $\Theta$-positively translating, see \Cref{subsection: ExNonFramablePosRepr}. 
\end{remark}

In light of \Cref{propo: positive implies weak proximal}, we may use the $\Theta$-proximal limit map of $\rho$ to extend any $\rho$-equivariant, positive map defined on a $\Gamma$-invariant non-empty subset of $\Sb^1$ to a subset that contains $\Lambda_h$:

\begin{lemma}\label{corollary: limit map extends to proximal limit map}
    Let $\Gamma$ be of the first kind. If $\rho \colon\Gamma\to G_{\Fb}$ is a representation and $\xi \from D \to \Fc_{\Theta,\Fb}$ is a positive and $\rho$-equivariant map defined on a $\Gamma$-invariant non-empty subset $D$ of $\Sb^1$, then $\xi$ and $\xi^h_\rho$ agree on $D\cap\Lambda_h$, and the map
    \[\begin{array}{rccl}
    \xi \cup \xi^h_\rho \colon & D\cup \Lambda_h & \to  &\Fc_{\Theta,\Fb} \\
    & x & \mapsto & \begin{cases*} \xi^h_\rho(x) & \textrm{if $x\in \Lambda_h$}, \\ \xi(x) & \textrm{otherwise}\end{cases*}
    \end{array}\]
    is again $\rho$-equivariant and positive, where $\xi^h_\rho$ is the $\Theta$-proximal limit map of $\rho$ (which is weakly $\Theta$-proximal by \Cref{propo: positive implies weak proximal}). 
\end{lemma}

\begin{proof}
First, we prove that $\xi$ and $\xi^h_\rho$ agree on $D\cap\Lambda_h$. Suppose that $\gamma\in\Gamma$ is a hyperbolic element such that $\gamma^+\in D\cap\Lambda_h$. Since $D\subset\Sb^1$ is dense, there exists $x,y\in D$ such that the tuple $(\gamma^-,x,\gamma\cdot x,\gamma_+,\gamma\cdot y,y)$ is positive. Since $\xi$ is positive and $\rho$-equivariant, it follows that both $(\xi(x),\xi(\gamma^+))$ and $(\xi(y),\xi(\gamma^+))$ are $\Theta$-positively translating pairs for $\rho(\gamma)$. Then by \Cref{propo:PosTransUniqueFP} and the definition of $\xi^h_\rho$, 
\[\xi^h_\rho(\gamma^+)\in\overline{\Diam}_{\xi(\gamma\cdot x)}(\xi(x),\xi(\gamma^+))=\overline{\Diam}_{\xi(y)}^{\rm opp}(\xi(x),\xi(\gamma^+)).\]
Similarly, $\xi^h_\rho(\gamma^+)\in\overline{\Diam}_{\xi(x)}^{\rm opp}(\xi(y),\xi(\gamma^+))$. By \Cref{lem: intersection closures diamonds}, 
\[\overline{\Diam}_{\xi(y)}^{\rm opp}(\xi(x),\xi(\gamma^+))\cap\overline{\Diam}_{\xi(x)}^{\rm opp}(\xi(y),\xi(\gamma^+))=\xi(\gamma^+),\]
so $\xi^h_\rho(\gamma^+)=\xi(\gamma^+)$.

Next, we show that $\xi\cup\xi^h_\rho$ is positive. Pick any cyclically ordered tuple in $D\cup\Lambda_h$. Since $D$ and $\Lambda_h$ are both dense in $\Sb^1$, by enlarging the tuple if necessary, we may assume that it is of the form
\[(x_1,\gamma_1^+,y_1,x_2,\gamma_2^+,y_2,\dots,x_k,\gamma_k^+,y_k)\]
for some positive integer $k\ge 2$, where $\gamma_1,\dots,\gamma_k\in\Gamma$ are hyperbolic elements and $x_i,y_i\in D$ for all $i=1,\dots,k$. We want to show that the tuple of flags
\begin{align}\label{long tuple}
\begin{split}
\big(\xi(x_1),\rho(\gamma_1)^+,\xi(y_1),\xi(x_2),\rho(\gamma_2)^+,\xi(y_2),\dots,\xi(x_k),\rho(\gamma_k)^+,\xi(y_k)\big)
\end{split}
\end{align}
is positive. 

Notice that for all $i=1,\dots,k$, the tuple 
    \[(x_i,\gamma_i\cdot x_i,\gamma_i^2\cdot x_i,\gamma_i\cdot y_i,y_i)\] 
    is cyclically ordered. From \Cref{lem: CharPosTransWeaklyProx} and the $\rho$-equivariance and positivity of $\xi$, we get that $\rho(\gamma_i)^+ \in \Diam_{\xi(\gamma_i \cdot x_i)}(\xi(x_i),\xi(y_i))$. In particular, if we set $x_{k+1}\coloneqq x_1$, then the tuple
    \[(\xi(x_i),\rho(\gamma_i)^+, \xi(y_i),\xi(x_{i+1}))\]
    is positive. At the same time, the positivity of $\xi$ implies that the tuple
\begin{align*}
\big(\xi(x_1),\xi(y_1),\xi(x_2),\xi(y_2),\dots,\xi(x_k),\xi(y_k)\big)
\end{align*}
is positive. The positivity of \eqref{long tuple} now follows from iterated applications of \Cref{lem: AddingElemToPosTuple}.
\end{proof}

We can now prove \Cref{theorem: limit map extends to proximal limit map}

\begin{proof}[Proof of \Cref{theorem: limit map extends to proximal limit map}]
    The direction (\ref{thm-intro: weak proximality: item positive on fixed points}) $\implies$ (\ref{thm-intro: weak proximality: item positive}) is clear. For the other direction, if we assume that $\rho$ is $\Theta$-positive, then \Cref{propo: positive implies weak proximal} implies that $\rho$ is weakly $\Theta$-proximal, while \Cref{corollary: limit map extends to proximal limit map} implies that $\xi^h_\rho$ is positive.
\end{proof}

\begin{corollary}
\label{cor: positivity finite-index subgroup}
    Let $\Gamma$ be Fuchsian group, $\rho \colon \Gamma \to G_{\Fb}$ a representation and $\Gamma' < \Gamma$ a finite-index subgroup.
    Then $\rho$ is $\Theta$-positive if and only if its restriction to $\Gamma'$ is $\Theta$-positive.
\end{corollary}

\begin{proof}
    By \Cref{propo: equivalence positivity first and second kind Fuchsian groups} we can without loss of generality assume that $\Gamma$ is of the first kind. 
    
    Observe that if $\gamma \in \Gamma$ be hyperbolic, the assumption that $\Gamma'<\Gamma$ is of finite-index, implies that there exists $k>0$ such that $\gamma^k \in \Gamma'$. This implies that $\Lambda_h(\Gamma)=\Lambda_h(\Gamma')$. In particular, $\Gamma'$ is also of the first kind. This also implies that a representation $\rho:\Gamma\to G_{\Fb}$ is weakly $\Theta$-proximal if and only if its restriction to $\Gamma'$ is weakly $\Theta$-proximal, and the $\Theta$-proximal limit map of $\rho$ and $\rho|_{\Gamma'}$ agree. The corollary now follows by applying \Cref{theorem: limit map extends to proximal limit map} to both $\rho$ and $\rho|_{\Gamma'}$.
\end{proof}

\subsection{%
\texorpdfstring{%
Boundary extensions of $\Theta$-positive representations}%
{Boundary extensions of Theta-positive representations}}
\label{subsec:boundary extension}
For the remainder of this section we assume that $\Gamma$ is of the first kind and $\rho \from \Gamma\to G_{\Fb}$ is a representation. The goal of this subsection is to characterize when $\rho$ is positive using boundary extensions, which we now define.

\begin{definition} \label{def: Boundary Extension}
    A \emph{$\Theta$-boundary extension} of $\rho$ is a pair $(\Vc,\zeta)$ where $\Vc$ is a closed $\rho(\Gamma)$-invariant subset of $\Fc_{\Theta,\Fb}$ and $\zeta\colon  \Vc \to \Sb^1$ is a continuous $(\rho,\Id)$-equivariant (not necessarily surjective) map. We call this $\Theta$-boundary extension 
    \begin{itemize}
        \item \emph{transverse} if $x$ and $y$ are transverse whenever $\zeta(x)\neq \zeta(y)$;
        \item \emph{positive} if $(x_1, \ldots, x_n) \in \Vc^n$ is positive whenever $(\zeta(x_1), \ldots, \zeta(x_n))$ is cyclically ordered.
    \end{itemize}
    A \emph{partial section of $\zeta$} is a $\rho$-equivariant map $\xi\colon D \to \Vc$ defined on a non-empty $\Gamma$-invariant subset $D$ of $\Sb^1$, such that $\zeta \circ \xi = \Id_D$.
\end{definition}

The notion of a boundary extension was introduced by \cite{Weisman} in the study of generalized relative Anosov representations over the real field, where it is defined as a surjective map. In our setting over general real closed fields, we drop the surjectivity assumption for reasons explained later, see \Cref{non-surjectivity} and \Cref{sss:Surjectivity Cantor Complete}.

Notice that if $(\Vc,\zeta)$ is a positive $\Theta$-boundary extension for $\rho$, then any partial section of $\zeta$ is positive. Thus, if $\zeta$ admits a partial section, then $\rho$ is positive. However, we don't know in general that $\zeta$ admits a partial section.
Indeed, one could imagine that $\zeta(\Vc) \subset \Lambda_h\cup \Lambda_p$ and the stabilizer in $\Gamma$ of every point in $\Vc$ is trivial. Nevertheless, we still have the following result:

\begin{proposition} \label{prop: Boundary extension implies positive}
    If the representation $\rho$ admits a positive $\Theta$-boundary extension $(\Vc,\zeta)$, then $\rho$ is $\Theta$-positive.
\end{proposition}

To prove \Cref{prop: Boundary extension implies positive}, we use the following lemma.

\begin{lemma}\label{skjfghjksh}
    Assume $\rho$ admits a positive $\Theta$-boundary extension $(\Vc,\zeta)$. Let $\gamma$ be a hyperbolic element in $\Gamma$, and let $x,y\in \Vc$ be such that $(\zeta(x),\gamma^+, \zeta(y), \gamma^-)$ is cyclically ordered. Then $\rho(\gamma)$ is weakly $\Theta$-proximal and 
    \[\rho(\gamma)^+ \in \Diam_{x}^\opp(\rho(\gamma)\cdot x,\rho(\gamma)\cdot y)~.\]
\end{lemma}

\begin{proof}
    Since the tuple $(\zeta(x),\gamma^+, \zeta(y), \gamma^-)$ is cyclically ordered, so is the tuple $\big(\zeta(x), \gamma\cdot \zeta(x), \gamma^2\cdot \zeta(x), \gamma \cdot \zeta(y), \zeta(y)\big)$. Since $(\Vc,\zeta)$ is a positive $\Theta$-boundary extension, the tuple $\big(x,\rho(\gamma)\cdot y, \rho(\gamma)^2\cdot x, \rho(\gamma)\cdot y, y\big)$ is positive, and the lemma follows from \Cref{lem: CharPosTransWeaklyProx}.
\end{proof}

\begin{proof}[Proof of \Cref{prop: Boundary extension implies positive}]
By \Cref{theorem: limit map extends to proximal limit map}, we need to show that $\rho$ is weakly $\Theta$-proximal and its $\Theta$-proximal limit map $\xi^h_\rho\colon\Lambda_h\to\Fc_{\Theta,\Fb}$ is positive. To do so, let $\gamma_1,\ldots, \gamma_n \in \Gamma$ be hyperbolic elements such that $(\gamma_1^+\ldots, \gamma_n^+)$ is cyclically ordered. It suffices to verify that for all $i=1,\dots,n$, $\rho(\gamma_i)$ is weakly $\Theta$-proximal, and the tuple $(\rho(\gamma_1)^+, \ldots, \rho(\gamma_n)^+)$ is positive.

Since $\zeta(\Vc)\subset\Sb^1$ is non-empty and $\Gamma$-invariant, it is dense. Thus, for all $i=1,\dots,n$, we can choose $x_i, y_i\in  \Vc$ such that the tuple
\begin{align*}
(\zeta(x_1), \gamma_1\cdot \zeta(x_1), \gamma_1^+, \gamma_1\cdot \zeta(y_1), \zeta(y_1),&\zeta(x_2), \gamma_2\cdot \zeta(x_2), {\gamma_2}^+, \gamma_2\cdot \zeta(y_2), \zeta(y_2), \\
&\ldots, \zeta(x_n), \gamma_n\cdot \zeta(x_n), {\gamma_n}^+, \gamma_n\cdot \zeta(y_n), \zeta(y_n))
\end{align*}
is cyclically ordered. In particular, for each $i=1,\dots,n$, the tuple 
\[(\zeta(x_i), \gamma_i\cdot \zeta(x_i), \gamma_i^+, \gamma_i\cdot \zeta(y_i), \zeta(y_i)\] 
is cyclically ordered, it follows that $(\zeta(x_i), \gamma_i^+, \zeta(y_i),\gamma_i^-)\big)$ is also cyclically ordered. Thus, \Cref{skjfghjksh} implies that the element $\rho(\gamma_i)$ is weakly $\Theta$-proximal and the tuple $(x_i, \rho(\gamma_i)\cdot x_i, \rho(\gamma_i)^+, \rho(\gamma_i)\cdot y_i)$ is positive. At the same time, by positivity of the $\Theta$-boundary extension, the tuple 
\[(x_1, \rho(\gamma_1)\cdot x_1, \rho(\gamma_1)\cdot y_1, y_1, \ldots, x_n, \rho(\gamma_n)\cdot x_n, \rho(\gamma_n) \cdot y_n, y_n)\] is positive.
Hence, repeated applications of \Cref{lem: AddingElemToPosTuple} yields 
\[(x_1, \rho(\gamma_1)\cdot x_1, \rho(\gamma_1)^+, \rho(\gamma_1)\cdot y_1, y_1, \ldots, x_n, \rho(\gamma_n)\cdot x_n, \rho(\gamma_n)^+, \rho(\gamma_n) \cdot y_n, y_n)\] is positive. In particular, $(\rho(\gamma_1)^+, \ldots, \rho(\gamma_n)^+)$ is positive.
\end{proof}

The converse of \Cref{prop: Boundary extension implies positive} is also true; in fact, we have the following stronger statement:

\begin{theorem}\label{prop: converse}
If $\rho$ is $\Theta$-positive, then it admits a (necessarily unique) positive $\Theta$-boundary extension $(\bdext, \zeta^{\mathrm M}_\rho)$ which is maximal in the following sense:
\begin{itemize}
    \item For every $\rho$-equivariant positive map $\xi \from D \subset \Sb^1 \to \Fc_{\Theta,\Fb}$, we have $\xi(D)\subset \bdext $ and $\zeta \circ \xi = \Id\vert_{ D}$,
    \item For every positive $\Theta$-boundary extension $(\Vc, \zeta)$, we have \[\Vc \subset \bdext \quad\text{and}\quad\zeta= {\zeta^{\mathrm M}_\rho}\vert_{ \Vc}~.\]
\end{itemize}
\end{theorem}

Together, \Cref{prop: Boundary extension implies positive} and \Cref{prop: converse} imply \Cref{thm: Boundary Extension Intro}.

To prove \Cref{prop: converse}, we first construct $(\bdext, \zeta^{\mathrm M}_\rho)$. If $\rho$ is $\Theta$-positive, there is a $\rho$-equivariant, positive limit map $\xi \from D\to \Fc_{\Theta,\Fb}$, where $D\subset \Sb^1$ is a $\Gamma$-invariant subset. For any point $x\in\Sb^1$, choose $z\in D\setminus\{x\}$, and choose sequences $(x_n)_{n\in\Nb}$ and $(x_n')_{n\in\Nb}$ in $D$ that converge to $x$, such that for all $n\in\Nb$,
\[(z,x_1,\dots,x_{n-1},x_n,x,x_n',x_{n-1}',\dots,x_1')\]
is cyclically ordered.
Then $(x_n,x,x_m',z)$ is cyclically ordered for all $m,n\in\Nb$. Set 
\[\bdext(x) \coloneqq \bigcap_{n\in\Nb} \Diam^\opp_{\xi(z)}(\xi(x_n),\xi(x_n')) \subset \Fc_{\Theta,\Fb}~.\]

We will now verify that $\bdext(x)$ depends only on $\rho$ and $x$, and not on any of the other choices made in its definition.

\begin{lemma}\label{lemma: boundary extension pointwise independence}
Suppose that $\rho$ is $\Theta$-positive, and $x\in\Sb^1$. The set $\bdext(x)$ constructed above does not depend on the choice of $z$, $(x_n)_{n\in\Nb}$ and $(x_n')_{n\in\Nb}$, and $\xi \colon D\to \Fc_{\Theta,\Fb}$. In particular, if $\gamma\in\Gamma$ is hyperbolic, then $\rho(\gamma)^+\in\bdext(\gamma^+)$.
\end{lemma}

\begin{proof}

We first fix $z$, $\xi$, and $(x_n)_{n\in\Nb}$, and show that $\bdext(x)$ does not depend on the choice of $(x_n')_{n\in\Nb}$. First, observe that $\bdext(x)$ remains unchanged if we replace the sequence $(x_n')_{n\in\Nb}$ in its construction with a subsequence; indeed, if $(k_n)_{n\in\Nb}$ is a strictly increasing sequence of integers, then 
\begin{align*}
\bigcap_{n\in\Nb} \Diam^\opp_{\xi(z)}(\xi(x_n),\xi(x_n'))&=\bigcap_{n\in\Nb} \Diam^\opp_{\xi(z)}(\xi(x_{k_n}),\xi(x_{k_n}'))\\
&\subset\bigcap_{n\in\Nb} \Diam^\opp_{\xi(z)}(\xi(x_n),\xi(x_{k_n}')) \\
&\subset\bigcap_{n\in\Nb} \Diam^\opp_{\xi(z)}(\xi(x_n),\xi(x_n'))=\bigcap_{n\in\Nb} \Diam^\opp_{\xi(z)}(\xi(x_n),\xi(x_n')).
\end{align*}

Let
$(y_n')_{n\in\Nb}$ be a sequence in $D$ converging to $x$ from the right with respect to $z$. If $(x_n')_{n\in\Nb}$ and $(y_n')_{n\in\Nb}$ share a common sub-sequence, then the above implies that constructing $\bdext(x)$ using $(x_n')_{n\in\Nb}$ or $(y_n')_{n\in\Nb}$ yields the same set. On the other hand, if $(x_n')_{n\in\Nb}$ and $(y_n')_{n\in\Nb}$ do not share a common subsequence, then by taking subsequences, we may assume that for all $n\in\Nb$, $(x, x_{n+1}',y'_n, x_n', z)$ is cyclically ordered. By the positivity of $\xi$, we get that for all $n\in\Nb$, 
    \begin{align*}
    \Diam^\opp_{\xi(z)}(\xi(x_{n+1}),\xi(x_{n+1}')) \subset\Diam^\opp_{\xi(z)}(\xi(x_n),\xi(y_n')) \subset \Diam^\opp_{\xi(z)}(\xi(x_n),\xi(x_n'))~.
    \end{align*}

By taking intersections over all $\Nb$, this implies that 
\[\bigcap_{n\in \Nb} \Diam^\opp_{\xi(z)}(\xi(x_n),\xi(y'_n))= \bigcap_{n\in \Nb} \Diam^\opp_{\xi(z)}(\xi(x_n),\xi(x'_n))~.\]
The same argument shows that if we fix $z$ and $\xi$, then $\bdext$ does not depend on the choice of $(x_n)_{n\in\Nb}$ as well.

Next, we fix $\xi$ and show that $\bdext(x)$ does not depend on $z$. Let $z'\in D\setminus\{x\}$. Then observe by taking a tail of the sequences $(x_n)_{n\in\Nb}$ and $(x_n')_{n\in\Nb}$, we may assume that both of these sequences converge to $x$ from the left and from the right respectively with respect to both $z$ and $z'$. Hence,
\[\bigcap_{n\in\Nb} \Diam^\opp_{\xi(z)}(\xi(x_n),\xi(x_n'))=\bigcap_{n\in\Nb} \Diam^\opp_{\xi(z')}(\xi(x_n),\xi(x_n')).\]

Finally, we prove that $\bdext(x)$ is independent of the choice of the positive map $\xi$. By \Cref{corollary: limit map extends to proximal limit map}, for any positive equivariant map $\xi \colon D\to \Fc_{\Theta,\Fb}$, the map $\xi \cup \xi^h_\rho\colon D\cup \Lambda_h \to \Fc_{\Theta, \Fb}$ is again positive and equivariant. Then since $\bdext(x)$ does not depend on the choice of the sequences $(x_n)_{n\in\Nb}$ or $(x_n')_{n\in\Nb}$, defining $\bdext(x)$ using $\xi^h_\rho$ and sequences in $\Lambda_h$ in place of $\xi$ and sequences in $D$ yield the same set. 

For the second statement of the lemma, simply observe that if $\gamma\in\Gamma$ is hyperbolic and we choose $x,y\in D$ such that $(\gamma^-,x,\gamma^+,y)$ is cyclically ordered, then $(\gamma^n\cdot x)_{n\in\Nb}$ and $(\gamma^n\cdot y)_{n\in\Nb}$ converge to $\gamma^+$ and for all $n\in\Nb$,
\[(\gamma^-,\gamma\cdot x,\dots,\gamma^n\cdot x,\gamma^+,\gamma^n\cdot y,\dots,\gamma\cdot y,\gamma^-)\]
are cyclically ordered. In particular, $(\xi(x),\rho(\gamma)\cdot \xi(x), \rho(\gamma)^2\cdot \xi(x), \rho(\gamma)\cdot \xi(y), \xi(y))$ is positive, so we may apply \Cref{lem: CharPosTransWeaklyProx} to conclude.
\end{proof}

Next, we verify the disjointness of $\bdext(x)$ and $\bdext(y)$ when $x$ and $y$ are distinct.

\begin{lemma}\label{disjoint}
Suppose that $\rho$ is $\Theta$-positive. If $x\neq y$, then $\bdext (x) \cap \bdext (y) = \emptyset$.
\end{lemma}

\begin{proof}
Choose a positive limit map $\xi \colon D \to \Fc_{\Theta,\Fb}$ and some $z\in D\setminus\{x,y\}$. By switching $x$ and $y$, we may assume that $(x,z,y)$ is cyclically ordered. Choose sequences $(x_n)_{n\in\Nb}$, $(x_n')_{n\in\Nb}$ in $D$ that converge to $x$ and sequences $(y_n)_{n\in\Nb}$ and $(y_n')_{n\in\Nb}$ in $D$ that converge to $y$ such that for all $n\in\Nb$,
\[(z,x_1,\dots,x_n,x,x_n',\dots,x_1')\quad\text{and}\quad(z,y_1,\dots,y_n,y,y_n',\dots,y_1')\]
are cyclically ordered. Then for sufficiently large $n\in\Nb$, $(x_n,x,x_n',y_n,y,y_n',z)$ is cyclically ordered. By construction,
\[\bdext(x)\subset \Diam^\opp_{\xi(z)}(\xi(x_n),\xi(x_n'))\quad \text{and}\quad \bdext(y)\subset \Diam^\opp_{\xi(z)}(\xi(y_n),\xi(y_n'))~,\] hence they are disjoint.
\end{proof}

Set
    \[\bdext  \coloneqq  \bigcup_{x\in \Sb^1} \bdext (x)\]
and define the map
\[ \zeta^{\mathrm M}_\rho\colon \bdext  \to \Sb^1\]
sending $\bdext (x)$ to the point $x$. This is well-defined by \Cref{disjoint}.

\begin{remark}\label{non-surjectivity}
In general, $\bdext(x)$ might be empty, and so $\zeta^{\mathrm M}_\rho$ might not be surjective.
However, if $\gamma\in\Gamma$ is hyperbolic, then $\bdext(\gamma^+)$ is non-empty as it contains $\rho(\gamma)^+$, see \Cref{lemma: boundary extension pointwise independence}. Hence $\bdext$ is non-empty.
\end{remark}

We will now prove that $(\bdext , \zeta^{\mathrm M}_\rho)$ is the positive $\Theta$-boundary extension that satisfies the stated properties in \Cref{prop: converse}.

\begin{proof}[Proof of \Cref{prop: converse}]
The $\rho(\Gamma)$-invariance of $\bdext $ and $(\rho,\Id)$-equivariance of $\zeta^{\mathrm M}_\rho$ follows from their construction and the $\rho$-equivariance of $\xi$. Thus, it suffices to verify that $\bdext $ is closed, $\zeta^{\mathrm M}_\rho$ is continuous, $(\bdext , \zeta^{\mathrm M}_\rho)$ is positive, and the two required maximality properties hold.\\
        
    \noindent {\bf Closedness of $\bdext $.} For each integer $m\ge 2$, let $\Cc_m$ denote the set of cyclically ordered $m$-tuples in $D$.
    Then define
     \[\bdext'\coloneqq \bigcap_{m\ge 3}\bigcap_{(y_1,...,y_m) \in \Cc_m} \bigcup_{k=1}^m  \overline{\Diam}_{\xi(y_k)}(\xi(y_{k-1}),\xi(y_{k+1})),\]
     where $y_0\coloneqq y_m$ and $y_{m+1}\coloneqq y_1$.
     Notice that $\bdext'\subset\Fc_{\Theta,\Fb}$ is closed, so we need only to prove that $\bdext = \bdext'$.

    We first prove that $\bdext \subset \bdext'$. It suffices to show that for any $x\in\Sb^1$, any integer $m\ge 3$, and any $(y_1,\dots,y_m)\in\Cc_m$, there exists some $k=1,\dots,m$ such that
    \[\bdext(x)\subset\overline{\Diam}_{\xi(y_{k})}(\xi(y_{k-1}),\xi(y_{k+1}))~.\]
    By cyclically permuting $(y_1,\dots,y_m)$, we may assume that $(y_1,x,y_3,...,y_m)$ is cyclically ordered. Suppose $(x_n)_{n\in\Nb}$ and $(x_n')_{n\in\Nb}$ are sequences in $D$ that converge to $x$, such that for all $n\in\Nb$, $(y_1,x_1,\dots,x_n,x,x_n',\dots,x_1',y_3)$ is cyclically ordered. Then by the definition of $\bdext(x)$,
    \begin{align*}
     \bdext(x)&\subset \Diam_{\xi(y_1)}^\opp(\xi(x_n),\xi(x_n')) \subset \Diam_{\xi(x_n)}(\xi(y_1),\xi(y_3))\subset \overline{\Diam}_{\xi(y_2)}(\xi(y_1),\xi(y_3)).
    \end{align*}
    
Next, we prove that $\bdext \supset \bdext'$. Let $F$ be an arbitrary flag in $\bdext'$. We will specify a point $x\in\Sb^1$ such that $F\in\bdext(x)$. To do so, it suffices to find a point $z\in D$ and sequences $(x_n)_{n\in\Nb}$ and $(x_n')_{n\in\Nb}$ in $D$ that converge to $x$, so that 
\[(z,x_1,\dots,x_{n-1},x_n,x,x_n',x_{n-1}',\dots,x_1')\]
is cyclically ordered for all $n\in\Nb$, and 
\[F\in \bigcap_{n\in\Nb} \Diam^\opp_{\xi(z)}(\xi(x_n),\xi(x_n')),\]
    
Fix an angular metric $d$ on $\Sb^1$. Observe that since $D$ is dense, one can iteratively construct, for each integer $\ell\ge 2$, a cyclically ordered tuple 
\[C_\ell\coloneqq(y_1^\ell,\dots,y_{2^\ell}^\ell)\in\Cc_{2^\ell}\] 
such that 
\begin{itemize}
\item $C_\ell$ is the sub-tuple of even entries of $C_{\ell+1}$, i.e.
\[C_\ell=(y_2^{\ell+1},y_4^{\ell+1},y_6^{\ell+1},\dots,y_{2^{\ell+1}}^{\ell+1}).\]
\item For each $k=1,\dots,2^\ell$, $d(y_k^\ell,y_{k+1}^\ell)\le (\frac{2}{3})^{\ell-3}\pi$.
\end{itemize}
By cyclically permuting each $C_\ell$, we may also assume that for all integers $\ell\ge 3$, $F\in\overline{\Diam}_{\xi(y_3^\ell)}(\xi(y_2^\ell),\xi(y_4^\ell))$. Then set $z\coloneqq y_6^3$, and for each $n\in\Nb$, set
\[x_n\coloneqq y_1^{2n+1}\quad\text{and}\quad x_n'\coloneqq y_5^{2n+1}.\]

To prove that $z$, $(x_n)_{n\in\Nb}$ and $(x_n)_{n\in\Nb}$ have the required properties, we need to understand, for all $\ell\ge 3$, the relative positions of $y_1^\ell$, $y_5^\ell$, $y_1^{\ell+2}$, $y_5^{\ell+2}$, and $z$. Notice that if $i$ is the integer such that $y_2^\ell=y_i^{\ell+2}$, then $y_1^\ell=y_{i-4}^{\ell+2}$, $y_3^\ell=y_{i+4}^{\ell+2}$, $y_4^{\ell}=y_{i+8}^{\ell+2}$, and $y_5^{\ell}=y_{i+12}^{\ell+2}$, where arithmetic in the subscripts is done modulo $2^{\ell+2}$. Thus, \Cref{lem: closure of diamond} implies that 
\[F\in\overline{\Diam}_{\xi(y_3^\ell)}\left(\xi(y_2^\ell),\xi(y_4^\ell)\right)\subset\Diam_{\xi(y_{i+4}^{\ell+2})}\left(\xi(y_{i-1}^{\ell+2}),\xi(y_{i+9}^{\ell+2})\right).\]
Furthermore, a straightforward induction argument yields that $z=y^{\ell+2}_{i+2^{\ell+1}}$. Since $\overline{\Diam}_{\xi(y_3^{\ell+2})}\left(\xi(y_2^{\ell+2}),\xi(y_4^{\ell+2})\right)$ contains $F$, it intersects $\Diam_{\xi(y_{i+4}^{\ell+2})}\left(\xi(y_{i-1}^{\ell+2}),\xi(y_{i+9}^{\ell+2})\right)$, and so $i-1\le 3\le i+9$, or equivalently, $-6\le i\le 4$. It follows that $(z,y_1^\ell,y_1^{\ell+2},y_5^{\ell+2},y_5^\ell)$ is cyclically ordered. 

As such, by iterated applications of \Cref{lem: AddingElemToPosTuple}, we have that for every $n\in\Nb$, the tuple $(z,x_1,\dots,x_n,x_n',\dots,x_1')$ is cyclically ordered, and so the sequences $(x_n)_{n\in\Nb}$ and $(x_n')_{n\in\Nb}$ converge. Observe from the definition of $C_\ell$ that for all $n\in\Nb$, we have $d(x_n,x_n')\le(\frac{2}{3})^{2n-2}(4\pi)$, so $(x_n)_{n\in\Nb}$ and $(x_n')_{n\in\Nb}$ have the same limit, which we denote by $x$. Then $(x_n)_{n\in\Nb}$ and $(x_n')_{n\in\Nb}$ converge to $x$, and for all $n\in\Nb$,
\[(z,x_1,\dots,x_{n-1},x_n,x,x_n',x_{n-1}',\dots,x_1')\]
is cyclically ordered.
Also, for all $n\in\Nb$, \Cref{lem: closure of diamond} gives
\[\overline{\Diam}_{\xi(y_3^{2n+1})}\left(\xi(y_2^{2n+1}),\xi(y_4^{2n+1})\right)\subset\Diam_{\xi(y_3^{2n+1})}\left(\xi(y_1^{2n+1}),\xi(y_5^{2n+1})\right)=\Diam_{\xi(z)}^{\rm opp}\left(\xi(x_n),\xi(x_n')\right),\]
and so
\[F\in \bigcap_{n\in\Nb} \Diam^\opp_{\xi(z)}(\xi(x_n),\xi(x_n')).\]

\noindent {\bf Positivity of $(\bdext , \zeta^{\mathrm M}_\rho)$.} For any positive integer $m\ge 3$, let $(a_1,\dots,a_m)$ be a tuple of points in $\bdext$ such that $(\zeta^{\mathrm M}_\rho(a_1),\dots,\zeta^{\mathrm M}_\rho(a_m))$ is cyclically ordered. We need to prove that $(a_1,\dots,a_m)$ is a positive tuple of flags. By the density of $D$, there exists $z,y_0,y_1,\dots,y_m\in D$ such that 
\[(z,y_0,\zeta^{\mathrm M}_\rho(a_1),y_1,\zeta^{\mathrm M}_\rho(a_2),y_2,\dots,y_{m-1},\zeta^{\mathrm M}_\rho(a_m),y_m)\]
is cyclically ordered. Then for each $i=1,\dots,m$, the definition of $\bdext(\zeta^{\mathrm M}_\rho(a_i))$ implies that the quadruple $(\xi(z),\xi(y_{i-1}),a_i,\xi(y_i))$ is positive. By the positivity of $\xi$, the tuple
\[(\xi(z),\xi(y_0),\xi(y_1),\xi(y_2),\dots,\xi(y_{m-1}),\xi(y_m))\]
is positive. Iterated applications of \Cref{lem: AddingElemToPosTuple} then imply that
\[(\xi(z),\xi(y_0),a_1,\xi(y_1),a_2,\xi(y_2),\dots,\xi(y_{m-1}),a_m,\xi(y_m))\]
is positive, so $(a_1,\dots,a_m)$ is positive.\\

\noindent {\bf First maximality property.}
Pick any point $x\in D$. We need to show that $\xi(x)\in\bdext(x)$. Fix a point $z\in D\setminus \{x\}$ and sequences $(x_n)_{n\in\Nb}$ and $(x_n')_{n\in\Nb}$ in $D$ that converge to $x$ such that for all $n\in\Nb$, $(z,x_1,\dots,x_n,x,x_n',\dots,x_1')$ is cyclically ordered. Then   
\[\xi(x)\in\bigcap_{n\in\Nb} \Diam^\opp_{\xi(z)}(\xi(x_n),\xi(x_n'))=\bdext(x).\]

\noindent {\bf Second maximality property.} Let $(\Vc, \zeta)$ be a positive $\Theta$-boundary extension. Pick any point $a\in \Vc$. We need to show that $a\in\bdext(\zeta(a))$.
    
Let $b\in\Vc$ such that $\zeta(b)\neq\zeta(a)$. By density of 
\[\{(\gamma^-,\gamma^+)\mid \gamma\in\Gamma\text{ is hyperbolic}\}\subset \Sb^1\times \Sb^1,\]
we can find 
\begin{itemize}
    \item a hyperbolic element $\eta\in\Gamma$ such that $(\eta^-, \zeta(b), \eta^+, \zeta(a))$ is cyclically ordered,
    \item a sequence $(\gamma_n)_{n\in \Nb} \subset \Gamma$ of hyperbolic elements such that the sequences $(\gamma_n^+)_{n\in\Nb}$ and $(\gamma_n^-)_{n\in\Nb}$ converge to $\zeta(a)$, and for all $n\in\Nb$, 
    \[(\eta^+,\gamma_1^+,\dots,\gamma_n^+,\zeta(a),\gamma_n^-,\dots,\gamma_1^-,\zeta(b))\]
    is cyclically ordered.
\end{itemize}
Then up to replacing $\eta$ and each $\gamma_n$ by a positive power, we can assume that
    \[(\zeta(b),\eta\cdot \zeta(b), \eta^2 \cdot \zeta(b), \eta\cdot \zeta(a), \gamma_n\cdot \zeta(b), \gamma_n^2 \cdot \zeta(b), \gamma_n\cdot \zeta(a), \zeta(a), \gamma_n^{-1} \cdot \zeta(a), \gamma_n^{-2} \cdot \zeta(a), \gamma_n^{-1}\cdot \zeta(b))\]
    is cyclically ordered for all $n\in\Nb$.

    By the positivity of $\zeta$, \Cref{lem: CharPosTransWeaklyProx}, and \Cref{lem: AddingElemToPosTuple}, we deduce that for all $n\in\Nb$, the tuple 
    \[(b, \rho(\eta) \cdot b, \rho(\eta)^+, \rho(\eta)\cdot a, \rho(\gamma_n)\cdot b, \rho(\gamma_n)^+, \rho(\gamma_n)\cdot a, a, \rho(\gamma_n)^{-1} \cdot a, \rho(\gamma_n)^-, \rho(\gamma_n)^{-1}\cdot b)\]
    is positive. In particular,
    \[a\in \Diam_{\rho(\eta)^+}^\opp(\rho(\gamma_n)^+, \rho(\gamma_n)^-)~.\]
    We thus conclude that
    \[a\in \bigcap_{n\in \Nb} \Diam_{\xi^h_\rho(\eta^+)}^\opp(\xi^h_\rho({\gamma_n}^+), \xi^h_\rho({\gamma_n}^-)) = \bdext (\zeta(a))~.\]

\noindent {\bf Continuity of $\zeta^{\mathrm{M}}_\rho$.}
Finally, let us prove that $\zeta_\rho^{\mathrm{M}}$ is continuous. Let $U$ be an open interval in $\Sb^1$ and $a\in \bdext$ such that $\zeta_\rho^{\mathrm{M}}(a) \in U$.
Let $x,y\in D$ be such that the closed (counter-clockwise) interval $[x,y]\subset\Sb^1$ is contained in $U$ and contains $\zeta_\rho^{\mathrm{M}}(a)$. Observe that by construction of $(\bdext,\zeta^{\mathrm{M}}_\rho)$, the triple $(\xi(x),a,\xi(y))$ is positive, so we may set
\[V\coloneqq \Diam_a(\xi(x),\xi(y))\cap \bdext~.\]
Since $V$ is an open neighborhood of $a$ in $\bdext$, it suffices to show that $\zeta_\rho^{\mathrm{M}}(V) \subset U$. 

Suppose for the purpose of contradiction that there exists $c\in V$ such that $\zeta_\rho^{\mathrm{M}}(c)\notin U$.
Then $\zeta_\rho^{\mathrm{M}}(c)\notin[x,y]$, so $(x,\zeta_\rho^{\mathrm{M}}(a),y,\zeta_\rho^{\mathrm{M}}(c))$ is cyclically ordered. By positivity of $\zeta_\rho^{\mathrm{M}}$, $(\xi(x),a,\xi(y),c)$ is positive, or equivalently, 
\[c\in \Diam_a^\opp(\xi(x),\xi(y))~.\]
This contradicts the assumption that $c\in V\subset \Diam_a(\xi(x),\xi(y))$.
\end{proof}

Note that \Cref{prop: converse} implies that positive boundary maps are always ``compatible'' in the following sense:

\begin{corollary}\label{coro: compatibility positive boundary maps}
    Suppose that $\rho$ is $\Theta$-positive. For each $i=1,\dots,n$, let $D_i\subset\Lambda(\Gamma)$ be a non-empty, $\Gamma$-invariant subset and let 
    $\xi_i\colon D_i \to \Fc_{\Theta,\Fb}$ be a positive, $\rho$-equivariant map. If $(x_1,\dots,x_n)$ is a cyclically ordered tuple in $\Sb^1$ such that for all $i=1,\dots,n$, $x_i\in D_i$, then the tuple \[(\xi_1(x_1), \ldots, \xi_n(x_n))\]
    is positive.
\end{corollary}

\begin{proof}
    By \Cref{prop: converse}, $\xi_i(x_i) \in \bdext(x_i)$ for all $i=1,\dots,n$ and 
    \[\zeta^{\mathrm{M}}_\rho( \xi_1(x_1), \ldots, \xi_n(x_n)) = (x_1, \ldots, x_n),\] 
    which is cyclically ordered. Hence $(\xi_1(x_1), \ldots, \xi_n(x_n))$ is positive by positivity of $(\bdext,\zeta_\rho^{\mathrm{M}})$.
\end{proof}

\subsubsection{Surjectivity over Cantor complete real closed fields}
\label{sss:Surjectivity Cantor Complete}
We have already seen in \Cref{non-surjectivity} that $\zeta^{\mathrm{M}}_\rho$ might not be surjective.
However, we show below that surjectivity of $\zeta^{\mathrm{M}}_\rho$ can always be achieved after extending $\Fb$ to a Cantor complete one (see \Cref{s: Real case} for the case $\Fb=\Rb$).

\begin{proposition}\label{proposition: cantor complete field admit surjective boundary extension}
    If $\Fb$ is a Cantor complete real closed field, then $\zeta^{\mathrm{M}}_\rho\colon \bdext  \to \Sb^1$ is surjective.
\end{proposition}

\begin{proof}
    By construction, $(\zeta^{\mathrm{M}}_\rho)^{-1}(x) = \bdext (x)$ is a nested intersection of diamonds. If $\Fb$ is Cantor complete it is thus non-empty for every $x$ by \Cref{corollary: diamond nonempty intersection over cantor complete fields}.
\end{proof}

\begin{remark}
    It was shown by Shelah that any ordered field can be embedded in a Cantor complete real closed field (see \cite[Theorem 1.1]{Shelah_QuiteCompleteRealClosedFields}).
\end{remark}

\begin{corollary}\label{cor:ClosedSurfExtensionLimitMap}
    Let $\Fb$ is a Cantor complete real closed field, $\Gamma$ a Fuchsian group of the first kind and $\rho \colon \Gamma \to G_{\Fb}$ a $\Theta$-positive representation.
    Then there is a $\rho$-equivariant, positive map $\xi \colon \Lambda \setminus \Lambda_p\to\Fc_{\Theta,\Fb}$.
\end{corollary}

\begin{proof}
    Recall that the stabilizer in $\Gamma$ of a point $x\in \Sb^1$ is either trivial or cyclic, and in the latter case, $x$ is a hyperbolic or parabolic fixed point. 

    Choose a representative $\hat x$ of each $\Gamma$-orbit $\Gamma \cdot x$ in $\Lambda \setminus (\Lambda_h \sqcup \Lambda_p)$. Then, using \Cref{proposition: cantor complete field admit surjective boundary extension}, we may choose $y_{\hat x} \in (\zeta^{\mathrm M}_\rho)^{-1}(\hat x)$ for each of these representatives. Finally, define
    \[\begin{array}{cccc}
    \xi\colon & \Lambda\setminus (\Lambda_p\sqcup \Lambda_h) & \to & \Fc_{\Theta, \Fb}\\
    & \gamma \cdot \hat x & \mapsto & \rho(\gamma) \cdot y_{\hat x}~.
    \end{array}\]
    The map $\xi$ is well defined since the stabilizer of each $\hat x$ is trivial, and by construction it is a $\rho$-equivariant section of $\zeta^{\mathrm M}_\rho$. Hence 
    \[\xi \cup \xi^h_\rho\colon \Lambda \setminus \Lambda_p\to\Fc_{\Theta,\Fb}\]
    is a partial section of $\zeta^{\mathrm M}_\rho$.
\end{proof}

\begin{remark}
    We will see that positivity does not imply that images of parabolic elements have fixed points in $\Fc_\Theta$.
    Thus, we can in general not expect to be able to extend the limit map to $\Lambda_p$.
    In the following section we study exactly $\Theta$-positive representations for which the limit map can be defined on $\Lambda_p$.
\end{remark}

\subsection{%
\texorpdfstring{%
Framed $\Theta$-positivity}%
{Framed Theta-positivity}}
\label{s: Framed Positivity}

As before, we fix $\Gamma$ a Fuchsian group of the first kind. In the previous section, we proved that, up to enlarging the field of definition, the boundary map of a $\Theta$-positive representation can be extended to the whole circle \emph{except maybe at the fixed points of parabolic elements}, see \Cref{cor:ClosedSurfExtensionLimitMap}.
To resolve this issue, we introduce here the notion of \emph{$\Theta$-positively framed representation} in the sense of Fock--Goncharov \cite{fock2006moduli}, and characterize these representations.

Recall that $\Lambda_p = \Lambda_p(\Gamma) \subset \Sb^1$ denotes the set of fixed points of parabolic elements of $\Gamma$. Throughout this section, we assume that $\Lambda_p$ is non-empty, hence dense in $\Sb^1$.

\begin{definition}
    Given a representation $\rho \colon \Gamma \to G_{\Fb}$, a \emph{$\Theta$-framing} of $\rho$ is a $\rho$-equivariant map $\xi_p \colon \Lambda_p \to \mathcal F_{\Theta,\Fb}$.
    A \emph{$\Theta$-framed representation} is a pair $(\rho, \xi_p)$, where $\rho$ is a representation and $\xi_p$ is a $\Theta$-framing.
\end{definition}

For each $p\in \Lambda_p$, there exists a unique positive parabolic element $\gamma_p\in \Gamma$ such that $\Stab_\Gamma(p) = \langle \gamma_p\rangle$. Let $I$ denote the set of $\Gamma$-orbits in $\Lambda_p$, and choose a representative $p_i$ in each orbit $i\in I$. Then, a $\Theta$-framing $\xi_\rho$ must map $p_i$ to a fixed point of $\rho(\gamma_{p_i})$. Conversely, given $x_i\in \mathcal F_{\Theta,\Fb}$ fixed by $\rho(\gamma_{p_i})$ for all $i \in I$, there exists a unique framing of $\rho$ mapping $p_i$ to $x_i$. Hence a $\Theta$-framing is equivalent to a choice of fixed point for some representative of each conjugacy class of each positive primitive parabolic element.

When $G$ is a complex or a split real simple Lie group and $P_\Theta$ is a Borel subgroup, this definition of framed representation coincides with the one of Fock--Goncharov (see \cite[Lemma 1.1]{fock2006moduli}).

 \begin{definition}
    A \emph{$\Theta$-positively framed} representation is a $\Theta$-framed representation $(\rho, \xi_p)$ such that the map $\xi_p$ is positive.
    A representation $\rho$ is called \emph{$\Theta$-positively frameable} if it admits a positive $\Theta$-framing.
\end{definition}

Note that a $\Theta$-positively frameable representation is $\Theta$-positive.

\begin{proposition}\label{framed positive implies positively translating}
    Let $\rho \colon \Gamma \to G_{\Fb}$ be a $\Theta$-positively frameable representation. Then, for every $\gamma \in \Gamma$ of infinite order, $\rho(\gamma)$ is $\Theta$-positively translating. 
\end{proposition}
\begin{proof}
    Every infinite order element in $\Gamma$ is either hyperbolic or parabolic. In the case when $\gamma$ is hyperbolic, the proposition follows from \Cref{propo: positive implies weak proximal}.
    It thus suffices to prove that the image of a parabolic element is $\Theta$-positively translating (and not just $\Theta$-positively rotating).

    If $\gamma \in \Gamma$ is parabolic, let $\gamma^+\in \Lambda_p$ be its fixed point and let $x\in \Lambda_p\setminus\{\gamma^+\}$. Then either $(x,\gamma \cdot x, \gamma^2 \cdot x, \gamma^+)$ or $(\gamma^2\cdot x, \gamma\cdot x, x, \gamma^+)$ is cyclically ordered (depending on whether $\gamma$ is a positive or negative parabolic).
    Hence $\xi_p(\gamma^+)$ is fixed by $\rho(\gamma)$ and $(\xi_p(x), \rho(\gamma)\cdot \xi_p(x), \rho(\gamma)^2 \cdot \xi_p(x), \xi_p(\gamma^+))$ is positive, showing that $\rho(\gamma)$ is $\Theta$-positively translating.
\end{proof}

We can now prove the counterpart of \Cref{cor:ClosedSurfExtensionLimitMap} for Fuchsian groups of the first kind with parabolic elements:
\begin{corollary}\label{extension cor}
    Let $\Fb$ a Cantor complete real closed field and $\rho \colon \Gamma \to G_{\Fb}$ a representation.
    Then $\rho$ is $\Theta$-positively frameable if and only if there exists a $\Theta$-positive $\rho$-equivariant map $\xi \colon \Sb^1 \to \mathcal F_{\Theta,\Fb}$.
\end{corollary}

\begin{proof}
    If there exists a $\rho$-equivariant $\Theta$-positive map $\xi \colon \Sb^1 \to \mathcal F_{\Theta,\Fb}$, then its restriction to $\Lambda_p$ is a $\Theta$-positive framing.

    Conversely, assume $\xi_p$ is a $\Theta$-positive framing for $\rho$.
    Let $(\bdext,\zeta^{\mathrm M}_\rho)$ be the maximal positive $\Theta$-boundary extension given by \Cref{prop: converse}. By \Cref{cor:ClosedSurfExtensionLimitMap}, there is a $\rho$-equivariant partial section $\xi$ of $\zeta^{\mathrm M}_\rho$ defined on $\Lambda \setminus \Lambda_p$. By \Cref{prop: converse}, $\xi_p$ is a partial section of $\zeta^{\mathrm M}_\rho$. Hence, the map
    \[\begin{array}{cccc}
    \xi \sqcup \xi_p \colon & \Sb^1 & \to & \Fc_{\Theta,\Fb}\\
    &x & \mapsto  &
        \begin{cases*}
        \xi_p(x) &\textrm{ if $x\in \Lambda_p$},\\
        \xi(x) &\textrm{ if $x\in \Sb^1\setminus \Lambda_p$}~.
        \end{cases*}
    \end{array}
    \]
    is now a globally defined $\rho$-equivariant section of $\zeta^{\mathrm M}_\rho$, or equivalently, a positive boundary map defined on $\Sb^1$.
\end{proof}

We conclude this section by showing that a $\Theta$-positively frameable representation admits two canonical $\Theta$-positive framings.

Let $\rho$ be a $\Theta$-positively frameable representation.
Recall that, by \Cref{propo:PosTransUniqueFP} and \Cref{propo: positive implies weak proximal}, the image under $\rho$ of every parabolic element $\gamma \in \Gamma$ has well-defined forward and backward fixed points in $\Fc_{\Theta,\Fb}$, denoted respectively by $\rho(\gamma)^+$ and $\rho(\gamma)^-$.
For $p \in \Lambda_p$, let $\gamma_p$ denote the positive parabolic element that generates the stabilizer of $p$ in $\Gamma$.

\begin{proposition}
\label{lem:PosFramedRepr}
   Let $\rho\colon \Gamma \to G_{\Fb}$ be a $\Theta$-positively frameable representation. Then for any parabolic element $\gamma\in\Gamma$, we have $\rho(\gamma)^\pm\in\bdext(\gamma^+)$, where $(\bdext,\zeta^{\mathrm M}_\rho)$ is the maximal positive $\Theta$-boundary extension given by \Cref{prop: converse}. 
   
   In particular, the maps $\xi_\rho^l, \xi_\rho^r\colon \Lambda_p \to \Fc_{\Theta,\Fb}$ defined by
   \[\xi_\rho^l\colon p \mapsto \rho(\gamma_p)^+\quad \textnormal{ and } \quad\xi_\rho^r\colon p \mapsto \rho(\gamma_p)^-\]
   are $\Theta$-positive framings.
\end{proposition}
\begin{proof}

Let $\xi_p\colon \Lambda_p \to \Fc_{\Theta,\Fb}$ be a $\Theta$-positive framing and choose any $x \neq \gamma^+ \in \Lambda_p$.
By positivity of $\xi_p$, the pair $(\xi_p(x), \xi_p(\gamma^+))$ is a $\Theta$-positively translating pair for both $\rho(\gamma)$ and $\rho(\gamma^{-1})$.
Hence, by \Cref{propo:PosTransUniqueFP}, $(\xi_p(x), \rho(\gamma)^+)$ and $(\xi_p(x), \rho(\gamma)^-)$ are positive translating pairs for $\rho(\gamma)$ and $\rho(\gamma)^{-1}$ respectively.
Therefore for all $n\in\Nb$, we have
    \[\{\rho(\gamma)^-, \rho(\gamma)^+\} \subset \Diam_{\xi_p(x)}^\opp(\rho(\gamma)^n \cdot \xi_p(x), \rho(\gamma)^{-n}\cdot \xi_p(x)) = \Diam_{\xi_p(x)}^\opp(\xi_p(\gamma^n \cdot x), \xi_p(\gamma^{-n}\cdot x))~.\]
Since $\gamma^n\cdot x$ and $\gamma^{-n} \cdot x$ converge to $\gamma^+$ and
\[(x,\gamma\cdot x,\dots,\gamma^n\cdot x,\gamma^+,\gamma^{-n}\cdot x,\dots,\gamma^{-1}\cdot x)\]
is cyclically ordered, we have
    \[\{\rho(\gamma)^-, \rho(\gamma)^+\} \subset\bigcap_{n>0}\Diam_{\xi_p(x)}^\opp(\xi_p(\gamma^n\cdot x), \xi_p(\gamma^{-n}\cdot x)) = \bdext (\gamma^+).\qedhere\]
\end{proof}

\begin{corollary}\label{corollary: fi frameable}
    Let $\Gamma'$ be a finite index subgroup of $\Gamma$ and $\rho \colon \Gamma \to G_{\Fb}$ a representation.
    Then $\rho$ is $\Theta$-positively frameable if and only if $\rho\vert_{ \Gamma'}$ is $\Theta$-positively frameable.
\end{corollary}
\begin{proof}
    The proof is akin to that of \Cref{cor: positivity finite-index subgroup}. We simply note that, since every positive parabolic of $\Gamma$ has a positive power in $\Gamma'$, we have
    \[\Lambda_p(\Gamma') = \Lambda_p(\Gamma )\textrm{ and }\xi_\rho^l(\rho) = \xi_\rho^l(\rho\vert_{ \Gamma'})~.\qedhere\]
\end{proof}

We will see in \Cref{subsection: continuity boundary maps} that, over $\Rb$, every $\Theta$-positive representation is $\Theta$-positively frameable, and that the canonical framings $\xi_\rho^l$ and $\xi_\rho^r$ are the restrictions to $\Lambda_p$ of the left and right semi-continuous boundary maps of $\rho$. In contrast to the real case, there are examples of representations $\rho \colon \Gamma \to \PSL_2(\Fb)$, over a non-Archimedean real closed field $\Fb$, that are positive but not frameable. We will now construct one such example. These are also constructed independently by Burger--Iozzi--Parreau--Pozzetti in \cite{BIPP_RSCMaximalRepr}.

\subsubsection{Example of a non-frameable $\Theta$-positive representation}
\label{subsection: ExNonFramablePosRepr}

We saw in \Cref{propo: PosTransPosRot} an example of a $\Theta$-positively rotating element of $\PSL_2(\Fb)$ which is not $\Theta$-positively translating, and our goal is now to realize such an element as the image of a parabolic element by a positive representation. Informally, we can obtain such representations as limits of sequences of representations in $\PSL_2(\Rb)$ which are closer and closer to being Fuchsian, with some parabolic element mapped to a rotation with smaller and smaller angle.

To better visualize elements in $\PSL_2(\Fb)$ and their action on $\mathbf{P}^1(\Fb)$ we look at their actions by M\"obius transformations on the non-Archimedean upper-half plane over $\Fb$.
For this choose a square root $i$ of $-1 \in \Fb$ and consider the algebraically closed field $\Fb[i]$.
Just like in the case when $\Fb=\Rb$ we define imaginary and real parts and complex conjugation in $\Fb$, and we use the standard notations.
We define the \emph{non-Archimedean upper half plane over $\Fb$} by
\[\mathbf{H}_{\Fb} \coloneqq \{z \in \Fb[i] \mid \textnormal{Im}(z)>0\}\]
as in \cite{Brumfiel_TreeNonArchimedeanHyperbolicPlane}.
Its ``boundary'' $\{z \in \Fb[i] \mid \textnormal{Im}(z)=0\} \cup \{ \infty \}$ is then identified with $\mathbf{P}^1(\Fb)$.
Mimicking the real case, through any two distinct points $x \neq y \in \mathbf{P}^1(\Fb)$ there is an \emph{$\Fb$-line} in $\mathbf{H}_{\Fb} \cup \mathbf{P}^1(\Fb)$ whose endpoints are $x$ and $y$.
More precisely this line is a straight line parametrized by $\Fb$ if one of $x$ or $y$ is equal to $\infty \in \mathbf{P}^1(\Fb)$, and otherwise a half-circle of the form $\{z \in \Fb[i] \mid (z-c)(\bar{z} -c) = r^2\}$ for some $c,r \in \Fb$.

The group $\PGL_2(\Fb)$ acts on $\mathbf{H}_{\Fb}$ by M\"obius transformations, i.e.\
\[\begin{bmatrix}
    a & b \\
    c & d
\end{bmatrix} \cdot z = \begin{cases}
    \frac{az+b}{cz+d}, \textnormal{ if } ad-bc>0~,\\
    \frac{a\bar{z}+b}{c\bar{z}+d}, \textnormal{ otherwise }~.
\end{cases}\]
Note that since $\Fb$ is real closed the sign of the determinant of an element of $\PGL_2(\Fb)$ is well-defined. We can also define an ($\Fb$-valued) absolute value by setting $\vert x \vert_{\Fb} \coloneqq \max\{x,-x\} \in \Fb_{\geqslant 0}$ for all $x \in \Fb$ since $\Fb$ is ordered, thus the absolute value of the trace of an element in $\PSL_2(\Fb)$ is well-defined.
As over $\Rb$, we have the classification of elements of $\PSL_2(\Fb)$ into hyperbolic, parabolic and elliptic transformations in terms of fixed points in $\mathbf{H}_{\Fb} \cup \mathbf{P}^1(\Fb)$ as well as in terms of absolute values of traces by the transfer principle (\Cref{thm_TarskiSeidenberg}).
We also denote by $\textnormal{CR}$ the cross-ratio of four points (of which three are pairwise distinct) in $\mathbf{P}^1(\Fb)$ with the convention that $\textnormal{CR}(0,1,t,\infty)=t\in\Fb\cup\{\infty\}$. 

Finally, notice that the order on $\Fb$ determines a cyclic order on $\mathbf{P}^1(\Fb)$. Just as in the real case, an $n$-tuple of points in $\mathbf{P}^1(\Fb)$ is positive if and only if it is either cyclically ordered or reverse cyclically ordered. Also, an element in $\PGL_2(\Fb)$ either preserves or reverses the cyclic order on $\mathbf{P}^1(\Fb)$, and the former holds if and only if it lies in $\PSL_2(\Fb)$. For any distinct pair of points $x,y\in\mathbf{P}^1(\Fb)$, we let $(x,y)$ denote the cyclically ordered open interval with $x$ and $y$ as its backward and forward endpoints, i.e.\
\[(x,y)\coloneqq \{w\in\mathbf{P}^1(\Fb) \mid (x,w,y)\text{ is cyclically ordered}\}.\]
Then we denote $[x,y]\coloneqq (x,y)\cup\{x,y\}$.

In this paragraph, we identify the hyperbolic plane $\Hb$ with $\mathbf{H}_{\Rb}$ and $\Isom_+(\Hb)$ with $\PSL_2(\Rb)$.
Our example will be a representation of the Fuchsian group $\Gamma_{0,3}$ generated by $a$ and $b$, where $a$ maps the complement of the interval $A_1 = (3,4) \subset \Rb\cup\{\infty\}$ to $A_2 = (-4,-3)$ and $b$ maps the complement of $B_1 = (1,2)$ to $B_2= (-2,-1)$. In particular, $a$ and $b$ are hyperbolic with $a^+\in A_2$ and $b^+\in B_2$, and they span a free group of rank $2$. Topologically, the quotient $\Gamma_{0,3} \backslash \mathbf{H}_{\Rb}$ is a $3$-holed sphere. Denote $c \coloneqq b^{-1}a^{-1}$, and note that the group $\Gamma_{0,3}$ admits the following presentation
\[\langle a,b,c \mid abc=1\rangle~.\]

We will construct representations of $\Gamma_{0,3}$ into $\PSL_2(\Fb)$ by restricting to an index two subgroup of a group spanned by $3$ reflections.
Let $l_1$ and $l_2$ be two $\Fb$-lines in $\mathbf{H}_{\Fb}$ with endpoints $x_1,y_1$ and $x_2, y_2$ respectively. Up to switching the extremities we can assume without loss of generality that $(x_1,y_1,x_2)$ is cyclically ordered.
Denote respectively by $\sigma_1$ and $\sigma_2$ the reflections in $\PGL_2(\Fb)$ about $l_1$ and $l_2$.
In particular, we have the following lemma.

\begin{lemma}
\label{lem: CRBigInf}
We have
\[ \vert \Tr(\sigma_1\sigma_2) \vert_{\Fb} = \vert 4\textnormal{CR}(x_1,y_1,y_2,x_2)-2 \vert_{\Fb}~.\]
In particular, if $\textnormal{CR}(x_1,y_1,y_2,x_2)-1$ is infinitesimal and negative, then $\sigma_1 \sigma_2$ is an infinitesimal rotation. Moreover, \[\bigcup_{n\in \Zb} [(\sigma_1\sigma_2)^ny_2, (\sigma_1\sigma_2)^{n+1} y_2] \subset (x_1, x_2)~.\] 
\end{lemma}

\begin{proof}
Up to conjugating by an element of $\PSL_2(\Fb)$, we can assume that $x_1=-1$, $y_1=1$, $x_2=\infty$ and $y_2 =t$, so that $\textnormal{CR}(x_1,y_1,y_2,x_2)= \frac{t+1}{2}$. We then have $\sigma_1(z) = \frac{1}{\bar z}$ and $\sigma_2(z) = 2 t - \bar z$. Hence
\[\sigma_1\sigma_2(z) = \frac{1}{2t -z}\]
and
\[\sigma_1 \sigma_2 = \begin{bmatrix} 0 & 1\\ -1 & 2 t \end{bmatrix} \in \PSL_2(\Fb)~.\]
In particular, 
\[\vert \textnormal{Tr}(\sigma_1 \sigma_2) \vert_{\Fb}= 2t = \vert 4\textnormal{CR}(x_1,y_1,y_2,x_2) -2 \vert_{\Fb}~.\\ \]

If $ -\varepsilon \coloneqq \textnormal{CR}(x_1,y_1,y_2,x_2) - 1$ is negative and infinitesimal, then $\vert \textnormal{Tr}(\sigma_1 \sigma_2)\vert = 2-4\varepsilon$, hence $\sigma_1\sigma_2$ is an infinitesimal rotation (see \Cref{rem: infinitesimal rotation}).
For every infinitesimal $\eta$, we have
\[\sigma_1\sigma_2(1+\eta) = \frac{1}{1-4\varepsilon -\eta}~.\]
Hence $\sigma_1\sigma_2$ preserves the infinitesimal ball about $1$, namely $\{t\in \Fb \mid \vert t-1\vert_{\Fb} \leq \frac{1}{R} \textrm{ for all $R\in \Rb_{>0}$}\}$.
Since $y_2=1-2\varepsilon$ belongs to this infinitesimal ball, we conclude that its orbit is contained in this infinitesimal ball.
In particular,
\[\bigcup_{n\in \Zb} [(\sigma_1\sigma_2)^ny_2, (\sigma_1\sigma_2)^{n+1} y_2] \subset (x_1, x_2)~. \qedhere\] 
\end{proof}

Consider now the configuration of $\Fb$-lines $l_1,l_2,l_3$ as in \Cref{fig:PosReprNotFrame} with respective endpoints $x_i, y_i$ for $i=1,2,3$, and let $\sigma_i$ be the reflection about $l_i$.
We can arrange $l_1$ and $l_3$ in such a way that $\textnormal{CR}(x_1,y_1,y_3,x_3)-1$ is negative and infinitesimal, so that $\sigma_1 \sigma_3$ is an infinitesimal rotation around $p = l_1 \cap l_3$, see \Cref{lem: CRBigInf}.

\begin{figure}[ht]
    \centering
    \begin{tikzpicture}
        \draw[blue] (0,0)--(0,3);
        \draw[blue,right] (0,1.5) node{$l_2$}; 
        \draw[blue] (0,3.2) node{$x_2$}; 
        \draw[blue, below] (0,0) node{$y_2$};
        \draw[green] (-0.5,0) arc(0:180:1.5);
        \draw[green, below] (-3.5,0) node{$y_3$};
        \draw[green, below] (-.5,0) node{$x_3$};
        \draw[green] (-2,1.7) node{$l_3$}; 
        \draw[green, dashed] (-2.75,0) arc(0:180:.95);
        \draw[red] (-3,0) arc(0:180:1.5);
        \draw[red] (-4.5,1.7) node{$l_1$};
        \draw[red, below] (-6,0) node{$x_1$};
        \draw[red, below] (-3,0) node{$y_1$};
        \draw[red, dashed] (-1.875,0) arc(0:180:.95);
        \draw[below] (-3.25,.85) node{\small{$p$}};
        \draw (-3.25,.85) node{$\bullet$};
        \draw[green] (3.5,0) arc(0:180:1.5);
        \draw[red] (6,0) arc(0:180:1.5);
        \draw[->] (0,1) arc(90:65:1.5);
        \draw[above] (-.4,.95) node{$\rho(b)$};
        \draw (0,1) arc(90:115:1.5);
        \draw (0,2.8) arc(90:52:6);
        \draw[->] (0,2.8) arc(90:128:6);
        \draw[above] (-1,2.1) node{$\rho(a)$};
        \draw[->] (-3.6,1.3) arc(120:20:.5);
        \draw (-3.25,1.5) node{\tiny{$\rho(c)$}};
        \draw (-6.5,0)--(6.5,0);
        \draw[green] (-4.5,1) node{\tiny{$\sigma_1(l_3)$}}; 
        \draw[red] (-2,1) node{\tiny{$\sigma_3(l_1)$}}; 
    \end{tikzpicture}
    \caption{A $\Theta$-positive representation that is not frameable.}
    \label{fig:PosReprNotFrame}
\end{figure}

We now define $\rho \from \Gamma_{0,3} \to \PSL_2(\Fb)$ by
\[ \rho(a) = \sigma_1 \sigma_2, \quad \rho(b) = \sigma_2 \sigma_3.\]
Then $\rho(c)^{-1}=\rho(ab)=\sigma_1 \sigma_3$ is an infinitesimal rotation around $p$.

The main result of this paragraph is:
\begin{proposition}
\label{propo:PosReprNotFrame}
    The representation $\rho \colon \Gamma_{0,3} \to \PSL_2(\Fb)$ constructed above is positive, but not frameable.
\end{proposition}

This will follow from the following lemma of independent interest:

\begin{lemma}
\label{lem: SchottkyPositive}
    Let $\rho$ be a representation of $\Gamma_{0,3}$ into $\PSL_2(\Fb)$ such that $\rho(a)$ is hyperbolic. Assume that there exists disjoint sets $U_1, U_2, V_1, V_2 \subset \mathbf{P}^1(\Fb)$ with $\rho(a)^+ \in U_2$ and such that:
\begin{itemize}
    \item $\rho(a)(U_1^c)\subset U_2$, $\rho(a)^{-1}(U_2^c) \subset U_1$, $\rho(b)(V_1^c)\subset V_2$, $\rho(b)^{-1}(V_2^c)\subset V_1$;
    \item for all $u_1\in U_1, u_2 \in U_2, v_1\in V_1, v_2 \in V_2$, the quadruple $(u_1,u_2,v_1,v_2)$ is cyclically ordered.
\end{itemize}
Then $\rho$ is positive.
\end{lemma}

\begin{proof}
    By the classical ping pong argument, the hypotheses imply that the representation $\rho$ is faithful. Let us characterize the cyclic order on the orbit $\rho(\Gamma_{0,3})\cdot \rho(a)^+$.

Consider the action of $\Gamma_{0,3} \times \Zb^3$ on $(\Gamma_{0,3})^3$ given by
\[(w,t_1,t_2,t_3)\cdot (\gamma_1,\gamma_2,\gamma_3) = (w\gamma_1 a^{t_1} , w\gamma_2 a^{t_2}, w\gamma_3 a^{t_3})~.\]
Call a triple $(\gamma_1,\gamma_2,\gamma_3)$ \emph{minimal} if it minimizes $\ell(\gamma_1)+ \ell(\gamma_2) + \ell(\gamma_3)$ in its orbit, where $\ell$ is the length of the reduced expression with respect to the generating set $\{a,b\}$. Note that a minimal triple $(\gamma_1,\gamma_2,\gamma_3)$ satisfies the following:
\begin{itemize}
    \item the reduced expression of each $\gamma_i$ does not end with $a$ nor $a^{-1}$ (because otherwise replacing $\gamma_i$ by $\gamma_i a^{-1}$ or $\gamma_i a$ would reduce its length and keep the triple within the same orbit),
    \item For all $i\neq j$, if $\gamma_i$ and $\gamma_j$ are non-trivial, then their reduced expression starts with a different letter (otherwise, if, say, $\gamma_1$ and $\gamma_2$ start with the letter $b$, then $(b^{-1}\gamma_1, b^{-1} \gamma_2, b^{-1}\gamma_3)$ has shorter total length, but remains in the same orbit as $(\gamma_1,\gamma_2,\gamma_3)$),
    \item If $\gamma_i = \1$, then $\gamma_{i-1}$ and $\gamma_{i+1}$ cannot both start with $a$ (otherwise, if, say, $\gamma_1 = \1$ and $\gamma_2$ starts with $a$, then $(\1, a^{-1}\gamma_2, a^{-1} \gamma_3)$ is in the same orbit as $(\1,\gamma_2,\gamma_3)$ and has shorter total length).
    \end{itemize}

    For $\gamma \in \Gamma$, set 
    \[U_\gamma = 
        \begin{cases*}
        U_1\textrm{ if $\gamma$ starts with $a^{-1}$}\\
        U_2\textrm{ if $\gamma$ starts with $a$ or $\gamma = \1$}\\
        V_1\textrm{ if $\gamma$ starts with $b^{-1}$}\\
        V_2\textrm{ if $\gamma$ starts with $b$}~.
        \end{cases*}
    \]
    A classical recurrence shows that, if $\gamma$ does not end with $a^{-1}$, then $\rho(\gamma)\cdot \rho(a)^+ \in \rho(\gamma)\cdot U_2 \subset U_\gamma$.

    Now, given $3$ distinct points $x$, $y$, and $z$ in the $\rho(\Gamma_{0,3})$ orbit of $\rho(a)^+$ we can write
    \[(x,y,z) = \rho(\eta)\cdot (\rho(\gamma_1)\cdot \rho(a)^+,\rho(\gamma_2)\cdot \rho(a)^+,\rho(\gamma_3)\cdot \rho(a)^+)\]
    for some minimal triple $(\gamma_1,\gamma_2, \gamma_3)$. Since $x,y,z$ are all distinct, at most one of the $\gamma_i$ is trivial.
    The properties of minimal triples mentioned above imply that $\rho(\gamma_i)\cdot \rho(a)^+ \in U_{\gamma_i}$ and that the $U_{\gamma_i}$ are all disjoint. Hence the cyclic order of the triple $(\rho(\gamma_1)\cdot \rho(a)^+,\rho(\gamma_2)\cdot \rho(a)^+,\rho(\gamma_3)\cdot \rho(a)^+)$ depends only on the cyclic order of $U_{\gamma_1}, U_{\gamma_2}, U_{\gamma_3}$, which only depends on the first letters of $\gamma_1,\gamma_2,\gamma_3$. 
    
    From this, it follows that the cyclic order on the orbit of $\rho(a)^+$ is the same for all representations satisfying the hypotheses of \Cref{lem: SchottkyPositive}. In particular, this cyclic order is the same for the inclusion of $\Gamma_{0,3}$ in $\PSL_2(\Rb)$. Hence, if $a^+$ denotes the attracting fixed point of $a \in \mathbf{P}^1(\Rb)$, the map
    \[
    \begin{array}{ccc}
    \Gamma \cdot a^+ & \to & \mathbf{P}^1(\Fb)\\
    \gamma \cdot a^+ & \mapsto & \rho(\gamma) \cdot \rho(a)^+
    \end{array}
    \]
    is positive.
    \end{proof}

\begin{figure}[ht]
    \centering
    \begin{tikzpicture}
        \draw[blue] (0,0)--(0,3);
        \draw[blue,right] (0,1.5) node{$l_2$}; 
        \draw[blue] (0,3.2) node{$x_2$}; 
        \draw[blue, below] (0,0) node{$y_2$};
        \draw[green] (-0.5,0) arc(0:180:1.5);
        \draw[green] (-2,1.7) node{$l_3$};
        \draw[green, below] (-3.5,0) node{$y_3$};
        \draw[green, below] (-.5,0) node{$x_3$};
        \draw[green, dashed] (-2.75,0) arc(0:180:.95);
        \draw[red] (-3,0) arc(0:180:1.5);
        \draw[red, dashed] (-1.875,0) arc(0:180:.95);
        \draw[above] (-3.25,.85) node{$p$};
        \draw (-3.25,.85) node{$\bullet$};
        \draw[red] (-4.5,1.7) node{$l_1$};
        \draw[red, below] (-6,0) node{$x_1$};
        \draw[red, below] (-3,0) node{$y_1$};
        \draw[red] (-2,1) node{\tiny{$\sigma_3(l_1)$}}; 
        \draw[green] (3.5,0) arc(0:180:1.5);
        \draw[green] (2,1.7) node{$\sigma_2(l_3)$}; 
        \draw[red] (6,0) arc(0:180:1.5);       
        \draw[red] (4.5,1.7) node{$\sigma_2(l_1)$}; 
        \draw[->] (0,1) arc(90:65:1.5);
        \draw[above] (-.4,.95) node{$\rho(b)$};
        \draw (0,1) arc(90:115:1.5);
        \draw (0,2.8) arc(90:55:6);
        \draw[->] (0,2.8) arc(90:128:6);
        \draw[above] (-1,2.1) node{$\rho(a)$};
        \draw[->] (-3.6,1.3) arc(120:20:.5);
        \draw (-3.25,1.5) node{\tiny{$\rho(c)$}};
        \draw[red,line width=3pt] (-6,-.05) -- (-2.5,-.05);
        \draw[red] (-4,-.5) node{$U_2$}; 
        \draw[red,line width=3pt] (4,-.05) -- (6,-.05);
        \draw[red, dashed] (3.75,0) arc(0:180:.95);
        \draw[red] (5,-.5) node{$U_1$};
        \draw[green, dashed] (4.625,0) arc(0:180:.95);
        \draw[green,line width=3pt] (-2.5,-.05) -- (-.5,-.05);
        \draw[green] (-1.5,-.5) node{$V_1$}; 
        \draw[green,line width=3pt] (.5,-.05) -- (4,-.05);
        \draw[green] (2.5,-.5) node{$V_2$};
        \draw (-6.5,0)--(6.5,0);
        \draw[gray,line width=3pt] (-4,0) -- (-2.5,0);
        \draw[gray] (-3.25,-.5) node{$I$};
        \draw[gray,line width=3pt] (2.5,0) -- (4,0);
        \draw[gray] (3.25,-.5) node{$J$};
    \end{tikzpicture}
    \caption{A positive representation that is not frameable.}
    \label{fig:PosReprNotFrame2}
\end{figure}

\begin{proof}[Proof of \Cref{propo:PosReprNotFrame}]
The representation $\rho$ is not frameable because the image of $c$ under $\rho$ is an infinitesimal rotation, so $\rho(c)$ does not have a fixed point in $\mathbf{P}^1(\Fb)$.

To check that $\rho$ is positive we would like to apply the above lemma.
For this we define the following subset of $\mathbf{P}^1(\Fb)$: 
\[I \coloneqq\bigcup_{n \in \Zb} [\rho(c)^{n+1} \cdot y_1, \rho(c)^{n} \cdot y_1] \subset \mathbf{P}^1(\Fb)~.\]
We also set $J\coloneqq \sigma_2(I)$.

We claim that $\rho(b)(I)=J$ and $\rho(a)(J)=I$.
Indeed, we get for all $n \in \Zb$ that 
\[\rho(b) (\rho(c)^{n+1} \cdot y_1) = \sigma_2 (\sigma_1 \rho(c)^n \cdot y_1) =\sigma_2 (\rho(c)^{-n} \cdot y_1) \in \sigma_2 (I)=J,\]
since $\sigma_1 \rho(c)^n \cdot y_1 = (\sigma_1\sigma_3)^n \sigma_1 \cdot y_1 = \rho(c)^{-n} \cdot y_1$.
Thus
\begin{align*}
    \rho(b)[\rho(c)^{n+1} \cdot y_1, \rho(c)^{n} \cdot y_1] &= [\rho(b) \rho(c)^{n+1} \cdot y_1, \rho(b) \rho(c)^{n} \cdot y_1]\\
    &= [\sigma_2 \rho(c)^{-n} \cdot y_1, \sigma_2 \rho(c)^{-(n-1)} \cdot y_1]\\
    &= \sigma_2 [\rho(c)^{-n+1} \cdot y_1,\rho(c)^{-n} \cdot y_1]~,
\end{align*}
since $\rho(b)$ preserves the cyclic order on $\mathbf{P}^1(\Fb)$ and $\sigma_2$ reverses it.
Taking the union over all $n$, we conclude that $\rho(b)(I) = J$.
The same computations with $y_3$ replacing $y_1$ and switching the roles of $\sigma_1$ and $\sigma_3$ shows that $\rho(a)^{-1}(I) = J$.

\Cref{lem: CRBigInf} applied to $\rho(c)^{-1}=\sigma_1\sigma_3$ we have that $I \subset (x_1,x_3)$.
Thus the sets 
\[ U_1 \coloneqq \sigma_2([x_1,y_1]) \setminus J, \quad U_2 \coloneqq [x_1,y_1]\cup I,\quad V_1\coloneqq [y_3,x_3] \setminus I, \quad V_2 \coloneqq \sigma_2([y_3,x_3]) \cup J \]
in \Cref{fig:PosReprNotFrame2} satisfy the conditions of the above lemma, and we conclude that $\rho$ is positive.
\end{proof}

\subsection{%
\texorpdfstring{%
$\PG$-condition and $\Theta$-positive frameability}%
{PG-condition and Theta-positive frameability}}
\label{section: PG-condition}

Recall that $\Fc_{\Theta,\Fb}$ satisfies the \emph{$\PG$-condition} if there exists an integer $N>0$ such that, for every positive tuple
$(x_1,\ldots,x_N)\in\Fc_{\Theta,\Fb}^N$ and every $y\in\Fc_{\Theta,\Fb}$, the flag $y$ is transverse to at least one of the $x_i$.
Though we could only prove this condition in some cases (see \Cref{subs:PGConditionHermFlagVarities}), we strongly believe that it is satisfied for all flag varieties with a $\Theta$-positive structure (see \Cref{conj:PG-condition}).

\begin{lemma}\label{prop: positively rotating => translating}
    Assume that $\Fc_{\Theta,\Fb}$ satisfies the $\PG$-condition. Let $g$ be an element of $G_{\Fb}$ and $(x,p)\in \Fc_{\Theta,\Fb}$ be such that $x$ has a positive $\langle g\rangle$-orbit and $p$ is fixed by $g$. Then $(x,p)$ is a positively translating pair for $g$. In particular, $g$ is positively translating if and only if it is positively rotating and has a fixed point.
\end{lemma}

Recall that a $(x,p)$ is a \emph{positively translating pair} for $g$ if $p$ is fixed by $g$ and $(x,g \cdot x, g^2\cdot x, p)$ is positive, and that if $(x,p)$ is a positively translating pair for $g$, then $p\in \bigcap_{n\in \Nb}\Diam_x^\opp(g^{-n}\cdot x, g^n\cdot x)$ (see \Cref{propo: PosTransPosRot}).

\begin{proof}
    Let $N$ be an integer for which the $\PG$-condition is satisfied. 

    We first claim that $p$ is transverse to every flag $y\in\Diam_x(g^{-2}\cdot x,g^2\cdot x)$. Indeed, for such a $y$, the sequence $(g^{4n}\cdot y)_{n\in\Zb}$ is positive: for each $n$, the flag $g^{4n}\cdot y$ belongs to \[\Diam_{g^{4n}\cdot x}(g^{4n-2}\cdot x,g^{4n+2}\cdot x)~,\] and the full $\langle g\rangle$-orbit of $x$ is positive. Hence $(y,g^4\cdot y,\ldots,g^{4(N-1)}\cdot y)$ is a positive $N$-tuple. By the $\PG$-condition, $p$ is transverse to $g^{4i}\cdot y$ for some $i\in\{0,\ldots,N-1\}$. Since $p$ is fixed by $g^4$, transversality is preserved by applying $g^{-4i}$. Hence $p$ is transverse to $y$. 
    
    Since \[\overline{\Diam}_x(g^{-1}\cdot x,g\cdot x)\subset \Diam_x(g^{-2}\cdot x,g^2\cdot x)~,\] we have that $p$ is transverse to every flag in $\overline{\Diam}_x(g^{-1}\cdot x,g\cdot x)$. By \Cref{prop: transverse set to a full diamond}, this implies \[p\in \Diam_x^{\opp}(g^{-1}\cdot x,g\cdot x)~.\] Equivalently, $(g^{-1}\cdot x,x,g\cdot x,p)$ is positive. Applying $g$ gives $(x,g\cdot x,g^2\cdot x,p)$ positive, so $(x,p)$ is a $\Theta$-positively translating pair for $g$.
\end{proof}

This allows to prove \Cref{thm-intro:PG-condition}~(\ref{thm-intro:PG-condition: posfram}).
The second part of \Cref{thm-intro:PG-condition} is proven in \Cref{cor:PG-PosFrameable-open-closed-in-frameable} in \Cref{subs: Semialg Framed Repr}.

\begin{corollary}\label{cor:PG-positive-frameable-implies-positive-framings}
    Assume that $\Fc_{\Theta,\Fb}$ satisfies the $\PG$-condition. Let $\Gamma$ be a Fuchsian group of the first kind.
    If $\rho \colon \Gamma\to G_{\Fb}$ is a $\Theta$-positive representation, then every $\Theta$-framing of $\rho$ is positive.
    In particular, $\rho$ is $\Theta$-positively frameable if and only if it is $\Theta$-positive and frameable.
\end{corollary}
\begin{proof}
    Let $\xi \colon D\to\Fc_{\Theta,\Fb}$ be a positive $\rho$-equivariant map, with $D\subset\Sb^1$ non-empty and $\Gamma$-invariant. Let $\xi_{\fr}\colon\Lambda_p\to\Fc_{\Theta,\Fb}$ be an arbitrary $\Theta$-framing of $\rho$.
    
    Fix $q\in\Lambda_p$, and let $\eta\in\Gamma$ be the positive generator of $\Stab_\Gamma(q)$. Choose $x\in D\setminus\{q\}$. Then $(\eta^n\cdot x)_{n\in\Nb}$ is cyclically ordered, so $(\rho(\eta)^n\cdot\xi(x))_{n\in\Nb}$ is a positive sequence. Since $\xi_{\fr}(q)$ is fixed by $\rho(\eta)$, \Cref{prop: positively rotating => translating} implies that $(\xi(x),\xi_{\fr}(q))$ is a $\Theta$-positively translating pair for $\rho(\eta)$. Hence
    $\xi_{\fr}(q)$ belongs to  $\bigcap_{n>0}\Diam_{\xi(x)}^{\opp} \big(\rho(\eta)^{-n}\cdot\xi(x),\rho(\eta)^n\cdot\xi(x)\big)$, which is equal to  $\bdext(q)$ by the construction of the maximal positive boundary extension $(\bdext,\zeta^{\mathrm M}_\rho)$ in \Cref{prop: converse}. Hence $\xi_{\fr}$ is a partial section of the positive boundary extension $(\bdext,\zeta^{\mathrm M}_\rho)$ and is thus positive.
\end{proof}

\section{%
\texorpdfstring{%
$\Theta$-positivity is a closed condition}%
{Theta-positivity is a closed condition}}
\label{section:PositivityClosed}

Let $\Gamma$ be a Fuchsian group (not necessarily finitely generated), and $G_{\Fb}$ as before.
We equip $\Hom(\Gamma,G_{\Fb})$ with the topology of pointwise convergence.\footnote{When $\Gamma$ is finitely generated, this coincides with the semi-algebraic topology of $\Hom(\Gamma,G_{\Fb})$.}

\subsection{Closedness}
\label{s: Closedness}

 In this section, we prove \Cref{thm-intro: Positivity closed}. Let us denote the subset of $\Hom(\Gamma, G_{\Fb})$ consisting of positive representations by $\Pos_\Theta(\Gamma, G_{\Fb})$.
 
\begin{theorem}[\Cref{thm-intro: Positivity closed}]
\label{thm: closedness}
    The set $\Pos_\Theta(\Gamma,G_{\Fb})$ is closed in $\Hom(\Gamma,G_{\Fb} )$.
\end{theorem}

The proof is essentially a semi-algebraic reformulation of the strategy developed by Beyrer, Guichard, Labourie, Pozzetti and Wienhard for closed surfaces \cite{beyrer2021positive, GLW, BeyrerGuichardLabouriePozzettiWienhard_PositivityCrossRatiosCollarLemma}. In particular, we will use their \emph{collar lemma}, which luckily they formulated in a semi-algebraic way with coefficients in $\Qbar$, so that it holds over every real closed field, even though their proof is only over $\Rb$ (see especially \cite[Appendix A]{BeyrerGuichardLabouriePozzettiWienhard_PositivityCrossRatiosCollarLemma}).

Recall from \Cref{section: Jordan projection} that $J_{\Fb}\colon G_{\Fb} \to C_{\Fb}$ denotes $\Fb$-extension of the multiplicative Jordan projection and $\alpha_{\Fb}$ the $\Fb$-extension of a multiplicative simple root $\alpha$. We also denote by $\omega_{\Fb}^\alpha$ the $\Fb$-extension of the multiplicative \emph{fundamental weight} associated to $\alpha$, and set $p^\alpha_{\Fb}(g) = (\omega^\alpha_{\Fb} \circ J_{\Fb}(g))(\omega^\alpha_{\Fb} \circ J_{\Fb}(g^{-1}))$. The precise definition of $\omega_{\Fb}^\alpha$ does not matter much: for our purposes it is enough to know that it is the $\Fb$-extension of a continuous semi-algebraic map $\omega^\alpha$ from $G$ to $\Qbar_{\geqslant 1}$.

\begin{lemma}[{\cite[Corollary D]{BeyrerGuichardLabouriePozzettiWienhard_PositivityCrossRatiosCollarLemma}}]
\label{lem: collar lemma}
    Let $\gamma, \gamma'$ be hyperbolic elements in $\Gamma$ such that $(\gamma^-, \gamma'^-, \gamma^+ ,\gamma'^+)$ is cyclically ordered.
    Then, for any $\Theta$-positive representation $\rho \colon \Gamma \to G_{\Fb}$ and any (multiplicative) root $\alpha\in \Theta$, we have
    \begin{equation} \label{eq: collar lemma}
    (\alpha_{\Fb}\circ J_{\Fb}(\rho(\gamma))-1)(p^\alpha_{\Fb}(\rho(\gamma'))-1) \geqslant 1~.
    \end{equation}
\end{lemma}

\begin{remark}
    Corollary D of \cite{BeyrerGuichardLabouriePozzettiWienhard_PositivityCrossRatiosCollarLemma} actually has a strict inequality in \eqref{eq: collar lemma}. However, the positive constant on the right side does not matter for our purpose, and it will be useful to have a closed condition.
\end{remark}

We say that a representation $\rho \colon \Gamma \to G_{\Fb}$ satisfying \eqref{eq: collar lemma} satisfies the \emph{collar lemma}.
Its main property of interest to us is the following:

\begin{lemma}
    \label{lem: collar lemma implies proximal}
    Let $\Gamma$ be a Fuchsian group of the first kind and $\rho\colon\Gamma \to G_{\Fb}$ a representation that satisfies the collar lemma.
    Then $\rho$ is weakly $\Theta$-proximal.
\end{lemma}
\begin{proof}
    Fix some hyperbolic element $\gamma \in \Gamma$. Since $\Gamma$ is of the first kind, we can find some hyperbolic element $\gamma'\in \Gamma$ such that $(\gamma^-, \gamma'^-, \gamma^+ ,\gamma'^+)$ is cyclically ordered ($\Gamma$ is of the first kind). Then, using the inequality \eqref{eq: collar lemma}, we obtain that
    \[\alpha_{\Fb}\circ J_{\Fb}(\rho(\gamma)) \geqslant 1 + \frac{1}{p^{\alpha}_{\Fb}(\rho(\gamma'))-1}> 1\]
    for every $\alpha\in \Theta$.
    Hence $\rho(\gamma)$ is weakly $\Theta$-proximal.
\end{proof}

We are now ready to prove \Cref{thm: closedness}.

\begin{proof}[Proof of \Cref{thm: closedness}]
    Without loss of generality, we can assume that $\Gamma$ is of the first kind, see \Cref{propo: equivalence positivity first and second kind Fuchsian groups}.
    
    Roughly speaking, we show that the set of positive representations coincides with the set of representations satisfying the collar lemma, and such that the map sending a fixed point of a hyperbolic element to its weak attracting fixed flag is semi-positive.\footnote{In the proof below we refrain from using converging sequences of representations, because there are real closed fields whose order topology is not sequential, e.g.\ the hyperreals.}

    Let \[\CL_\Theta(\Gamma,G_{\Fb})\subset \Hom(\Gamma,G_{\Fb})
    \]
    be the subset of representations satisfying the collar lemma.
    Then $\CL_\Theta(\Gamma,G_{\Fb})$ is closed in $\Hom(\Gamma,G_{\Fb})$, since $J_{\Fb}$ is continuous, see \Cref{propo: Jordan proj}.

    Recall that the map $\Prox_\Theta(G_{\Fb}) \to \Fc_{\Theta,\Fb}$ sending a weakly $\Theta$-proximal element $g$ to its weak attracting fixed flag $g^+$ is continuous (\Cref{prop: continuity g->g+ hyperbolic}).
    For any integer $n\ge 3$, denote by $X_n$ the set of all hyperbolic $n$-tuples $(\gamma_1,\ldots,\gamma_n)$ in $ \Gamma$ such that the $n$-tuple $((\gamma_1)^+,\ldots,(\gamma_n)^+)$ is cyclically ordered. Then by \Cref{lem: collar lemma implies proximal}, for any $x=(\gamma_1,\dots,\gamma_n) \in X_n$, the composition of maps
    \[
    \begin{array}{cccccc}
    \varphi_x \colon &\CL_\Theta(\Gamma,G_{\Fb}) &\to& (\Prox_\Theta(G_{\Fb}))^n &\to& \Fc_{\Theta,\Fb}^n,\\
    &\rho &\mapsto& (\rho(\gamma_1),\ldots,\rho(\gamma_n)) &\mapsto& (\rho(\gamma_1)^+,\ldots,\rho(\gamma_n)^+)
    \end{array}
    \]
    is defined and continuous (recall that we endowed $\CL_\Theta(\Gamma,G_{\Fb})$ with the subspace topology of the topology of pointwise convergence).
    
    We claim that 
    \[\Pos_\Theta(\Gamma,G_{\Fb}) = \bigcap_{n \ge 3} \bigcap_{x \in X_n} \varphi_x^{-1}((\Fc_{\Theta,\Fb}^n)_{\geqslant 0})~,\]
    where $(\Fc_{\Theta,\Fb}^n)_{\geqslant 0}$ denotes the closed subset of semi-positive $n$-tuples in $\Fc_{\Theta,\Fb}$.
    Since $\CL_\Theta(\Gamma,G_{\Fb})$ is also closed in $\Hom(\Gamma,G_{\Fb})$ this shows that $\Pos_\Theta(\Gamma,G_{\Fb})$ is closed.
    
    Let us now prove the claim. 
    By \Cref{lem: collar lemma}, we have that every $\Theta$-positive representation satisfies the collar lemma, which proves the inclusion $\subseteq$.

    For the other inclusion $\supseteq$, let $\rho$ be a representation that satisfies the collar lemma and such that
    \[
    \xi^h_\rho \colon \Lambda_h \to \Fc_{\Theta,\Fb}, \quad 
    \gamma^+ \mapsto \rho(\gamma)^+
    \]
    is semi-positive.
    We want to prove that $\xi^h_\rho$ is positive and, by \Cref{prop: transverse + semi positive implies positive} it is enough to prove that $\xi^h_\rho$ is transverse.

    Let $x$ and $y$ be two distinct points in $\Lambda_h$. By density of pairs $(\gamma^-,\gamma^+)$ in $\Sb^1\times \Sb^1$, we can find hyperbolic elements $\gamma, \eta\in \Gamma$ such that the sextuple $(\gamma^-,\gamma^+, x, \eta^-, \eta^+, y)$ is cyclically ordered. Hence $(\xi^h_\rho(\gamma^-), \xi^h_\rho(\gamma^+), \xi^h_\rho(x), \xi^h_\rho(\eta^-), \xi^h_\rho(\eta^+), \xi^h_\rho(y))$ is semi-positive.
    Since $\rho$ is weakly $\Theta$-proximal, we know that $\xi^h_\rho(\gamma^-)$ and $\xi^h_\rho(\eta^-)$ are transverse to $\xi^h_\rho(\gamma^+)$ and $\xi^h_\rho(\eta^+)$ respectively.
    By \Cref{lem: semi-positive + a bit transverse implies transverse}, we conclude that $\xi^h_\rho(x)$ and $\xi^h_\rho(y)$ are transverse.
\end{proof}

\subsection{Irreducibility}
\label{s: irreducibility}

As an application of closedness of the set of $\Theta$-positive representations we conclude that $\Theta$-positive representations are irreducible in the following sense.
Again we do not assume that $\Gamma$ is finitely generated.

\begin{proposition}
\label{propo: positive implies irreducible}
    Let $\Gamma$ be a Fuchsian group and $\rho \colon \Gamma \to G_{\Fb}$ a $\Theta$-positive representation.
    Then $\rho(\Gamma)$ is not contained in any proper parabolic subgroup of $G_{\Fb}$.
\end{proposition}

\begin{proof}
    We can without loss of generality assume that $\Gamma$ is of the first kind, see \Cref{propo: equivalence positivity first and second kind Fuchsian groups}.
    
    Assume by contradiction that $\rho(\Gamma)$ is contained in some parabolic subgroup $P \subset G_{\Fb}$ (potentially different from $P_{\Theta,\Fb}$).
    In the closure of $\{g \rho g^{-1}\mid g\in G_{\Fb}\} \subset \Hom(\Gamma,G_{\Fb})$ there exists a representation $\rho_{ss} \colon \Gamma \to G_{\Fb}$ whose image is contained in the Levi factor $L$ of $P$ ($\rho_{ss}$ is the semi-simplification), see \Cref{section:SemiAlgGroups}.
    Since the set of $\Theta$-positive representations is closed (\Cref{thm: closedness}), we obtain that $\rho_{ss}$ is $\Theta$-positive.
    Now if we take any $g$ in the centralizer $Z_{G_{\Fb}}(\rho_{ss}(\Gamma))$ of $\rho_{ss}(\Gamma)$ in $G_{\Fb}$, then for any hyperbolic element $\gamma \in \Gamma$ we have 
    \[(\rho_{ss}(\gamma))^+ = (g \rho_{ss}(\gamma) g^{-1})^+ = g \rho_{ss}(\gamma)^+~.\]
    Hence $Z_{G_{\Fb}}(\rho_{ss}(\Gamma))$ fixes pointwise the image of the positive map $\xi^h_\rho$.
    Since\break $Z_{G_{\Fb}}(\rho_{ss}(\Gamma))$ is a semi-algebraic subset of $G_{\Fb}$, we can use the transfer principle and \cite[Theorem 1.6]{guichard2025generalizing} to conclude that $Z_{G_{\Fb}}(\rho_{ss}(\Gamma))$ is closed and bounded (the analogue of compactness in real algebraic geometry, see \Cref{thm_ExtConnComp}~(\ref{thm_ExtConnComp: closed and bounded})).
    On the other hand, since $\rho_{ss}$ has image in $L$, its centralizer contains the center $Z(L)$ of $L$. By construction of parabolic subgroups, $Z(L)$ is unbounded, a contradiction.
\end{proof}

\section{%
\texorpdfstring{%
More on $\Theta$-positivity in the real case}%
{More on Theta-positivity in the real case}}
\label{s: Real case}

In this section, we study further properties of $\Theta$-positive representations that are specific to the real case, and relate them to several generalizations of the notion of $\Theta$-Anosov representations.
In particular, we show that $\Theta$-positive representations over $\Rb$ are extended geometrically finite and $\Theta$-divergent, see \Cref{thm-intro: Theta-positive EGF}, and give conditions on when they are in fact relatively $\Theta$-Anosov.

Throughout this section, $G_{\Rb}$ is a \emph{real} semisimple Lie group with a $\Theta$-positive structure.

\subsection{Divergence and continuity of boundary maps}
\label{subsection: continuity boundary maps}

We start by proving that $\Theta$-positive representations over $\Rb$ of a Fuchsian group of the first kind admit ``canonical'' equivariant left and right continuous boundary maps defined on $\Sb^1$, which are continuous at conical limit points. In particular, such $\Theta$-positive representations are $\Theta$-positively frameable. From this, we deduce that over $\Rb$, every $\Theta$-positive representation of a closed surface group is $\Theta$-Anosov. In the process, we also prove that $\Theta$-positive representations over $\Rb$ are \emph{$\Theta$-divergent} (see \Cref{definition: divergent sequence and representation}).\\

Let $\Gamma$ be a non-elementary Fuchsian group with limit set $\Lambda$, and let $\rho\from\Gamma\to G_{\Rb}$ be a $\Theta$-positive representation. We say that a point $x\in\Lambda$ \emph{is a left limit} (respectively, is a right limit) if there is a sequence of points in $\Lambda$ that converges to $x$ from the left (respectively, from the right). If $x\in\Lambda$ is a left limit (respectively, right limit), choose a positive $\rho$-equivariant map $\xi \colon D\subset\Lambda\to\Fc_{\Theta,\Rb}$ and a sequence $(x_n)_{n\in\Nb}$ in $D$ that converges to $x$ from the left (respectively, from the right), and define 
\[\xi_\rho^l(x) \coloneqq \lim_{n\to +\infty} \xi(x_n), \quad \textnormal{(respectively, } \quad \xi_\rho^r(x) \coloneqq \lim_{n\to +\infty} \xi(x_n))~.\]
Here, notice that since $D$ is $\Gamma$-invariant, it is dense in $\Lambda$, so such a sequence $(x_n)_{n\in\Nb}$ exists, and the above limits exist by \Cref{proposition: least upper bounded property for positive sequence}.

\begin{lemma}\label{lemma: left and right limit independent of the choice of limit map}
If $x\in\Lambda$ is a left limit (respectively, right limit), then  $\xi_\rho^l(x)$ (respectively, $\xi_\rho^r(x)$) does not depend on the choice of the map $\xi$ nor choice of the sequence $(x_n)_{n\in \Nb}$. 
\end{lemma}
\begin{proof}
We prove the claim for the case when $x$ is a left limit; the proof of the case when $x$ is a right limit is identical.
Let $\xi \colon D \to \Fc_{\Theta,\Rb}$ and $\xi'\colon D' \to \Fc_{\Theta, \Rb}$ be two positive $\rho$-equivariant boundary maps, and $(x_n)_{n\in \Nb} \subset D$ and $(y_n)_{n\in \Nb}\subset {D'}$ be two sequences converging to $x$ from the left. We want to show that
\[\lim_{n\to+\infty}\xi(x_n)=\lim_{n\to+\infty}\xi'(y_n).\]
By \Cref{proposition: least upper bounded property for positive sequence}, both of these limits exist, so by taking sub-sequences, we may assume that the two sequences are disjoint, and that the sequence $(z_n)_{n\in\Nb}$ in $D\cup D'$ defined by $z_{2k-1}=x_k$ and $z_{2k}=y_k$ for all $k\in\Nb$, also converges to $x$ from the left. For all $k\in\Nb$, let $a_{2k-1}\coloneqq \xi(x_k)$ and $a_{2k}\coloneqq \xi'(y_k)$. Then by \Cref{coro: compatibility positive boundary maps}, $(a_n)_{n\in\Nb}$ is a positive sequence that contains both $(\xi(x_n))_{n\in\Nb}$ and $(\xi'(y_n))_{n\in\Nb}$ as subsequences. Again by \Cref{proposition: least upper bounded property for positive sequence}, $(a_n)_{n\in\Nb}$ converges, so the required equality holds.

\end{proof}

Let $\Lambda_l=\Lambda_l(\Gamma)$ (respectively, $\Lambda_r=\Lambda_r(\Gamma)$) denote the set of points in $\Lambda(\Gamma)$ that are left limits (respectively, right limits). In the case when $\Gamma$ is of the first kind, then $\Lambda_l=\Lambda_r=\Lambda=\Sb^1$. In general, observe that both $\Lambda_l$ and $\Lambda_r$ are $\Gamma$-invariant subsets of $\Lambda$, and the perfectness of $\Lambda$ implies that their union is all of $\Lambda$. Then by \Cref{lemma: left and right limit independent of the choice of limit map}, we have well-defined, $\rho$-equivariant maps
\[\xi_\rho^l \colon \Lambda_l\to\Fc_{\Theta,\Rb}\quad\text{and}\quad\xi_\rho^r \colon \Lambda_r\to\Fc_{\Theta,\Rb}.\]

\begin{proposition} \label{prop: the maps are left and right continuous}
The $\rho$-equivariant maps $\xi_\rho^l$ and $\xi_\rho^r$ are positive and respectively left and right continuous.
\end{proposition}

To prove \Cref{prop: the maps are left and right continuous}, recall from \Cref{Fuchsian groups lemma} that there is a Fuchsian group $\Gamma'$ of the first kind, and a semi-conjugacy $\iota \colon \Gamma\to\Gamma'$.
Let $\alpha\colon\Lambda(\Gamma)\to  \Lambda(\Gamma') = \Sb^1$ be the continuous, $\iota$-equivariant, monotonic, surjective map.

\begin{lemma}\label{injective alpha}
The restrictions of $\alpha$ to both $\Lambda_l(\Gamma)$ and $\Lambda_r(\Gamma)$ are injective. 
\end{lemma}

\begin{proof}
We will prove this for $\Lambda_l(\Gamma)$, the case of $\Lambda_r(\Gamma)$ is similar. Pick a distinct pair of points $x,y\in\Lambda_l(\Gamma)$, and let $I_1$ and $I_2$ be the two open sub-intervals of $\Sb^1$ that have $x$ and $y$ as endpoints. Notice that $I_1\cap\Lambda(\Gamma)$ and $I_2\cup\Lambda(\Gamma)$ are both non-empty. 

Suppose for the purpose of contradiction that $\alpha(x)=\alpha(y)\eqqcolon p\in\Lambda(\Gamma')$. Then the fact that $\alpha$ is monotonic implies that $\alpha(I_1)=p$ or $\alpha(I_2)=p$. Assume without loss of generality that $\alpha(I_1)=p$. Since $I_1\cap\Lambda$ is an open subset of $\Lambda$, the density of $\Lambda_h(\Gamma)$ in $\Lambda(\Gamma)$ implies that $I_1$ contains a pair of distinct points in $\Lambda_h(\Gamma)$. Let $\gamma,\gamma'\in\Gamma$ be hyperbolic elements that fix these two points. Then $\gamma$ and $\gamma'$ do not commute. However, this means that $\iota(\gamma)$ and $\iota(\gamma')$ in $\Gamma'$ are non-commuting elements that fix $p$. This violates a standard fact from hyperbolic geometry, namely that the group generated by a pair of non-commuting, infinite order elements in $\PSL_2(\Rb)$ that share a common fixed point cannot be discrete. 
\end{proof}

Using \Cref{injective alpha}, we can now prove the positivity claim of \Cref{prop: the maps are left and right continuous} by reducing it to the case where $\Gamma$ is of the first kind and using the maximal boundary extensions $(\bdext,\zeta_\rho^\textnormal{M})$ given by \Cref{prop: converse} (these are constructed only in the case when $\Gamma$ is of the first kind).

\begin{proof}[Proof of \Cref{prop: the maps are left and right continuous}]
Again, we will only give the proof for $\xi_\rho^l$; the proof for $\xi_\rho^r$ is identical. 

First, we prove that $\xi_\rho^l$ is positive. By \Cref{injective alpha}, $\alpha|_{\Lambda_l(\Gamma)} \colon \Lambda_l(\Gamma)\to\Lambda(\Gamma')$ is a monotonic, $\iota$-equivariant bijection. This implies that 
\[\xi_\rho^l\circ(\alpha|_{\Lambda_l(\Gamma)})^{-1}=\xi_{\rho\circ\iota^{-1}}^l,\]
so it suffices to prove that $\xi_{\rho\circ\iota^{-1}}^l$ is positive. Furthermore, by \Cref{propo: equivalence positivity first and second kind Fuchsian groups}, $\rho$ is positive if and only if $\rho\circ\iota^{-1}$ is positive, so we have reduced to the case where $\Gamma$ is of the first kind.

In the case when $\Gamma$ is of the first kind, pick any $x\in\Lambda(\Gamma)$. Then choose a positive, $\rho$-equivariant map $\xi \colon D\to\Fc_{\Theta,\Rb}$, and let $(x_n)_{n\in\Nb}$ and $(x_n')_{n\in\Nb}$ be sequences in $D$ that converge to $x$ from the left and the right respectively. Then observe that $\xi^l_\rho(x)$ belongs to
    \[\bigcap_{n\ge 2} \Diam_{\xi(x_1)}^\opp(\xi(x_n), \xi(x'_n)) = \Vc^\textnormal{M}_\rho(x)~.\]
    Hence $\xi^l_\rho$ is a section of $(\zeta^{\mathrm M}_\rho,\Vc^\textnormal{M}_\rho)$ and is thus positive by \Cref{prop: Boundary extension implies positive}.

Next, we prove that $\xi_\rho^l$ is left continuous. By \Cref{lemma: left and right limit independent of the choice of limit map}, $\xi^l_\rho$ is does not depend on the choice of $\rho$-equivariant positive map $\xi \colon D\to\Fc_{\Theta,\Rb}$. Thus, if we take $\xi=\xi^l_\rho$, we get that, for any sequence $(x_n)_{n\in \Nb}\in \Sb^1$ converging to $x$ from the left,
    \[\lim_{n\to +\infty} \xi^l_\rho(x_n) = \xi^l_\rho(x)~,\]
    i.e.\ $\xi^l_\rho$ is left continuous.
\end{proof}

As a corollary of the existence of $\xi_\rho^l$ and $\xi_\rho^r$ we can finally conclude that $\Theta$-positive representations over $\Rb$ are frameable:

\begin{corollary} \label{coro: positive => frameable over R}
    Let $\Gamma$ be a Fuchsian group of the first kind and $\rho \colon \Gamma \to G_{\Rb}$ a $\Theta$-positive representation. 
    Then $\rho$ is $\Theta$-positively frameable.
    
    Moreover, for any positive parabolic element of $\Gamma$ with fixed point $\gamma^+\in \Lambda_p$, we have
    \[\xi^r_\rho(\gamma^+) = \rho(\gamma)^+ \quad \textrm{ and } \quad \xi^l_\rho(\gamma^+) = \rho(\gamma)^-~.\]
    In other words both $\xi^l_\rho$ and $\xi^r_\rho$ extend the canonical framings defined in \Cref{lem:PosFramedRepr}.
\end{corollary}

\begin{proof}
    Since the boundary map $\xi^l_\rho$ is defined on the whole circle $\Sb^1$, its restriction to $\Lambda_p$ defines a positive framing. Moreover, for any $x\neq \gamma^+ \in \Sb^1$, the sequence $(\gamma^n\cdot x)_{n\in \Nb}$ converges to $\gamma^+$ from the left (since $\gamma$ is a positive parabolic element), hence
    \[\xi_\rho^l(\gamma^+) = \lim_{n\to +\infty}\xi_\rho^l(\gamma^n\cdot x) = \lim_{n\to +\infty} \rho(\gamma)^n \cdot \xi_\rho^l(x) = \rho(\gamma)^+\]
    by \Cref{propo: OverRPosRotImpliesDiv}. The proof for $\xi_\rho^r$ is similar.
\end{proof}

Using the monotonic approximation property of Fuchsian groups acting on their limit set (\Cref{lem: Monotonous limit points}), we can now prove the following two results:

\begin{theorem}
\label{thm: positive rep are divergent}
    Let $\Gamma$ be a Fuchsian group (not necessarily of the first kind) and $\rho \colon \Gamma \to G_{\Rb}$ a $\Theta$-positive representation. 
    Then $\rho$ is $\Theta$-divergent.
\end{theorem}

\begin{theorem} \label{thm: continuity at conical limit points}
    Let $\Gamma$ be a Fuchsian group of the first kind and $\rho \colon \Gamma \to G_{\Rb}$ a $\Theta$-positive representation.
    Then $\bdext(x)$ is a singleton whenever $x$ is conical. In particular, if $D,D'\subset\Lambda$ are $\Gamma$-invariant subsets and $\xi \colon D \to \Fc_{\Theta,\Rb}$ and $\xi' \colon D' \to \Fc_{\Theta,\Rb}$ are two positive $\rho$-equivariant maps, then for every conical limit point $x \in D \cap D'$, we have $\xi(x)= \xi'(x)$, and $\xi$ and $\xi'$ are continuous at $x$. 
\end{theorem}

\begin{proof}[Proof of \Cref{thm: positive rep are divergent}]
    Since the divergence property is independent of how $\Gamma$ is embedded in $\Isom_+(\Hb)$, we can assume without loss of generality that $\Gamma$ is of the first kind. Let $\xi\colon \Sb^1 \to \Fc_{\Theta,\Rb}$ be a $\rho$-equivariant positive boundary map.
    
    Let $(\gamma_n)_{n\in \Nb}$ be a sequence in $\Gamma$ that leaves every finite set.
    By \Cref{lem: Monotonous limit points}(1), up to extracting a subsequence, we can find a points $a,b,x\in \Sb^1$ such that $(a,b,x)$ is cyclically ordered and $\gamma_n\cdot [a,b]$ converges to $x$ from the left or from the right, where $[a,b]\subset\Sb^1$ is the subinterval with endpoints $a$ and $b$ that does not contain $x$. We may assume that this convergence is from the right; the other case is similar. 
    
    Let $c$ be a point in the interior $(a,b)$ of $[a,b]$, and pick any $y\in\Diam_{\xi_\rho^r(c)}(\xi_\rho^r(a),\xi_\rho^r(b))$. For $n>1$, $(a,\gamma_n\cdot a,\gamma_n\cdot c, \gamma_n\cdot b,x)$ is cyclically ordered, so
    \[\rho(\gamma_n)\cdot y\in\Diam_{\xi_\rho^r(\gamma_n\cdot c)}(\xi_\rho^r(\gamma_n\cdot a),\xi_\rho^r(\gamma_n\cdot b))\subset\Diam_{\xi_\rho^r(a)}^\opp(\xi_\rho^r(\gamma_n\cdot a),\xi_\rho^r(x)).\]
    Since $\xi^r_\rho(\gamma_n\cdot a)\to \xi^r_\rho(x)$ as $n\to+\infty$, \Cref{onesided squeeze} implies that $\rho(\gamma_n)\cdot y\to \xi^r_\rho(x)$ as $n\to+\infty$. Since $\Diam_{\xi^r_\rho(c)}(\xi^r_\rho(a),\xi^r_\rho(b))\subset\Fc_{\Theta,\Fb}$ is an open set,
    by \Cref{lemma: divergence in flag manifold}~(\ref{lemma: divergence in flag manifold: open sets}), we conclude that the sequence $\rho(\gamma_n)$ is $\Theta$-attracting. Hence the representation $\rho$ is $\Theta$-divergent.
\end{proof}

\begin{proof}[Proof of \Cref{thm: continuity at conical limit points}]
    By \Cref{lem: Monotonous limit points}(2), we can find a sequence $(\gamma_n)_{n\in\Nb}$ and points $a,b,a',b'\in\Sb^1$ such that $(a,b,x)$ and $(x,b',a')$ are cyclically ordered, $\gamma_n\cdot [a,b]$ converges to $x$ from the left, and $\gamma_n\cdot[b',a']$ converges to $x$ from the right (again, $[a,b]$ and $[b',a']$ do not contain $x$). Let $c\in (a,b)$ and $c'\in(b',a')$. Repeating the arguments of the proof of \Cref{thm: positive rep are divergent} above, we deduce that $\rho(\gamma_n)\cdot y$ converges to $\xi^l_\rho(x)$ for all $y\in \Diam_{\xi(c)}(\xi(a),\xi(b))$ and to $\xi^r_\rho(x)$ for all $y\in \Diam_{\xi(c')}(\xi(a'),\xi(b'))$.
    By \Cref{lemma: divergence in flag manifold}, we conclude that both $\xi^l_\rho(x)$ and $\xi^r_\rho(x)$ coincide with the unique attracting point of the sequence $(\gamma_n)_{n\in\Nb}$.
    In particular, $\xi(\gamma_n\cdot c)$ and $\xi(\gamma_n\cdot c')$ have the same limit, and so by \Cref{lem: squeeze lemma for diamond not real}
    \[\Vc^\textnormal{M}_\rho(x)=\bigcap_{n\in \Nb} \Diam_{\xi(x)}(\xi(\gamma_n\cdot c),\xi(\gamma_n\cdot c')) = \{\xi^l_\rho(x)\} = \{\xi^r_\rho(x)\}~.\]
This proves the first claim of the theorem.

Let $\xi \colon D\to \Fc_{\Theta,\Fb}$ be any positive $\rho$-equivariant boundary map, and let $x\in D$ be a conical limit point. By \Cref{prop: converse}, $\xi(x)\in \Vc^\textnormal{M}_\rho(x)$, which we have established above is a singleton that does not depend on $\xi$.
Furthermore, since $\Vc^\textnormal{M}_\rho(x)=\{\xi^l_\rho(x)\} = \{\xi^r_\rho(x)\}$, if $x_n$ is any sequence converging to $x\in D$ from the left or from the right, we have
    \[\lim_{n\to +\infty} \xi(x_n) = \xi^l_\rho(x) = \xi^r_\rho(x)= \xi(x)\]
    by \Cref{lemma: left and right limit independent of the choice of limit map}. Hence $\xi$ is continuous at $x$.
\end{proof}

\subsection{Extended geometrically finite representations}\label{subsection: EGF}

We now collect the consequences of the previous section in terms of $\Theta$-boundary extensions, and conclude that, when $\Gamma$ is geometrically finite, $\Theta$-positive representations over $\Rb$ are \emph{extended geometrically finite} in the sense of Weisman and admit a \emph{minimal} $\Theta$-boundary extension. \\

Let $\Gamma$ be a Fuchsian group of the first kind and $\rho\colon \Gamma \to G_{\Rb}$ a representation.
If $(\Vc,\zeta)$ is a $\Theta$-boundary extension of $\rho$, then the convergence property of the action of $\Gamma$ on $\Sb^1$ implies that for every diverging sequence $(\gamma_n)$ in $\Gamma$, there exist $x,y\in \Sb^1$ such that, up to extraction, $\rho(\gamma_n)$ contracts $\zeta^{-1}(\Sb^1\setminus \{y\})$ towards $\zeta^{-1}(x)$.
Informally speaking, Weisman's \emph{extended geometrically finite property} requires that this convergence property extends to a relative neighborhood of $\Vc$.

\begin{definition}[{\cite[Definition 1.3]{Weisman}}]\label{def: EGF}
    A representation $\rho \from \Gamma\to G_{\Rb}$ is called $\Theta$-\emph{extended geometrically finite} (later abbreviated as $\Theta$-EGF) if it admits a transverse $\Theta$-boundary extension $\zeta \from \Vc \to \Sb^1$ and a family $(C_x)_{x\in \Sb^1}$ of open subsets of $\Fc_{\Theta,\Rb}$ that \emph{extends the convergence dynamics} in the following sense:
    \begin{enumerate}
    \item For all $x\in \Sb^1$, $\Vc\setminus \zeta^{-1}(x)\subset C_x$,  
    \item for every sequence $(\gamma_n)_{n\in \Nb}$ in $\Gamma$ for which there exist $x,y \in \Sb^1$ such that $\gamma_n\cdot z \underset{n\to +\infty}{\longrightarrow} x$ for all $z\in\Sb^1\setminus\{y\}$, for any compact set $K\subset C_y$ and any open set $O\supset \zeta^{-1}(x)$, we have
    \[\rho(\gamma_n)K\subset O\]
    for $n$ large enough.
    \end{enumerate}
\end{definition}

\begin{remark}
    In the work \cite{Weisman}, Weisman defined EGF representations only for relatively hyperbolic groups. The notion of EGF representation actually makes sense more generally for a group $\Gamma$ equipped with a convergence action on a compact space. Here we adapt the definition for Fuchsian groups, the case in which we are interested throughout this paper.
\end{remark}

While the definition is a bit involved, the third author remarked in \cite{W23b} that the EGF condition is better behaved when we assume divergence. To describe this, we use the Benoist limit set, which we now define.

\begin{definition}
For any subgroup $H\subset G_{\Rb}$, the \emph{$\Theta$-Benoist limit set} of $H$, denoted $\mathcal L_H=\mathcal L_{\Theta,H}$, is the set of attracting flags in $\Fc_{\Theta,\Rb}$ of all $\Theta$-attracting sequences in $H$. 
\end{definition}

The set $\mathcal L_H$ was introduced by Benoist in \cite{benoist1997proprietes}, where the following is proven:

\begin{theorem}
\label{remark: Benoist limit cone}
The set $\Lc_H$ is closed and $H$-invariant. If $H$ contains $\Theta$-proximal elements, then  $\Lc_H$ is the closure in $\Fc_{\Theta,\Rb}$ of the set of attracting fixed flags of all $\Theta$-proximal elements in $H$. 
\end{theorem}

When $H$ is the image of a representation $\rho$, we denote $\Lc_\rho\coloneqq\Lc_H$.

The following lemma is a slight refinement of \cite[Theorem~1.3]{W23b}:

\begin{lemma}
\label{lemma: equivalences EGF and divergent}
    Let $\rho\colon \Gamma \to G_{\Rb}$ be a $\Theta$-divergent representation. The following are equivalent:
    \begin{itemize}
        \item[(i)] the representation $\rho$ is $\Theta$-EGF;
        \item[(ii)] there exists a transverse $\Theta$-boundary extension $(\Vc,\zeta)$ of $\rho$ such that $\Lc_\rho\subset \Vc$;
        \item[(iii)] there exists a map \[\zeta^{\mathrm{m}}_\rho\colon \Lc_\rho \to \Sb^1\]
        such that $(\Lc_\rho, \zeta_\rho^{\mathrm{m}})$ is a transverse $\Theta$-boundary extension of $\rho$.
    \end{itemize}
\end{lemma}
\begin{proof}
The implication (iii) $\Rightarrow$ (ii) is immediate. Conversely, if $(\Vc,\zeta)$ is a transverse $\Theta$-boundary extension such that $\Lc_\rho \subset \Vc$, then $\zeta_\rho^{\mathrm{m}}\coloneqq \zeta\vert_{ \Lc_\rho}$ satisfies (iii).
Hence (ii) and (iii) are equivalent.

For (i) $\Rightarrow$ (ii), let $(\Vc,\zeta)$ be a transverse $\Theta$-boundary extension satisfying the EGF condition. We need to show that $\Lc_\rho \subset \Vc$. Pick any point $p_+\in\Lc_\rho$. By the definition of $\Lc_\rho$, there is a diverging sequence $(\gamma_n)_{n\in \Nb}$ in $\Gamma$ and a flag $p_-\in\Fc_{\Theta,\Fb}$ such that $\rho(\gamma_n)\cdot p$ converges to $p_+$ for all $p$ transverse to $p_-$. Up to extracting a subsequence, we can assume that there are points $x,y\in\Sb^1$ such that $\gamma_n\cdot z$ converges to $x$ for all $z\in\Sb^1\setminus\{y\}$. Since the set $p_-^\pitchfork$ of flags in $\Fc_{\Theta,\Fb}$ that are transverse to $p_-$ is dense in $\Fc_{\Theta,\Fb}$, $C_y$ necessarily intersects $p_-^\pitchfork$ and, for $q \in C_y \cap p_-^\pitchfork$
\[p_+=\lim_{n\to +\infty} \rho(\gamma_n)\cdot q \in \zeta^{-1}(x) \subset \Vc\]
by the EGF property. Since $p_+\in\mathcal L_\rho$ is arbitrary, we conclude that $\Lc_\rho \subset \Vc$.

Finally, we prove (iii) $\Rightarrow$ (i). For each $x\in\Sb^1$, define \[C_x = \{p\in \Fc_{\Theta,\Rb} \mid p\pitchfork q\quad \forall q\in (\zeta_\rho^{\mathrm{m}})^{-1}(x)\}~.\] 
 Since $(\zeta_\rho^{\mathrm{m}})^{-1}(x)$ is compact, $C_x$ is open. We will show that the transverse $\Theta$-boundary extension $(\Lc_\rho, \zeta_\rho^{\mathrm{m}})$ and the family of open sets $\{C_x\}_{x\in\Sb^1}$ satisfy the EGF property. The fact that $(\Lc_\rho, \zeta_\rho^{\mathrm{m}})$ is transverse implies that condition (1) of \Cref{def: EGF} holds, so we focus on showing condition (2) of \Cref{def: EGF}.

Suppose for the purpose of contradiction that condition (2) of \Cref{def: EGF} fails. Then there is a sequence $(\gamma_n)_{n\in \Nb}$, points $x,y\in\Sb^1$, a compact set $K\subset C_y$, and an open set $O\supset(\zeta_\rho^{\rm m})^{-1}(x)$ such that $\gamma_n\cdot z$ converges to $x$ for all $z \in \Sb^1 \setminus \{y\}$, but for all $N\in\Nb$, $\rho(\gamma_n)K\not\subset O$ for all $n\ge N$. Since $\rho$ is $\Theta$-divergent, by taking a subsequence, we may assume that $(\rho(\gamma_n))_{n\in \Nb}$ is a $\Theta$-attracting sequence. Let $p_+$ and $p_-$ denote its  attracting and repelling flag respectively. It suffices to show that $\zeta_\rho^{\mathrm{m}}(p_+)=x$ and $\zeta_\rho^{\mathrm{m}}(p_-)=y$. Indeed, this implies that $p_+\in O$ and $C_y\subset p_-^\pitchfork$, which contradicts \Cref{lemma: divergence in flag manifold}.

It remains to show that $\zeta_\rho^{\mathrm{m}}(p_+)=x$ and $\zeta_\rho^{\mathrm{m}}(p_-)=y$. Let $z\in \Sb^1$ be different from $x$, $y$, $\zeta(p_-)$, and $\zeta(p_+)$, and let $q \in {(\zeta_\rho^{\mathrm{m}})}^{-1}(z)$. By transversality of $(\Lc_\rho, \zeta_\rho^{\mathrm{m}})$, $q$ is transverse to $p_-$ and $p_+$, hence by the definition of attracting and repelling flags (see \Cref{lemma: divergence in flag manifold}),
\[\lim_{n\to +\infty}\rho(\gamma_n)\cdot q = p_+ \textrm{ and } \lim_{n\to +\infty}\rho(\gamma_n)^{-1}\cdot q = p_-~.\]
On the other hand
\[\lim_{k\to +\infty} \zeta_\rho^{\mathrm{m}}(\rho(\gamma_n)\cdot q) = \lim_{n\to +\infty} \gamma_n\cdot z = x \]\[\textrm{ and } \lim_{n\to +\infty} \zeta_\rho^{\mathrm{m}}(\rho(\gamma_n)^{-1}\cdot q) = \lim_{k\to +\infty} \gamma_n^{-1}\cdot z = y~.\]
By the continuity of $\zeta_\rho^{\mathrm{m}}$, $\zeta_\rho^{\mathrm{m}}(p_+) = x$ and $\zeta_\rho^{\mathrm{m}}(p_-) = y$ as required.
\end{proof}

We are now ready to prove \Cref{thm-intro: Theta-positive EGF}, namely:
\begin{theorem}\label{theorem: extended positive representations are extended geometrically finite}
    If $\rho \from \Gamma\to G_{\Rb}$ is $\Theta$-positive, then 
    \[\Lc_\rho = \xi^l_\rho(\Sb^1)\cup \xi^r_\rho(\Sb^1),\] 
    and is contained in $\Vc$ for every positive $\Theta$-boundary extension $(\Vc,\zeta)$ (these exist by \Cref{prop: converse}). In particular, $\rho$ is $\Theta$-EGF.
\end{theorem}
\begin{proof}
    We first prove that $\Lc_\rho = \xi^l_\rho(\Sb^1)\cup \xi^r_\rho(\Sb^1)$.
    Recall that $\xi^h_\rho\colon \Lambda_h \to \Fc_{\Theta,\Rb}$ denotes the map sending the attracting fixed point of every hyperbolic element $\gamma$ to the attracting fixed flag $\rho(\gamma)$, see \Cref{section : proximal representations}.
    Following \Cref{remark: Benoist limit cone}, we have 
    \[\Lc_\rho = \overline{\xi_\rho^h(\Lambda_h)}~.\]
    Since every hyperbolic fixed point is a conical limit point, $\xi_\rho^h$ coincides with $\xi_\rho^l$ and $\xi_\rho^r$ on $\Lambda_h$ by \Cref{thm: continuity at conical limit points}, hence $\xi_\rho^h(\Lambda_h) \subset \xi^l_\rho(\Sb^1) \cup \xi^r_\rho(\Sb^1)$.
    On the other hand, by \Cref{lemma: left and right limit independent of the choice of limit map}, $\xi^l_\rho(\Sb^1) \cup \xi^r_\rho(\Sb^1) \subset \overline{\xi_\rho^h(\Lambda_h)} = \Lc_\rho$.
    Finally, since $\xi^l_\rho(\Sb^1)\cup \xi^r_\rho(\Sb^1)$ contains the left and right limits of any positive map $\xi$ (in particular, of $\xi^l_\rho$ and $\xi^r_\rho$ themselves), we deduce that $\xi^l_\rho(\Sb^1) \cup \xi^r_\rho(\Sb^1)$ is closed, and we conclude that
    \[\Lc_\rho = \xi^l_\rho(\Sb^1) \cup \xi^r_\rho(\Sb^1)~.\]

    Now, let $(\Vc,\zeta)$ be a positive $\Theta$-boundary extension of $\rho$. We will next show that $\Lc_\rho\subset\Vc$. Since $\Vc$ is compact (it is a closed subset of $\Fc_{\Theta,\Rb}$, which is compact) and $\zeta(\Vc)$ is $\Gamma$-invariant, the minimality of the $\Gamma$-action on $\Sb^1$ implies that $\zeta$ is surjective. In particular, the image of $\zeta$ contains a point $x \in \Sb^1$ with trivial stabilizer in $\Gamma$ and we can construct a partial section $\xi$ of $\zeta$ on $\Gamma\cdot x$ by setting \[\xi(\gamma\cdot x) \coloneqq  \rho(\gamma)\cdot \xi(x)~,\]
    where $\xi(x)$ is any point in $\zeta^{-1}(x)$.
    By \Cref{lemma: left and right limit independent of the choice of limit map}, $\overline{\xi(\Gamma\cdot x)}$ contains $\xi^l_\rho(\Sb^1) \cup \xi^r_\rho(\Sb^1) =\Lc_\rho$. Since $\Vc$ is closed, we conclude that $\Lc_\rho \subset \Vc$. 
    
     By \Cref{thm: positive rep are divergent}, $\rho$ is $\Theta$-divergent. We have shown that $\Lc_\rho$ is contained in some (in fact any) positive (hence transverse) $\Theta$-boundary extension. We conclude that $\rho$ is $\Theta$-EGF by \Cref{lemma: equivalences EGF and divergent}.
\end{proof}

\subsection{Relatively Anosov representations}

We conclude this section on the real case with a discussion on the relation between $\Theta$-positivity and the relative $\Theta$-Anosov property.\\

In this section we assume that $\Gamma$ is a geometrically finite Fuchsian group (not necessary of the first kind).
This means that $\Gamma$ is finitely generated and its limit set $\Lambda=\Lambda(\Gamma)$ consists of only conical limit points and parabolic fixed points.

\begin{definition}[{\cite[Definition~7.1]{KL}, \cite[Theorem 1.2]{CZZ}, \cite[Definition~1.1]{ZZ1}}]\label{def: Relatively Anosov}
    A representation $\rho\colon\Gamma\to G_{\Rb}$ is \emph{relatively $\Theta$-Anosov} if $\rho$ is $\Theta$-divergent and there exists a map $\xi \colon \Lambda \to \Lc_\rho$ that is a $\rho$-equivariant, transverse homeomorphism.
\end{definition}

\begin{remark}
    In \cite{ZZ1}, relatively $\Theta$-Anosov representations are defined more generally for a group $\Gamma$ that is \emph{hyperbolic relatively to a collection of subgroups}. The above definition coincides with that of Zhu--Zimmer when we see a geometrically finite Fuchsian group $\Gamma$ as a hyperbolic group relatively to the stabilizers of parabolic fixed points in $\Lambda(\Gamma)$. In particular, when $\Gamma$ is convex cocompact (equivalently, does not contain any parabolic element), \Cref{def: Relatively Anosov} can be taken as a definition of $\Theta$-Anosov representations of $\Gamma$.
\end{remark}

When the Fuchsian group $\Gamma$ is cocompact, the whole circle $\Sb^1$ consists of conical limit points. Thus, if $\rho \colon \Gamma\to G_{\Rb}$ is $\Theta$-positive, then by \Cref{thm: continuity at conical limit points}, $\xi_\rho^l = \xi_\rho^r \colon \Sb^1 \to \Fc_{\Theta,\Rb}$ is a continuous positive (hence transverse) boundary map, and we recover the result of \cite[Theorem~A]{GLW} that $\rho$ is $\Theta$-Anosov.

When $\Gamma$ is not cocompact, on the other hand, the boundary maps $\xi^l_\rho$ and $\xi^r_\rho$ associated to a positive representation may not coincide in general.

\begin{definition}\label{definition: weakly unipotent}
    An element $g\in G_{\Rb}$ is called \emph{weakly unipotent} if its Jordan decomposition $g= g_h g_e g_u$ has trivial hyperbolic part, i.e.\ $g_h = \Id_{G_{\Rb}}$.
\end{definition}

Recall that an element $\gamma \in \Gamma$ is called a \emph{peripheral hyperbolic element} if it preserves a connected component of $\Sb^1\setminus \Lambda(\Gamma)$.

\begin{definition}
    A representation $\rho \colon \Gamma\to G_{\Rb}$ is \emph{type-preserving} if it sends every peripheral hyperbolic element in $\Gamma$ to a $\Theta$-proximal element in $G_{\Rb}$, and sends every parabolic element in $\Gamma$ to a weakly unipotent element in $G_{\Rb}$.
\end{definition}

The following theorem follows from the definition of $\Theta$-extended geometrically finite representations and relatively $\Theta$-Anosov representations.

\begin{theorem}\label{theorem: positive relatively Anosov}
    Let $\rho \colon \Gamma \to G_{\Rb}$ be a $\Theta$-positive representation. The following are equivalent:
    \begin{enumerate}
        \item
        \label{theorem: positive relatively Anosov: continuous on Lambda}
        there exists a continuous positive $\rho$-equivariant map defined from $\Lambda=\Lambda(\Gamma)$ to $\Fc_{\Theta,\Rb}$;
        \item
        \label{theorem: positive relatively Anosov: relative Anosov}
        $\rho$ is relatively $\Theta$-Anosov;
        \item
        \label{theorem: positive relatively Anosov: type preserving}
        $\rho$ is type-preserving.
    \end{enumerate}
\end{theorem}

Similar results were presented in \cite{CZZ} and \cite{yao2025critical}.
We prove the theorem based on \Cref{subsection: EGF}.

\begin{proof}[Proof of \Cref{theorem: positive relatively Anosov}]
    We first show (\ref{theorem: positive relatively Anosov: continuous on Lambda})$\implies$(\ref{theorem: positive relatively Anosov: relative Anosov}).
    A continuous positive $\rho$-equivariant limit map $\xi \colon \Lambda\to \Fc_\Theta$ is transverse. Furthermore, since every point in $\Lambda_h$ is conical, by \Cref{thm: continuity at conical limit points}, $\xi|_{\Lambda_h}=\xi_\rho^h$. Since $\Lambda_h\subset\Lambda$ is dense and $\xi$ is continuous, it follows from \Cref{remark: Benoist limit cone} that 
    \[\xi(\Lambda)=\overline{\xi(\Lambda_h)}=\overline{\xi_\rho^h(\Lambda_h)}=\mathcal L_\rho.\]
    Thus $\rho$ is relatively $\Theta$-Anosov.

    (\ref{theorem: positive relatively Anosov: relative Anosov})$\implies$(\ref{theorem: positive relatively Anosov: type preserving}) follows from \cite{ZZ1} (even without the assumption that $\rho$ is $\Theta$-positive). It remains to prove that (\ref{theorem: positive relatively Anosov: type preserving})$\implies$(\ref{theorem: positive relatively Anosov: continuous on Lambda}).

    Recall from \Cref{subsection: continuity boundary maps} that $\Lambda_l$ and $\Lambda_r$ are respectively the set of left limits and right limits in $\Lambda$, and $\Lambda=\Lambda_l\cup\Lambda_r$. Furthermore, by \Cref{prop: the maps are left and right continuous}, the maps 
    \[\xi_\rho^l\colon\Lambda_l\to\Fc_{\Theta,\Rb}\quad\text{and}\quad\xi_\rho^r\colon\Lambda_r\to\Fc_{\Theta,\Rb}\] 
    are $\rho$-equivariant, positive, and respectively left and right semi-continuous. We first show that $\xi_\rho^l$ and $\xi_\rho^r$ agree on $\Lambda_l\cap\Lambda_r$.

    Pick any point $x\in\Lambda_l\cap\Lambda_r$. Since $\Gamma$ is geometrically finite, $x$ is either parabolic or conical. 
    
    If $x$ is parabolic, then there is some positive parabolic $\eta\in\Gamma$ that fixes $x$. Notice that for any $z\in\Lambda(\Gamma)\setminus x$, $(\gamma^n\cdot z)_{n\in\Nb}$ and $(\gamma^{-n}\cdot z)_{n\in\Nb}$ converges to $x$ from the left and the right respectively, so 
    \[\lim_{n\to+\infty}\rho(\gamma^n)\cdot z=\xi_\rho^r(x)\quad\text{and}\quad\lim_{n\to+\infty}\rho(\gamma^{-n})\cdot z=\xi_\rho^l(x).\]
    At the same time, since $\rho(\eta)$ is weakly unipotent, \cite[Observation~3.13]{ZZ1} gives that $\rho(\eta)^+=\rho(\eta)^-$. Thus, by \Cref{propo: OverRPosRotImpliesDiv},
    \[\lim_{n\to+\infty}\rho(\gamma^n)\cdot z=\rho(\gamma)^+=\rho(\gamma)^-=\lim_{n\to+\infty}\rho(\gamma^{-n})\cdot z,\]
    so $\xi_\rho^l(x)=\xi_\rho^r(x)$.

    Finally, we consider the case when $x$ is conical. Let $\Gamma'$ be a Fuchsian group of the first kind for which there is a semi-conjugacy $\iota \colon \Gamma \to \Gamma'$, and let $\alpha\colon \Lambda\to \Sb^1$ be a continuous surjective monotonous $\iota$-equivariant map, see \Cref{Fuchsian groups lemma}. By \Cref{propo: equivalence positivity first and second kind Fuchsian groups}, $\rho\circ\iota^{-1}\colon  \Gamma' \to G_{\Rb}$ is $\Theta$-positive. 
    
    Notice that since $x\in\Lambda_l\cap\Lambda_r$ is conical, it is not the fixed point of a peripheral element of $\Gamma$. Thus, $\alpha(x)\in\Sb^1$ is not the fixed point of a peripheral element in $\Gamma'$, and is thus a conical point. Also, we previously observed in the proof of \Cref{prop: the maps are left and right continuous} that \Cref{injective alpha} implies 
\[\xi_\rho^l\circ(\alpha|_{\Lambda_l(\Gamma)})^{-1}=\xi_{\rho\circ\iota^{-1}}^l\quad\text{and}\quad\xi_\rho^r\circ(\alpha|_{\Lambda_r(\Gamma)})^{-1}=\xi_{\rho\circ\iota^{-1}}^r.\]
Thus,
\[\xi^l_\rho(x)=\xi_{\rho\circ\iota^{-1}}^l(\alpha(x))=\xi_{\rho\circ\iota^{-1}}^r(\alpha(x))=\xi_\rho^r(x),\]
where the second equality holds by \Cref{thm: continuity at conical limit points}.

Since $\xi_\rho^l$ and $\xi_\rho^r$ coincide on $\Lambda_l \cap \Lambda_r$, we can define
\[\xi_\rho\colon \Lambda = \Lambda_l \cup \Lambda_r \to \Fb_{\Theta,\Rb}\]
by $\xi_\rho(x) = \xi_\rho^l(x)$ if $x\in \Lambda_l$ and $\xi_\rho(x) = \xi_\rho^r(x)$ if $x\in \Lambda_r$. By \Cref{lemma: left and right limit independent of the choice of limit map}, $\xi_\rho$ is left-continuous at every left limit and right-continuous at every right limit, hence it is continuous. It is also $\rho$-equivariant. 

It remains to show that $\xi_\rho$ is positive.
Note that $\xi_\rho$ is positive in restriction to the dense subset $\Lambda_l$ by \Cref{prop: the maps are left and right continuous}. Since it is continuous, it is semi-positive on $\Lambda$ and, by \Cref{prop: transverse + semi positive implies positive}, it is enough to prove that $\xi_\rho$ is transverse. 

Let $x\neq y$ be two points in $\Lambda$. If $(x,y)$ or $(y,x) \subset \Sb^1\setminus \Lambda$, there exists a peripheral hyperbolic element $\gamma\in \Gamma$ such that $(x,y) = (\gamma_+,\gamma_-)$. Then $\xi_\rho(x) = \rho(\gamma)_+$ and $\xi_\rho(y) = \rho(\gamma)_-$, and these are transverse since $\rho(\gamma)$ is $\Theta$-proximal.

Otherwise, since $\Lambda$ is perfect and $\Lambda_l$ is dense in $\Lambda$, we can find $a,b,c,d\in \Lambda_l$ such that $(a,x,b,c,y,d)$ is cyclically ordered. By positivity of $\xi_\rho^l$, we get that $\xi_\rho(a), \xi_\rho(b),\xi_\rho(c),\xi_\rho(d)$ are pairwise transverse, hence $\xi_\rho(x)$ and $\xi_\rho(y)$ are transverse by \Cref{lem: semi-positive + a bit transverse implies transverse}. This concludes the proof that $\xi_\rho$ is transverse, hence positive. 

\end{proof}

\section{%
\texorpdfstring{%
$\Theta$-positivity is a semi-algebraic condition}%
{Theta-positivity is a semi-algebraic condition}}
\label{section:Positivitysemi-algebraic}

From now on, we assume that $\Gamma\subset\Isom_+(\Hb)$ is a finitely generated Fuchsian group. Let $\Fb$ be any real closed field and let $G$ be any linear semisimple semi-algebraic group (over $\overline{\mathbb Q}^r$).
Recall that we denote by $G_{\Fb}$ the $\Fb$-extension of $G$. In this section, we introduce the semi-algebraic structure of $\Hom(\Gamma, G_{\Fb})$ and discuss the semi-algebraicity of the locus of $\Theta$-positive and $\Theta$-positively frameable representations.

For any choice of finite generating set $\mathcal S$ of $\Gamma$, the map $i_{\mathcal S}\colon \Hom(\Gamma,G_{\Fb}) \to G_{\Fb}^{\mathcal S}$ sending a representation $\rho$ to $(\rho(s))_{s\in \mathcal S}$ identifies $\Hom(\Gamma, G_{\Fb})$ with a semi-algebraic subset of $G_{\Fb}^{\mathcal S}$. Furthermore, if $\mathcal S'$ is another generating set, the map $\iota_{\mathcal S'} \circ \iota_{\mathcal S}^{-1}$ is a semi-algebraic homeomorphism. This equips $\Hom(\Gamma, G_{\Fb})$ with a topology and semi-algebraic structure which is independent of the finite generating set, and for which the following holds:
\begin{enumerate}
\item For every $\gamma\in\Gamma$, the evaluation map 
\[e_\gamma \colon \Hom(\Gamma,G_{\Fb})\to G_{\Fb}\]
given by $e_\gamma\colon\rho\mapsto\rho(\gamma)$ is semi-algebraic and continuous.
\item For every finitely generated subgroup $\Gamma'\subset\Gamma$, the restriction map
\[
\begin{array}{cccc}
r_{\Gamma'}\colon &\Hom(\Gamma,G_{\Fb})& \to &\Hom(\Gamma',G_{\Fb})\\
&\rho & \mapsto & \rho\vert_{ \Gamma'}
\end{array}
\]
is semi-algebraic and continuous.
\end{enumerate}

Since the group law is defined over $\Qbar$, we have moreover:
\[\Hom(\Gamma,G_{\Fb}) = \Hom(\Gamma,G)_{\Fb}~.\]

Let is now denote by $\Pos_\Theta(\Gamma,G_{\Fb})$ (resp. $\Pos_\Theta^{\fr}(\Gamma,G_{\Fb})$) the subset of $\Theta$-positive representations (resp.\ $\Theta$-positively frameable representations). The goal of this section is to show that, when $\Gamma$ is a lattice, 
\begin{itemize}
    \item $\Pos_\Theta(\Gamma,G_{\Fb})$ is semi-algebraic and defined over $\Qbar$ when $\Gamma$ is uniform, i.e.\ $\Gamma \backslash \Hb$ is compact (see \Cref{corol : positive semi algebraic}),
    \item $\Pos_\Theta^{\fr}(\Gamma,G_{\Fb})$ is semi-algebraic and defined over $\Qbar$ when $\Gamma$ is non-cocompact (see \Cref{corollary : positively framed repr semi-algebraic}).
\end{itemize}

\subsection{Non-cocompact case}
\label{subs: Semialg Framed Repr}

For now, we assume that $\Gamma\subset\Isom_+(\Hb)$ is a non-cocompact lattice, so that the set $\Lambda_p=\Lambda_p(\Gamma)\subset\Sb^1$ of fixed points of parabolic elements of $\Gamma$ is non-empty. Recall that a \emph{framed representation} is a pair $(\rho,\xi)$ where $\rho\colon \Gamma \to G_{\Fb}$ is a representation and $\xi\colon \Lambda_p(\Gamma) \to \Fc_{\Theta,\Fb}$ is a $\Theta$-framing. Let us denote by $\widetilde{\Hom}_\Theta^{\fr}(\Gamma,G_{\Fb})$ the set of $\Theta$-framed representations from $\Gamma$ to $G_{\Fb}$.

Let $m$ denote the number of $\Gamma$-orbits in $\Lambda_p$ (equivalently, the number of cusps in $\Gamma \backslash \Hb$). Let us choose points $p_1,\dots,p_m \in\Lambda_p$ representing each $\Gamma$-orbit and write ${\bf p}\coloneqq (p_1,\dots,p_m)$. Let also $\gamma_i$ denote the positive parabolic element of $\Gamma$ generating the cyclic subgroup $\Stab_\Gamma(p_i)$.

The map
\[\begin{array}{cccc}
\Phi_{\bf p} \colon & \widetilde{\Hom}_\Theta^{\fr}(\Gamma,G_{\Fb}) & \to & \Hom(\Gamma,G_{\Fb}) \times \Fc_{\Theta,\Fb}^m\\
&(\rho, \xi) & \mapsto & (\rho, \xi(p_1),\ldots, \xi(p_m))
\end{array}
\]
identifies $\widetilde{\Hom}_\Theta^{\fr}(\Gamma,G_{\Fb})$ with the closed semi-algebraic set
\[\{(\rho, f_1, \ldots, f_m) \in \Hom(\Gamma,G_{\Fb}) \times \Fc_{\Theta,\Fb}^m \mid \rho(\gamma_i)\cdot f_i = f_i, 1\leq i \leq m\}~.\]
(The bijectivity of $\Phi_{\bf p}$ is an immediate consequence of the $\rho$-equivariance of $\xi$, see \Cref{s: Framed Positivity}).
If $\bf q$ is another family of representatives of the $\Gamma$-orbits in $\Lambda_p$, and $\eta_i\in\Gamma$ is such that $\eta_i\cdot p_i = q_i$ for all $1\leq i \leq m$, then $\Phi_{\bf q} \circ \Phi_{\bf p}^{-1}$ is the restriction of the semi-algebraic homeomorphism of $\Hom(\Gamma,G_{\Fb}) \times \Fc_{\Theta,\Fb}^m$ given by
\[(\rho, f_1,\ldots, f_m) \mapsto (\rho, \rho(\eta_1)\cdot f_1, \ldots, \rho(\eta_m)\cdot f_m)~.\]
Hence the semi-algebraic structure of $\widetilde{\Hom}_\Theta^{\fr}(\Gamma,G_{\Fb})$ does not depend on the choice of $\bf p$. Moreover:
\begin{itemize}
    \item The map \[\pi_{\Fb}\colon\widetilde{\Hom}_\Theta^{\fr}(\Gamma,G_{\Fb})\to\Hom_\Theta^{\rm fr}(\Gamma,G_{\Fb})\]
    given by $(\rho,\xi) \mapsto \rho$ is the restriction of the projection $\Hom(\Gamma,G_{\Fb}) \times \Fc_{\Theta,\Fb}^m\to \Hom(\Gamma,G_{\Fb})$, and hence is semi-algebraic.
    \item For each $p\in \Lambda_p$, the evaluation map
    $(\rho,\xi)\mapsto \xi(p)$ is semi-algebraic.
\end{itemize}
Finally, all the above maps can be defined over $\Qbar$. In particular, we get
\[\widetilde{\Hom}_\Theta^{\fr}(\Gamma, G_{\Fb}) = \widetilde{\Hom}_\Theta^{\fr}(\Gamma,G)_{\Fb}~.\]

Projecting to the first factor, we obtain the following:
\begin{proposition}\label{framed in representation variety}
    The set $\Hom_\Theta^{\fr}(\Gamma,G_{\Fb})$ is a closed semi-algebraic subset of\break $\Hom(\Gamma,G_{\Fb})$, and we have
    \[\Hom_\Theta^{\fr}(\Gamma,G_{\Fb}) = \Hom_\Theta^{\fr}(\Gamma,G)_{\Fb}~.\]
\end{proposition}

To prove \Cref{framed in representation variety}, we use the following useful fact.

\begin{lemma}\label{lemma: forgetting framing is closed}
The projection $\Hom(\Gamma,G_{\Fb}) \times \Fc_{\Theta,\Fb}^m\to \Hom(\Gamma,G_{\Fb})$ is closed. In particular, the map $\pi_{\Fb}$ is closed.
\end{lemma}

\begin{proof}
Since $\pi_{\Fb}$ is the restriction of the projection $\Hom(\Gamma,G_{\Fb}) \times \Fc_{\Theta,\Fb}^m\to \Hom(\Gamma,G_{\Fb})$ to the closed subset $\widetilde{\Hom}_\Theta^{\fr}(\Gamma,G_{\Fb})\subset \Hom(\Gamma,G_{\Fb}) \times \Fc_{\Theta,\Fb}^m$, the second claim of the lemma follows from the first.

To prove the first claim, we observe that, by the Tarski--Seidenberg transfer principle, it suffices to prove it in the case $\Fb=\Rb$ (the extension of a closed semi-algebraic map is closed, with a proof similarly to \Cref{thm_ExtSemiAlgMaps}).
In that case, the fibers of the projection are compact. It follows from a straightforward argument using the tube lemma that the projection is a closed map. 
\end{proof}

\begin{proof}[Proof of \Cref{framed in representation variety}]
Since $\pi_{\Fb}$ is semi-algebraic and defined over $\Qbar$, it follows that $\Hom_\Theta^{\fr}(\Gamma,G_{\Fb})=\pi_{\Fb}(\widetilde{\Hom}_\Theta^{\fr}(\Gamma,G_{\Fb}))$ is semi-algebraic, and so by Tarski--Seidenberg's theorem, it is the $\Fb$-extension of $\Hom_\Theta^{\fr}(\Gamma,G)$. Also, $\Hom_\Theta^{\fr}(\Gamma,G_{\Fb})$ is the image of the closed subset $\widetilde{\Hom}_\Theta^{\fr}(\Gamma,G_{\Fb})\subset \Hom(\Gamma,G_{\Fb}) \times \Fc_{\Theta,\Fb}^m$ under the projection $\Hom(\Gamma,G_{\Fb}) \times \Fc_{\Theta,\Fb}^m\to \Hom(\Gamma,G_{\Fb})$, so by \Cref{lemma: forgetting framing is closed}, it is closed in $\Hom(\Gamma,G_{\Fb})$. 
\end{proof}

Let $\widetilde{\Pos}_\Theta^{\fr}(\Gamma,G_{\Fb})$ (resp.\ $\Pos_\Theta^{\fr}(\Gamma,G_{\Fb})$) denote the set of $\Theta$-positively framed (resp.\ frameable) representations from $\Gamma$ to $G_{\Fb}$. Our goal is to show that these sets are semi-algebraic.
\\

Assume first that $\Gamma$ is torsion free. We may then choose a $\Gamma$-invariant ideal triangulation $\mathcal T$ of $\Hb$, i.e.\ a $\Gamma$-invariant, maximal collection of pairwise non-intersecting geodesics in $\Hb$ whose endpoints are both in $\Lambda_p$.
Also, we may choose $\Gamma$-invariant orientations on the geodesics in $\mathcal T$. Then for each $\ell\in\mathcal T$,
denote by $x_\ell$ and $y_\ell$ the backward and forward endpoints of $\ell$ respectively, and let $z_\ell$ and $w_\ell$ be the third vertex of the ideal triangle that lies to the left and right of $\ell$ respectively. Notice that $\mathcal T$ has finitely many $\Gamma$-orbits.
Choose a representative in each $\Gamma$-orbit in $\mathcal T$, and let $\mathcal T'\subset\mathcal T$ denote the set of all such choices. 

With this fixed topological data, we then have the following.

\begin{proposition}
\label{propo: ThetaPosSemiAlg}
Let $\Gamma\subset \Isom_+(\Hb)$ be a torsion free non-cocompact lattice. A $\Theta$-framed representation 
\[(\rho \colon \Gamma\to G_{\Fb},\ \xi \colon \Lambda_p(\Gamma)\to\Fc_{\Theta,\Fb})\]
is $\Theta$-positively framed if and only if, for every $\ell\in\mathcal T'$, the quadruple of flags
\[(\xi(x_\ell),\xi(z_\ell),\xi(y_\ell),\xi(w_\ell))\]
in $\Fc_{\Theta,\Fb}^4$ is positive. 
\end{proposition}

\Cref{propo: ThetaPosSemiAlg} is \cite[Proposition 6.4]{GRW} in the case $\Fb=\Rb$, but we will prove it here for completeness.
Before doing so, we discuss a couple of its consequences. 

\begin{corollary}
\label{corollary : positively framed repr semi-algebraic}
    Let $\Gamma \subset \Isom_+(\Hb)$ be a non-cocompact lattice. Then the subset 
    \[\widetilde{\Pos}_\Theta^{\fr}(\Gamma,G_{\Fb})\subset \widetilde{\Hom}_\Theta^{\fr}(\Gamma,G_{\Fb})\]
    is an open semi-algebraic subset. In particular, 
    \[\Pos_\Theta^{\fr}(\Gamma,G_{\Fb})\subset \Hom(\Gamma,G_{\Fb})\]
    is a semi-algebraic subset. Furthermore, we have
    \[\left(\widetilde{\Pos}_\Theta^{\fr}(\Gamma,G)\right)_{\Fb}=\widetilde{\Pos}_\Theta^{\fr}(\Gamma,G_{\Fb})\quad\text{and}\quad\left(\Pos_\Theta^{\fr}(\Gamma,G)\right)_{\Fb}=\Pos_\Theta^{\fr}(\Gamma,G_{\Fb})~.\]
\end{corollary}

\begin{proof}[Proof of \Cref{corollary : positively framed repr semi-algebraic} assuming \Cref{propo: ThetaPosSemiAlg}]
By Selberg's lemma, $\Gamma$ contains a finite index, torsion free, normal subgroup $\Gamma'$. The restriction map
\[r_{\Fb}\colon\Hom(\Gamma,G_{\Fb})\to\Hom(\Gamma',G_{\Fb})\]
is continuous, semi-algebraic, and is the $\Fb$-extension of the restriction map $r\coloneqq r_{\overline{\Qb}^r}$.
Since $\Lambda_p(\Gamma)=\Lambda_p(\Gamma')$, it follows that the map
\[\widetilde{r}_{\Fb}\colon \widetilde{\Hom}_\Theta^{\fr}(\Gamma,G_{\Fb})\to\widetilde{\Hom}_\Theta^{\fr}(\Gamma',G_{\Fb})\]
given by $\widetilde{r}_{\Fb} \colon (\rho,\xi)\mapsto(\rho|_{\Gamma'},\xi)$is continuous, semi-algebraic, and is the $\Fb$-extension of the restriction map $\widetilde r\coloneqq \widetilde r_{\overline{\Qb}^r}$. By \Cref{cor: positivity finite-index subgroup}, we have
\[\widetilde{\Pos}_\Theta^{\fr}(\Gamma,G_{\Fb})=(\widetilde{r}_{\Fb})^{-1}(\widetilde{\Pos}_\Theta^{\fr}(\Gamma',G_{\Fb}))~.\]
Hence, in order to prove the first claim of the lemma, we may assume that $\Gamma$ is torsion free and put ourselves in the set-up of \Cref{propo: positive n-tuples is semi-algebraic}.

Consider the map
\[f\colon\widetilde{\Hom}_\Theta^{\fr}(\Gamma,G_{\Fb})\to(\Fc_{\Theta,\Fb}^4)^{\Tc'}\]
given by $(\rho,\xi)\mapsto \big((\xi(x_\ell),\xi(z_\ell),\xi(y_\ell),\xi(w_\ell))\big)_{\ell\in\Tc'}$. By \Cref{propo: positive n-tuples is semi-algebraic}, the set $(\Fc_{\Theta,\Fb}^4)_{>0}$ of positive quadruples of flags is an open, semi-algebraic subset of $\Fc_{\Theta,\Fb}^4$, so $((\Fc_{\Theta,\Fb}^4)_{>0})^{\Tc'}\subset(\Fc_{\Theta,\Fb}^4)^{\Tc'}$ is open and semi-algebraic. Since $\Gamma$ is torsion free, \Cref{propo: ThetaPosSemiAlg} can be restated as
\[\widetilde{\Pos}_\Theta^{\fr}(\Gamma,G_{\Fb})=f^{-1}\left(((\Fc_{\Theta,\Fb}^4)_{>0})^{\Tc'}\right),\]
and the first claim follows.

Observe that $\Pos_\Theta^{\fr}(\Gamma,G_{\Fb}) = \pi_{\Fb}(\widetilde{\Pos}_\Theta^{\fr}(\Gamma,G_{\Fb}))$. Thus, by the first claim and the Tarski--Seidenberg theorem, $\Pos_\Theta^{\fr}(\Gamma,G_{\Fb})\subset \Hom_\Theta^{\rm fr}(\Gamma,G_{\Fb})$ is semi-algebraic. Apply \Cref{framed in representation variety} to deduce the second claim.

The third claim follows from the observation that the polynomial equalities and inequalities involved in defining $\widetilde{\Pos}_\Theta^{\fr}(\Gamma,G_{\Fb})$ have coefficients in $\overline{\Qb}^r$.
\end{proof}

Combining \Cref{corollary : positively framed repr semi-algebraic} with \Cref{thm: closedness}, we deduce the second part of \Cref{thm-intro:PG-condition} from the introduction.  

\begin{corollary}\label{cor:PG-PosFrameable-open-closed-in-frameable}
    If $\Gamma\subset\Isom_+(\Hb)$ is a non-cocompact lattice and $\Fc_{\Theta,\Fb}$ satisfies the $\PG$-condition, then \[\Pos_\Theta^{\fr}(\Gamma,G_{\Fb}) \subset \Hom_\Theta^{\fr}(\Gamma,G_{\Fb})\] is open and closed.
\end{corollary}

\begin{proof}
    By \Cref{cor:PG-positive-frameable-implies-positive-framings}, a $\Theta$-frameable representation is $\Theta$-positively frameable if and only if it is $\Theta$-positive. Hence \[
    \Pos_\Theta^{\fr}(\Gamma,G_{\Fb}) = \Pos_\Theta(\Gamma,G_{\Fb})\cap \Hom_\Theta^{\fr}(\Gamma,G_{\Fb})~.\]
    \Cref{thm: closedness} then implies that $\Pos_\Theta^{\fr}(\Gamma,G_{\Fb})\subset \Hom_\Theta^{\fr}(\Gamma,G_{\Fb})$ is closed.
    
    It remains to prove that $\Pos_\Theta^{\fr}(\Gamma,G_{\Fb})\subset \Hom_\Theta^{\fr}(\Gamma,G_{\Fb})$ is open. Again by \Cref{cor:PG-positive-frameable-implies-positive-framings}, a $\Theta$-frameable representation fails to be $\Theta$-positively frameable if and only if it admits a non-positive framing.
    Therefore, 
    \[\Hom_\Theta^{\fr}(\Gamma,G_{\Fb}) \setminus \Pos_\Theta^{\fr}(\Gamma,G_{\Fb}) = \pi_{\Fb}\left(\widetilde{\Hom}_\Theta^{\fr}(\Gamma,G_{\Fb}) \setminus \widetilde{\Pos}_\Theta^{\fr}(\Gamma,G_{\Fb})\right)~.\]
    By \Cref{corollary : positively framed repr semi-algebraic}, $\widetilde{\Hom}_\Theta^{\fr}(\Gamma,G_{\Fb}) \setminus \widetilde{\Pos}_\Theta^{\fr}(\Gamma,G_{\Fb})\subset\widetilde{\Hom}_\Theta^{\fr}(\Gamma,G_{\Fb})$ is closed, and by \Cref{lemma: forgetting framing is closed}, $\pi_{\Fb}$ is closed.
    Therefore, $\Hom_\Theta^{\fr}(\Gamma,G_{\Fb}) \setminus \Pos_\Theta^{\fr}(\Gamma,G_{\Fb})\subset\Hom_\Theta^{\fr}(\Gamma,G_{\Fb})$ is closed and $\Pos_\Theta^{\fr}(\Gamma, G_{\Fb})$ is open in $\Hom_\Theta^{\fr}(\Gamma,G_{\Fb})$.
\end{proof}

To prove \Cref{propo: ThetaPosSemiAlg}, we will use the following lemma.
Let $P_k$ be a subset of $\partial_\infty\Hb$ with cardinality $k\ge 4$, and let $\Rc$ be an ideal triangulation of the convex hull ${\rm Hull}(P_k)\subset\Hb$ of $P_k$, i.e.\ $\Rc$ is a maximal collection of pairwise non-intersecting geodesics in $\Hb$ whose endpoints are in $P_k$.
We say that an edge of $\Rc$ is an \emph{interior edge} if it lies in the interior of ${\rm Hull}(P_k)$, and a \emph{boundary edge} otherwise.
Then as before, choose an orientation on each interior edge $\ell$ of $\Rc$, let $x_\ell$ and $y_\ell$ be the backward and forward endpoints of $\ell$ respectively, and let $z_\ell$ and $w_\ell$ be the third vertex of the ideal triangle that lies to the left and right of $\ell$ respectively.

\begin{lemma}\label{finite lemma}
Let $G$ be a linear semisimple semi-algebraic group that admits a $\Theta$-positive structure.
A map $\xi \colon P_k\to\Fc_{\Theta,\Fb}$ is positive if and only if every interior edge $\ell$ of $\Rc$, the quadruple of flags
\[(\xi(x_\ell),\xi(z_\ell),\xi(y_\ell),\xi(w_\ell))\]
in $\Fc_\Theta^4$ is positive.
\end{lemma}
\begin{proof}
The forward direction follows from the positivity of $\xi$ and the observation that for every interior edge $\ell$ of $\Rc$, the quadruple $(x_\ell,z_\ell,y_\ell,w_\ell)$ is cyclically ordered.

We will prove the backward direction by induction on $k$. The base case when $k=4$ clearly holds. 

For the inductive step, notice that there is an ideal triangle $T$ of $\Rc$ such that two of its edges are boundary edges of $\Rc$ (and the third is necessarily an interior edge of $\Rc$). Enumerate the set $P_k=\{p_1,\dots,p_k\}$ according to the cyclic order on $P_k$, so that $p_{k-1},p_k,p_1$ are the vertices of $T$. Let $\Rc'$ be $\Rc$ with the two boundary edges of $\Rc$ which are edges of $T$ that removed, and observe that $\Rc'$ is an ideal triangulation of $P_{k-1}\coloneqq\{p_1,\dots,p_{k-1}\}$. Let $\ell_0$ be the interior edge of $\Rc$ that is an edge of $T$. We have 
\[(x_{\ell_0},z_{\ell_0}, y_{\ell_0}, w_{\ell_0})= (p_{k-1}, p_k, p_1, p_m)\]
for some $m\in \{2,\ldots, k-2\}$. 

Observe that the set of interior edges of $\Rc$ is the set of interior edges of $\Rc'$ together with $\ell_0$. By the inductive hypothesis, the tuple $(\xi(p_1),\dots,\xi(p_{k-1}))$ is positive. Moreover, the quadruple \[(\xi(x_{\ell_0}),\xi(z_{\ell_0}),\xi(y_{\ell_0}),\xi(w_{\ell_0})) = (\xi(p_{k-1}), \xi(p_k), \xi(p_1), \xi(p_m))\]
is positive. By \Cref{lem: AddingElemToPosTuple} we conclude that $(\xi(p_1),\dots,\xi(p_k))$ is also positive.
\end{proof}

With this lemma, we may now prove \Cref{propo: ThetaPosSemiAlg}.
 
\begin{proof}[Proof of \Cref{propo: ThetaPosSemiAlg}]
Again, the forward direction follows from the positivity of $\xi$ and the observation that for all $\ell\in\mathcal T$, the quadruple $(x_\ell,z_\ell,y_\ell,w_\ell)$ is cyclically ordered.

For the backward direction, we need to show that for any cyclically ordered tuple $(p_1,\dots,p_k)$ in $\Lambda_p$, the tuple of flags $(\xi(p_1),\dots,\xi(p_k))$ is positive.
Let $P_k \coloneqq \{p_1,\dots,p_k\}$.
By adding more points to the tuple, we may assume that the restriction of $\Tc$ to ${\rm Hull}(P_k)$ is an ideal triangulation of ${\rm Hull}(P_k)$. 

Observe from the $\rho$-equivariance of $\xi$ that requiring $(\xi(x_\ell),\xi(z_\ell),\xi(y_\ell),\xi(w_\ell))$ to be positive for all $\ell\in\Tc'$ is equivalent to requiring that the same holds for all $\ell\in\Tc$, and in particular for all $\ell$ that is an interior edge of $\Tc|_{P_k}$.
Thus, we may now apply \Cref{finite lemma} to deduce that $(\xi(p_1),\dots,\xi(p_k))$ is positive.
\end{proof}

Note that, contrary to $\Pos_\Theta^{\fr}(\Gamma,G_{\Fb})$, the set $\Pos_\Theta(\Gamma,G_{\Fb})$ is not semi-algebraically defined over $\Rb$:

\begin{proposition}
    Let $\Gamma_{0,3}$ be the Fuchsian group introduced in \Cref{subsection: ExNonFramablePosRepr}, and $\Fb$ a non-Archimedean real closed field containing $\mathbb \Rb$.
    Then there does not exist a semi-algebraic set $X\subset \Hom(\Gamma_{0,3},G_{\Rb})$ such that
    \[X_{\Fb} = \Pos_\Theta(\Gamma_{0,3}, G_{\Fb})~.\]
\end{proposition}
\begin{proof}
    Assume by contradiction that such an $X$ exists.
    Then we should have
    \[X= \Pos_\Theta(\Gamma_{0,3}, G_{\Fb})\cap \Hom(\Gamma_{0,3}, G_{\Rb})= \Pos_\Theta(\Gamma_{0,3}, G_{\Rb})~.\]
    Since every $\Theta$-positive representation over $\Rb$ is $\Theta$-positively frameable (\Cref{coro: positive => frameable over R}), we would thus have $X= \Pos_\Theta^{\fr}(\Gamma_{0,3},G_{\Rb})$, and hence
    \[X_{\Fb} = \Pos_\Theta^{\fr}(\Gamma_{0,3},G_{\Fb})\]
    by \Cref{corollary : positively framed repr semi-algebraic}.
    We would thus conclude that every $\Theta$-positive representation of $\Gamma_{0,3}$ into $G_{\Fb}$ is $\Theta$-positively frameable.
    However, \Cref{propo:PosReprNotFrame} shows precisely that this is not true.
\end{proof}

\begin{remark}
    It is not clear whether the set $\Pos_\Theta(\Gamma,G_{\Fb})$ could be defined by finitely many semi-algebraic conditions with coefficients in the field $\Fb$.
\end{remark}

Let us note that $\Pos_\Theta^{\fr}(\Gamma,G_{\Fb})$ is not open in $\Hom(\Gamma,G_{\Fb})$, simply because $\Theta$-positively translating elements do not form an open set. 
However, we will prove that it is open when restricting to ``relative representation varieties'', where the conjugacy classes of the images of the parabolic elements are fixed.\\

For all $\rho_0\in\Pos_\Theta^{\fr}(\Gamma,G_{\Fb})$, let 
\[\Hom_{\rho_0}(\Gamma,G_{\Fb}) \coloneqq \left\{\rho\in \Hom(\Gamma,G_{\Fb}) \;\middle\vert\, \begin{array}{l}\rho(\gamma) \textrm{ is conjugate to }\rho_0(\gamma)\\
\text{for all peripheral }\gamma\in\Gamma
\end{array}\right\}~.\]
Observe that $\Hom_{\rho_0}(\Gamma,G_{\Fb})\subset \Hom(\Gamma,G_{\Fb})$ is a semi-algebraic set. Indeed, let $\eta_1,\dots,\eta_n\in \Gamma$ be elements such that every parabolic subgroup of $\Gamma$ is conjugated to the group generated by $\eta_i$ for some $i\in\{1,\dots,n\}$. For each $i\in\{1,\dots,n\}$, note that the conjugacy class of $\rho(\eta_i)$, denoted $\mathcal C(\rho_0(\eta_i))\subset G_{\Fb}$, is a semi-algebraic subset. As such,
\[\Hom_{\rho_0}(\Gamma,G_{\Fb})=\bigcap_{i=1}^ne_{\eta_i}^{-1}(\mathcal C(\rho_0(\eta_i)))\subset \Hom(\Gamma,G_{\Fb})\]
is semi-algebraic, and if $\rho_0$ has image in $G$, then $\Hom_{\rho_0}(\Gamma,G_{\Fb})=\Hom_{\rho_0}(\Gamma,G)_{\Fb}$.

\begin{proposition}
\label{propo : open in relative character variety}
    Suppose that $\Gamma\subset\Isom_+(\Hb)$ is a non-cocompact lattice. For any $\rho_0\in\Pos_\Theta^{\fr}(\Gamma,G_{\Fb})$, the semi-algebraic set
    \[\Pos_\Theta^{\fr}(\Gamma, G_{\Fb}) \cap \Hom_{\rho_0}(\Gamma,G_{\Fb}) \]
    is open in $\Hom_{\rho_0}(\Gamma,G_{\Fb})$.
\end{proposition}

\begin{proof}
For any $g\in G_{\Fb}$, let $\mathcal C(g)\subset G_{\Fb}$ denote the set of conjugates of $g$.
It is semi-algebraic (as the image of $G_{\Fb}$ by the semi-algebraic map $h\mapsto h g h^{-1}$).
We claim that the semi-algebraic map $h \mapsto h^+$ is continuous in restriction to $\mathcal C(g)$.
Indeed, the definition of $h \mapsto h^+$ implies that $(h g h^{-1})^+ = h \cdot g^+$.
Now, the map $h \mapsto h \cdot g^+$ is continuous and invariant by the centralizer $Z(g)$ of $g$ in $G_{\Fb}$.
Hence it factors to a continuous map on $\mathcal C(g)\simeq G_{\Fb}/Z(g)$.

By \Cref{framed positive implies positively translating}, $\rho_0(\gamma)$ is $\Theta$-positively translating for all $\gamma\in\Gamma$ of infinite order.
Thus, every $\rho\in \Hom_{\rho_0}(\Gamma,G_{\Fb})$ admits a well-defined framing 
\[\xi_\rho^l \colon \Lambda_p\to \Fc_{\Theta,\Fb}\]
given by $\eta^+\mapsto\rho(\eta)^+$ for each positive parabolic element $\eta\in \Gamma$. This allows us to define the map 
\[\begin{array}{cccc}
\sigma \colon &\Hom_{\rho_0}(\Gamma, G_{\Fb}) & \to & \widetilde{\Hom}_\Theta^{\fr}(\Gamma, G_{\Fb})\\
& \rho & \mapsto & (\rho, \xi_\rho^l)
\end{array}\]
which is continuous by the continuity discussed in the paragraph above. Then
\[\Pos_\Theta^{\fr}(\Gamma,G_{\Fb}) \cap \Hom_{\rho_0}(\Gamma, G_{\Fb}) = \sigma^{-1} \left(\widetilde{\Pos}_\Theta^{\fr}(\Gamma, G_{\Fb})\right)\]
is open, since $ \widetilde{\Pos}_\Theta^{\fr}(\Gamma, G_{\Fb})$ is open in $\widetilde{\Hom}_\Theta^{\fr}(\Gamma, G_{\Fb})$, see \Cref{corollary : positively framed repr semi-algebraic}.
\end{proof}

\subsection{Cocompact case}
Assume now that $\Gamma\subset \Isom_+(\Hb)$ is a cocompact lattice.
The goal of this section is to prove that $\Pos_\Theta(\Gamma,G_{\Fb})$ is an open semi-algebraic subset of $\Hom(\Gamma,G_{\Fb})$ defined over $\Qbar$.
Once we know that the subset is semi-algebraic, the openness follows from \cite[Corollary B]{GLW}, which states that $\Pos_\Theta(\Gamma,G_{\Rb})$ is open in $\Hom(\Gamma,G_{\Rb})$, together with the Tarski--Seidenberg principle, see \Cref{thm_ExtConnComp}~(\ref{thm_ExtConnComp: closedness}).

In order to prove the desired result, we will reduce to the non-cocompact case by a ``gluing lemma'' of independent interest. Let $S$ denote the (closed, oriented) hyperbolic orbifold $\Gamma \backslash \Hb$, choose a simple oriented multi-curve $\{c_1,\dots,c_k\}$ in the hyperbolic orbifold $S$, i.e.\ a pairwise non-intersecting collection of simple, oriented closed geodesics in $S$ that avoid the orbifold points of $S$, and let $S_1,\dots,S_\ell$ denote the connected components of $S\setminus\bigcup_{i=1}^k c_i$.
(Here, we abuse notation and denote the image of $c_i$ also by $c_i$.)

Informally, \Cref{propo: cutting lemma} below states that a representation $\rho\colon  \Gamma = \pi_1(S) \to G_{\Fb}$ is $\Theta$-positive if and only if the restriction of $\rho$ to each $\pi_1(S_i)$ is $\Theta$-positively frameable, and the boundary maps associated to two adjacent connected components have image into opposite diamonds. To state this more precisely, we need to fix some topological data:

Let $\pi\colon \Hb\to S\coloneqq \Gamma \backslash \Hb$ denote the universal cover of $S$. For each $j\in\{1,\dots,\ell\}$, choose a connected component $H_j$ of $\Hb\setminus\bigcup_{i=1}^k \pi^{-1}(c_i)$ such that $\pi(H_j) = S_j$, and let $\Gamma_j\subset\Gamma$ denote the subgroup that leaves $H_j$-invariant, so that $\Gamma_j$ is the deck group of the universal cover $\pi|_{H_j} \colon H_j\to S_j$.
Choose also a non-peripheral group element $\omega_j\in\Gamma_j$.
Notice that $\Gamma_j \subset \Isom_+(\Hb)$ is convex cocompact, but not a lattice. However, by \Cref{Fuchsian groups lemma}, we can fix a non-cocompact lattice $\Gamma_j'\subset \Isom_+(\Hb)$ and a semi-conjugacy $\iota_j\colon\Gamma_j\to\Gamma_j'$.

For each $i\in\{1,\dots,k\}$, let $a(i),b(i)\in\{1,\dots,\ell\}$ be such that $S_{a(i)}$ and $S_{b(i)}$ are the two connected components of $S\setminus\bigcup_{i=1}^kc_i$ that lie to the left and right of $c_i$ respectively (it is possible that $a(i)=b(i)$). Let $\gamma_i\in\Gamma$ be such that $H_{a(i)}$ and $\gamma_i\cdot H_{b(i)}$ share a common boundary component. This boundary component $g_i$ is an oriented geodesic such that $\pi(g_i)= c_i$. Let  $\eta_i\in\Gamma$ denote the generator of the cyclic group $\Stab_\Gamma(g_i)$ that translates positively along the geodesic $g_i$, so that its attracting and repelling fixed points $\eta_i^+$ and $\eta_i^-$ are respectively the forward and backward endpoint of $g_i$. Define
\[\hat\eta_i\coloneqq \omega_{a(i)}^{-1}\eta_i\omega_{a(i)},\quad\nu_i\coloneqq \gamma_i^{-1}\eta_i\gamma_i,\quad\text{and}\quad \hat\nu_i\coloneqq \omega_{b(i)}^{-1}\nu_i\omega_{b(i)}.\]
Observe that $\eta_i,\hat\eta_i\in\Gamma_{a(i)}$ and $\nu_i,\hat\nu_i\in\Gamma_{b(i)}$. These auxiliary elements provide us points $\hat{\eta}_i^+ \in \Lambda(\Gamma_{a(i)})\setminus \{\eta_i^+,\eta_i^-\}$ and $\hat \nu_i^+ \in \Lambda(\Gamma_{b(i)})\setminus \{\nu_i^+, \nu_i^-\}$. Moreover, by construction, the quadruple \[(\eta_i^+,\hat\eta_i^+,\eta_i^-,\gamma_i\cdot\hat\nu_i^+)=\gamma_i\cdot(\nu_i^+,\gamma_i^{-1}\cdot\hat\eta_i^+,\nu_i^-,\hat\nu_i^+)\]
is cyclically ordered. 

Given these topological choices, we may now state the following characterization of $\Theta$-positive representations from $\Gamma$ to $G_{\Fb}$. Recall that if $g\in G_{\Fb}$ is $\Theta$-positively translating, then it admits a forward and backward fixed point, which we denote by $g^+$ and $g^-$ respectively, see \Cref{subsec:PosTranslElem}.

\begin{proposition}
\label{propo: cutting lemma}
Let $\rho \colon \Gamma \to G_{\Fb}$ be a representation. Then $\rho$ is $\Theta$-positive if and only if all of the following conditions hold:
\begin{enumerate}
\item for each $j\in\{1,\dots,\ell\}$, the representation
\[\rho_j \coloneqq \rho|_{\Gamma_j}\circ\iota_j^{-1} \colon \Gamma_j'\to G_{\Fb}\] 
is $\Theta$-positively frameable,

\item for each $i\in \{1, \ldots, k\}$, the element $\rho(\eta_i)$ is weakly $\Theta$-proximal, and

\item for each $i\in\{1,\dots,k\}$, the quadruple of flags
\[\left(\rho(\eta_i)^+,\rho(\hat\eta_i)^+,\rho(\eta_i)^-,\rho(\gamma_i)\cdot \rho(\hat\nu_i)^+\right)\] 
is positive.
\end{enumerate}

\end{proposition}

\begin{proof}
    We first prove the forward implication. Since $\Gamma$ is a cocompact lattice, every infinite order element in $\Gamma$ is hyperbolic. As such, since $\rho$ is $\Theta$-positive, \Cref{theorem: limit map extends to proximal limit map} implies that for every infinite order element $\gamma\in\Gamma$, $\rho(\gamma)$ is weakly $\Theta$-proximal. In particular, $\rho(\eta_i)$ is weakly $\Theta$-proximal for all $i\in \{1,\ldots, k\}$.
    
    Recall that $\Lambda_h(\Gamma)\subset \Sb^1$ denotes the set of fixed points of hyperbolic elements in $\Gamma$. \Cref{theorem: limit map extends to proximal limit map} also implies that the map $\xi_\rho^h \colon \Lambda_h(\Gamma)\to\Fc_{\Theta,\Fb}$, which sends every $\gamma^+\in\Lambda_h(\Gamma)$ to $\rho(\gamma)^+$, is positive. By \Cref{Fuchsian groups lemma}, for each $j\in\{1,\dots,\ell\}$, there is a left-continuous, $\iota_j^{-1}$-equivariant, monotonic, injective map
    \[\beta_j\colon \Lambda(\Gamma_j')\to\Lambda(\Gamma_j).\]
    Since $\beta_j$ is $\iota_j^{-1}$-equivariant, it sends the fixed points in $\Lambda(\Gamma_j')$ of infinite order elements in $\Gamma_j'$ to the fixed points in $\Lambda(\Gamma_j)$ of infinite order elements in $\Gamma_j$. Hence, $\beta_j(\Lambda_p(\Gamma_j'))\subset\Lambda_h(\Gamma_j)\subset\Lambda_h(\Gamma)$, so the map 
    \[\xi^h_\rho\circ(\beta_j|_{\Lambda_p(\Gamma_j')}) \colon \Lambda_p(\Gamma_j')\to\Fc_{\Theta,\Fb}\]
    is a well-defined positive framing for $\rho_j$.

For each $i\in\{1,\dots,k\}$, 
    \[\rho(\eta_i)^\pm=\xi^h_\rho(\eta_i^\pm),\quad\rho(\hat\eta_i)^\pm=\xi^h_\rho(\hat\eta_i^\pm),\quad\text{and}\quad \rho(\hat\nu_i)^\pm=\xi^h_\rho(\hat\nu_i^\pm).\]
    We observed earlier that $(\eta_i^+,\hat\eta_i^+,\eta_i^-,\gamma_i\cdot\hat\nu_i^+)$ is cyclically ordered, so 
    \[\left(\rho(\eta_i)^+,\rho(\hat\eta_i)^+,\rho(\eta_i)^-,\rho(\gamma_i)\cdot \rho(\hat\nu_i)^+\right)=\left(\xi^h_\rho(\eta_i^+),\xi^h_\rho(\hat\eta_i^+),\xi^h_\rho(\eta_i^-),\xi^h_\rho(\gamma_i\cdot\hat\nu_i^+)\right)\] 
    is positive.   This finishes the proof of the forward implication.\\

    Let us now turn to the proof of the backward implication. Consider the $\Gamma$-invariant subset \[D\coloneqq \bigcup_{i=1}^k(\Gamma\cdot \eta_i^+\cup\Gamma\cdot\eta_i^-) \subset \Lambda_h(\Gamma)~.\] 
    Since $\rho(\eta_i)$ is $\Theta$-positively translating by condition (1) and weakly $\Theta$-proximal by condition (2), the map 
    \[\xi \colon D\to\Fc_{\Theta,\Fb}\]
    given by $\gamma^+\mapsto\rho(\gamma)^+$ (for every $\gamma$ conjugated to one of the $\eta_i$ or $\eta_i^{-1}$) is well-defined and $\rho$-equivariant, and it suffices to prove that it is positive. The strategy to do so has two steps; first, we use condition (1) in the statement of the proposition to show that for each $j\in\{1,\dots,\ell\}$, the map \[\xi|_{D\cap\Lambda(\Gamma_j)}\colon D\cap\Lambda(\Gamma_j)\to\Fc_{\Theta,\Fb}\] 
    is positive. Then, using condition (3), we prove using an induction argument that we can recover the positivity of $\xi$ from the positivity of its restrictions considered in the first step.\\

    \noindent {\bf First step.} Fix $j\in\{1,\dots,\ell\}$. Since $\rho_j$ is $\Theta$-positive, there is a maximal positive $\Theta$-boundary extension $\zeta_j\coloneqq \zeta^\textnormal{M}_{\rho_j} \colon \mathcal V_{\rho_j}^{\mathrm{M}}\to\Lambda(\Gamma_j)$ for $\rho_j$, see \Cref{prop: converse}. By \Cref{Fuchsian groups lemma}, there exists a continuous, $\iota_j$-equivariant, monotonic, surjective map 
    \[\alpha_j\colon \Lambda(\Gamma_j)\to  \Lambda(\Gamma_j')\]
    whose fibers contain at most two points.
    Observe that every point in $D\cap\Lambda(\Gamma_j)$ is the attracting fixed point $\gamma^+$ for some hyperbolic element $\gamma\in\Gamma_j$, and by \Cref{lem:PosFramedRepr}  we have
    \begin{align}\label{eqn: propo: cutting lemma}
\xi(\gamma^+)=\rho(\gamma)^+=\rho_j(\iota_j(\gamma))^+\in\zeta_j^{-1}(\iota_j(\gamma)^+)=\zeta_j^{-1}(\alpha_j(\gamma^+)),
\end{align}
where the last equality holds by the $\iota_j$-equivariance and monotonicity of $\alpha_j$.  

We need to prove that $\xi$ sends cyclically ordered tuples in $D\cap\Lambda(\Gamma_j)$ to positive tuples of flags in $\Fc_{\Theta,\Fb}$. Since positivity is preserved by taking sub-tuples, it suffices to prove that if $\mu_1,\dots,\mu_n\in\Gamma_j$ are hyperbolic elements such that $\iota_j(\mu_t)\in\Gamma_j'$ is a parabolic element for all $t\in\{1,\dots,n\}$ and
\[(\mu_1^+,\mu_1^-,\mu_2^+,\mu_2^-,\dots,\mu_n^+,\mu_n^-)\]
is cyclically ordered, then 
\[\left(\xi(\mu_1^+),\xi(\mu_1^-),\xi(\mu_2^+),\xi(\mu_2^-),\dots,\xi(\mu_n^+),\xi(\mu_n^-)\right)\] 
is positive.

Let $y\in D\cap\Lambda(\Gamma_j)$ be a point that is not fixed by $\mu_t$ for all $t\in\{1,\dots,n\}$. Then for sufficiently large $m$, 
\begin{align*}
(\mu_1^+,\mu_1^-,\mu_1^{-m}\cdot y,\mu_2^m\cdot y,\mu_2^+,\mu_2^-,\mu_2^{-m}\cdot y,\dots,\mu_n^m\cdot y,\mu_n^+,\mu_n^-,\mu_n^{-m}\cdot y,\mu_1^m\cdot y)    
\end{align*}
is a cyclically ordered tuple in $D\cap \Lambda(\Gamma_j)$. Observe that for each $t\in\{1,\dots,n\}$, $\alpha_j(\mu_t^+)=\alpha_j(\mu_t^-)\neq\alpha_j(\mu_t^s\cdot y)$ for any $s\in\Zb$, so
\begin{align*}
\big(\alpha_j(\mu_1^+),\alpha_j(\mu_1^{-m}\cdot y),\alpha_j(\mu_2^m\cdot y),&\,\alpha_j(\mu_2^+),\alpha_j(\mu_2^{-m}\cdot y),\dots\\
&\dots,\alpha_j(\mu_n^m\cdot y),\alpha_j(\mu_n^+),\alpha_j(\mu_n^{-m}\cdot y),\alpha_j(\mu_1^m\cdot y)\big)
\end{align*}
is a cyclically ordered tuple in $\Lambda_p(\Gamma_j')$, so by Equation~\eqref{eqn: propo: cutting lemma} and the fact that $\zeta_j$ is a positive $\Theta$-boundary extension for $\rho_j$, we see that
\begin{align}\label{eqn: positive1}
\begin{split}
\big(\xi(\mu_1^+),\xi(\mu_1^{-m}\cdot y),\xi(\mu_2^m\cdot y),&\,\xi(\mu_2^+),\xi(\mu_2^{-m}\cdot y),\dots\\
&\dots,\xi(\mu_n^m\cdot y),\xi(\mu_n^+),\xi(\mu_n^{-m}\cdot y),\xi(\mu_1^m\cdot y)\big)
\end{split}
\end{align}
is positive. 

On the other hand, notice that for each $t\in\{1,\dots,n\}$, $(y,\mu_t\cdot y,\mu_t^2\cdot y,\mu_t^+,\mu_t^-)$ is a cyclically ordered tuple in $D\cap \Lambda(\Gamma_j)$. Hence, $(\alpha_j(y),\alpha_j(\mu_t\cdot y),\alpha_j(\mu_t^2\cdot y),\alpha_j(\mu_t^+))$ is a cyclically ordered tuple in $\Lambda_p(\Gamma_j')$, so Equation~\eqref{eqn: propo: cutting lemma} again implies that
\[\big(\xi(y),\rho(\mu_t)\cdot\xi(y),\rho(\mu_t)^2\cdot\xi( y),\rho(\mu_t)^+\big)=\big(\xi(y),\xi(\mu_t\cdot y),\xi(\mu_t^2\cdot y),\xi(\mu_t^+)\big)\]
is positive, i.e.\ that $(\xi(y),\rho(\mu_t)^+)$ is a positive translating pair for $\rho(\mu_t)$. By assumption, $\rho(\mu_t)=\rho_j(\iota_j(\mu_t))$ is weakly $\Theta$-proximal, so we may apply \Cref{corol: CharWeakThetaProxForPosTrans} to deduce that 
\begin{align}\label{eqn: positive2}
\begin{split}
\big(\xi(\mu_t^+),\xi(\mu_t^-), \xi(\mu_t^{-n}\cdot y),&\,\xi(\mu_t^n\cdot y)\big)\\
&=\big(\rho(\mu_t)^+,\rho(\mu_t)^-,\rho(\mu_t)^{-n}\cdot \xi(y),\rho(\mu_t)^n\cdot\xi(y)\big)
\end{split}
\end{align}
    is positive. Together, the positivity of the tuples \eqref{eqn: positive1} and \eqref{eqn: positive2} allow us to iteratively apply \Cref{corollarly: gluing} to deduce that 
\begin{align*}
\big(\xi(\mu_1^+),\xi(\mu_1^-),\xi(\mu_1^{-m}\cdot y),&\,\xi(\mu_2^m\cdot y),\xi(\mu_2^+),\xi(\mu_2^-),\xi(\mu_2^{-m}\cdot y),\dots\\
&\dots,\xi(\mu_n^m\cdot y),\xi(\mu_n^+),\xi(\mu_n^-),\xi(\mu_n^{-m}\cdot y),\xi(\mu_1^m\cdot y)\big)    
\end{align*}
is positive. In particular, $\left(\xi(\mu_1^+),\xi(\mu_1^-),\xi(\mu_2^+),\xi(\mu_2^-),\dots,\xi(\mu_n^+),\xi(\mu_n^-)\right)$ is positive. This finishes the first step.\\

\noindent {\bf Second step.} Let $\mathcal M$ be a finite, pairwise distinct collection of $\Gamma$-translates of the closures in $\Hb$ of the subsets $H_1,\dots,H_\ell$, such that $\bigcup_{M\in\mathcal M}M$ is connected. Then let $D(\mathcal M)$ denote the set of points in $D$ that are limits of points in $\bigcup_{M\in\mathcal M}M$. It suffices to prove that $\xi|_{D(\mathcal M)}$ is positive. We will do so by induction on the cardinality of $\mathcal M$. 

The base case when $\mathcal M$ is a singleton follows immediately from the first step. For the inductive step, suppose that $\abs{\mathcal M}\ge 2$, and let $(x_1,\dots,x_n)$ be a cyclically ordered tuple in $D(\mathcal M)$. We will prove that $(\xi(x_1),\dots,\xi(x_n))$ is positive.

Let $g$ be an oriented geodesic that lies in the interior of $\bigcup_{M\in\mathcal M}M$ so that $g=\gamma\cdot g_i$ for some $\gamma\in\Gamma$ and some $i\in\{1,\dots,k\}$. Let $\mathcal M_L$ and $\mathcal M_R$ denote the subsets of $\mathcal M$ consisting of all the elements that lie to the left and right of $g$ respectively. Notice that $\gamma\cdot \overline{H_{a(i)}}\in\mathcal M_L$ and $\gamma\gamma_i\cdot\overline{H_{b(i)}}\in\mathcal M_R$, so $\gamma\cdot\hat\eta_i^+\in D(\mathcal M_L)$ and $\gamma\gamma_i\cdot\hat\nu_i^+\in D(\mathcal M_R)$. 

Since positivity is preserved by taking sub-tuples and cyclic permutations, we may assume that $n\ge 4$, and there is some $m\in\{3,\dots,n-1\}$ such that 
\begin{itemize}
    \item $x_1$ and $x_m$ are the backward and forward endpoints of $g$, i.e.\ $x_1=\gamma\cdot\eta_i^-$ and $x_m=\gamma\cdot\eta_i^+$, 
    \item there is some $c\in\{2,\dots,m-1\}$ such that 
    $x_c=
    \gamma\gamma_i\cdot\hat\nu_i^+$
    \item there is some $d\in\{m+1,\dots,n\}$ such that $x_d=
    \gamma\cdot\hat\eta_i^+$.
\end{itemize}
Necessarily, we have $x_1,\dots,x_m\in D(\mathcal M_R)$ and $x_m,\dots,x_n,x_1\in D(M_L)$.

Since $\abs{\mathcal M_R},\abs{\mathcal M_L}<\abs{\mathcal M}$, by the inductive hypothesis, $\xi|_{D(\mathcal M_L)}$ and $\xi|_{D(\mathcal M_R)}$ are both positive, so the tuples of flags
\[\big(\xi(x_1),\dots,\xi(x_m)\big)\quad\text{and}\quad\big(\xi(x_m),\dots,\xi(x_n),\xi(x_1)\big)\]
are positive. By condition (3), the quadruple of flags
\begin{align*}
    (\xi(x_1),\xi(x_c),\xi(x_m),\xi(x_d))&=(\xi(\gamma\cdot \eta_i^-),\xi(\gamma\gamma_i\cdot\hat\nu_i^+),\xi(\gamma\cdot \eta_i^+),\xi(\gamma\cdot\hat\eta_i^+))\\
&=\rho(\gamma)\cdot\left(\rho(\eta_i)^-,\rho(\gamma_i)\cdot \rho(\hat\nu_i)^+,\rho(\eta_i)^-,\rho(\hat\eta_i)^+\right)
\end{align*}
is positive. Thus, by \Cref{corollarly: gluing}, $\big(\xi(x_1),\dots,\xi(x_m)\big)$ is positive.
\end{proof}

As a consequence of \Cref{propo: cutting lemma}, we obtain the following corollary.

\begin{corollary}
\label{corol : positive semi algebraic}
    Suppose that $\Gamma\subset \Isom_+(\Hb)$ is a cocompact lattice. The set
    \[\Pos_\Theta(\Gamma, G_{\Fb})\subset \Hom(\Gamma,G_{\Fb})\]
    is semi-algebraic. Furthermore,
    \[\Pos_\Theta(\Gamma,G)_{\Fb} = \Pos_\Theta(\Gamma,G_{\Fb}).\]
\end{corollary}
\begin{proof}
By Selberg's lemma, $\Gamma$ contains a finite index, torsion free, normal subgroup $\Gamma'$. We previously observed that the restriction map
\[r_{\Fb}\colon\Hom(\Gamma,G_{\Fb})\to\Hom(\Gamma',G_{\Fb})\]
is continuous, semi-algebraic, and the $\Fb$-extension of $r\coloneqq r_{\overline{\Qb}^r}$.
Since
\[\Pos_\Theta(\Gamma,G_{\Fb})=r_{\Fb}^{-1}(\Pos_\Theta(\Gamma',G_{\Fb})),\]
to prove the first claim of the lemma, we may assume that $\Gamma$ is torsion free.

Since $\Gamma$ is torsion free, $S\coloneqq \Hb/\Gamma$ admits a simple oriented multicurve $\{c_1,\dots,c_k\}$, and so we can put ourselves in the set-up of \Cref{propo: cutting lemma}. Given a representation $\rho \colon \Gamma\to G_{\Fb}$ and $j\in\{1,\dots,\ell\}$, let 
\[\rho_j\coloneqq \rho|_{\Gamma_j}\circ\iota_j^{-1} \colon \Gamma_j'\to G_{\Fb}.\] 
Recall that the map 
\[\Phi_j \colon \Hom(\Gamma,G_{\Fb})\to\Hom(\Gamma_j',G_{\Fb}), \quad \rho\mapsto \rho_j\]
is semi-algebraic.

By \Cref{corollary : positively framed repr semi-algebraic}, $\Pos_\Theta^{\fr}(\Gamma_j',G_{\Fb})\subset \Hom(\Gamma_j',G_{\Fb})$
    is a semi-algebraic subset for each $j\in\{1,\dots,\ell\}$. Also, by \Cref{prop: continuity g->g+ hyperbolic}, the subset $\Prox_\Theta(G_{\Fb})\subset G_{\Fb}$ of weakly $\Theta$-proximal elements in $G_{\Fb}$ is semi-algebraic. 
    Thus the subset $A$ of representations satisfying conditions (1) and (2) in \ref{propo: cutting lemma} can be described as
    \[A= \bigcap_{j=1}^\ell\Phi_j^{-1}(\Pos_\Theta^{\fr}(\Gamma_j',G_{\Fb}))\cap\bigcap_{i=1}^ke_{\eta_i}^{-1}(\Prox_\Theta(G_{\Fb}))\subset\Hom(\Gamma,G_{\Fb})~,\]
    which is a semi-algebraic subset of $\Hom(\Gamma,G_{\Fb})$. Here, recall that for all $\gamma\in\Gamma$,\break $e_\gamma\colon \Hom(\Gamma,G_{\Fb})\to G_{\Fb}$ denotes the evaluation at $\gamma$ defined by $e_\gamma \colon \rho\mapsto\rho(\gamma)$, which we already observed is semi-algebraic.

For each $i\in\{1,\dots,k\}$, \Cref{prop: continuity g->g+ hyperbolic} also implies that the map
\[\Xi_i \colon A\to\Fc_{\Theta,\Fb}^4\]
defined by $\Xi_i \colon \rho\mapsto\left(\rho(\eta_i)^+,\rho(\hat\eta_i)^+,\rho(\eta_i)^-,\rho(\gamma_i)\cdot \rho(\hat\nu_i)^+\right)$ is semi-algebraic. Also, \Cref{propo: positive n-tuples is semi-algebraic} implies that the subset $(\Fc_{\Theta,\Fb}^4)_{>0}\subset\Fc_{\Theta,\Fb}^4$ of positive quadruples of flags is semi-algebraic. Thus,
\[B \coloneqq \bigcap_{i=1}^k\Xi_i^{-1}((\Fc_{\Theta,\Fb}^4)_{>0})\]
is a semi-algebraic subset of $\Hom(\Gamma,G_{\Fb})$. On the other hand, \Cref{propo: cutting lemma} implies that $B=\Pos_\Theta(\Gamma, G_{\Fb})$, so we have proven the first claim of the corollary. \\

One can moreover verify that all the above semi-algebraic conditions are $\Fb$-extensions of conditions which are defined over $\Qbar$. It thus follows from the transfer principle (\Cref{thm_TarskiSeidenberg}) that the set that they define over $\Fb$ is the $\Fb$-extension of the one they define over $\Qbar$, i.e.\
\[\Pos_\Theta(\Gamma,G_{\Fb})= \Pos_\Theta(\Gamma,G)_{\Fb}~.\qedhere\]
\end{proof}

\section{%
\texorpdfstring{%
Real spectrum compactification of $\Theta$-positive representations}%
{Real spectrum compactification of Theta-positive representations}}
\label{section:RSC}

The goal of the section is to prove Corollaries~\ref{coro:Positive = limit of positive} and~\ref{coro: positive form connected components}.
For experts, it should be clear at this moment that these corollaries follow readily from the characterization of the set of $\Theta$-positively frameable and $\Theta$-positive representations in Corollaries~\ref{corollary : positively framed repr semi-algebraic} and~\ref{corol : positive semi algebraic}, and \Cref{thm: closedness}, and are merely a rephrasing of these results in terms of the real spectrum compactification of $\Hom(\Gamma,G)$.

However, to keep this article accessible even for non-experts, we give the necessary background on the real spectrum compactification of representation varieties.
For more details for the general setup we refer to \cite{BochnakCosteRoy_RealAlgebraicGeometry, Scheiderer_RAG}, and to \cite{BurgerIozziParreauPozzetti_RSCCharacterVarieties2} in the context of representations and character varieties.

\subsection{Real spectrum of a ring}

In this section we introduce the notion of the real spectrum of a ring.
We mainly follow \cite[Chapters 7.1 and 7.2]{BochnakCosteRoy_RealAlgebraicGeometry}.
Let $A$ be a commutative ring with $1$.
In our examples, $A$ will often be a polynomial ring with coefficients in $\overline{\Qb}^r$.

\begin{definition}
\label{dfn_RealSpectrum}
The \emph{real spectrum} $\rsp(A)$ is the topological space
\begin{align*}
\rsp(A) = \{ (\mathfrak{p},\leq) \mid \mathfrak{p} \subset &A \textrm{ prime ideal, }\leq \textrm{ order on the fraction field } \textnormal{Frac}(A/\mathfrak{p}) \}
\end{align*}
together with the following subbasis of open sets: For $a \in A$ let
\[ \widetilde{\mathcal{B}}(a) \coloneqq \{ (\mathfrak{p}, \leq) \in \rsp(A) \mid \overline{a}^\mathfrak{p} >0\}, \]
where $\overline{a}^\mathfrak{p}$ is the image of $a$ in $\textnormal{Frac}(A/\mathfrak{p})$ under the homomorphism
\[ A \to A/\mathfrak{p} \to \textnormal{Frac}(A/\mathfrak{p}).\]
\end{definition}

Note that this definition differs from the classical spectrum of a ring: firstly because the only ideals for which $\textnormal{Frac}(A/\mathfrak{p})$ admits an order are \emph{real ideals}, i.e.\ such that whenever $a_1^2 + \ldots +a_m^2 \in \mathfrak{p}$ for some $a_1,\ldots, a_m \in A$, then $a_i \in \mathfrak{p}$ for all $i=1,\ldots m$. Secondly, because the field $\textnormal{Frac}(A/\mathfrak{p})$ may have several distinct orders.

If $k$ is a field,  the real spectrum $\rsp(k)$ of $k$ is homeomorphic to the set of orders on $k$ together with the Harrison topology \cite[Example 7.1.4 a)]{BochnakCosteRoy_RealAlgebraicGeometry}.
It is non-empty if and only if $k$ is orderable.

\begin{remark}
There is an equivalent, and for our purposes better, way of viewing points in the real spectrum \cite[3.1.15]{Scheiderer_RAG}.
Namely, if $\Fb$ is a real closed field, and $\varphi \from A \to \Fb$ a ring homomorphism, the kernel of $\varphi$ is a prime ideal in $A$, and the unique order on $\Fb$ restricts to an order on $\textnormal{Frac}(A/\ker(\varphi)) \subset \Fb$.
We say that $\varphi$ \emph{represents} the point $\alpha=(\ker(\varphi),\leq_{\Fb})$, and we write $[\varphi]=\alpha$.
On the other hand, if $\alpha=(\mathfrak{p}, \leq)$ is a point in $\rsp(A)$, then the ring homomorphism
\[A \to A/\mathfrak{p} \to \textnormal{Frac}(A/\mathfrak{p}) \hookrightarrow \overline{(\textnormal{Frac}(A/\mathfrak{p}),\leq)}^r,\]
where $\overline{(\textnormal{Frac}(A/\mathfrak{p}),\leq)}^r$ denotes the real closure of $\textnormal{Frac}(A/\mathfrak{p})$ endowed with the order~$\leq$, represents $\alpha$.

Two ring homomorphisms $\varphi_1 \colon A \to \Fb_1$ and $\varphi_2 \colon A \to \Fb_2$ represent the same point in $\rsp(A)$ if and only if there exists a real closed field $\Fb$ and field homomorphisms $f_i \colon \Fb_i \to \Fb$ such that
\[f_1\circ \varphi_1 = f_2 \circ \varphi_2~.\]
For a proof see e.g.\ \cite[Lemma 3.1.16]{Scheiderer_RAG}.
\end{remark}

\begin{example}
    Let $A = \overline{\Qb}^r[Y_1,\ldots,Y_d]$ be the polynomial ring with coefficients in $\overline{\Qb}^r \subset \Rb$.
    Then 
    \[\rsp(A) = \{ [\textnormal{ev}_{(x_1,\ldots,x_d)}] \mid \Fb \textrm{ a real closed field and }(x_1,\ldots,x_d) \in \Fb^d\},\]
    where $\textnormal{ev}_{(x_1,\ldots,x_d)}$ denotes the evaluation homomorphism from $A$ to $\Fb$.
\end{example}

\begin{theorem}[{\cite[Proposition 7.1.12]{BochnakCosteRoy_RealAlgebraicGeometry}}]
\label{thm_RSCTopolProp}
The topological space $\rsp(A)$ is compact (but not necessarily Hausdorff).
\end{theorem}

In order to define the real spectrum compactification of semi-algebraic sets, we need the following.
\begin{definition}
A subset of $\rsp(A)$ is called \emph{constructible} if it can be obtained as a Boolean combination, i.e.\ finite unions, finite intersections and complements, from the sets $\widetilde{\mathcal{B}}(a)$ defined above.
\end{definition}

\begin{proposition}[{\cite[Proposition 7.1.25 (ii)]{BochnakCosteRoy_RealAlgebraicGeometry}}]
\label{propo_RSpClHausdorff}
Let $C \subset \rsp(A)$ be a constructible subset endowed with the subspace topology.
Then the topological space $C^\cl$ of closed points of $C$ is a compact Hausdorff space.
In particular $\rsp^\cl(A)$, the set of closed points of $\rsp(A)$, is a compact Hausdorff space.
\end{proposition}

\subsection{Real spectrum compactification of semi-algebraic sets}
\label{subs_RSCsemi-algebraicSets}

We now describe the real spectrum compactification of semi-algebraic sets.
Note that these include the algebraic subsets.
For this section we also recommend \cite[Chapter~4]{Scheiderer_RAG}.

For the rest of this section let 
\[A \coloneqq \overline{\Qb}^r[Y_1,\ldots,Y_d]~.\]

\begin{proposition}[{\cite[Proposition 7.1.5]{BochnakCosteRoy_RealAlgebraicGeometry}}]
\label{propo_ImageVInRealSpec}
    The map 
    \[\Psi \from (\overline{\Qb}^r)^d \to \rsp(A), \quad v \mapsto [\textnormal{ev}_v \from A \to \overline{\Qb}^r]\]
    is injective and induces a homeomorphism from $(\overline{\Qb}^r)^d$ with its analytic topology (the subspace topology of the analytic topology on $\Rb^d$), onto its image in $\rsp(A)$.
\end{proposition}

\begin{proposition}[{\cite[Proposition 2.33]{BurgerIozziParreauPozzetti_RSCCharacterVarieties2}}]
\label{propo_ImageVRInRealSpec}
    The map $\Psi$ from \Cref{propo_ImageVInRealSpec} extends to a topological embedding of $\Rb^d$ in $\rsp^\cl(A)$ with open and dense image.
    The space $\rsp^\cl(A)$ is compact, Hausdorff and metrizable.
\end{proposition}

If $X \subset (\overline{\Qb}^r)^d$ is a semi-algebraic set given by a Boolean combination of the basic semi-algebraic sets $\mathcal{B}(f_i)$ for some $f_i \in A$ (\Cref{dfn_SemiAlgSet}), then we define $\widetilde{X}$ to be the constructible subset of $\rsp(A)$ given by the same Boolean combination of the open sets $\widetilde{\mathcal{B}}(f_i)$ (\Cref{dfn_RealSpectrum}).
It turns out that $\widetilde{X}$ is intrinsically defined by the semi-algebraic set $X$ (up to homeomorphism) and does not depend on the ambient space $(\overline{\Qb}^r)^d$ in which $X$ is embedded, see \cite[Corollary 7.2.4, Remark 7.2.5]{BochnakCosteRoy_RealAlgebraicGeometry}.
This is the reason why we are able to extend the compactification to semi-algebraic sets.

\begin{theorem}[{\cite[Theorem 7.2.3, Proposition 7.2.7]{BochnakCosteRoy_RealAlgebraicGeometry}}]
\label{thm_SemiAlgSetsConstructibleSubsets}
    Let $X \subset (\overline{\Qb}^r)^d$ be a semi-algebraic set.
    Then
\begin{enumerate}
    \item 
    \label{item_thm_SemiAlgSetsConstructibleSubsets}
    $X\subset (\overline{\Qb}^r)^d$ is closed (resp.\ open) if and only if $\widetilde{X}\subseteq\rsp(A)$ is closed (resp.\ open).
    \item 
    \label{propo_CharTilde}
    If $X\subset (\overline{\Qb}^r)^d$ is closed (resp.\ open), then $\widetilde{X}$ is the smallest (resp.\ largest) closed (resp.\ open) subset of $\rsp(A)$ whose intersection with $\Psi((\overline{\Qb}^r)^d)$ is $\Psi(X)$. 
\end{enumerate}
\end{theorem}

The following proposition relates the constructible set associated to a semi-algebraic set to its extensions.

\begin{proposition}[{\cite[Proposition 4.1.10]{Scheiderer_RAG}}]
\label{propo: Characterizing closures of closed semi-algebraic sets}
    Let $X \subset (\overline{\Qb}^r)^d$ be a semi-algebraic set.
    If $(\varphi,\Fb)$ represents a point in $\RSp(A)$, then $[(\varphi,\Fb)] \in \widetilde{X}$ if and only if $(\varphi(Y_1),\ldots, \varphi(Y_d)) \in X_{\Fb}$.
\end{proposition}

With this at hand we are now ready to define the real spectrum compactification of closed semi-algebraic sets.
If $X \subset (\overline{\Qb}^r)^d$ is any subset, we denote by $\overline{X}$ the closure of $\Psi(X)$ in $\RSp(A)$.
Then $\overline{X}$ is compact and contains $X$ as a dense subset. The definition of the compactification of $X$ seems to depend on the embedding $X \subset (\overline{\Qb}^r)^d$.
However if $X$ is closed and semi-algebraic, we have the following.

\begin{lemma}\label{lem: closures}
    Let $X \subset (\overline{\Qb}^r)^d$ be a closed semi-algebraic subset.
    Then $\overline{X}=\widetilde{X}$.
\end{lemma}
\begin{proof}
    Since $X\subset (\overline{\Qb}^r)^d$ is a closed semi-algebraic set, the associated constructible set $\widetilde{X} \subset \RSp(A)$ is closed, see \Cref{thm_SemiAlgSetsConstructibleSubsets}~(\ref{item_thm_SemiAlgSetsConstructibleSubsets}).
    Thus $\overline{X}\subset \widetilde{X}$.
    On the other hand, $\widetilde{X}$ is the smallest closed set in $\RSp(A)$ whose intersection with $\Psi((\overline{\Qb}^r)^d)$ is $\Psi(X)$, see \Cref{thm_SemiAlgSetsConstructibleSubsets}~(\ref{propo_CharTilde}).
    Since $X\subset(\overline{\Qb}^r)^d$ is closed we have that $\overline{X}\cap \Psi((\overline{\Qb}^r)^d)=\Psi(X)$, so $\widetilde{X} \subset \overline{X}$, which concludes that $\overline{X}=\widetilde{X}$.
\end{proof}

In particular, if $X \subset (\overline{\Qb}^r)^d$ is closed and semi-algebraic, the real spectrum compactification $\rsp(X) \coloneqq \overline{X}=\widetilde{X}$ is intrinsic to $X$, and does not depend on the embedding $X \subset (\overline{\Qb}^r)^d$.
We also set $\rsp^\cl(X) \coloneqq \widetilde{X}^\cl $ to be the subset of closed points in $\rsp(X)$.
We then have $\rsp^\cl(X)  = \widetilde{X} \cap \rsp^\cl(A)$, since $\widetilde{X}$ is closed by \Cref{thm_SemiAlgSetsConstructibleSubsets}~(\ref{item_thm_SemiAlgSetsConstructibleSubsets}).

Beware that in \cite[\S 2.7]{BurgerIozziParreauPozzetti_RSCCharacterVarieties2}, the authors define the real spectrum compactification of a semi-algebraic set by its associated constructible set, which in the case of closed semi-algebraic sets gives the same definition as we give here.

Even though $\overline{\Qb}^r \neq \Rb$, the space $\rsp(X)$ contains $X_{\Rb}$, see \Cref{propo_ImageVRInRealSpec}.
Thus we obtain the following.

\begin{theorem}[{\cite{BurgerIozziParreauPozzetti_RSCCharacterVarieties2}}]
    Let $X \subset (\Qbar)^d$ be a closed semi-algebraic set.
    Then $\rsp(X)$ is a compact space that contains $X_{\Rb}$ as a dense subset.
    Furthermore, $\rsp^\cl(X)$ is a compact Hausdorff space that contains $X_{\Rb}$ as a dense open subset.
\end{theorem}

\subsection{Real spectrum compactification of spaces of representations}
\label{s: Real Sectrum}

We have seen in \Cref{subs_RSCsemi-algebraicSets} how to define the real spectrum compactification of a closed semi-algebraic set,  which we now would like to apply in the context of character varieties and $\Theta$-positive representations.

Let $\mathbf{G}<\mathbf{\SL}_n$ be a connected semisimple linear algebraic group defined over $\overline{\Qb}^r$ with a $\Theta$-positive structure, see \Cref{subsection: PositiveStructures}.
As in \Cref{section:Positivitysemi-algebraic}, let $\Gamma \subset \Isom_+(\Hb)$ be a lattice.
Since $\Gamma$ is finitely generated, say by a symmetric generating set $\mathcal{S} \subset \Gamma$, 
the set of representations of $\Gamma$ in $\mathbf{G}$, denoted by $\Hom(\Gamma,\mathbf{G})$ has the structure of an affine variety defined over $\overline{\Qb}^r$.

As in \Cref{section:SemiAlgGroups}, let $G$ be a group with $\mathbf{G}(\overline{\Qb}^r)^\circ < G < \mathbf{G}(\overline{\Qb}^r)$, where $\mathbf{G}(\overline{\Qb}^r)^\circ$ denotes the semi-algebraically connected component of $\mathbf{G}(\overline{\Qb}^r)$ containing the identity.
Then $G$ is a semisimple linear semi-algebraic group over $\Qbar$ as in \Cref{dfn: linear semi-algebraic group}.

\subsubsection{The representation variety}
The representation variety $\Hom(\Gamma,G)$ is naturally (up to the choice of $\mathcal{S}$) a subset of $(\overline{\Qb}^r)^d$, where $d = (n^2-1) |\mathcal{S}|$, and has the structure of a closed semi-algebraic subset of $(\overline{\Qb}^r)^d$, which is even algebraic if $G$ is algebraic.

Thus it admits a real spectrum compactification $\rsp\Hom(\Gamma, G)$, which is a compact (non-Hausdorff) space that contains $\Hom(\Gamma, G)$ as a dense (non-open) subset \cite{BochnakCosteRoy_RealAlgebraicGeometry}.
We call $\rsp\Hom(\Gamma, G)$ the \emph{real spectrum compactification of the representation variety}.
The set $\rsp^\cl\Hom(\Gamma, G)$ of closed points in $\rsp\Hom(\Gamma, G)$ is a compact Hausdorff space which contains $\Hom(\Gamma, G)$ as an open and dense subset \cite[Section 1.1]{BurgerIozziParreauPozzetti_RSCCharacterVarieties2}.
Note that $\rsp\Hom(\Gamma, G)$ (and the subset of its closed points) also contains $\Hom(\Gamma,G)_{\Rb}=\Hom(\Gamma,G_{\Rb})$ (\Cref{propo_ImageVRInRealSpec}), thus providing a compactification of the space of representations of $\Gamma$ into $G_{\Rb}$.

It follows from the description of the real spectrum and \Cref{propo: Characterizing closures of closed semi-algebraic sets} that we can describe the points in $\rsp\Hom(\Gamma, G)$ as follows.

\begin{theorem}[{\cite[Proposition 6.3]{BurgerIozziParreauPozzetti_RSCCharacterVarieties2}}]
\label{thm_RSCCharBoundary}
Let $G$ be as above.
Then every $\alpha \in \rsp\Hom(\Gamma, G)$ is represented by a representation $\rho \from \Gamma \to  G_{\Fb}$, where $\Fb$ is a real closed field.
Two representations $\rho_1 \from \Gamma \to G_{\Fb_1}$ and $\rho_2 \from \Gamma \to G_{\Fb_2}$ represent the same point if there exist a real closed field $\Fb$ and field homomorphisms $f_i \from \Fb_{i} \to \Fb$ such that $f_2 \circ \rho_2=f_1 \circ \rho_1$, and we write $\alpha=[(\Fb,\rho)]$.
\end{theorem}

Since $\Gamma$ is finitely generated we can always assume that $\alpha \in \rsp\Hom(\Gamma,G)$ is represented by a representation into $G_{\Fb}$, where the real closed field $\Fb$ is of finite transcendence degree over $\Qbar$, see \cite[\S 6]{Brumfiel_RSCTeichmullerSpace} and \cite[Proposition 6.3]{BurgerIozziParreauPozzetti_RSCCharacterVarieties2}.
Thus $\Fb$ contains a big element (\Cref{def: big element}) \cite[\S 5]{Brumfiel_RSCTeichmullerSpace}, and we get an order-compatible absolute value $\abs{\cdot} \colon \Fb \to \Rb_{\geqslant 0}$ as in \Cref{sec: ordered fields} (which is unique up to positive rescaling).

From this data one defines a metric space $\mathcal{B}_{G_{\Fb}}$ on which $G_{\Fb}$ acts transitively by isometries.
The stabilizers of points are conjugate to the subgroup $G(\mathcal{O})\coloneqq G_{\Fb} \cap \mathcal{O}^{n \times n}$, where $\mathcal{O}=\{ x \in \Fb \mid \abs{x}\leq 1\}$ is the valuation ring of $\Fb$. 
In particular, $\mathcal{B}_{G_{\Fb}} \cong G_{\Fb}/G(\mathcal{O})$ as $G_{\Fb}$-homogeneous sets.
It turns out that $\mathcal{B}_{G_{\Fb}}$ has the structure of a so-called \emph{generalized affine building} \cite{Appenzeller_Buildings}, which one should think of as the non-Archimedean analogue of the symmetric space of $G_{\Rb}$, and a generalization to higher rank of the Bruhat--Tits tree associated to $\SL_2(\Fb)$ \cite{Brumfiel_TreeNonArchimedeanHyperbolicPlane}.
For more details we refer to \cite{kleiner1997rigidity, Parreau_CompEspReprGroupesTypeFini, BurgerIozziParreauPozzetti_RSCCharacterVarieties2, Appenzeller_Buildings}.

We then have the following very useful characterization of closed points in the boundary of $\rsp\Hom(\Gamma,G)$.
Recall that the subset of closed points is a Hausdorff topological space.

\begin{theorem}[{\cite[Proposition 6.4]{BurgerIozziParreauPozzetti_RSCCharacterVarieties2}}]
    Let $[(\Fb, \rho)]$ be a point in\break $\rsp\Hom(\Gamma,G)$ with $\Fb$ real closed, non-Archimedean and of finite transcendence degree over $\Qbar$.
    Then the following are equivalent:
    \begin{itemize}
        \item $(\Fb,\rho)$ represents a closed point;
        \item the $\Gamma$-action on $\mathcal{B}_{G_{\Fb}}$ does not globally fix $\textnormal{Id}\cdot G(\mathcal{O})$.
    \end{itemize}
\end{theorem}

Recall that in \Cref{subsection: ExNonFramablePosRepr} we constructed families of representations $\rho$ from the fundamental group of a three-holed sphere $\Gamma_{0,3}$ to $\PSL_2(\Fb)$ in such a way that $\rho(ab)$ is an infinitesimal rotation, where $a$ and $b$ are the generators of $\Gamma_{0,3}$ such that $a$, $b$, and $b^{-1}a^{-1}$ give three peripheral curves in $\Gamma_{0,3} \backslash \mathbf{H}_{\Rb}$.
Briefly, we choose for each $i=1,2,3$, an $\Fb$-line $\ell_i$ in $\mathbf{H}_{\Fb}$ with endpoints $x_i,y_i$, such that $(x_1,y_3,y_1,x_3,y_2,x_2)$ are cyclically ordered (see  \Cref{fig:PosReprNotFrame}), and $\textnormal{CR}(x_1,y_1,y_3,x_3)-1$ is negative and infinitesimal. Let $\sigma_i$ be the reflection in $\PGL_2(\Fb)$ about $\ell_i$. Then by \Cref{propo:PosReprNotFrame}, the representation $\rho \colon \Gamma_{0,3}\to\PSL_2(\Fb)$ defined by $\rho(a)=\sigma_1\sigma_2$ and $\rho(b)=\sigma_2\sigma_3$ is positive but not frameable.

In particular, they give examples of $\Theta$-positive non-frameable representations.
With the above characterization of closed points, we can now show that we can choose $\rho$ in such a way that it represents a closed point in $\rsp\Hom(\Gamma_{0,3},\PSL_2(\overline{\Qb}^r))$.

\begin{corollary}
\label{corollary: positive non frameable representation is closed}
The representation $\rho \colon \Gamma_{0,3} \to \PSL_2(\Fb)$ constructed in \Cref{subsection: ExNonFramablePosRepr} can be chosen to represent a closed point in $\rsp\Hom(\Gamma_{0,3},\PSL_2(\overline{\Qb}^r))$.
\end{corollary}

\begin{proof}
    Recall that $\abs{x}_{\Fb}\coloneqq \max \{x,-x\} \in \Fb_{\geqslant 0}$ denotes the $\Fb$-valued absolute value of $x$.
    To ensure that $\rho$ corresponds to a closed point, we further choose $\ell_2$ so that $\abs{\textnormal{CR}(x_1,y_1,y_2,x_2)}_{\Fb}$ is a big element. We saw in \Cref{lem: CRBigInf} that 
    \[\abs{\Tr(\sigma_1\sigma_2)}_{\Fb}=\abs{4\textnormal{CR}(x_1,y_1,y_2,x_2)-2}_{\Fb},\]
    so $\abs{\Tr(\sigma_1 \sigma_2)}_{\Fb}$ is a big element in $\Fb$. In particular, it is not in the ideal $\mathcal{O}$. Thus $\rho(a) \notin G(\mathcal{O})$ (since the trace is the sum of the eigenvalues in $\Fb[\sqrt{-1}]$).
    We conclude using the characterization in the above theorem.
\end{proof}

\subsubsection{Characterizing $\Theta$-positive representations}
We are particularly interested in subsets (e.g.\ the set of $\Theta$-positive representations) of $\Hom(\Gamma,G)$ and their closures in the real spectrum compactification.
Recall that if $X \subset \Hom(\Gamma,G)$, we denote by $\overline{X}$ its closure in $\RSp\Hom(\Gamma,G)$.
Then $\overline{X}$ is a compactification of $X$.
In the case that $X$ is closed and semi-algebraic we have the following characterization of points in $\overline{X}$, which is a direct consequence of the definition of the real spectrum.
It is also the crucial observation in the characterization of limits of $\Theta$-positive representations in \Cref{coro:Positive = limit of positive}.

\begin{proposition}
\label{lem: Characterizing closures of closed semi-algebraic sets}
    Let $X \subset \Hom(\Gamma,G)$ be a closed semi-algebraic set.
    If $(\Fb,\rho)$ represents a point in $\RSp\Hom(\Gamma,G)$, then $[(\Fb,\rho)] \in \overline{X}$ if and only if $\rho \in X_{\Fb}$.
\end{proposition}
\begin{proof}
    This is implicitly stated in \cite[Remark 2.36]{BurgerIozziParreauPozzetti_RSCCharacterVarieties2} and follows from \Cref{propo: Characterizing closures of closed semi-algebraic sets} and \Cref{lem: closures}.
\end{proof}

Assume now that $G$ has a $\Theta$-positive structure as in \Cref{def: theta positive structure}.
We denote by $\overline{\Pos}_\Theta(\Gamma,G)$ the closure of $\Pos_\Theta(\Gamma,G)$ inside the compact space $\rsp\Hom(\Gamma,G)$.
We would like to say that the representations representing points in the boundary of the set of (Archimedean) $\Theta$-positive representations are exactly the $\Theta$-positive representations over non-Archimedean fields. This is true for cocompact Fuchsian groups. In contrast, in the non-cocompact case, we constructed in \Cref{subsection: ExNonFramablePosRepr} examples of $\Theta$-positive representations over non-Archimedean fields that are not $\Theta$-frameable, hence cannot are not limits of $\Theta$-positive representations over $\Rb$ (which are frameable by \Cref{coro: positive => frameable over R}). This is summarized in \Cref{coro:Positive = limit of positive}, whose statement we recall here.

\begin{corollary}[\Cref{coro:Positive = limit of positive}]
\label{coro in text: coro:Positive = limit of positive}
    Let $\Gamma \subset \Isom_+(\Hb)$ be a lattice and denote by $\overline{\Pos}_\Theta(\Gamma,G)$ the closure of $\Pos_\Theta(\Gamma,G)$ in $\rsp\Hom(\Gamma,G)$. Then we have the following:
    \begin{enumerate}
        \item If $\Gamma$ is a cocompact lattice, then 
        \[\overline{\Pos}_\Theta(\Gamma, G) = \{[(\Fb,\rho)] \mid \rho\colon \Gamma \to G_{\Fb} \textrm{ $\Theta$-positive}\}~,\]

        \item If $\Gamma$ is a non-cocompact lattice, then
        \[\overline{\Pos}_\Theta(\Gamma, G) = \{[(\Fb,\rho)]\mid \rho \colon \Gamma \to G_{\Fb} \textrm{ $\Theta$-positively frameable}\}~.\]
    \end{enumerate}
\end{corollary}

\begin{proof}
    If $\Gamma$ is cocompact, then  we have seen in \Cref{corol : positive semi algebraic} that $\Pos_\Theta(\Gamma,G_{\Fb})$ is semi-algebraic and that
    \[\Pos_\Theta(\Gamma,G_{\Fb})= (\Pos_\Theta(\Gamma,G))_{\Fb}~.\]
    By \Cref{lem: Characterizing closures of closed semi-algebraic sets}, this proves~(1).

    In the non-cocompact case, by \Cref{coro: positive => frameable over R}, every $\Theta$-positive representation over $\Qbar$ is $\Theta$-positively frameable over $\Rb$. Since $\Theta$-positively frameable representations are semi-algebraically defined over $\Qbar$ by \Cref{corollary : positively framed repr semi-algebraic}, it is also $\Theta$-positively frameable over $\Qbar$ by the transfer principle. Hence we have
    \[\overline{\Pos}_\Theta(\Gamma, G) = \overline{\Pos}^\fr_\Theta(\Gamma, G)~.\]
    By \Cref{corollary : positively framed repr semi-algebraic}, we also have 
    \[\Pos_\Theta^\fr(\Gamma,G)_{\Fb} = \Pos_\Theta^\fr(\Gamma,G_{\Fb})\]
    for every real closed field $\Fb$. We conclude again by \Cref{lem: Characterizing closures of closed semi-algebraic sets} that
    \[\overline{\Pos}^\fr_\Theta(\Gamma, G) = \{[(\Fb,\rho)]\mid \rho \colon \Gamma \to G_{\Fb} \textrm{ $\Theta$-positively frameable}\}~,\]
    proving (2).
\end{proof}

We know that if $\Gamma$ is cocompact, then $\Pos_\Theta(\Gamma,G_{\Rb})$ is both open and closed in $\Hom(\Gamma,G_{\Rb})$ \cite{GLW, BeyrerGuichardLabouriePozzettiWienhard_PositivityCrossRatiosCollarLemma}.
Thus by \Cref{thm_SemiAlgSetsConstructibleSubsets}, $\overline{\Pos}_\Theta(\Gamma,G)$ is a union of  connected components of $\rsp\Hom(\Gamma,G)$.
Combining \Cref{coro:Positive = limit of positive} and \Cref{propo: positive implies injective and discrete} we thus obtain \Cref{coro: positive form connected components}.

\begin{corollary}[\Cref{coro: positive form connected components}]
    If $\Gamma\subset \Isom_+(\Hb)$ is a cocompact lattice, the set of points $[(\Fb,\rho)]$, where $\Fb$ is a real closed field and $\rho\colon \Gamma \to G_{\Fb}$ is $\Theta$-positive, forms a union of connected components of $\rsp \Hom(\Gamma, G)$ consisting only of injective and discrete representations. 
\end{corollary}

In the non-cocompact case, we combine \Cref{coro:Positive = limit of positive},  \Cref{propo: positive implies injective and discrete} and \Cref{thm-intro: Positivity closed}.
\begin{corollary}
\label{rem:connected components of positive representations non uniform lattice}
    Let $\Gamma\subset \Isom_+(\Hb)$ be a non-cocompact lattice and $\rho_0 \colon \Gamma \to G$ a $\Theta$-positive representation.
    Then the set of points $[(\Fb,\rho)]$, where $\Fb$ is a real closed field and $\rho\colon \Gamma \to G_{\Fb}$ is $\Theta$-positively frameable and in $\Hom_{\rho_0}(\Gamma, G_{\Fb})$, forms a union of connected components of $\rsp\Hom_{\rho_0}(\Gamma, G)$ consisting only of injective and discrete representations.
\end{corollary}

Under the $\PG$-condition we can even say more, in which case we can apply \Cref{thm-intro:PG-condition}~(\ref{thm-intro:PG-condition: open in frameable}) (see also \Cref{cor:PG-PosFrameable-open-closed-in-frameable}) instead of \Cref{thm-intro: Positivity closed}.

\begin{corollary}
    Assume $\Fc_{\Theta,\Fb}$ satisfies the $\PG$-condition.
    Let $\Gamma\subset \Isom_+(\Hb)$ be a non-cocompact lattice.
    Then the set of points $[(\Fb,\rho)]$, where $\Fb$ is a real closed field and $\rho\colon \Gamma \to G_{\Fb}$ is $\Theta$-positively frameable forms a union of connected components of $\rsp\Hom^{\fr}_\Theta(\Gamma, G)$ consisting only of injective and discrete representations.
\end{corollary}

\subsubsection{Some words on the character variety}
    We obtain analogous results for the compactification of the character variety, but the exposition is slightly more delicate, and we will not give as many details in this section.

    Using the theory of minimal vectors the \emph{character variety} 
    \[\mathfrak{X}(\Gamma,G) \coloneqq \Hom_\red(\Gamma,G)/G~,\]
    i.e.\ the space of completely reducible representations from $\Gamma$ to $G$ up to $G$-conjuga-tion, is homeomorphic to a closed semi-algebraic subset of some $(\overline{\Qb}^r)^m$ for some $m\in\Nb$, see \cite[Proposition 7.4]{BurgerIozziParreauPozzetti_RSCCharacterVarieties2}.
    It thus also admits a real spectrum compactification, which we denote by $\RSp\mathfrak{X}(\Gamma,G)$.
    The precise study of the identification of $\mathfrak{X}(\Gamma,G)$ with a closed semi-algebraic subset of $(\overline{\Qb}^r)^m$ allows to describe the points of its compactification as follows.
    
    \begin{theorem}[{\cite[Theorem 1.1]{BurgerIozziParreauPozzetti_RSCCharacterVarieties2}}]
    \label{thm_RSCCharBoundary character variety}
    Let $G$ be as above.
    Then every $\alpha \in \rsp\mathfrak{X}(\Gamma,G)$ is represented by a completely reducible representation $\rho \colon \Gamma \to G_{\Fb}$, where $\Fb$ is real closed.
    Two representations $\rho_1 \from \Gamma \to G_{\Fb_1}$ and $\rho_2 \from \Gamma \to G_{\Fb_2}$ represent the same point if there exists a real closed field $\Kb$ and field homomorphisms $f_i \from \Fb_{i} \to \Kb$ such that $f_2 \circ \rho_2$ and $f_1 \circ \rho_1$ are $G_{\Kb}$-conjugate.

\end{theorem}

Their theorem is slightly stronger: it says that we can even choose $\Fb$ to be \emph{minimal}, in the sense that $\rho$ cannot be $G_{\Fb}$-conjugated into a representation $\rho \colon \Gamma \to G_{\Lb}$, where $\Lb \subset \Fb$ is a proper real closed subfield. The boundary points then correspond to those representations, where the minimal field is non-Archimedean.

Again we have the following very useful characterization of closed points in the boundary of $\rsp\mathfrak{X}(\Gamma,G)$, whose description is not as straight-forward as in the case of the representation variety.

\begin{theorem}[{\cite[Theorem 1.2]{BurgerIozziParreauPozzetti_RSCCharacterVarieties2}}]
    Let $(\Fb, \rho)$ represent a point in\break $\rsp\mathfrak{X}(\Gamma,G) \setminus \mathfrak{X}(\Gamma,G)$ with $\Fb$ real closed, minimal and non-Archimedean.
    Then the following are equivalent:
    \begin{itemize}
        \item $(\Fb,\rho)$ represents a closed point;
        \item the $\Gamma$-action on $\mathcal{B}_{G_{\Fb}}$ does not have a global fixed point;
        \item there exists $\gamma \in \Gamma$ for which $\rho(\gamma)$ has positive translation length on $\overline{\mathcal{B}_{G_{\Fb}}}$, the metric completion of $\mathcal{B}_{G_{\Fb}}$.
    \end{itemize}
\end{theorem}

By \cite[Proposition 5.21]{BurgerIozziParreauPozzetti_RSCCharacterVarieties2}, 
the condition on having positive translation length on the completion of $\mathcal{B}_{G_{\Fb}}$ in the above theorem can be reformulated as follows:
Let $\| \cdot \|$ be some Weyl group invariant norm on $\mathfrak{a}_{\Rb}$, $\abs{\cdot} \colon \Fb^\times \to \Rb$ an order-compatible absolute value, and recall the multiplicative Jordan projection $J_{\Fb} \colon G_{\Fb} \to C_{\Fb}$ from \Cref{section: Jordan projection}.
Then the composition of the maps
\[N \colon G_{\Fb} \xrightarrow{J_{\Fb}} C_{\Fb} \xrightarrow{\ln\abs{\cdot}} \mathfrak{a}_{\Rb} \xrightarrow{\|\cdot\|} \Rb~,\]
where we apply $\ln\abs{\cdot}$ component-wise, can be used to calculate the translation length.
Namely, the translation length of $\rho(\gamma)$ is equal to $N(\rho(\gamma))$.
Thus the last condition is equivalent to asking that there exists $\gamma \in \Gamma$ with
\[ N(\rho(\gamma))>0~.\]

Using the above characterization we can show that the representation constructed in \Cref{subsection: ExNonFramablePosRepr} also represents a closed point in $\rsp\mathfrak{X}(\Gamma_{0,3},\PSL_2(\overline{\Qb}^r))$.

\begin{corollary}
    The representation $\rho \colon \Gamma_{0,3} \to \PSL_2(\Fb)$ from \Cref{corollary: positive non frameable representation is closed} also represents a closed point in $\rsp\mathfrak{X}(\Gamma_{0,3},\PSL_2(\overline{\Qb}^r))$.
\end{corollary}
\begin{proof}
    Since the representation $\rho$ is positive, it is irreducible, and in particular completely reducible (\Cref{propo: positive implies irreducible}).
By definition, $\rho(a)=\sigma_1\sigma_2$ and $\vert \Tr(\sigma_1\sigma_2)\vert_{\Fb}$ is a big element in $\Fb$. In particular, $\vert \Tr(\sigma_1\sigma_2)\vert_{\Fb}>2$, so the two eigenvalues of (some representative of) $\rho(a)$ are in $\Fb$, are inverses of each other. Moreover, these two eigenvalues sum to a big element, so one of them must be a big element, which implies by the above remark that 
    \[N(\rho(a)) >0~,\]
    and thus that $\rho$ represents a closed point in $\rsp\mathfrak{X}(\Gamma_{0,3},\PSL_2(\overline{\Qb}^r))$. 
\end{proof}

We now turn to the characterization of the boundary of the subset of equivalence classes of $\Theta$-positive representations.
We know from \Cref{propo: positive implies irreducible} that $\Theta$-positive representations are irreducible, so in particular completely reducible.
Since $\Pos_\Theta(\Gamma,G)$ is closed in $\Hom(\Gamma,G)$, it follows that $\Pos_\Theta(\Gamma,G) \subset \Hom_\red(\Gamma,G)$ is a closed $G$-invariant semi-algebraic subset.
We denote by $\mathfrak{Pos}_\Theta(\Gamma,G)$ the image of $\Pos_\Theta(\Gamma,G)$ in $\mathfrak{X}(\Gamma,G)$, and by $\overline{\mathfrak{Pos}}_\Theta(\Gamma,G)$ its closure in $\rsp\mathfrak{X}(\Gamma,G)$. 
    
With these notations we obtain the following corollary on the level of character varieties, similarly to \Cref{coro in text: coro:Positive = limit of positive}.

\begin{corollary}
    Let $\Gamma \subset \Isom_+(\Hb)$ be a lattice and $\rho$ a representation from $\Gamma$ to $G_{\Fb}$, where $\Fb$ is a real closed field.
    If $\Gamma$ is cocompact then the following are equivalent:
    \begin{enumerate}
        \item $\rho$ is $\Theta$-positive,
        \item $[(\Fb,\rho)]$, seen as a point of $\RSp\mathfrak{X}(\Gamma,G)$, belongs to $\overline{\mathfrak{Pos}}_\Theta(\Gamma,G)$.
    \end{enumerate}
    If $\Gamma$ is non-cocompact then the following are equivalent:
    \begin{enumerate}
        \item $\rho$ is $\Theta$-positively frameable,
        \item $[(\Fb,\rho)]$, seen as a point of $\RSp\mathfrak{X}(\Gamma,G)$, belongs to $\overline{\mathfrak{Pos}}_\Theta(\Gamma,G)$.
    \end{enumerate}
\end{corollary}
\begin{proof}
    By the construction of the semi-algebraic model for $\mathfrak{X}(\Gamma,G)$, for every real closed field $\Fb$ the fibers of the map
    \[\Hom_\red(\Gamma,G)_{\Fb} \to \mathfrak{X}(\Gamma,G)_{\Fb}\]
    coincide with the $G_{\Fb}$-orbits \cite[Proposition 7.4~(2)]{BurgerIozziParreauPozzetti_RSCCharacterVarieties2}.
    Since the subset of $\Theta$-positive representations consists solely of irreducible representations (\Cref{propo: positive implies irreducible}), in particular completely reducible representations, the set $\Pos_\Theta(\Gamma,G)$ is a closed $G$-invariant semi-algebraic subset of $\Hom_{\red}(\Gamma,G)$.
    By the description of the points of $\RSp\mathfrak{X}(\Gamma,G)$ in \Cref{thm_RSCCharBoundary character variety}, the result follows from the statement on the level of representation varieties, see \Cref{coro in text: coro:Positive = limit of positive}.
\end{proof}

\appendix
\section{%
\texorpdfstring{%
Positive representations into $\PSL_2(\Rb)$}%
{Positive representations into PSL(2,R)}}
\label{appendix}

This section proves some technical results about Fuchsian groups, which can be seen either as describing Fuchsian groups modulo the equivalence relation that we call \emph{semi-conjugacy}, or describing positive representations of a Fuchsian group of the first kind into $\PSL_2(\Rb)$. These results might not be surprising to experts and are certainly classical when the Fuchsian group is finitely generated (see for instance \cite{Katok}). However, the case of infinitely generated Fuchsian groups is more subtle and, since we could not find a reference in the literature, it seems useful to include them here.

\subsection{Semi-conjugacy between Fuchsian groups}
Let $\Gamma$, $\Gamma'$ be two non-elemen-tary Fuchsian groups, with $\Gamma'$ of the first kind. Let $\mathcal C_\Gamma \subset \Hb$ denote the the \emph{convex core} of $\Gamma$ (i.e.\ the convex hull of $\Lambda(\Gamma)$ in $\Hb$), and $\mathcal{C}^\circ_\Gamma$ its interior.

The following theorem gives six different characterizations of when an isomorphism between $\Gamma$ and $\Gamma'$ is a \emph{semi-conjugacy}.

\begin{theorem} \label{thm: Equivalences semi-conjugation}
    Let $\iota\colon \Gamma \to \Gamma'$ be an isomorphism. The following conditions are equivalent:
    \begin{itemize}
        \item[(i)] there exists a continuous, $\iota$-equivariant, monotonous surjective map
        \[\alpha\colon \Sb^1 \to \Sb^1~;\]
        \item[(ii)] there exist left and right-continuous, $\iota^{-1}$-equivariant, monotonic injective maps
        \[\beta^l, \beta^r\colon \Sb^1 \to \Lambda(\Gamma)~;\]
        \item[(iii)] The representation $\iota^{-1} \colon \Gamma' \to \PSL_2(\Rb)$ is positive;
        \item[(iv)] there exists a $\iota$-equivariant injective local diffeomorphism from $\mathcal C^\circ_\Gamma$ to $\Hb$;
        \item[(v)] there exists a $\iota$-equivariant diffeomorphism from $\mathcal C^\circ_\Gamma$ to $\Hb$.
        \item[(vi)] there exists a $\iota$-equivariant diffeomorphism from $\Hb$ to $\Hb$;
    \end{itemize}
\end{theorem}
As before, when these conditions are satisfied, we refer to $\iota$ as a \emph{semi-conjugacy} between $\Gamma$ and $\Gamma'$. 

The second theorem that we will prove in this section is the following:
\begin{theorem} \label{thm: Existence semiconjugation}
    Let $\Gamma$ be a non-elementary Fuchsian group. Then there exists a Fuchsian group of the first kind $\Gamma'$ and an isomorphism $\iota \colon \Gamma \to \Gamma'$ such that there exists an injective $\iota$-equivariant \emph{holomorphic} map $f\colon \Hb\to \Hb$.
\end{theorem}

Since the restriction of $f$ to $\mathcal C^\circ_\Gamma$ is an injective local diffeomorphism, by characterization (iv) of semi-conjugacy in \Cref{thm: Equivalences semi-conjugation}, we obtain the following:
\begin{corollary} \label{coro:Fuchsian group conjugated to first kind}
    Every non-elementary Fuchsian group is semi-conjugated to a Fuchsian group of the first kind.
\end{corollary}
Together with the various characterizations of semi-conjugacy in \Cref{thm: Equivalences semi-conjugation}, this corollary implies \Cref{Fuchsian groups lemma}.\\

Let us define more generally a semi-conjugacy between Fuchsian groups (not necessarily of the first kind) in the following way:
\begin{definition}\label{semicon}
    Let $\Gamma_1$ and $\Gamma_2$ be two Fuchsian groups. An isomorphism $\iota\colon  \Gamma_1 \to \Gamma_2$ is a \emph{semi-conjugacy} if there exists a $\iota$-equivariant diffeomorphism from $\Hb$ to $\Hb$.
\end{definition}

With this definition, semi-conjugacy is clearly an equivalence relation.
If $\Gamma_2$ is of the first kind, then this does coincide with the previous notion of semi-conjugacy by characterization (vi) in \Cref{thm: Equivalences semi-conjugation}. 

\begin{remark} \label{rmk: semiconjugation -> conjugation first kind}
    If $\iota\colon \Gamma \to \Gamma'$ is a semi-conjugacy between Fuchsian groups of the first kind, then by minimality of the actions of $\Gamma$ and $\Gamma'$ on $\Sb^1$, the maps $\alpha$ and $\beta^l=\beta^r$ in \Cref{thm: Equivalences semi-conjugation} are both homeomorphisms that are inverses of each other.
\end{remark}

Let us also recall the following theorem of Gabai: 
\begin{theorem}[\cite{Gabai92}]
\label{thm: Gabai}
    Let $\Gamma$ be a group acting on $\mathbb S^1$ as a convergence group by orientation preserving homeomorphisms. Then there exists a Fuchsian group $\Gamma'$ isomorphic to $\Gamma$ such that the actions of $\Gamma$ and $\Gamma'$ on $\Sb^1$ are conjugated by a homeomorphism.
\end{theorem}

Putting these results together, we obtain the following (informal) theorem:

\begin{theorem}
\label{thm: informal theorem semi-conjugacy}
    There are natural bijections between the following classes:
    \begin{enumerate}
        \item non-elementary Fuchsian groups, up to semi-conjugacy,
        \item Fuchsian groups of the first kind, up to semi-conjugacy,
        \item orientation-preserving minimal convergence actions of discrete groups on $\Sb^1$, modulo conjugacy by orientation preserving-homeomorphisms,
        \item $2$-dimensional orbifolds of negative Euler characteristic, up to diffeomorphism.
    \end{enumerate}
\end{theorem}

\begin{proof}
The equivalence between (1) and (2) follows from the fact that every Fuchsian group is semi-conjugated to a Fuchsian group of the first kind (\Cref{coro:Fuchsian group conjugated to first kind}).

The equivalence between (2) and (3) follows from Gabai's \Cref{thm: Gabai}: every orientation-preserving minimal convergence action on $\Sb^1$ is topologically conjugated to the action of a Fuchsian group of the first kind, and two such actions are conjugated if and only if the Fuchsian groups are semi-conjugated, by characterization (i) or (ii) in \Cref{thm: Equivalences semi-conjugation} and \Cref{rmk: semiconjugation -> conjugation first kind}.

The equivalence between (1) and (4) follows from the definition of semi-conjugacy for general Fuchsian groups (see \Cref{semicon}): to a non-elementary Fuchsian group $\Gamma$, one can associate the orbifold $\Gamma \backslash \Hb$ of negative Euler characteristic. If $\iota \colon \Gamma\to \Gamma'$ is a semi-conjugation, there exists a $\iota$-equivariant diffeomorphism $f\colon \Hb\to \Hb$, which factors to an orbifold diffeomorphism
\[\Gamma \backslash \Hb \to \Gamma' \backslash \Hb~.\]
Conversely, any $2$-dimensional orbifold with negative Euler characteristic admits a complete hyperbolic metric and is thus isomorphic to $\Gamma \backslash \Hb$ for some non-elementary Fuchsian group $\Gamma$. Finally, any orbifold diffeomorphism $\bar f \colon \Gamma \backslash \Hb \to \Gamma' \backslash \Hb$ induces an isomorphism $\bar f_*$ between their fundamental groups and lifts to a $\bar f_*$-equivariant diffeomorphism between their universal covers. Hence $\Gamma$ and $\Gamma'$ are semi-conjugated.
\end{proof}

Finally, since positive representations are discrete and faithful by \Cref{propo: positive implies injective and discrete}, the characterization (iii) of semi-conjugacy in \Cref{thm: Equivalences semi-conjugation} gives:

\begin{proposition}
    Let $\Gamma$ be a Fuchsian group and $\rho\colon \Gamma \to \PSL_2(\Rb)$ a representation. Then $\rho$ is positive if and only if $\rho$ is discrete and faithful and $\rho\colon \Gamma \to \rho(\Gamma)$ is a semi-conjugacy.
\end{proposition}

\begin{proof}
    Let $\iota\colon \Gamma \to \Gamma'$ be a semi-conjugacy with a Fuchsian group of the first kind $\Gamma'$. By \Cref{propo: equivalence positivity first and second kind Fuchsian groups}, $\rho$ is positive if and only if $\rho\circ\iota^{-1}$ is positive. Also, observe that $\rho$ is discrete and faithful if and only the same is true for $\rho\circ\iota^{-1}$, and $\rho\colon \Gamma\to\rho(\Gamma)$ is a semi-conjugacy if and only if $\rho\circ\iota^{-1}\colon \Gamma'\to\rho\circ\iota^{-1}(\Gamma')$ is a semi-conjugacy. It thus suffices to prove the proposition assuming that $\Gamma$ is of the first kind.

    Assume that $\Gamma$ is of the first kind. If $\rho$ is discrete and faithful and $\rho\colon \Gamma \to \rho(\Gamma)$ is a semi-conjugacy, then so is $\rho^{-1}\colon  \rho(\Gamma) \to \Gamma\subset\PSL_2(\Rb)$, in which case $\rho$ is positive by \Cref{thm: Equivalences semi-conjugation} (iii). Conversely, if $\rho$ is positive, then \Cref{propo: positive implies injective and discrete} implies that it is discrete and faithful, and \Cref{thm: Equivalences semi-conjugation} (iii) implies that $\rho^{-1}\colon \rho(\Gamma)\to\Gamma$, and hence $\rho\colon\Gamma\to\rho(\Gamma)$, is a semi-conjugacy.
\end{proof}

The rest of this section is devoted to the proof of Theorems~\ref{thm: Equivalences semi-conjugation} and~\ref{thm: Existence semiconjugation}.

\subsection{%
\texorpdfstring{%
Proof of \Cref{thm: Existence semiconjugation}}%
{Proof of Theorem A.2}}
    Recall that $o\in \Hb$ is a fixed basepoint. Consider the set $E$ of discrete and faithful representations $\rho\colon\Gamma \to \Isom_+(\Hb)$ for which there exists a $\rho$-equivariant injective holomorphic map $f\colon  \Hb \to \Hb$ satisfying $f(o) = o$ and $f'(o)\in \R_{>0}$. For any $\rho,\rho'\in\Gamma$, we say that a map $f\colon\Hb\to\Hb$ is \emph{$(\rho,\rho')$-equivariant} if for all $\gamma\in\Gamma$ and $x\in\Hb$, we have $f(\rho(\gamma)\,x)=\rho'(\gamma)\,f(x)$. We write
    \[\rho \preceq \rho'\]
    if there exists a $(\rho,\rho')$-equivariant holomorphic map $f\colon \Hb \to \Hb$ satisfying $f(o) = o$ and $f'(o)\in \R_{>0}$.

    We claim that $(E,\preceq)$ is a partially ordered set satisfying the hypotheses of Zorn's lemma. It is clear that $\preceq$ is transitive and reflexive. Assume $\rho \preceq \rho'$ and $\rho' \preceq \rho$. Then there are holomorphic maps $h\colon \Hb\to\Hb$ and $g\colon \Hb\to\Hb$ that are $(\rho,\rho')$-equivariant and $(\rho',\rho)$-equivariant respectively, such that $h(o)=o=g(o)$ and $h'(o),g'(o)\in\Rb_{>0}$. The map $f= g\circ h\colon \Hb \to \Hb$ also fixes $o$. By the Schwarz lemma, the maps $f$, $g$, and $h$ are either isometries or they contract distances. Since
    \begin{eqnarray*}
        d (f(\rho(\gamma) \cdot o), f(o)) & = & d(\rho(\gamma) \cdot f(o), f(o))\\
        &=& d(\rho(\gamma) \cdot o, o)~,
    \end{eqnarray*}
    $f$ cannot contract distances. Hence $h$, $g$ and $f$ are isometries. Since $h(o) = o$ and $h'(o)\in \R_{>0}$, $h$ is the identity and $\rho = \rho'$. This proves that $\preceq$ is anti-symmetric, and hence is a partial order.

    Let us equip $E$ with the topology of pointwise convergence of representations, which is metrizable since $\Gamma$ is countable. We will now prove that for any $\rho\in E$, the set
    \[E_\rho \coloneqq \{\rho' \in E \mid \rho \preceq \rho'\}\] 
    is compact. Let $(\rho_n)$ be a sequence in $E_\rho$ and $(h_n)$ a sequence of $(\rho,\rho_n)$-equivariant holomorphic maps with $h_n(o)=o$ and $h'_n(o)\in \Rb_{>0}$. By Montel's Theorem on normal families, up to extracting a subsequence, $h_n$ converges uniformly to a holomorphic map $h_\infty\colon\Hb \to \Hb$ such that $h_\infty(o)=o$. Notice that for any $x\in\Hb$ and any $\gamma\in\Gamma$, 
    \begin{align} \tag{$\star$}\label{eqn: Montel}
    \lim_{n\to\infty}\rho_n(\gamma)\, h_\infty(x)=\lim_{n\to\infty}\rho_n(\gamma)\, h_n(x)=\lim_{n\to\infty}h_n(\rho(\gamma)\, x)= h_\infty(\rho(\gamma)\,x).
    \end{align}
    
    It is a standard consequence of Cauchy's argument principle that $h_\infty$ is either constant or injective. Suppose for the sake of contradiction that $h_\infty$ is constant, i.e.\ $h_\infty(x)=o$ for all $x\in\Hb$. Then \eqref{eqn: Montel} implies that $\rho_n(\gamma)\, o\to o$ as $n\to\infty$. Since $\Gamma$ is non-elementary, there is a finite rank, non-Abelian free subgroup $\Gamma'\subset\Gamma$. By the Margulis lemma, $\rho_n(\Gamma')$ is virtually nilpotent for sufficiently large $n$, but this is impossible because $\rho_n$ is injective for all $n$. Thus, $h_\infty$ is injective, and in particular $h_\infty'(o)\in \Rb_{>0}$. Then \eqref{eqn: Montel} implies that $\rho_n$ converges pointwise to some representation $\rho_\infty\colon\Gamma\to\Isom_+(\Hb)$, and that $h_\infty$ is $(\rho,\rho_\infty)$-equivariant. If $f\colon \Hb\to\Hb$ is a $\rho$-equivariant, injective holomorphic map satisfying $f(o)=o$ and $f'(o)\in\Rb_{>0}$, then $f_\infty\coloneqq h_\infty\circ f\colon \Hb\to\Hb$ is a holomorphic, injective, $\rho_\infty$-equivariant map that satisfies $f_\infty(o)=o$ and $f_\infty'(o)\in\Rb_{>0}$, so $\rho_\infty$ lies in $E$, hence in $E_\rho$.

    Now, let $V$ be a totally ordered subset of $E$. For every finite subset $\{\rho_1\preceq \ldots \preceq \rho_n\} \subset V$, we have that
    \[\rho_n \in \bigcap_{i=1}^n E_{\rho_i}~.\]
    By compactness, we deduce that $\bigcap_{\rho \in V} E_\rho$
    is non-empty. Hence $V$ admits an upper bound. By Zorn's lemma, we can thus find a maximal element $\iota\in E$. \\
    
    To conclude the proof of \Cref{thm: Existence semiconjugation}, it now suffices to prove that $\Gamma'\coloneqq \iota(\Gamma)$ is of the first kind. 
    Assume for the sake of contradiction that it is not. Consider $\Hb$ as the upper-half plane in $\mathbf P^1(\Cb)$, in which case $\partial_\infty \Hb$ is identified with $\mathbf P^1(\Rb)\subset\mathbf P^1(\Cb)$. Since $\Gamma'$ acts properly discontinuously on $\mathbf P^1(\Cb) \backslash \Lambda(\Gamma')$, given some point $x\in \mathbf P^1(\Rb) \setminus \Lambda(\Gamma)$, we can find a small open ball $B$ about $x$ such that the domain $D = \Hb \cup \bigcup_{\gamma \in \Gamma} \iota(\gamma) \cdot B$ is simply connected and acted on properly discontinuously by $\iota(\Gamma)$. By the Riemann mapping theorem, there is a biholomorphism $f\colon D\to \Hb$ such that $f(o)=o$ and $f'(o)\in\Rb_{>0}$. Let $\iota'\colon\Gamma\to\Isom_+(\Hb)$ be the representation defined by $\iota'(\gamma)=f\circ \iota(\gamma)\circ f^{-1}$, and observe that $f|_{\Hb}$ is $(\iota,\iota')$-equivariant, is injective but is not surjective. The representation $\iota'$ thus satisfies $\iota \preceq \iota'$ and $\iota \neq \iota'$, which contradicts the maximality of $\iota$. Hence $\Gamma'$ is of the first kind.

\subsection{%
\texorpdfstring{%
Proof of \Cref{thm: Equivalences semi-conjugation}}%
{Proof of Theorem A.1}}

In this section, we prove the equivalence between the six characterizations of a semi-conjugacy. Note that the implications (v) $\Rightarrow$ (iv) and (vi) $\Rightarrow$ (iv) are immediate. Let $\iota\colon \Gamma \to \Gamma'$ be an isomorphism between Fuchsian groups such that $\Gamma'$ is of the first kind.\\

\paragraph{Proof of (i) $\Leftrightarrow$ (ii) $\Leftrightarrow$ (iii)}

This was essentially already proven through our various characterizations of positive representations. Indeed, (ii) $\Rightarrow$ (iii) by definition of positivity and (i) $\Rightarrow$ (iii) by \Cref{prop: Boundary extension implies positive}.
Conversely, if $\iota^{-1}$ is positive, then $\beta^l$ and $\beta^r$ are the boundary maps given by \Cref{lemma: left and right limit independent of the choice of limit map} and $(\Sb^1,\alpha)$ is the maximal boundary extension given by \Cref{prop: converse}.\\

\paragraph{Proof of (iv) $\Rightarrow$ (iii)}

Let $\iota\colon \Gamma\to \Gamma'$ be an isomorphism between Fuchsian groups, with $\Gamma'$ of the first kind, and $f\colon\mathcal{C}_\Gamma^\circ \to \Hb$ a $\iota$-equivariant injective local diffeomorphism. Up to precomposing $f$ and conjugating $\iota$ with an element of $\PGL_2(\Rb) = \Aut_1(\mathfrak{sl}(2,\Rb))$, we can assume that $f$ is orientation preserving.

For every $\varepsilon>0$, set
\[\mathcal{C}_\Gamma^\varepsilon = \{x\in \mathcal{C}_\Gamma^\circ \mid d_{\Hb}(x, \partial \mathcal{C}_{\Gamma}) \geqslant \varepsilon\}~.\] 
This is a $\Gamma$-invariant closed subset of $\mathcal C_\Gamma^\circ$.

Fix a basepoint $o\in C_{\Gamma}^\circ$.
Let $\gamma \in \Gamma$ be a non-peripheral hyperbolic element such that $\iota(\gamma)$ is also hyperbolic. Let $g \colon \Rb_+ \to \mathcal C_\Gamma^\circ$ be any geodesic ray contained in $\mathcal C_\Gamma^\circ$ and converging to $\gamma^+$. Then $g$ is asymptotic to the axis of $\gamma$, which is contained in $\mathcal C_\Gamma^\circ$ since $\gamma$ is non-peripheral, and remains uniformly bounded away from $\partial\mathcal C_\Gamma$ since it is acted on cocompactly by $\langle\gamma\rangle$. Hence $g(\R_+)$ is contained in $\mathcal C_\Gamma^\varepsilon$ for some $\varepsilon >0$.

 We claim that $f(g(t))$ converges to $\iota(\gamma)^+$ as $t\to +\infty$. Indeed, since $\gamma$ is hyperbolic, the sequence $(\gamma^n\cdot o)$ is a quasi-geodesic ray converging to $\gamma_+$, and we can find some $R>0$ and some $\lambda>0$ such that for all $t\in\Rb$,
    \[d_{\Hb}(g(t), \gamma^{\lfloor \lambda t\rfloor}\cdot o)\leq R~.\]
    Since $f$ is continuous, there is some $R'>0$ such that $f(B_R(o) \cap \mathcal C_\Gamma^\varepsilon)\subset B_{R'}(f(o))$. Then the $\iota$-equivariance of $f$ implies that
    \[d_{\Hb}(f(g(t)), \iota(\gamma)^{\lfloor \lambda t \rfloor}\cdot f(o))\leq R'~,\]
    from which we get that
    \[\lim_{t\to +\infty}f(g(t)) = \lim_{n\to +\infty}\iota(\gamma)^n \cdot f(o) = \iota(\gamma)^+~.\]

    Recall that, as a general property of Fuchsian groups, the stabilizer of $\gamma^+$ in $\Gamma$ (resp. of $\iota(\gamma)^+$ in $\Gamma'$) is the centralizer of $\gamma$ (resp. $\iota(\gamma)$). Since $\iota$ is an isomorphism, we get that $\iota(\Stab_\Gamma(\gamma^+))= \Stab_{\Gamma'}(\iota(\gamma)^+)$ and deduce that the map
    \[\alpha \colon \Gamma \cdot \gamma^+  \to \Gamma'\cdot \iota(\gamma)^+\]
    given by $\alpha(\eta\cdot \gamma^+) = \iota(\eta)\cdot \iota(\gamma)^+$
    is well-defined, injective and $\iota$-equivariant. 
    
    We will now argue that $\alpha$ is monotonic. Let $(x,y,z)= (\eta_1\cdot \gamma^+, \eta_2\cdot \gamma^+, \eta_3\cdot \gamma^+) \in (\Gamma \cdot \gamma^+)^3$ be a cyclically oriented triple. Since $\gamma$ is non-peripheral, the ideal geodesic triangle with vertices $x,y,z$ is contained in $\mathcal C_\Gamma^\circ$. By the previous paragraph (applied to $\eta_1\gamma \eta_1^{-1}$, $\eta_2\gamma \eta_2^{-1}$ and $\eta_3 \gamma \eta_3^{-1}$), for any distinct $a,b\in\{x,y,z\}$, $f$ sends the geodesic with ideal endpoints $a$ and $b$ to a proper curve with ideal endpoints $\alpha(a)$ and $\alpha(b)$. Since $f$ is injective, the ideal geodesic triangle with vertices $x,y,z$ is mapped via $f$ to a (not necessarily geodesic) triangle with ideal vertices $\alpha(x), \alpha(y), \alpha(z)$. Since $f$ is orientation preserving, we conclude that $(\alpha(x),\alpha(y),\alpha(z))$ is cyclically ordered.

    Now, the map $\alpha^{-1}\colon \Gamma'\cdot \iota(\gamma)_+ \to \Sb^1$ is a $\iota^{-1}$-equivariant positive map, and we conclude that $\iota^{-1}$ is positive.\\

\paragraph{Proof of (i) $\Rightarrow$ (v)}
To prove this implication, we verify that the \emph{Douady--Earle extension} of a homeomorphism of the circle generalizes to a  monotonous, continuous surjective map which may not be injective. We refer to \cite{DouadyEarle86} for details.\\

For a point $p \in \Sb^1$, we denote by $B_p \colon \Hb \to \R$ the Busemann function based at $p$ vanishing at the basepoint $o$. Given a probability measure $\nu$ on $\Sb^1$, define $b(\nu)\in \Hb$ as the unique minimum of the function
\[x\mapsto \int_{\Sb^1} B_p(x) \mathrm d \nu(p)~.\]
In \cite[Remark (2) p.26]{DouadyEarle86}, Douady and Earle note that $b(\nu)$ is defined if and only if $\nu$ does not have an atom of mass $\geqslant \frac{1}{2}$. Note that this does not depend on the choice of the basepoint $o$.

Given a point $x\in \Hb$, denote by $\nu_x$ its visual (probability) measure on $\Sb^1$. Finally, define
\[f(x) = b(\alpha_*\nu_x)~.\]

Since $\alpha$ is continuous, monotonous and surjective, there is a countable collection $(I_n)_{n\in \Nb}$ of disjoint closed intervals in $\Sb^1$ which are mapped to points by $\alpha$, and the above map $f$ is well-defined on the set $\Omega$ such that for all $x\in \Omega$, $\nu_x(I_n)< \frac{1}{2}$ for all $n$. 
Since the set 
\[H_n = \left\{x \mid \nu_x(I_n) \geqslant \tfrac{1}{2}\right\}\]
is the closed half-plane bounded by $I_n$, we get that
\[\Omega = \Hb \setminus \bigcup_{n\in \Nb} H_n = \mathcal C_\Gamma^\circ~.\]
Hence $f$ is well-defined on the interior of the convex core of $\Gamma$.

Since $\alpha$ is $\iota$-equivariant and the construction is natural, $f$ is $\iota$-equivariant.\break
Douady--Earle's proof that $f$ is a local diffeomorphism carries through. 

Finally, when $x_n$ diverges in $\mathcal C_\Gamma^\circ$, $\alpha_*\nu_{x_n}$ converges (after extraction) to a measure with an atom of mass $\geqslant \frac{1}{2}$, from which we deduce that $f\colon \mathcal C^\circ_\Gamma \to \Hb$ is proper. Hence $f$ is a global diffeomorphism.\\

\paragraph{Proof of (v) $\Rightarrow$ (vi)}

Note that we have proven the equivalence between conditions (i) through (v) and that (vi) $\Rightarrow $ (iv) is immediate. To prove the last implication, we will use the following result, which we extracted from \cite[p.~636, last paragraph]{ALPSS}.

\begin{lemma} \label{lemma: convex core holomorphic}
    Let $\Gamma$ be a non-elementary Fuchsian group. Then there exists a Fuchsian group $\Gamma_1$, an isomorphism $j\colon \Gamma \to \Gamma_1$ and a $j$-equivariant biholomorphism $f\colon \mathcal \Hb \to \mathcal C^\circ_{\Gamma_1}$.
\end{lemma}
\begin{proof}
    Identify $\Hb$ with the upper half-plane in $\mathbf{P}^1(\Cb)$. The group $\Gamma$ acts on the Riemann surface $\mathbf{P}^1(\Cb) \setminus \Lambda(\Gamma)$ by biholomorphisms, hence by isometries for its Poincaré metric $g_\Lambda$. The anti-holomorphic involution $z\mapsto \bar z$ of $\mathbf{P}^1(\Cb) \setminus \Lambda(\Gamma)$ is also an isometry for $g_\Lambda$, hence its fixed locus $\mathbf{P}^1(\Rb)\setminus \Lambda(\Gamma)$ is geodesic for $g_\Lambda$. We get that  $(\overline{\Hb}\setminus \Lambda(\Gamma), g_\Lambda)$ is a complete contractible hyperbolic surface with geodesic boundary, hence there is an isometry $f \colon (\overline{\Hb}\setminus \Lambda(\Gamma), g_\Lambda) \to \mathcal C \subset \Hb$ whose image is a convex subset $\mathcal C$ with geodesic boundary. This map $f$ is tautologically equivariant with respect to a homomorphism $j\colon \Gamma \to \Isom_+(\mathcal C) \subset \Isom_+(\Hb)$. Since $g_\Lambda$ is conformal to the Poincar\'e metric on $\Hb$, the map $f$ is holomorphic. Since $\Gamma$ acts properly on $\overline{\Hb}\setminus \Lambda(\Gamma)$, the representation $j$ is discrete and faithful.

    By Caratheodory's theorem, the map $f$ extends to a homeomorphism $\overline{f} \colon \overline{\Hb} \to \overline{\mathcal C}$, where $\overline{\mathcal C}$ is the closure of $\mathcal C$ in $\overline{\Hb}$. By $j$-equivariance of $\overline{f}$, we have
    \[\overline{f}(\Lambda(\Gamma)) \subset \Lambda(j(\Gamma))~.\]
    On the other hand, we have
    \[\overline{f}(\overline{\Hb} \setminus \Lambda(\Gamma)) \subset \Hb~,\]
    hence
    \[\overline{\mathcal C} \cap \partial_\infty \Hb \subset \overline{f}(\Lambda(\Gamma))~. \]
    We conclude that $\overline{\mathcal C} \cap \partial_\infty \Hb = \Lambda(j(\Gamma))$, hence $\mathcal C$ is the convex core of $\mathcal C_{j(\Gamma)}$ of $j(\Gamma)$.
    Taking $\Gamma_1 = j(\Gamma)$ and restricting $f$ to $\Hb$ gives the conclusion of \Cref{lemma: convex core holomorphic}.     
\end{proof}

Let us now conclude the proof of (v) $\Rightarrow$ (vi). Fix $\Gamma$ a Fuchsian group, $\Gamma'$ a Fuchsian group of the first kind, $\iota\colon\Gamma \to \Gamma'$ and $h\colon \mathcal C^\circ_\Gamma \to \Hb$ a $\iota$-equivariant diffeomorphism. Let $\Gamma_1$, $j$ and $f$ be as in \Cref{lemma: convex core holomorphic}. By \Cref{thm: Existence semiconjugation}, we can find a Fuchsian group of the first kind $\Gamma_1'$, an isomorphism $\iota_1\colon\Gamma_1 \to \Gamma_1'$ satisfying condition (iv), hence condition (v). Let $h_1\colon \mathcal C^\circ_{\Gamma_1} \to \Hb$ be a $\iota_1$-equivariant bijection. Then $h_1\circ f\colon\Hb \to \Hb$ is a $\iota_1 \circ j$-equivariant diffeomorphism, showing that $\iota_1 \circ j$ satisfies condition (vi).

Moreover, $h_1 \circ f \circ h^{-1}$ is a $\iota_1 \circ j \circ \iota^{-1}$ injective local diffeomorphism defined on $\mathcal C^\circ_{\Gamma'}= \Hb$. We deduce that $\iota_1 \circ j \circ \iota^{-1}$ satisfies condition (iv). Since (iv)$\Rightarrow$(v), there exists a $\iota_1 \circ j \circ \iota^{-1}$-equivariant diffeomorphism
$g\colon \mathcal C^\circ_{\Gamma'}=\Hb \to \Hb$. The composition $g^{-1} \circ h_1 \circ f\colon\Hb \to \Hb$ is the global $\iota$-equivariant diffeomorphism that we were seeking.

\bibliographystyle{alpha}
\bibliography{reference.bib}

\end{document}